\documentclass[11pt,letterpaper, reqno]{amsart}
\usepackage{amssymb, amsfonts, stmaryrd, color,fancyhdr,amsmath,amsthm,amsbsy,amsfonts, titletoc,  mathrsfs, mathabx, verbatim,extarrows, wasysym, euscript, enumerate, latexsym, mathtools, epsfig, subfigure, mathabx, graphicx, color, wrapfig, hyperref,graphicx,float, bm}

\usepackage[margin=1in]{geometry}

\numberwithin{equation}{section}
\theoremstyle{plain}
\DeclareMathAlphabet{\mathpzc}{OT1}{pzc}{m}{it}

\newcommand{\h}{\hspace}
\newcommand{\normmm}[1]{{\left\vert\kern-0.25ex\left\vert\kern-0.25ex\left\vert #1 
    \right\vert\kern-0.25ex\right\vert\kern-0.25ex\right\vert}}

\newcommand{\D}{\hspace{1pt} \mathrm{d}}

\renewcommand{\div}{\operatorname{div}}
\newcommand{\cE}{\mathcal{E}}

\newcommand{\vH}{\mathbf{H}}
\newcommand{\vn}{\mathbf{n}}

\newtheorem{defn}{Definition}[section]
\newtheorem{thm}[defn]{Theorem}

\newtheorem{cor}[defn]{Corollary}
\newtheorem{prop}[defn]{Proposition}
\newtheorem{lem}[defn]{Lemma}

\title{Phase Transition with Rapini-Papoular Surface Anchoring}

\author{Shun Li\ \ \ }
\thanks{S. Li would like to express his deepest gratitude to his postgraduate supervisor Y. Yu for the patient guidance of this research work and for having involved him in the project funded by the Hong Kong RGC grant No. 14310925.}
\address{Department of Mathematics\\ The Chinese University of Hong Kong\\ Hong Kong }
\email{shunli@cuhk.edu.hk}

\author{\ \ Yong Yu}
\thanks{Y. Yu is partially supported by Hong Kong RGC grants No. 14306622 and 14310925.}
\address{Department of Mathematics\\ The Chinese University of Hong Kong\\ Hong Kong }
\email{yongyu@cuhk.edu.hk}

\date{}
\address{}
\email{}

\keywords{P-HAN transition, Simplified Ericksen-Leslie system, Suitable weak solution}
\subjclass[2020]{82D30, 35K10, 35K58, 35K61}

\begin{document}

\begin{abstract} 
    We analyze the dynamical (in)stability of nematic liquid crystals in the presence of external magnetic fields and Rapini-Papoular surface potential. The P-HAN transition is investigated using a simplified 3D Ericksen-Leslie system. We find the thickness threshold of the P-HAN transition. If the thickness of the nematic layer exceeds this threshold, there is a global-in-time suitable weak solution converging exponentially to a nontrivial equilibrium state as time tends to infinity. If the thickness is no more than the threshold, the global-in-time suitable weak solution has a trivial long-time asymptotic limit. Our results rigorously justify the P-HAN transition discussed in the physics literature.
\end{abstract}

\maketitle
\thispagestyle{empty}
\tableofcontents

\section{Introduction}
When a nematic liquid crystal cell is equipped with a homeotropic boundary condition at one of the substrates and a unidirectional planar (P) boundary condition at the other, it is possible to obtain a hybrid aligned nematic (HAN) cell if the cell thickness exceeds some threshold. The transition from the P cell to the HAN cell is called the P-HAN transition in the physics literature. See \cite{BarberoBarberi1983}, \cite{SparavignaKomitovLavrentovichStrigazzi1992} and the references therein. Generally speaking, if the thickness is no more than the threshold, an undeformed planar alignment is expected. If the thickness exceeds the threshold, the HAN cell is preferred. 

\subsection{Hydrodynamical flow of director angle}

Motivated by physics literature, the bulk domain is given by $\Omega := \mathbb T^2 \times (0, d\h{0.5pt})$, where $d > 0$ is the thickness of the liquid crystal cell. Specifically, we assume the liquid crystal material is periodic in the variables $x_1$ and $x_2$ with a wavelength of 1 along both directions. The substrates H and P are put at $\big\{ x_3 = 0\big\}$ and $\big\{ x_3 = d \h{0.5pt}\big\}$, respectively. 

In 1995, Lin-Liu \cite{LinLiu1995} introduced a simplified Ericksen-Leslie system for the director fields of nematic liquid crystals. To describe thin nematic films, the system is extended in \cite{lcakt2013} by Lin-Cummings-Archer-Kondic-Thiele to include a free boundary. Based on these arguments, we investigate the following simplified Ericksen-Leslie system with the external magnetic field $\vH_*$:
\begin{equation} \label{EL} \left\{ \begin{aligned}
     \partial_t u +  u \cdot \nabla  u - \Delta u &= - \nabla p - \nabla \cdot \big(\nabla \vn \odot \nabla \vn\big), \\[1mm]
    \operatorname{div} u &= 0, \\[1mm]
    \partial_t \vn + u \cdot \nabla  \vn - \Delta \vn &= |\h{0.5pt}\nabla \vn\h{0.5pt}|^2 \vn + \left(\vn \cdot \vH_* \right) \vH_* - \left(\vn \cdot \vH_*\right)^2 \vn.
\end{aligned} \right. \end{equation}
Here, $u$ denotes the velocity field of the fluid. $p$ is the pressure induced from the incompressibility condition of $u$. $\vn$ is the $\mathbb S^2$-valued director field. The dot product is the standard inner product on $\mathbb R^3$. In the first equation of \eqref{EL}, $\nabla \vn \odot \nabla \vn$ is the stress tensor with its entries given by
\begin{equation*}
    \big(\nabla \vn \odot \nabla \vn\big)_{ij} :=  \partial_i \vn \cdot \partial_j \vn, \h{20pt}\text{where $i, j = 1, 2, 3$.}
\end{equation*}  

We supply the unknowns $(u, \vn)$ in \eqref{EL} with suitable boundary conditions.  The velocity $u$ is imposed with the no-slip boundary condition: \begin{equation}\label{bdry of u}
u = 0 \h{20pt}\text{on $\text{H} \cup \text{P}$.}\end{equation}  
The director field $\vn$ is supposed to satisfy the unidirectional planar boundary condition: \begin{equation}\label{dry ofn} \vn = e_1 := (1, 0, 0)^\bot\h{20pt}\text{ on P.} \end{equation}
Moreover, it satisfies the Rapini-Papoular weak anchoring condition:
\begin{equation}\label{RP for n}   
    \partial_3 \vn = - L_\text{H} \h{0.8pt} \vn_3 \left(e_3 - \vn_3 \vn\right) \h{20pt} \text{on H}.\end{equation}
Here, $e_3 := (0, 0, 1)^\bot$. $\vn_j$ denotes the $j$-th component of $\vn$. $L_\text{H}$ is a positive constant describing the strength of the weak anchoring on H. 

Concerning the third equation in \eqref{EL} and the Rapini-Papoular boundary condition on H in \eqref{RP for n}, we introduce the following total free energy for the director field $\vn$:
\begin{equation} \label{OS}
    \frac{1}{2} \int_{\Omega} |\nabla \vn|^2 + |\vH_*|^2 - (\vn \cdot \vH_*)^2 - \frac{L_\text{H}}{2} \int_{\text{H}} (\vn \cdot \nu)^2.
\end{equation}In this energy, we use the one-constant approximation of the Oseen-Frank energy to measure the elastic energy. $\nu$ is the outer normal direction. The negative sign in front of the last integral in \eqref{OS} indicates that $\nu$ is the easy axis of the director field $\vn$ on H.

We simply put $\vH_*$ and $\vn$ under the ansatz:  \begin{equation}\label{ansatz}\vH_* = h \h{0.5pt}e_3 \h{20pt}\text{and}\h{20pt}\vn = \cos \phi \h{1.2pt} e_1 + \sin \phi \h{1.2pt}e_3.\end{equation} The positive constant $h$ describes the strength of the external magnetic field. The function $\phi$ is called the director angle of $\vn$. With \eqref{ansatz}, the system  \eqref{EL} can then be rewritten by 
\begin{equation} \left\{ \label{sEL} \begin{aligned}
    \partial_t u + u \cdot \nabla u - \Delta u &= - \nabla p -\nabla \cdot \big(\nabla \phi \odot \nabla \phi \big), \\[1mm]
    \operatorname{div} u &= 0, \\[1mm]
    \partial_t \phi + u \cdot \nabla \phi - \Delta \phi &=  h^2  \sin \phi \cos \phi.
\end{aligned} \right.  \end{equation}
The stress tensor $\nabla \phi \odot \nabla \phi$ is defined in terms of its entries by $$\big(\nabla \phi \odot \nabla \phi \big)_{ij} := \partial_i \phi \h{2pt} \partial_j \phi, \h{15pt}\text{where $i, j = 1, 2, 3$.}$$ The boundary conditions of $\vn$ in \eqref{dry ofn}-\eqref{RP for n} can be further rephrased as follows:
\begin{equation} \label{bc} \left\{ \begin{aligned}
    \phi &= 0 \quad\hspace{71pt} \text{on P}, \\[.5mm]
    \partial_\nu \phi &= L_{\text{H}} \sin \phi \cos \phi \quad \h{14pt} \text{on H}.
\end{aligned} \right. \end{equation}

So far, we have introduced our hydrodynamic system \eqref{sEL}. The boundary conditions of $(u, \phi)$ are given in \eqref{bdry of u} and \eqref{bc}. We are now in a position to discuss the initial conditions of the system \eqref{sEL}. Recalling \eqref{bdry of u} and the incompressibility condition in \eqref{sEL}, we define $H_{0, \mathrm{div}}^1(\Omega)$ to be the subspace of $H^1(\Omega; \mathbb R^3)$ in which all vector fields are divergence-free and equal to $0$ on the substrates $\text{H}$ and $ \text{P}$ in the sense of trace. In light of the first condition in \eqref{bc}, we define $H_\text{P}^1(\Omega)$ to be the subspace of $H^1(\Omega)$ in which all functions are equal to $0$ on P in the sense of trace. With these functional spaces, we set
\begin{equation}\label{ini cond}
    u  = u_0 \in H_{0, \mathrm{div}}^1(\Omega) \hspace{10pt} \text{ and } \hspace{10pt} \phi  = \phi_0 \in H_{\text{P}}^1(\Omega) \cap H^2(\Omega) \hspace{20pt} \text{at } t = 0.
\end{equation}
In the remainder of the article, the initial boundary value problem \eqref{sEL}, \eqref{bdry of u}, \eqref{bc}, and \eqref{ini cond} is referred to as IBVP. The P-HAN transition will be justified based on the solutions to this initial-boundary-value problem.  

\subsection{Suitable weak solutions} Fix a time $T > 0$ and suppose $(u, \phi)$ is a smooth solution to IBVP on $(0, T\h{0.5pt}]$. To derive a local energy identity of $(u, \phi)$, we choose a smooth test function on $\overline{\Omega} \times [0, T]$ vanishing near the initial time $t = 0$. This test function is denoted by $\varphi$ in the following arguments.

First, we take the inner product with $\varphi \h{0.5pt}u$ on both sides of the first equation in \eqref{sEL} and then integrate over $\Omega$. By using the no-slip boundary condition \eqref{bdry of u}, it turns out \begin{align}\label{energy of u}
\frac{\mathrm d}{\mathrm d \h{0.5pt}t} \int_\Omega \varphi \h{0.5pt} |u|^2 &+ 2 \int_\Omega \varphi\h{0.5pt} | \nabla u |^2 - 2 \int_\Omega \varphi \h{1.5pt} \nabla u : \big( \nabla \phi \odot \nabla \phi \big) \\[1.5mm]
&= \int_\Omega \left( u \cdot \nabla \varphi\right) \left( 2 p +  |u|^2 \right) + 2\int_\Omega \left(u \cdot \nabla \phi\right) \nabla \phi \cdot \nabla \varphi + \int_\Omega  |u|^2  \left( \partial_t \varphi + \Delta \varphi\right),   \nonumber
\end{align}where if we denote by $u^j$ the $j$-th component of $u$, then $$\nabla u : \big( \nabla \phi \odot \nabla \phi \big) = \partial_i u^j \h{1.5pt}\partial_i \phi \h{1.5pt} \partial_j \phi.$$   

Next, we act $\partial_j$ on the third equation in \eqref{sEL}. Multiplying $\varphi\h{1pt}\partial_j \phi$ on both sides of the resulting equation and integrating over $\Omega$, we obtain \begin{align}\label{phi energy}
    \frac{\mathrm d}{\mathrm d \h{0.5pt}t} \int_\Omega \varphi \h{0.5pt} |\nabla \phi|^2 &+ 2 \int_\Omega \varphi\h{0.5pt} | \nabla^2 \phi |^2 + 2 \int_\Omega \varphi \h{1.5pt} \nabla u : \big( \nabla \phi \odot \nabla \phi \big)  \\[1.5mm]
    &\h{7pt} =  \int_\Omega \nabla \cdot \left( \varphi \h{1.5pt}\nabla | \nabla \phi |^2 \right) - \int_\Omega \nabla \varphi \cdot \nabla | \nabla \phi |^2 \nonumber\\[1.5mm]
    &\h{7pt} + \int_\Omega \left( u \cdot \nabla \varphi\right) |\nabla \phi|^2 + 2 h^2 \int_\Omega \varphi \h{1.5pt}|\nabla \phi |^2 \h{1.5pt}\cos 2 \phi + \int_\Omega | \nabla \phi |^2 \h{1pt}\partial_t \varphi. \nonumber
\end{align}Here, we also sum over the index $j$ and use the no-slip boundary condition \eqref{bdry of u}. Applying the integration by parts induces \begin{align*}
    \int_\Omega \nabla \cdot \left( \varphi \h{1.5pt}\nabla | \nabla \phi |^2 \right) &- \int_\Omega \nabla \varphi \cdot \nabla | \nabla \phi |^2 \\[1.5mm]
    &= \int_\text{P}  \varphi \h{1.5pt}\partial_3 | \nabla \phi |^2   -   | \nabla \phi |^2 \h{1pt} \partial_3 \varphi - \int_\text{H}  \varphi \h{1.5pt}\partial_3 | \nabla \phi |^2   -   | \nabla \phi |^2 \h{1pt} \partial_3 \varphi + \int_\Omega | \nabla \phi |^2 \h{1pt}\Delta \varphi.   
\end{align*}Since $\phi \equiv 0$ on P, then $$\partial_t \phi = \partial_\tau \phi = \partial_{\tau\tau} \phi = 0 \h{20pt}\text{ on P, where $\tau = 1, 2$.}$$ By  \eqref{bdry of u} and the third equation in \eqref{sEL}, it holds $\partial_{33}\phi = 0$ on P. Therefore, \begin{equation*}
    \int_\text{P}  \varphi \h{1.5pt}\partial_3 | \nabla \phi |^2   -   | \nabla \phi |^2 \h{1pt} \partial_3 \varphi = - \int_\text{P}   \left(\partial_3  \phi \right)^2 \h{1pt} \partial_3 \varphi.
\end{equation*}By the second condition in \eqref{bc}, it turns out $$\partial_3 | \nabla \phi |^2 =  - 2 L_\text{H} \h{1pt} |\nabla' \phi|^2 \h{1pt}\cos 2 \phi  - L_\text{H} \h{1pt} \partial_{33}\phi \h{1pt}\sin 2\phi \h{20pt}\text{ on H, where $\nabla' = (\partial_1, \partial_2)$.}$$ We then get \begin{align*}
    - \int_\text{H}  \varphi \h{1.5pt}\partial_3 | \nabla \phi |^2   &-   | \nabla \phi |^2 \h{1pt} \partial_3 \varphi = L_{\mathrm{H}}\int_{\mathrm{H}}   \varphi \h{1pt} \partial_{33}\phi \sin 2\phi \\[1mm]
    & + 2 L_\text{H} \int_\text{H}
   \varphi \h{1pt}  |\nabla' \phi|^2 \cos 2 \phi   + \int_\text{H} | \nabla' \phi |^2 \h{1pt} \partial_3 \varphi + \frac{L_\text{H}^2}{4} \h{1pt}\partial_3 \varphi \h{1pt} \sin^2 2 \phi. \end{align*}Define \begin{align}\label{defn of R} R\left(\phi, \varphi\right)  :=    - \int_\text{P}   \left(\partial_3  \phi \right)^2 \h{1pt} \partial_3 \varphi +  \frac{L_\text{H}^2}{4} \int_{\mathrm H} \partial_3 \varphi \h{1pt} \sin^2 2 \phi   + \int_\text{H} | \nabla' \phi |^2 \h{1pt} \partial_3 \varphi + 2L_\text{H} \int_\text{H}
   \varphi \h{1pt}  |\nabla' \phi|^2 \cos 2 \phi. 
 \end{align} 
 The above calculations reduce \eqref{phi energy} to \begin{align*}
      \frac{\mathrm d}{\mathrm d \h{0.5pt}t} \int_\Omega \varphi \h{0.5pt} |\nabla \phi|^2 &+ 2 \int_\Omega \varphi\h{0.5pt} | \nabla^2 \phi |^2 + 2 \int_\Omega \varphi \h{1.5pt} \nabla u : \big( \nabla \phi \odot \nabla \phi \big) =   \int_\Omega \left( u \cdot \nabla \varphi\right) |\nabla \phi|^2 \\[2mm]
    &   + 2 h^2 \int_\Omega \varphi \h{1.5pt}|\nabla \phi |^2 \h{1.5pt}\cos 2 \phi + \int_\Omega | \nabla \phi |^2 \left(\partial_t \varphi + \Delta \varphi\right) + L_\text{H} \int_\text{H} \varphi \h{1pt} \partial_{33}\phi \sin 2\phi + R\left(\phi, \varphi\right).
 \end{align*} Summing this equation with \eqref{energy of u} and integrating the resulting equation from $0$ to $T$, we obtain \begin{align*}
      &\int_{\Omega \times \{\h{0.5pt}T\h{0.5pt}\}} \varphi \left( \h{1pt} |u|^2 + |\nabla \phi|^2\right)  + 2 \int_0^T\int_\Omega \varphi\left(\h{1pt} | \nabla u |^2 + |\nabla^2 \phi |^2 \right)  \nonumber\\[2mm]
 &  = \int_0^T\int_\Omega \left( u \cdot \nabla \varphi\right) \left( 2 p +  |u|^2 + |\nabla \phi|^2\right) + 2 \int_0^T\int_\Omega \left(u \cdot \nabla \phi\right) \nabla \phi \cdot \nabla \varphi + 2 h^2 \int_0^T\int_\Omega \varphi \h{1.5pt}|\nabla \phi |^2 \h{1.5pt}\cos 2 \phi \\[2mm]
&   + \int_0^T\int_\Omega  \left(\h{1pt}|u|^2 + |\nabla \phi |^2 \right)  \left( \partial_t \varphi + \Delta \varphi\right) + L_\text{H} \int_0^T\int_\text{H} \varphi \h{1pt} \partial_{33}\phi \sin 2\phi + \int_0^T R\left(\phi, \varphi\right).
 \end{align*} 
We now apply the equation of $\phi$ and no-slip boundary condition of $u$ to get \begin{align*}
    \int_0^T\int_\text{H} \varphi \h{1pt} \partial_{33}\phi \sin 2\phi = \int_0^T\int_\text{H} \varphi \h{1pt} \left( \partial_t \phi - \frac{h^2}{2} \sin 2 \phi - \Delta' \phi \right)\sin 2\phi.
\end{align*}Through integration by parts, we note that  
    \begin{align*}
    \int_0^{T} \int_\mathrm{H} \varphi \left(\partial_t \phi\right) \left(\sin 2 \phi\right) =    \int_{\mathrm{H} \times \{ T\}} \varphi  \sin^2 \phi -  \int_0^{T} \int_{\mathrm{H}}  \partial_t \varphi \sin^2 \phi  
\end{align*} and 
    \begin{align*}
    - \int_\mathrm{H} \left(\Delta' \phi\right) \left(\sin 2 \phi\right) \varphi   =   \int_{\mathrm{H}}   2 \left(\cos 2\phi\right) \left| \nabla' \phi\right|^2 \varphi  +   \left(\sin 2\phi\right)  \nabla' \phi \cdot \nabla' \varphi.
\end{align*} Therefore, \begin{align*}
    \int_0^T\int_\text{H} \varphi \h{1pt} \partial_{33}\phi \sin 2\phi  &=    \int_{\mathrm{H} \times \{ T\}} \varphi \sin^2 \phi -   \int_0^{T} \int_{\mathrm{H}} \partial_t \varphi \sin^2 \phi  \\[2mm]
    & + \int_0^T\int_{\mathrm{H}}   2 \left(\cos 2\phi\right) \left| \nabla' \phi\right|^2 \varphi  +   \left(\sin 2\phi\right)  \nabla' \phi \cdot \nabla' \varphi - \frac{h^2}{2} \int_0^T\int_{\mathrm H} \varphi \left(\sin 2 \phi \right)^2.
\end{align*}
Eventually, we arrive at our local energy identity:  
    \begin{align*}
      &\int_{\Omega \times \{\h{0.5pt}T\h{0.5pt}\}} \varphi \left( \h{1pt} |u|^2 + |\nabla \phi|^2\right)  + 2 \int_0^T\int_\Omega \varphi\left(\h{1pt} | \nabla u |^2 + |\nabla^2 \phi |^2 \right) + \frac{h^2 L_{\mathrm H}}{2} \int_0^T\int_{\mathrm H} \varphi \left(\sin 2 \phi \right)^2 \nonumber\\[2mm]
 &\h{20pt}  = \int_0^T\int_\Omega \left( u \cdot \nabla \varphi\right) \left( 2 p +  |u|^2 + |\nabla \phi|^2\right) + 2 \left(u \cdot \nabla \phi\right) \nabla \phi \cdot \nabla \varphi +  2 h^2   \varphi \h{1.5pt}|\nabla \phi |^2 \h{1.5pt}\cos 2 \phi \\[2mm]
& \h{20pt}   + \int_0^T\int_\Omega  \left(\h{1pt}|u|^2 + |\nabla \phi |^2 \right)  \left( \partial_t \varphi + \Delta \varphi\right)   + L_{\mathrm{H}} \int_{\mathrm{H} \times \{ T\}} \varphi \sin^2 \phi\\[2mm]
& \h{20pt}   + L_{\mathrm H}\int_0^T\int_{\mathrm{H}}   2 \left(\cos 2\phi\right) \left| \nabla' \phi\right|^2 \varphi  +   \left(\sin 2\phi\right)  \nabla' \phi \cdot \nabla' \varphi -    \partial_t \varphi \sin^2 \phi + \int_0^T R\left(\phi, \varphi\right).
 \end{align*}

 Generally, the above energy identity cannot be satisfied by weak solutions of IBVP. Similar to the work of Caffarelli-Kohn-Nirenberg \cite{CaffarelliKohnNirenberg1982} for the 3D Navier-Stokes equation, we introduce the global suitable weak solutions of IBVP as follows: 
 \begin{defn}\label{suitable wk sol}
The pair $(u, \phi)$ is a global suitable weak solution of $\mathrm{IBVP}$ if the followings hold:
\begin{itemize}
\item[$\mathrm{(1).}$] $(u, \phi)$ satisfies the integrability condition: \begin{align}\label{infty222} 
\sup_{t \h{0.5pt} \geq \h{0.5pt} 0} \int_{\Omega \times \{\h{0.5pt}t\h{0.5pt}\}}  |u|^2 + |\nabla \phi|^2 +   \int_0^\infty\int_\Omega  | \nabla u |^2 +   \Big|\h{1pt}  \Delta \phi  + \frac{h^2}{2} \sin 2 \phi \h{1pt} \Big|^2 < \infty.
\end{align} \vspace{0.4pc}

\item[$\mathrm{(2).}$] $(u, \phi)$ solves the $\mathrm{IBVP}$ weakly in $\Omega \times (0, \infty)$. \vspace{0.4pc}

\item[$\mathrm{(3).}$] For any $T > 0$, the angle $\phi$ satisfies the following energy equality:\begin{align}\label{energy eqa phi}
\int_{\Omega \times \{\h{0.5pt}T\h{0.5pt}\}} \phi^2 + 2 \int_0^T \int_\Omega | \nabla \phi |^2 = \int_\Omega \phi_0^2 + h^2 \int_0^T \int_\Omega \phi \sin 2 \phi + L_\mathrm{H} \int_0^T \int_{\mathrm{H}} \phi \sin 2 \phi.
\end{align}

\item[$\mathrm{(4).}$] For any $T > 0$ and any non-negative $\varphi \in C^\infty \big( \overline{\Omega} \times [\h{0.5pt}0, T\h{0.5pt}] \big)$ vanishing near $t=0$, we have
\begin{align} \label{energy ineq} 
      &\int_{\Omega \times \{\h{0.5pt}T\h{0.5pt}\}} \varphi \left( \h{1pt} |u|^2 + |\nabla \phi|^2\right)  + 2 \int_0^T\int_\Omega \varphi\left(\h{1pt} | \nabla u |^2 + |\nabla^2 \phi |^2 \right) + \frac{h^2 L_{\mathrm H}}{2} \int_0^T\int_{\mathrm H} \varphi \left(\sin 2 \phi \right)^2  \nonumber\\[2mm]
 & \h{20pt}  \leq \int_0^T\int_\Omega \left( u \cdot \nabla \varphi\right) \left( 2 p +  |u|^2 + |\nabla \phi|^2\right) + 2   \left(u \cdot \nabla \phi\right) \nabla \phi \cdot \nabla \varphi +  2 h^2   \varphi \h{1.5pt}|\nabla \phi |^2 \h{1.5pt}\cos 2 \phi \nonumber\\[2mm]
&\h{20pt}   + \int_0^T\int_\Omega  \left(\h{1pt}|u|^2 + |\nabla \phi |^2 \right)  \left( \partial_t \varphi + \Delta \varphi\right)   + L_{\mathrm{H}} \int_{\mathrm{H} \times \{ T\}} \varphi \sin^2 \phi\\[2mm]
&\h{20pt}   + L_{\mathrm H}\int_0^T\int_{\mathrm{H}}   2 \left(\cos 2\phi\right) \left| \nabla' \phi\right|^2 \varphi  +   \left(\sin 2\phi\right)  \nabla' \phi \cdot \nabla' \varphi -    \partial_t \varphi \sin^2 \phi + \int_0^T R\left(\phi, \varphi\right). \nonumber
\end{align}
\end{itemize}

\end{defn}
\noindent \eqref{energy ineq} is referred to as the generalized energy inequality of IBVP.
\subsection{Main results and organization of the article} The pair $(0, \phi)$ is an equilibrium solution of IBVP if $\phi$ solves the  boundary value problem: 
\begin{equation} \label{SG} \left\{ \begin{aligned}
    - \Delta \phi &=  h^2  \sin \phi \cos \phi \quad \hspace{29.5pt}\text{in } \Omega; \\[.5mm]
    \phi &= 0 \quad \hspace{82.5pt}\text{on P}; \\[.5mm]
    \partial_\nu \phi &= L_{\text{H}} \sin \phi \cos \phi \quad \h{25pt}\text{on H}.
\end{aligned} \right. \end{equation}
Solutions to \eqref{SG} are critical points of the following energy functional on $H_\text{P}^1(\Omega)$:
\begin{equation} \label{phiE}
    E[\phi] := \int_{\Omega} \frac{1}{2} |\nabla \phi|^2 + \frac{h^2}{4} (\cos 2 \phi + 1) + \frac{L_{\text{H}}}{4} \int_\text{H} (\cos 2 \phi + 1).
\end{equation}Our first result is about the thickness threshold for the existence of multiple solutions to \eqref{SG}. \begin{thm}\label{stationary solutions}
Define the critical thickness:
\begin{equation} \label{dc}
    d_c := \frac{1}{h} \tan^{-1} \frac{h}{L_{\mathrm{H}}}.
\end{equation}Then the followings hold for the least-energy solution of \eqref{SG}: \begin{itemize}
    \item[$\mathrm{(1).}$] If $d \leq d_c$, then $0$ is the unique critical point of the energy $E$.\vspace{0.4pc}
        \item[$\mathrm{(2).}$] If $d > d_c$, then there is a unique positive least-energy solution of \eqref{SG}. \vspace{0.4pc}
        \item[$\mathrm{(3).}$] If $d > d_c$, then the least-energy solution obtained in $\mathrm{(2)}$ depends only on the variable $x_3$.
\end{itemize}
\end{thm}

Theorem \ref{stationary solutions} is proved in Section 2 for general dimensions. In Lemma \ref{thrshold by eigvalu}, the linear (in)stability of the $0$ solution is characterized by the first Steklov-Dirichlet eigenvalue. We then prove in Lemma \ref{dim red} that this eigenvalue is independent of the dimension, using a dimension-reduction argument. In Section 2.2, we characterize the linear (in)stability of the $0$ solution in terms of the thickness $d$. More properties are shown in Section 2.3 for the least-energy solution. Lemmas \ref{sign of min}, \ref{sign for all}, and \ref{unique of positive bd solu} establish the uniqueness of the positive least-energy solution of \eqref{SG}. Moreover, in Lemma \ref{sym of minimizer}, the least-energy solution is shown to depend only on the normal variable and is strictly decreasing on the interval $[0, d\h{0.5pt}]$. Section 2.4 is devoted to studying the strong stability of the least-energy solution when $d \neq d_c$. See Proposition \ref{strong stab of gmin}. In the end, Proposition \ref{non-unif-unstable} shows that solution of \eqref{SG} must be strongly unstable if it also depends on some tangential variables.

We prove in the next that there is an asymptotic limit of the global suitable weak solution to IBVP while the time tends to infinity. The limit of the director angle must be a solution to \eqref{SG}. More precisely, we have \begin{thm} \label{thm1.1}
Given an arbitrary thickness $d$ and a global suitable weak solution $(u, \phi)$ to $\mathrm{IBVP}$, there exists a large time $T_0$ such that
\begin{itemize}
    \item[$\mathrm{(1).}$] The solution $(u, \phi)$ is regular on $\overline{\Omega} \times [\h{0.5pt}T_0, \infty)$. \vspace{0.4pc}
    
    \item[$\mathrm{(2).}$] There exists a smooth solution $\phi_{\infty}$ to \eqref{SG} and a constant $\theta \in \big(0, \frac{1}{2}\big)$ such that 
    \begin{equation} \label{DecayEst}
        \|\hspace{.5pt} u(t) \hspace{.5pt}\|_{H^1} + \|\hspace{.5pt} \phi(t) - \phi_{\infty} \hspace{.5pt}\|_{H^2} \hspace{1.5pt}\lesssim\hspace{1.5pt}  (1 + t)^{- \frac{\theta}{1 - 2 \theta}}, \h{20pt}\text{for any  $t > T_0$.}
    \end{equation}
\end{itemize}

The constant $\theta$ is given by the \L{}ojasiewicz-Simon inequality.
\end{thm} 

Hereinafter, given two quantities $A$ and $B$, the notation $A \lesssim B$ means that there is a constant $C > 0$ such that $A \leq C B$. The constant $C$ might depend on $h$, $L_\mathrm{H}$, $\Omega$, and the initial data in \eqref{ini cond}. If $C$ depends on some specific constants $c_1, ..., c_j$, we also use the notation $A \lesssim_{c_1, ..., c_j} B$.  

The proof of this theorem relies on the topics discussed in Sections 3, 4, and 5. In Theorem \ref{LS general}, we prove a \L{}ojasiewicz-Simon inequality for a critical point of the $E$-energy, using the result of Chill \cite{Chill2003}. This inequality is applied to the classic solution of IBVP and infers the decay estimate \eqref{DecayEst}. See the item (1) in Proposition \ref{convergence rate}. The regularity result in the item (1) of Theorem \ref{thm1.1} is shown in Section 5 by a small-energy regularity result. We point out that for the Navier-Stokes equation, the small-energy regularity result was first proved by Caffarelli-Kohn-Nirenberg in \cite{CaffarelliKohnNirenberg1982}. Different proofs were established by Lin in \cite{Lin1998} and by Ladyzhenskaya-Seregin in \cite{LadyzhenskayaSeregin1999}. The approach of Ladyzhenskaya-Seregin was later used in \cite{Seregin2002} by Seregin to study the regularity near the flat boundary. For the 3D simplified Ericksen-Leslie equation, its small-energy regularity on the interior points is obtained in \cite{LinLiu1996}. In the 2D case, Lin-Lin-Wang \cite{LinLinWang2010} prove both the interior and boundary regularities for the simplified Ericksen-Leslie system with strong achoring condition. In our current work, we are forced to study the boundary partial regularity of the suitable weak solution $(u, \phi)$ with the weak anchoring condition for the angle variable $\phi$. This boundary condition brings the boundary integrals in \eqref{energy ineq}. We emphasize that there is a null structure hidden in these boundary integrals. In fact, we observe that for any constant $C$, it holds that \begin{align*}
    \int_{\mathrm H \times \{\h{0.5pt} T \h{0.5pt}\}} \varphi\h{1pt} \sin^2 \phi &- \int_{0}^{T} \int_{\mathrm H} \partial_t \varphi  \h{1pt}\sin^2 \phi  =  \int_{\mathrm H \times \{\h{0.5pt} T\h{0.5pt}  \}} \varphi \Bigl(\sin^2 \phi - C \h{1pt}\Bigr) -  \int_{0}^{T} \int_{\mathrm H} \partial_t \varphi \Bigl(\sin^2 \phi - C \h{1pt} \Bigr).
\end{align*}This structure is crucial in our proof of boundary partial regularity, particularly the blow-up argument in the proof of Lemma \ref{decay lem}. \vspace{0.2pc}

After investigating the general asymptotic behavior of the global suitable weak solution in Theorem \ref{thm1.1}, we now rigorously justify the P-HAN transition induced by the thickness $d$.

\begin{thm} \label{thm1.2}

Suppose the same assumption as in Theorem \ref{thm1.1}. 

\begin{itemize}
    \item[$\mathrm{(1).}$] If $0 < d \leq d_c$, then the estimate \eqref{DecayEst} holds with $\phi_{\infty} \equiv 0$. Moreover, if $d < d_c$, then 
    \begin{equation}\label{small conv}
        \|\hspace{.5pt} u(t) \hspace{.5pt}\|_{H^1} + \|\hspace{.5pt} \phi(t) \hspace{.5pt}\|_{H^2} \hspace{1.5pt}\lesssim\hspace{1.5pt} e^{- \kappa \h{0.5pt}t}, \h{20pt}\text{for any $t > T_0$.}
    \end{equation}Here, $T_0$ is a large time. $\kappa > 0$ is a constant depending on $h$, $L_{\mathrm{H}}$, $\Omega$, and the initial data. \vspace{0.4pc}

    \item[$\mathrm{(2).}$] Assume the initial director angle $\phi_0$ satisfies
    \begin{equation*}
        \phi_0 \not\equiv 0 \hspace{15pt}\text{and}\h{15pt} 0 \leq \phi_0 \leq \pi \h{20pt}\text{in $\Omega$}.
    \end{equation*}If $d > d_c$, then \eqref{DecayEst} holds with $\phi_{\infty} = \phi_*$, where 
    $\phi_*$ is the unique positive least-energy solution of \eqref{SG}. Moreover, 
    \begin{equation}\label{large conv}
        \|\hspace{.5pt} u(t) \hspace{.5pt}\|_{H^1} + \|\hspace{.5pt} \phi(t) - \phi_* \hspace{.5pt}\|_{H^2} \hspace{1.5pt}\lesssim\hspace{1.5pt} e^{- \kappa \h{0.5pt} t}, \h{20pt}\text{for any $t > T_0$}.
    \end{equation}Here, $T_0$ is a large time. $\kappa > 0$ is a constant depending on $h$, $L_{\mathrm{H}}$, $\Omega$, and the initial data.

\end{itemize}
\end{thm}

The exponential convergence rates in \eqref{small conv} and \eqref{large conv} are obtained by Corollary \ref{LS exp}, based on the strong stability of the least-energy solution obtained in Proposition \ref{strong stab of gmin}. In (1) of Theorem \ref{thm1.2}, we have $\phi_\infty \equiv 0$. This is a result of (1) in Theorem \ref{stationary solutions}. The result that $\phi_\infty = \phi_*$ in (2) of Theorem \ref{thm1.2} is discussed in Section 6 for classic global solutions of IBVP, and in Section 7 for the global suitable weak solutions of IBVP. Note that for a suitable weak solution, we do not have strong maximum principle for the director angle in general. The director angle might be identically $0$ at some finite time $T_*$ if the solution is not regular on the time interval $(0, T_*\h{0.5pt}]$. However, this situation cannot occur. In Section 7, we utilize an approximation argument and the Harnack inequality due to Ignatov-Kukavica-Ryzhik (see Lemma 3.1 in \cite{IgnatovaKukavicaRyzhik2016}), with which we show that $\phi\left(t, \cdot\right)$ cannot be identically $0$ at any large $t $ if the initial director angle is not identically $0$ in $\Omega$.

\subsection{A short literature review} To complete this introduction, we note that some research works on the long-time asymptotic behaviors of the hydrodynamical flow of liquid crystals are available in the literature. In \cite{ChenKimYu2018}, Fr\'{e}edericksz transition is considered for the same equation \eqref{sEL}. The results in \cite{ChenKimYu2018} are in 2D and with a strong unidirectional planar boundary condition. The anisotropic case is discussed in Kim-Pan \cite{kimPan2020} and Kim \cite{kim2021}. In 2023, the Fr\'{e}edericksz transition were considered with the applied inhomogeneous electric field. See the work by Sadovskii-Sadovskaya in \cite{SS2023}. We also refer readers to the work of Wu \cite{Wu2010} for the Ginzburg-Landau approximation of the Ericksen-Leslie model, and to Kim-Pan \cite{kimpan2021} for the smectic liquid crystals. In all these works, the director angle is supplied with a strong anchoring condition. Our current work focuses on the Rapini-Papoular weak anchoring condition. This boundary condition was first introduced in \cite{RapiniPapoular1969}. Its validity in the theory of nematic liquid crystals was later justified by Barbero-Durand in \cite{BarberoDurand1986}. We also note that our results apply to the global suitable weak solutions. The P-HAN transition is preserved even along the suitable weak flow of the 3D simplified Ericksen-Leslie system.

\section{Least-energy solution of the sine-Gordon equation}

In this section we take $\Omega = \mathbb{T}^{n-1} \times [\hspace{0.5pt} 0, d \hspace{0.5pt}]$, where $n$ is a natural number. We use $x_n$ to denote the normal variable that lies on the interval $[\hspace{0.5pt}0, d\hspace{0.5pt}]$. The remaining spatial variables are called tangential variables. This section is mainly concerned with the least-energy solution of the boundary value problem \eqref{SG}. The strong stability of the least-energy solution is also obtained.
\subsection{A generalized Steklov-Dirichlet eigenvalue problem}

Consider the Rayleigh quotient: 
\begin{align}\label{rayquo}
    \mathrm{R}[\phi] := \dfrac{\displaystyle\int_{\Omega} \big|\hspace{.5pt} \nabla \phi \hspace{.5pt}\big|^2}{\displaystyle h^2 \int_{\Omega} \phi^2 + L_\text{H} \int_{\text{H}} \phi^2}, \hspace{20pt}\text{where $\phi \in H_\text{P}^1(\Omega)$ and $\phi \not \equiv 0$.}
\end{align} 
If $n = 1$, the integral of a function on $\text{H}$ is known as the evaluation of this function at $0$. Define $\mathrm{R}^{n\mathrm{D}}$ to be the infimum of $\mathrm{R}[\cdot]$ over all functions in $H_\text{P}^1(\Omega)$ that are not identically $0$. Then 
\begin{align}\label{pos of evalue} 
    \mathrm{R}^{n\mathrm{D}} = \lambda_1^2 \hspace{1pt}>\hspace{1pt} 0.
\end{align}
Since $W^{1, 2}(\Omega)$ is compactly embedded into $L^2(\Omega)$ and $ L^2(\partial \hspace{0.5pt}\Omega)$, the infimum $\mathrm{R}^{n\mathrm{D}}$ can be attained by a non-negative and non-trivial function $\phi_1$ on $\Omega$. Moreover, $\phi_1$ satisfies the following generalized Steklov-Dirichlet eigenvalue problem: 
\begin{align} \label{Seignprob} 
    - \Delta \phi_1 = \left(h \lambda_1 \right)^2 \phi_1 \hspace{10pt} \text{in } \Omega; \hspace{20pt}
    \phi_1 = 0 \hspace{10pt} \text{on } \text{P}; \hspace{20pt}
    - \partial_n \phi_1 = L_\text{H} \lambda_1^2 \hspace{1pt} \phi_1 \hspace{10pt} \text{on } \text{H}.
 \end{align} 
 
The linear (in)stability of the trivial solution $0$ can be determined by $\mathrm{R}^{n\mathrm{D}}$.
\begin{lem}\label{thrshold by eigvalu}
The following two statements hold:
\begin{itemize} 
    \item[$\mathrm{(1).}$] If $\mathrm{R}^{n\mathrm{D}} \geq 1$, then $0$ is the unique critical point of the energy $E$.\vspace{0.4pc}
    \item[$\mathrm{(2).}$] If $\mathrm{R}^{n\mathrm{D}} < 1$, then $0$ is linearly unstable. It is not a local minimizer of the energy $E$.
\end{itemize}
\end{lem}

\begin{proof}[\bf Proof] Suppose $\phi \in H_\text{P}^1(\Omega)$ is a critical point of $E$. It solves the boundary value problem \eqref{SG}. Multiply the first equation in \eqref{SG} by $\phi$ and integrate over $\Omega$. Applying the divergence theorem, we obtain
\begin{equation*}
    \frac{L_\text{H}}{2} \int_\text{H} \phi \sin 2 \phi + \int_{\Omega} \big|\nabla \phi\big|^2 = \frac{h^2}{2} \int_{\Omega} \phi \sin 2 \phi.
\end{equation*}

\noindent If $\mathrm{R}^{n\mathrm{D}} \geq 1$, then the last equality infers that
\begin{equation*}
    h^2 \int_{\Omega} \phi^2 - \phi \hspace{1pt} \frac{\sin 2 \phi}{2} + L_\text{H} \int_{\text{H}} \phi^2 + \phi \hspace{1pt} \frac{\sin 2 \phi}{2} \leq 0.
\end{equation*}

\noindent Therefore,
\begin{equation*}
    \phi^2 - \phi \hspace{1pt}\frac{\sin 2 \phi}{2} \equiv 0 \hspace{10pt} \text{in } \Omega, 
\end{equation*}

\noindent which implies that $\phi \equiv 0$ in $\Omega$.\vspace{0.3pc}

Using the non-trivial eigenfunction $\phi_1$ in \eqref{Seignprob}, we calculate 
\begin{equation*}
    E \big[\hspace{0.5pt}t \hspace{0.5pt}\phi_1\hspace{0.5pt}\big] - E\big[\hspace{0.5pt}0\hspace{0.5pt}\big] = \frac{t^2}{2} \left[ \hspace{1pt}\int_{\Omega} \big|\nabla \phi_1 \big|^2 - h^2 \int_{\Omega} \phi_1^2 - L_\text{H} \int_{\text{H}} \phi_1^2 \right] + \mathrm{O}\big(t^4\big).
\end{equation*}
If $\mathrm{R}^{n\mathrm{D}} < 1$, then the coefficient of $t^2$ on the right-hand side above is strictly negative. Hence, $E [\hspace{0.5pt}t \hspace{0.5pt}\phi_1\hspace{0.5pt}] < E [\hspace{0.5pt}0\hspace{0.5pt}]$, provided that $t^2$ is suitably small. The zero solution is not a local minimizer of the energy $E$.
\end{proof}

Using a dimension-reduction argument, we show that $\mathrm{R}^{n\mathrm{D}}$ is independent of the dimension $n$.  \begin{lem}\label{dim red}
    For any natural number $n$, it holds $\mathrm{R}^{n\mathrm{D}} = \mathrm{R}^{1\mathrm{D}}$.
\end{lem}
\begin{proof}[\bf Proof] 
We have $\mathrm{R}^{n\mathrm{D}} \leq \mathrm{R}^{1\mathrm{D}}$ since $\Big\{ \phi \in H^1\big[0, d\big] : \phi(d) = 0 \hspace{1pt}\Big\} \subset H_\text{P}^1(\Omega)$. It remains to prove \begin{align}\label{ano dir}
\mathrm{R}^{n\mathrm{D}} \geq \mathrm{R}^{1\mathrm{D}}.
\end{align}
Define a non-trivial single variable function as follows:
\begin{equation*}
    \eta_1(x_n) := \int_{\mathbb{T}^{n-1}} \phi_1(x', x_n) \hspace{2pt}\mathrm{d} x'.
\end{equation*}
Integrate the first $n-1$ variables on both sides of \eqref{Seignprob}. It turns out 
\begin{align*}  
    - \dfrac{\mathrm{d}^2 \eta_1}{\mathrm{d} \hspace{0.5pt} x_n^2} = \left(h \lambda_1 \right)^2 \eta_1 \hspace{10pt} \text{in } (0, d); \hspace{20pt}\eta_1 = 0 \hspace{10pt} \text{at } x_n = d; \hspace{20pt}
    - \dfrac{\mathrm{d} \hspace{0.5pt}\eta_1}{\mathrm{d} \hspace{0.5pt} x_n} = L_\text{H} \lambda_1^2 \hspace{1pt} \psi_1 \hspace{10pt} \text{at } x_n = 0.
\end{align*}
Multiply $\eta_1$ on the both sides of the first equation above and integrate over $(0, d)$. Through integration by parts and using the boundary conditions satisfied by $\eta_1$, we obtain 
\begin{align*}
    \int_0^d \big( \eta'_1 \big)^2 \hspace{1pt} \mathrm{d} x_n = \big( h\lambda_1 \big)^2 \int_0^d \eta_1^2 \hspace{1.5pt} \mathrm{d} x_n + L_\text{H} \h{0.5pt} \lambda_1^2 \hspace{1.5pt}\eta_1^2(0).
\end{align*} 
Here $'$ denotes the derivative with respect to the $x_n$-variable. Therefore,
\begin{align*}
    \mathrm{R}^{1\mathrm{D}} \leq \dfrac{\displaystyle \int_0^d \big( \eta'_1\big)^2 \hspace{1pt}\mathrm{d} x_n}{\displaystyle h^2 \int_0^d \eta_1^2 \hspace{1.5pt}\mathrm{d} x_n + L_\text{H}  \hspace{1pt}\eta_1^2(0)} = \lambda_1^2 = \mathrm{R}^{n\mathrm{D}}.
\end{align*}
\eqref{ano dir} is obtained and the proof is completed.
\end{proof}

\subsection{Critical thickness of the film}

With Lemma \ref{dim red} we characterize  the relation between $\lambda_1$ and $d$ in the following lemma. 
 
\begin{lem} Recall $\lambda_1$ in \eqref{pos of evalue} and $d$ the thickness of $\Omega$. Then it holds
\begin{align}\label{lambda(d)}
    \lambda_1 \tan\left(h \hspace{0.5pt}\lambda_1 \hspace{0.5pt}d\right) = \frac{h}{L_\mathrm{H}}.
\end{align}
\end{lem} 

\begin{proof}[\bf Proof] By Lemma \ref{dim red}, we may consider the eigenvalue problem in \eqref{Seignprob} with $n = 1$. The first equation in \eqref{Seignprob} is now a second-order ODE. The general representation of $\phi_1$ reads as follows: 
\begin{align*}
    \phi_1(x_n) = A \sin \big(h \lambda_1 x_n \big) + B \cos\big(h \lambda_1 x_n \big).
\end{align*}
Here $A$ and $B$ are two constants. In light of the boundary conditions in \eqref{Seignprob} and the positivity of $\lambda_1$ in \eqref{pos of evalue}, $(A, B)$ is a non-trivial solution of the linear system: 
\begin{align}\label{AB eqn} \left\{ \begin{array}{lcl}
\hspace{43.5pt} h \hspace{0.5pt} A + \left( L_\mathrm{H} \lambda_1\right) B &= 0,\\[2mm]
    \sin\left(h \hspace{0.5pt}\lambda_1 \hspace{0.5pt} d\right) A  + \cos\left(h \hspace{0.5pt} \lambda_1 \hspace{0.5pt} d\right) B &= 0.
\end{array}\right.
\end{align}
Note that $\cos\left(h \hspace{0.5pt} \lambda_1 \hspace{0.5pt} d\right) \neq 0$. Otherwise, $A = B = 0$.  We then obtain \eqref{lambda(d)} since the coefficient matrix in \eqref{AB eqn} must have zero determinant.  
\end{proof} 
Furthermore, we have 
\begin{lem}\label{mon of lambda1}
For any fixed $d \in (0, \infty)$, there is a unique solution in $\left(0, \dfrac{\pi}{2 \hspace{0.5pt} h \hspace{0.5pt} d}\right)$ to the following equation of $x$: \begin{align}\label{alg eqn} 
    x \tan\left(h \hspace{0.5pt}x \hspace{0.5pt}d\right) = \frac{h}{L_\mathrm{H}}. 
\end{align} This solution is equal to $\lambda_1$. If we regard $\lambda_1 = \lambda_1(d)$ as a function of $d$, then $\lambda_1$ is strictly decreasing with respect to $d$.
\end{lem}

\begin{proof}[\bf Proof] Suppose $\lambda = \lambda(d)$ is the unique solution of \eqref{alg eqn} in $\left(0, \dfrac{\pi}{2 \hspace{0.5pt} h \hspace{0.5pt} d}\right)$. Then 
$$\psi_*(x_n) := - L_\text{H} \lambda \sin \big(h \lambda x_n \big) + h \cos\big(h \lambda x_n \big)$$ 
is a non-trivial solution to the problem: \begin{align*}
    - \dfrac{\mathrm{d}^2 \psi_*}{\mathrm{d} \hspace{0.5pt} x_n^2}  = \left(h \lambda \right)^2 \psi_* \hspace{10pt} \text{in } (0, d); \hspace{20pt}\psi_* = 0 \hspace{10pt} \text{at } x_n = d; \hspace{20pt}
    - \dfrac{\mathrm{d} \hspace{0.5pt}\psi_*}{\mathrm{d} \hspace{0.5pt} x_n} = L_\text{H} \lambda^2 \hspace{1pt} \psi_* \hspace{10pt} \text{at } x_n = 0.
\end{align*}
Note that $\mathrm{R}^{n\mathrm{D}} = \lambda_1^2$ is the minimum of  \eqref{rayquo} over all functions in $H_\text{P}^1(\Omega)$ that are not identically $0$. It yields \begin{align*}
    \lambda_1^2 \leq \dfrac{\displaystyle \int_0^d \big( \psi'_*\big)^2 \hspace{1pt}\mathrm{d} x_n}{\displaystyle h^2 \int_0^d \psi_*^2 \hspace{1.5pt}\mathrm{d} x_n + L_\text{H}  \hspace{.5pt}\psi_*^2(0)} = \lambda^2,
\end{align*} which together with \eqref{lambda(d)} infer that $\lambda_1$ is also a solution of \eqref{alg eqn} in $\left(0, \dfrac{\pi}{2 \hspace{0.5pt} h \hspace{0.5pt} d}\right)$. Hence, $\lambda_1(d) = \lambda(d)$.\vspace{0.4pc}

Differentiate the equation \eqref{lambda(d)} with respect to  $d$ and rearrange the resulting equation. We obtain \begin{align*}
    \dot{\lambda}_1 \left( h \hspace{0.5pt}\lambda_1 d + \dfrac{1}{2} \sin \big(2 \hspace{0.5pt}h\hspace{0.5pt} \lambda_1 d \big) \right) = - h \hspace{0.5pt}\lambda_1^2.
\end{align*}
Here $\dot{\lambda}_1$ is the derivative of $\lambda_1$ with respect to $d$. It holds $\dot{\lambda}_1 < 0$ from the last equality. $\lambda_1$ is therefore a  strictly decreasing function of $d$.
\end{proof}

Let $d_c$ be the critical thickness satisfying $\lambda_1(d_c) = 1$. Equivalently, 
\begin{align}\label{d_c}
    d_c = \dfrac{1}{h} \tan^{-1} \dfrac{h}{L_\text{H}}.
\end{align}
Since $\lambda_1(d)$ is strictly decreasing with $d$, Lemma \ref{thrshold by eigvalu} can now be rephrased in terms of $d$ as follows.  
\begin{prop}\label{threshold of thickness}
    The following two statements hold:
    \begin{itemize} 
        \item[$\mathrm{(1).}$] If $d \leq d_c$, then $0$ is the unique critical point of the energy $E$.\vspace{0.4pc}
        \item[$\mathrm{(2).}$] If $d > d_c$, then $0$ is linearly unstable. It is not a local minimizer of the energy $E$.
    \end{itemize}
\end{prop}

\subsection{The least-energy solution when \texorpdfstring{$d > d_c$}{TEXT}}\vspace{0.2pc}

We study the least-energy solution of \eqref{SG}. In view of part (1) in Proposition \ref{threshold of thickness}, we assume $d > d_c$ for the rest of the section. 

\begin{lem}\label{sign of min}
    If $\phi$ is a non-negative global minimizer of $E$ in $H_\mathrm{P}^1(\Omega)$, then $0 < \phi < \dfrac{\pi}{2}$ on $\Omega \cup \text{H}$. 
\end{lem}
\begin{proof}[\bf Proof]
Suppose $\phi$ is a non-negative global minimizer of $E$ in $H_\text{P}^1(\Omega)$. It is a solution to the boundary value problem \eqref{SG}. By iteratively applying Theorem 2.3.3.2 in \cite{Grisvard1985} and trace theorem, $\phi$ is a classic solution in $\Omega$. All derivatives of $\phi$ are continuous up to the boundary. \vspace{0.3pc}

Assuming $\displaystyle \max_{\overline{\Omega}} \phi > \frac{\pi}{2}$, then we define the truncation of $\phi$, denoted by $\phi_{\hspace{0.5pt}\flat}$, such that $\phi_{\hspace{0.5pt}\flat} = \dfrac{\pi}{2}$ if $\phi \geq \dfrac{\pi}{2}$. $\phi_{\hspace{0.5pt}\flat} = \phi $ at the points where $\phi$ is less than $\dfrac{\pi}{2}$. It turns out
\begin{align*}
    E[\phi] - E[\phi_{\hspace{0.5pt}\flat}] = \int_{\Omega \hspace{1pt}\cap \hspace{1pt}\big\{ \phi \hspace{1pt}\geq \hspace{1pt}\frac{\pi}{2}\big\}} \dfrac{1}{2} \hspace{1pt}\big|\nabla \phi\big|^2 + \frac{h^2}{4} \big(\cos 2\phi + 1\hspace{0.5pt}\big) + \dfrac{L_\text{H}}{4}\int_{\text{H} \hspace{1pt}\cap \hspace{1pt}\big\{ \phi \hspace{1pt}\geq \hspace{1pt}\frac{\pi}{2}\big\}}  \big(\cos 2 \phi + 1\hspace{0.5pt}\big) > 0.
\end{align*}
This contradicts the fact that $\phi$ is a global minimizer of $E$ in $H_\text{P}^1(\Omega)$. Therefore, $\displaystyle \max_{\overline{\Omega}} \phi \leq \frac{\pi}{2}$. Moreover, by \eqref{SG}, $\phi$ satisfies 
\begin{align}\label{ineq of min}
    (1). \hspace{5pt} \Delta \phi \leq 0  \hspace{15pt} \text{in } \Omega; \hspace{25pt}(2). \hspace{5pt}\phi \equiv 0 \hspace{15pt} \text{on P} ; \hspace{25pt}
    (3). \hspace{5pt}\partial_n \phi \leq 0 \hspace{15pt} \text{on } \text{H}.
\end{align}

We first show that $\phi > 0$ on $\Omega \cup \text{H}$. If $\phi$ is a constant function, then $\phi \equiv 0$ by (2) in \eqref{ineq of min}. This case was excluded by (2) in Proposition \ref{threshold of thickness}. Applying the strong maximum principle, we obtain $\phi > 0$ in $\Omega$. Suppose there is $x_0 \in \text{H}$ such that $\phi(x_0) = 0$. Then Hopf lemma induces $\partial_n \phi(x_0) > 0$. This contradicts (3) in \eqref{ineq of min}. We obtain $\phi > 0$ on $\Omega \cup \text{H}$.  \vspace{0.3pc}

To prove $\phi < \dfrac{\pi}{2}$ on $\Omega \cup \text{H}$, we change the variable by  $\psi := \dfrac{\pi}{2} - \phi$. The previous results imply that $0 \leq \psi < \dfrac{\pi}{2}$ on $\Omega \cup \text{H}$. Recall \eqref{SG}. The function $\psi $ satisfies
\begin{align}\label{ineq of min psi} 
    (1). \hspace{5pt}\Delta \psi - h^2 \hspace{0.5pt}\psi   \leq 0 \hspace{15pt} \text{in } \Omega; \hspace{20pt}(2). \hspace{5pt}\psi \equiv \dfrac{\pi}{2} \hspace{15pt} \text{on P}; \hspace{20pt}(3). \hspace{5pt}\partial_n \psi =  \dfrac{L_\text{H}}{2} \sin 2 \psi \hspace{15pt} \text{on H}.
\end{align} 
The function $\psi$ is not constant, otherwise $\phi \equiv 0$ in $\Omega$. Assume that $0$ is the minimum value of $\psi$ on $\overline{\Omega}$. From (1) in \eqref{ineq of min psi} and the strong maximum principle for the elliptic operator $\Delta - h^2$, it holds $\psi > 0$ in $\Omega$. Given (2) in \eqref{ineq of min psi}, the value $0$ can only be taken by $\psi$ at some $x_1 \in \text{H}$. Furthermore, it follows from Hopf lemma that $\partial_n \psi (x_1) > 0$. This is a contradiction, because according to (3) in \eqref{ineq of min psi}, $\partial_n  \psi (x_1) = 0$. Therefore, the minimum of $\psi$ over $\overline{\Omega}$ is not $0$. Consequently, $\phi < \frac{\pi}{2}$ on $\overline{\Omega}$.
\end{proof}

One application of Lemma \ref{sign of min} is to show 
\begin{lem}\label{sign for all}
    Any global minimizer of $E$ in $H_\mathrm{P}^1(\Omega)$ keeps the sign strictly in $\Omega \cup \text{H}$.
\end{lem}
\begin{proof}[\bf Proof] Suppose $\phi$ is a global minimizer of $E$ in $H_\text{P}^1(\Omega)$. If $\phi$ changes sign in $\Omega \cup \text{H}$, then $\phi$ vanishes at some $x_0$ in $\Omega \cup \text{H}$. $|\hspace{0.5pt} \phi \hspace{0.5pt}|$ is also a global minimizer of $E$ in $H_\text{P}^1(\Omega)$. Recall Lemma \ref{sign of min}. It holds $|\hspace{0.5pt}\phi(x_0)\hspace{.5pt}| > 0$. This is a contradiction to the fact that $\phi(x_0) = 0$.
\end{proof}

In the following we investigate the uniqueness of the positive solutions of \eqref{SG} which are bounded from above by $\frac{\pi}{2}$. We use the monotone iteration method introduced in \cite{Sattinger1972}. In contrast to \cite{Sattinger1972}, where the semilinear elliptic equations are supplied with Robin  boundary conditions, our problem \eqref{SG} involves nonlinear boundary condition on H. The monotone iteration method is also used in \cite{ChenDingHuNiZhou2003} for the sine-Gordon equation with the Dirichlet boundary condition.

\begin{lem}\label{unique of positive bd solu}
There is only one solution of \eqref{SG} with the values in $(0, \frac{\pi}{2})$ on $\Omega \cup \mathrm{H}$. 
\end{lem}
\begin{proof}[\bf Proof]
Given a smooth function $u$ on $\overline{\Omega}$ and let $g(u) := u + \frac{1}{2} \sin 2 u$, we denote by $
\mathscr{L} u$ the unique solution of the following boundary value problem:
\begin{align} \label{MonosG}
    \left(h^{-2} \Delta - 1\right) v = - g(u) \quad \text{in } \Omega; \hspace{20pt}v = 0 \quad \text{on P} ; \hspace{20pt}\left(- L_\text{H}^{-1}\hspace{0.5pt} \partial_n  +  1\right)v  = g(u) \quad \text{on H} .
\end{align}
If $u_1, u_2$ are smooth functions on $\overline{\Omega}$ and $0 \leq u_1 \leq u_2 \leq \frac{\pi}{2}$ on $\Omega$, then $w := \mathscr{L}u_1 - \mathscr{L} u_2$ satisfies
\begin{align}\label{ineq of w}
    &(1).\hspace{5pt} \left(h^{-2} \Delta - 1\right) w  \geq 0 \quad \text{in } \Omega; \hspace{20pt} \nonumber\\[1mm] &(2). \hspace{10pt}w = 0 \quad \text{on P} ; \h{70pt}(3). \hspace{5pt}
    \left(- L_\text{H}^{-1} \hspace{0.5pt}\partial_n + 1\right) w  \leq 0 \quad \text{on H} .        
\end{align}
According to (2) in \eqref{ineq of w}, the maximum value of $w$ over $\overline{\Omega}$ is non-negative. Moreover, if $w$ is constant, then $w \equiv 0$ on $\Omega$. Assume $w$ is not constant. By the strong maximum principle, the non-negative maximum value of $w$ on $\overline{\Omega}$ can be only attained by some point $x_0$ on $\partial \hspace{0.5pt} \Omega$. If $x_0 \in \text{H}$, then Hopf lemma infers $\partial_n  w(x_0) < 0$. This is impossible by (3) in \eqref{ineq of w}. Therefore, if $w$ is not constant, then the maximum point of $w$ must be on $\text{P}$. To summarize, we obtain
\begin{align}\label{monotonicity}
    \mathscr{L}u_1 \leq \mathscr{L} u_2 \hspace{8pt}\text{on $\overline{\Omega}$, \hspace{5pt} if $u_1, u_2$ are smooth on $\overline{\Omega}$ \hspace{0.5pt}and\hspace{0.5pt} $0 \leq u_1 \leq u_2 \leq \dfrac{\pi}{2}$ on $\Omega$.}
\end{align}

Let $v_0 \equiv \frac{\pi}{2}$ and define $v_1 := \mathscr{L} v_0$. It follows 
\begin{align*}
    \left(h^{-2} \Delta - 1\right) (v_1 - v_0) =  0 \quad \text{in } \Omega; \hspace{10pt}v_1 - v_0 = - \dfrac{\pi}{2} \quad \text{on P} ; \hspace{10pt}
    \left(- L_\text{H}^{-1} \hspace{0.5pt} \partial_n + 1 \right)(v_1 - v_0) =  0 \quad \text{on H} .
\end{align*}
If the maximum value of $v_1 - v_0$ over $\overline{\Omega}$ is non-negative, then by the boundary condition on P above, $v_1 - v_0$ is not constant. The maximum value of $v_1 - v_0$ cannot be attained on P. Apply the strong maximum principle. $v_1 - v_0$ takes its maximum value on H. Hopf Lemma infers $\partial_n (v_1 - v_0) < 0$ at the maximum point of $v_1 - v_0$ on H. This is a contradiction to the boundary condition of $v_1 - v_0$ on H. Therefore, it holds $v_1 < v_0$ on $\overline{\Omega}$. Inductively, we define $v_{k+1} := \mathscr{L}v_k$ for all $k \in \mathbb{N}$. Since $0 < v_0 \equiv \frac{\pi}{2}$, it follows by \eqref{monotonicity} that  $$0 = \mathscr L 0 \leq \mathscr L v_0 = v_1 \h{15pt} \text{on $\overline{\Omega}$.}$$ Therefore, $0 \leq v_1 \leq v_0 \equiv \frac{\pi}{2}$ on $\overline{\Omega}$. Still using \eqref{monotonicity} induces $$0 \leq v_2 = \mathscr L v_1 \leq \mathscr L v_0 = v_1 \h{15pt} \text{on $\overline{\Omega}$.}$$ Repeatedly applying the same arguments then yields $v_{k+1} \leq v_k$ on $\overline{\Omega}$ for any $k \in \mathbb{N}$. \vspace{0.3pc}

Assume $v$ is a positive solution of \eqref{SG} bounded from above by $\frac{\pi}{2}$ on $\Omega \cup \text{H}$. It turns out $v$ is a fixed point of the operator $\mathscr{L}$. Then, $v = \mathscr{L} v \leq \mathscr{L} v_0 = v_1$  by \eqref{monotonicity}. Inductively, we obtain $v \leq v_k$ for any $k \in \mathbb{N}$.\vspace{0.3pc} 

In summary, our arguments induce
\begin{align}\label{order ineq}
    v \leq \cdots \leq v_{k+1} \leq v_k \leq \cdots \leq v_1 < v_0 = \frac{\pi}{2} \h{15pt}\text{on $\overline{\Omega}$.}
\end{align}
Iteratively applying Theorem 2.3.3.2 in \cite{Grisvard1985}, we know that $\{v_k\}$ is uniformly bounded in $H^p(\Omega)$ for any $p \in (1, \infty)$. By Morrey's inequality and Arzel\`{a}-Ascoli theorem, $\{v_k\}$ converges uniformly to a limit function, denoted by $v_\infty$, in $C^1(\overline{\Omega})$. The inequalities in \eqref{order ineq} then imply
\begin{align}\label{comp phi and v}
   0 < v \leq v_\infty   < \frac{\pi}{2} \quad \text{on } \Omega \cup \mathrm H.
\end{align}
Meanwhile, $v_\infty$ also satisfies the boundary value problem in \eqref{SG}. Through integrations by parts, \begin{align*}
    \int_\Omega - v_\infty \Delta v + v \Delta v_\infty = \frac{h^2}{2} \int_\Omega v_\infty \sin 2v - v \sin 2 v_\infty = \frac{L_\text{H}}{2} \int_\text{H} v \sin 2 v_\infty - v_\infty \sin 2v,
\end{align*}
which gives us 
\begin{align*}
h^2 \int_\Omega v \hspace{0.5pt}v_\infty \left( \frac{\sin 2 v}{2v} - \frac{\sin 2v_\infty}{2v_\infty} \right)  +  L_\text{H} \int_\text{H} v \hspace{0.5pt}v_\infty \left(\frac{\sin 2 v}{2v} - \frac{\sin 2v_\infty}{2v_\infty}\right) = 0.
\end{align*}
Thus, $v = v_\infty$ on $\Omega \cup \text{H}$ due to \eqref{comp phi and v}, the above equality, and the monotonicity of $\frac{\sin x}{x}$ on $(0, \pi)$. The proof ends with the arbitrariness of $v$. 
\end{proof}

With Lemmas \ref{sign of min}-\ref{unique of positive bd solu}, we have 
\begin{lem}\label{sym of minimizer}
    The unique positive global minimizer of $E$ in $H_\mathrm{P}^1(\Omega)$ depends only on the normal variable. It is a strictly decreasing function on $[\hspace{0.5pt}0, d\hspace{0.5pt}]$.
\end{lem} 
\begin{proof}[\bf Proof] The energy $E$ is translation invariant along the tangential direction. Suppose $\phi$ is the positive global minimizer of $E$. Then for each $t \in \mathbb{R}$ and $i = 1, ..., n - 1$, $\phi(\cdot + t \hspace{1pt} l_i)$ is also a positive global minimizer of $E$. Here, $l_i$ is the unit vector in $\mathbb{R}^n$ whose $j$-th component is equal to $\delta_{ij}$. By the uniqueness result in Lemma \ref{unique of positive bd solu}, it follows $\phi(\cdot) = \phi(\cdot + t \hspace{1pt} l_i)$ on $\Omega$. Taking the partial derivative with respect to $t$ induces $\partial_i \phi = 0$ on $\Omega$ for each $i = 1, ..., n-1$. Hence, $\phi$ depends only on the normal variable. \vspace{0.3pc}

According to the first equation in \eqref{SG} and the fact that $0 < \phi < \frac{\pi}{2}$ on $\Omega \cup \text{H}$, we have $\phi'' < 0$ on $[\hspace{0.5pt}0, d\hspace{0.5pt})$.  Since $\phi'(0) < 0$, hence, $\phi' < 0$ on $[\hspace{0.5pt}0, d\hspace{0.5pt}]$. $\phi$ is strictly decreasing on $[\hspace{0.5pt}0, d\hspace{0.5pt}]$.
\end{proof}

\subsection{Strong stability of the least-energy solution}\vspace{0.2pc}

Denote by $\mathscr{V}$ the Hilbert spaces $H_\text{P}^1(\Omega)$. Its inner product is given by 
\begin{align}\label{inner prod on V}
    \big<g_1, g_2\big>_\mathscr{V}  := \int_\Omega \nabla g_1 \cdot \nabla g_2, \hspace{20pt} \text{for any $g_1, g_2 \in \mathscr{V}$.}
\end{align}
Let $\big<\cdot, \cdot \big>_{\mathscr{V}'\times \mathscr{V}}$ be the duality between $\mathscr{V}$ and its dual $\mathscr{V}'$. The first-order derivative of $E$ is read  as  
\begin{align}\label{1st deriv of E}
    \big< \hspace{0.5pt}E'[\hspace{0.5pt}\phi\hspace{0.5pt}], \varphi \hspace{0.5pt}\big>_{\mathscr{V}'\times \mathscr{V}} = \int_{\Omega} \nabla \varphi  \cdot \nabla \phi  - \frac{h^2}{2} \int_{\Omega} \varphi \sin 2 \phi - \frac{L_\text{H}}{2} \int_\text{H} \varphi \sin 2 \phi.
\end{align}
Here $E' \in C^1 \left( \mathscr{V}, \mathscr{V}' \right)$. We keep differentiating $E'$. The second-order derivative of $E$ is read as 
\begin{align}\label{2nd deri of E}
    \big<\hspace{0.5pt} E''[\hspace{0.5pt}\phi\hspace{0.5pt}] \hspace{0.5pt} \psi, \varphi \hspace{0.5pt}\big>_{\mathscr{V}' \times \mathscr{V}} := \int_{\Omega} \nabla \varphi \cdot \nabla \psi - h^2 \int_{\Omega} \varphi \hspace{0.5pt}\psi \cos 2 \phi - L_\text{H} \int_\text{H} \varphi \hspace{0.5pt} \psi \cos 2 \phi.
\end{align}
Given $\phi \in \mathscr{V}$, the linear operator $E''[\hspace{0.5pt}\phi\hspace{0.5pt}]$ is a bounded operator from $\mathscr{V}$ to $\mathscr{V}'$. Let $\phi$ be a critical point of the energy $E$, we define the principal eigenvalue of the linearized operator $E''[\hspace{0.5pt}\phi\hspace{0.5pt}]$ as follows: \begin{align}\label{prin eigen}
    \mu_1 := \inf_{\psi \hspace{1pt}\in\hspace{1pt}\mathscr{V}, \hspace{1.5pt}\psi \not \equiv 0} \hspace{3pt}\dfrac{\big<\hspace{0.5pt} E''[\hspace{0.5pt}\phi\hspace{0.5pt}] \hspace{0.5pt} \psi, \psi \hspace{0.5pt}\big>_{\mathscr{V}' \times \mathscr{V}}}{h^2\hspace{.5pt}\|\hspace{0.5pt}\psi\hspace{0.5pt}\|^2_{L^2(\Omega)} + L_\text{H} \hspace{.5pt} \|\hspace{0.5pt}\psi\hspace{0.5pt}\|^2_{L^2(\text{H})}}.
\end{align}It turns out that $\mu_1$ can be attained by a non-negative eigenfunction $\psi_1$. Moreover, $\psi_1> 0$ in $\Omega$ due to Serrin's maximum principle. The eigenspace associated with $\mu_1$ is simple. \vspace{0.3pc}

\begin{defn}\label{strong stab}
The critical point $\phi$ of $E$ is called strongly unstable if $\mu_1 < 0$. It is called strongly stable if $\mu_1 > 0$.
\end{defn}
The main result in this section is 
\begin{prop}\label{strong stab of gmin}
Recall the critical thickness $d_c$ in \eqref{d_c}. If $d \neq d_c$, then the global minimizer of $E$ is strongly stable in the sense of Definition \ref{strong stab}. \end{prop}
\begin{proof}[\bf Proof] Let $\phi$ in \eqref{prin eigen} be the global minimizer of $E$. $\psi_1$ is the eigenfunction that achieves $\mu_1$. In addition, we assume $\psi_1 > 0$ in $\Omega$. It can be shown that $\psi_1$ satisfies the following boundary value problem:\begin{align}\label{eq of psi1}
    - \Delta \psi_1 - h^2 \hspace{.5pt} \psi_1 \cos 2 \phi = \mu_1 h^2 \hspace{0.5pt}\psi_1 \hspace{10pt} \text{in $\Omega$;} \hspace{20pt}- \partial_n \psi_1 - L_\text{H} \hspace{0.5pt} \psi_1 \cos 2\phi = \mu_1 L_\text{H} \hspace{0.5pt} \psi_1 \hspace{10pt} \text{on H.}
\end{align}

If $0 < d < d_c$, by (1) in Proposition \ref{threshold of thickness}, it holds $\phi \equiv 0$ in $\Omega$. In light of \eqref{prin eigen}, it turns out
    \begin{align*}
        \mu_1 \left( h^2 \int_\Omega \psi_1^2 + L_\text{H} \int_\text{H} \psi_1^2\right) =  \int_\Omega \big|\hspace{1pt} \nabla \psi_1 \big|^2 - h^2 \int_\Omega \psi_1^2 - L_\text{H} \int_\text{H} \psi_1^2.
    \end{align*}
When $d \in \big(0, d_c\big)$, we have $\mathrm{R}^{n\mathrm{D}} > 1$. The right-hand side is positive, and thus $\mu_1 > 0$ in this case.
    
In the following, we assume $d > d_c$. $\phi$ is the global minimizer of $E$. Due to Lemma \ref{sign for all}, we may assume that $\phi$ is  strictly positive on $\Omega \cup \text{H}$. Now we multiply the first equation in \eqref{eq of psi1} by $\phi$ and integrate by parts. Using the boundary condition in \eqref{eq of psi1} and the fact that $\phi = 0$ on P, we get 
    \begin{align*}
        \mu_1 \left( h^2 \int_\Omega \phi \hspace{1pt}\psi_1 + L_\text{H} \int_\text{H} \phi \hspace{1pt} \psi_1  \right) = \int_\Omega \nabla \phi \cdot \nabla \psi_1 - h^2 \int_\Omega \psi_1 \phi \cos 2 \phi - L_\text{H} \int_\text{H} \psi_1 \phi \cos 2 \phi.
    \end{align*}
Note that $\phi$ satisfies the boundary value problem \eqref{SG}. We then multiply $\psi_1$ on both sides of the first equation in \eqref{SG} and integrate by parts. It then turns out \begin{align*}
    \int_\Omega \nabla \phi \cdot \nabla \psi_1 = \frac{h^2}{2} \int_\Omega \psi_1 \sin 2 \phi + \frac{L_\text{H}}{2} \int_\text{H} \psi_1 \sin 2 \phi.
\end{align*}Combining the last two equalities, we induce that \begin{align*}
    \mu_1 \left( h^2 \int_\Omega \phi \hspace{1pt}\psi_1 + L_\text{H} \int_\text{H} \phi \hspace{1pt} \psi_1  \right) = \frac{h^2}{2} \int_\Omega \psi_1 \big( \sin 2 \phi - 2 \phi \cos 2 \phi \big) + \frac{L_\text{H}}{2} \int_\text{H} \psi_1 \big( \sin 2 \phi - 2  \phi \cos 2 \phi \big).
\end{align*}According to Lemma \ref{sign of min}, it satisfies $\sin 2 \phi - 2 \phi \cos 2 \phi > 0$ on $\Omega \cup \text{H}$. The right-hand side above is therefore strictly positive, since $\psi_1$ is also strictly positive on $\Omega$. We conclude from the last equality that $\mu_1 > 0$. The proof is completed.             
\end{proof}

We call solutions of \eqref{SG} uniform if they depend only on the normal variable. If a solution is not uniform, it is called a non-uniform solution. From Lemma \ref{sym of minimizer} and Proposition \ref{strong stab of gmin}, the global minimizer of $E$ is a uniform solution of \eqref{SG}. It is strongly stable if $d \neq d_c$. We would like to point out that the strong stability of a critical point of $E$ is sufficient to imply that the critical point is a uniform solution of \eqref{SG}. In fact, we have

\begin{prop}\label{non-unif-unstable}The non-uniform solutions of \eqref{SG} are strongly unstable.
\end{prop}
\begin{proof}[\bf Proof] Assume that $\phi$ is a non-uniform solution of $\eqref{SG}$. The variables $x_1, ..., x_{n-1}$ are tangential variables. Taking $\partial_j$, with $j = 1, ..., n-1$, on both sides of the equation and the conditions in \eqref{SG}, we obtain \begin{align}\label{pro of tan}
    - \Delta \hspace{1pt}\partial_{j} \phi = h^2 \left(\cos 2 \phi \right) \partial_{j} \phi \quad \text{in $\Omega$}; \hspace{20pt} \partial_j \phi = 0 \quad \text{on P}; \hspace{20pt} - \partial_n \partial_j \phi = L_\text{H} \left( \cos 2 \phi \right) \partial_j \phi \quad \text{on H.}
\end{align}Multiply $\partial_j \phi$ on both sides of the first equation above and integrate by parts. It follows \begin{align*}
    \big<\hspace{0.5pt} E''[\hspace{0.5pt}\phi\hspace{0.5pt}] \hspace{0.5pt} \partial_j \phi, \partial_j \phi \hspace{0.5pt}\big>_{\mathscr{V}' \times \mathscr{V}} = 0, \hspace{20pt}j = 1, ..., n-1.
\end{align*}Here we also use the boundary conditions in \eqref{pro of tan}. Since $\phi$ is non-uniform, one of the tangential derivatives of $\phi$ must not be identically equal to zero. We assume $\partial_1 \phi \not \equiv 0$ on $\Omega$. Recall $\mu_1$ defined in \eqref{prin eigen}. If $\mu_1 \geq 0$, then $\mu_1 = 0$ and $\partial_1 \phi$ is an eigenfunction corresponding to $\mu_1$. Therefore, $\partial_1 \phi = c \hspace{0.8pt}\psi_1$ on $\Omega$, where $c$ is a non-zero constant. $\psi_1$ is an eigenfunction corresponding to $\mu_1$ that is positive on $\Omega$. We imply that  $\partial_1 \phi$ retains the sign on $\Omega$. This is a contradiction because $\phi$ is periodic along the direction $x_1$.   
\end{proof}

\section{\L ojasiewicz-Simon inequality}

The \L ojasiewicz-Simon inequality is proved in \cite{Chill2003} for functions on Banach spaces. In this section, we apply the results from \cite{Chill2003} and prove the \L ojasiewicz-Simon inequality for our energy $E$ defined on $H_\text{P}^1(\Omega)$. Note that the results in this section are valid for all dimensions.\vspace{0.3pc}

\begin{lem}\label{finite dim of ker}
    Suppose  $\phi$ is a critical point of $E$ that satisfies $E'[\hspace{0.5pt}\phi\hspace{0.5pt}] = 0$. Then $\mathrm{Ker} \hspace{0.5pt}E''[\hspace{0.5pt}\phi\hspace{0.5pt}]$ is of finite dimension. In addition, the functions in $\mathrm{Ker} \hspace{0.5pt}E''[\hspace{0.5pt}\phi\hspace{0.5pt}]$ are smooth on $\Omega$ with all their derivatives continuous up to $\partial \hspace{0.5pt} \Omega$.
\end{lem}
\begin{proof}[\bf Proof]
Assume $\psi \in \mathrm{Ker} \hspace{0.5pt}E''[\hspace{0.5pt}\phi\hspace{0.5pt}] \subseteq \mathscr{V}$. In view of \eqref{2nd deri of E}, $\psi$ solves the boundary value problem:    \begin{align}\label{ker equ}
    - \Delta \psi = h^2 \left( \cos 2 \phi \right) \psi \quad \text{in $\Omega$;} \hspace{20pt} \psi = 0\quad \text{on P}; \hspace{20pt} - \partial_n \psi = L_\text{H} \left(\cos 2\phi\right) \psi \quad\text{on H}.
\end{align}If $\phi$ is a critical point of $E$ and $\psi$ solves \eqref{ker equ}, then $\phi$ and $\psi$ are smooth on $\Omega$ by applying Theorem 2.3.3.2 in \cite{Grisvard1985} iteratively. All their derivatives are continuous up to the boundary $\partial \hspace{0.5pt}\Omega$.\vspace{0.3pc}

Fixing a function $f \in \mathscr{V}$, we define $K f$ to be the unique solution of the problem: \begin{align*} - \Delta \psi = h^2 \left( \cos 2 \phi \right) f \quad \text{in $\Omega$;} \hspace{20pt} \psi = 0\quad \text{on P}; \hspace{20pt} - \partial_n \psi = L_\text{H} \left(\cos 2\phi\right) f \quad\text{on H}.\end{align*}Therefore, $\psi \in \mathrm{Ker} \hspace{0.5pt}E''[\hspace{0.5pt}\phi\hspace{0.5pt}]$ if and only if $\psi \in N (I - K)$. Here, $N(I - K)$ denotes the null space of $I - K$. According to Theorem 2.3.3.2 in \cite{Grisvard1985} and the compactness of the embedding $H^1(\Omega) \hookrightarrow L^2(\Omega)$, $K$ is a compact operator from $\mathscr{V}$ to $\mathscr{V}$. Fredholm alternative infers that $N(I - K)$ is of finite dimension.
\end{proof}

Assume $\phi$ is a critical point of $E$ and define $\mathscr{V}_0 := \mathrm{Ker} \hspace{0.5pt}E''[\hspace{0.5pt}\phi\hspace{0.5pt}]$. Let $P_0$ be the orthogonal projection from $\mathscr{V}$ to $\mathscr{V}_0$. Then, $\mathscr{V} = \mathscr{V}_0 \oplus \mathscr{V}_1$, where $\mathscr{V}_1 = \mathrm{Ker}\hspace{0.5pt} P_0$. Note that $P_0$ is defined with respect to the inner product \eqref{inner prod on V}. Recall the operator $K$ in the proof of Lemma \ref{finite dim of ker}. It holds \begin{align}\label{self adj}
\int_{\Omega} \nabla \varphi \cdot \nabla K \psi = h^2 \int_{\Omega} \varphi \hspace{0.5pt}\psi \cos 2 \phi + L_\text{H} \int_\text{H} \varphi \hspace{0.5pt} \psi \cos 2 \phi    \hspace{20pt}\text{for any $\psi, \varphi \in H_\text{P}^1(\Omega)$}.
\end{align}Plugging the above equality into \eqref{2nd deri of E} yields \begin{align*}
     \big<\hspace{0.5pt} E''[\hspace{0.5pt}\phi\hspace{0.5pt}] \hspace{0.5pt} \psi, \varphi \hspace{0.5pt}\big>_{\mathscr{V}' \times \mathscr{V}} = \int_{\Omega} \nabla \varphi \cdot \nabla \big( \left( I - K \right) \psi\big).
\end{align*}This representation of $E''[\hspace{0.5pt}\phi\hspace{0.5pt}]$ and Fredholm alternative infer that $$\mathrm{Rg} \hspace{1.5pt}E''[\hspace{0.5pt}\phi\hspace{0.5pt}] = \mathrm{Rg} \hspace{1.5pt}\big( I - K \big) = N\left( I - K^*\right)^{\bot}.$$ Here, $\mathrm{Rg} \hspace{1.5pt}T$ denotes the range of an operator $T$. The operator $K^*$ is the adjoint of $K$. Given $f \in \mathscr{V}$, it is identified with the linear operator $T_f \in \mathscr{V}'$ via the relation: $$\big<f, g\big>_{\mathscr{V}} = \big<\hspace{0.5pt} T_f, g \hspace{0.5pt}\big>_{\mathscr{V}' \times \mathscr{V}} \hspace{20pt} \text{for all $g \in H_\text{P}^1(\Omega)$.}$$Note that $K$ is self-adjoint since by \eqref{self adj},  \begin{align*}
    \big< K^*f, g\big>_{\mathscr{V}} = \big<f, K g\big>_{\mathscr{V}} = \big<K f, g\big>_{\mathscr{V}} \hspace{20pt}\text{for any $f, g \in H_\text{P}^1(\Omega)$.}
\end{align*}We then get $\mathrm{Rg} \hspace{1.5pt}E''[\hspace{0.5pt}\phi\hspace{0.5pt}] = N\left( I - K\right)^{\bot}$. Denote the adjoint of $P_0$ by $P_0'$. It turns out \begin{align*}
    \big<\hspace{0.5pt} P_0'\hspace{1.2pt}T_f, g \hspace{0.5pt}\big>_{\mathscr{V}' \times \mathscr{V}} = \big<\hspace{0.5pt} T_f, P_0\hspace{1pt}g \hspace{0.5pt}\big>_{\mathscr{V}' \times \mathscr{V}} = \big< f, P_0\hspace{1pt}g\big>_{\mathscr{V}} = \big< P_0 \h{1pt}f, g\big>_{\mathscr{V}} \hspace{20pt}\text{for any $f, g \in H_\text{P}^1(\Omega)$.}
\end{align*}Therefore, $T_f \in \mathrm{Ker} \hspace{0.5pt}P_0'$ is equivalent to $f \in \mathrm{Ker} \hspace{0.5pt}P_0 = N\left( I - K\right)^{\bot}$. To summarize, we have \begin{align}\label{rg linear}
\mathrm{Rg} \hspace{1.5pt}E''[\hspace{0.5pt}\phi\hspace{0.5pt}] =  \mathrm{Ker} \hspace{0.5pt}P_0'.
\end{align}\

Consider the subspaces $X = H_\text{P}^1(\Omega) \cap H^n (\Omega) \hookrightarrow \mathscr{V}$ and $Y = \big\{ T_f : f \in X\big\} \hookrightarrow \mathscr{V}'$. $X$ and $Y$ are invariant under the projections $P_0$ and $P_0'$, respectively. This is due to the smoothness of the eigenfunctions in $\mathrm{Ker} \hspace{0.5pt}E''[\hspace{0.5pt}\phi\hspace{0.5pt}]$. See Lemma \ref{finite dim of ker}. By Morrey's inequality, $H^n(\Omega)$ is embedded into $L^{\infty}(\Omega)$ continuously. In view of \eqref{1st deriv of E} and the analyticity of the sine function, the restriction of $E'$ on $X$ is analytic in a neighborhood of $\phi$. We now show \begin{align}\label{rg of linear on X}
\mathrm{Rg} \hspace{1.5pt} E''[\hspace{0.5pt}\phi\hspace{0.5pt}] \hspace{1.5pt}\Big|_X =  \mathrm{Ker} \hspace{0.5pt}P_0' \cap Y.
\end{align}Recall \eqref{2nd deri of E}. $E''[\hspace{0.5pt}\phi\hspace{0.5pt}] \psi = T_f$ is equivalent to \begin{align*}
    \int_{\Omega} \nabla \varphi \cdot \nabla \psi - h^2 \int_{\Omega} \varphi \hspace{0.5pt}\psi \cos 2 \phi - L_\text{H} \int_\text{H} \varphi \hspace{0.5pt} \psi \cos 2 \phi = \int_\Omega \nabla \varphi \cdot \nabla f \hspace{20pt}\text{for any $\varphi \in H_\text{P}^1(\Omega)$}.
\end{align*}Theorem 2.3.3.2 in \cite{Grisvard1985} shows that $\psi \in H^n(\Omega)$ if $f \in H^n(\Omega)$, and vice versa. \eqref{rg of linear on X} follows by \eqref{rg linear}.\vspace{0.3pc}

With the above arguments, we apply Corollary 3.11 in \cite{Chill2003} and obtain

\begin{thm} \label{LS general}

Let $\phi$ be a critical point of $E$. Then there are $\rho > 0$, $\gamma > 0$ and $\theta \in \left(\hspace{0.5pt}0, \frac{1}{2}\hspace{0.5pt}\right]$ so that 
\begin{equation*}
    \big\|\hspace{1pt} E'[\psi] \hspace{1.5pt}\big\|_{\mathscr{V}'} \geq \gamma \hspace{1pt}\big|\hspace{1pt} E[\psi] - E[\phi] \hspace{1.5pt}\big|^{1 - \theta} \hspace{20pt}\text{for any $\psi \in \mathscr{V}$ with $\|\hspace{1pt} \psi - \phi \hspace{1pt}\|_{H^1} \leq \rho$. }
\end{equation*}
Here the constants $\rho, \gamma$, $\theta$ depend on $h$, $L_\mathrm{H}$, $\Omega$ and $\phi$. The notation $\mathscr{V}$ still denotes the space $H_{\mathrm{P}}^1(\Omega)$.

\end{thm}

In certain cases, the \L{}ojasiewicz-Simon inequality applies with the optimal exponent $\theta = \frac{1}{2}$, which leads to an exponential convergence rate in some gradient flows. For our current problem, if $\phi$ is the global minimizer of $E$ and $d \neq d_c$, we have strong stability of $\phi$ as shown in Proposition \ref{strong stab of gmin}. Therefore, $\mathrm{Ker} \hspace{0.5pt}E''[\hspace{0.5pt}\phi\hspace{0.5pt}] = 0$, which infers that $\mathrm{Ker} P_0 = \mathscr{V}$. Due to \eqref{rg linear}, the linearized operator $E''[\hspace{0.5pt}\phi\hspace{0.5pt}]$ is invertible from $\mathscr{V}$ onto $\mathscr{V}'$. Applying Corollary 3.13 in \cite{Chill2003}, we obtain

\begin{cor}\label{LS exp}
Theorem \ref{LS general} applies with $\theta = \frac{1}{2}$ if $\phi$ is the global minimizer of $E$ and $d \neq d_c$. 
\end{cor}

\section{Convergence along the classical hydrodynamic flow}

Suppose $(u, \phi)$ is a global classical solution of IBVP. The spatial dimension $n$ is set to be 3. In this section, we apply the \L ojasiewicz-Simon inequality in Theorem \ref{LS general} (see also Corollary \ref{LS exp}) to study the convergence of $(u, \phi)$ as $t$ tends to $\infty$. 

\subsection{Basic energy estimates}

With the energy $E$ in \eqref{phiE}, we define the total energy: 
\begin{align*}
    \mathcal{E}(t) := E\hspace{0.5pt}\big[\phi(t)\big] + \frac{1}{2} 
 \int_{\Omega} |\hspace{0.5pt}u(t)\hspace{0.5pt}|^2.
\end{align*}
\begin{lem}\label{basic energy est}
    If $(u, \phi)$ is a global classical solution of $\mathrm {IBVP}$, then, for any $t > T_0 \geq 0$, we have
    \begin{equation} \label{energy}
        \cE(t) + \int_{T_0}^t \int_{\Omega} \big|\hspace{0.5pt}\nabla u\hspace{0.5pt}\big|^2 + \big|\hspace{0.5pt}\Delta \phi + \frac{h^2}{2} \sin 2 \phi \hspace{0.5pt}\big|^2 = \cE(T_0).
    \end{equation}
\end{lem}

\begin{proof}[\bf Proof]
Take the inner product with $u$ on the first equation in \eqref{sEL} and integrate by parts. Then
\begin{equation*}
    \dfrac{1}{2}\frac{\mathrm{d}}{\mathrm{d} \hspace{0.7pt}t} \int_{\Omega}  |\hspace{0.5pt}u\hspace{0.5pt}|^2 + \int_{\Omega} \big|\hspace{0.5pt}\nabla u\hspace{0.5pt}\big|^2 = - \int_{\Omega} \left(\Delta \phi + \frac{h^2}{2} \sin 2 \phi\right) u \cdot \nabla \phi = - \int_{\Omega} \big|\hspace{0.5pt}u \cdot \nabla \phi \hspace{0.5pt}\big|^2 + \left(u \cdot \nabla \phi\right)\partial_t \phi.
\end{equation*}The first equality above uses the incompressibility condition and the boundary condition of $u$. The second equality results from the third equation in \eqref{sEL}. In the next step, we multiply the third equation in \eqref{sEL} by $ \partial_t \phi$ and integrate by parts. Hence,
\begin{align*}
    \int_{\Omega} \big(\partial_t \phi\big)^2 + \frac{\mathrm{d}}{\mathrm{d}\hspace{0.7pt}t} \int_{\Omega} \frac{1}{2} \hspace{0.5pt} \big|\hspace{0.5pt}\nabla \phi\hspace{0.5pt}\big|^2 + \frac{h^2}{4} \big(\cos 2\phi + 1\big) + \frac{\mathrm{d}}{\mathrm{d} \hspace{0.7pt}t} \int_\text{H} \frac{L_\text{H}}{4} \big(\cos 2 \phi + 1\big) = - \int_{\Omega} \left(u \cdot \nabla \phi\right)  \partial_t \phi. 
\end{align*}The boundary conditions in \eqref{bc} are used to derive the last equality. Now we add the two equalities above. It follows from the definition of $\cE(t)$ that
\begin{equation*}
    \frac{\mathrm{d} \hspace{0.7pt} \cE}{\mathrm{d} \hspace{0.7pt}t}  = -\int_{\Omega} \big|\hspace{0.5pt}\nabla u\hspace{0.5pt}\big|^2 + \big|\hspace{0.5pt} \partial_t \phi + u \cdot \nabla \phi \hspace{0.5pt}\big|^2 = - \int_{\Omega} \big|\hspace{0.5pt}\nabla u\hspace{0.5pt}\big|^2 + \big|\hspace{0.5pt}\Delta \phi + \frac{h^2}{2} \sin 2 \phi \hspace{0.5pt}\big|^2.
\end{equation*}We obtain the proof of \eqref{energy} by integrating the above equality from $T_0$ to $t$.
\end{proof}

In the rest of this section, we study the higher-order energy estimate for the global classical solution of IBVP. The main result is based on a Stokes-type estimate for the velocity field. We summarize it in the following lemma.

\begin{lem}\label{Stokes est}
  Suppose $u \in H_{0, \mathrm{div}}^1(\Omega)$ is a weak solution of \begin{align}\label{stokes} - \Delta u + \nabla q = f \hspace{20pt}\text{in $\Omega$,} \h{15pt}\text{where $f \in L^2(\Omega; \mathbb R^3)$.}
    \end{align}Then it holds  $\| \hspace{0.5pt}u \hspace{0.5pt}\|_{H^2} \lesssim \| \hspace{0.5pt}f \hspace{0.5pt}\|_{L^2}$.
\end{lem}\begin{proof}[\bf Proof] Note that $u$ is the unique minimizer of the variational problem: \begin{align*}
    \min \hspace{1pt} \left\{\int_\Omega   \big|\hspace{0.5pt} \nabla v \hspace{0.5pt}\big|^2 - 2 \hspace{0.5pt}f \cdot v : v \hspace{1pt} \in \hspace{1pt} H_{0, \mathrm{div}}^1(\Omega) \right\}. 
\end{align*} Comparing the energy between $u$ and $0$ induces \begin{align*}
     \int_\Omega   \big|\hspace{0.5pt} \nabla u \hspace{0.5pt}\big|^2 - 2 \hspace{0.5pt} f \cdot u \leq 0.
\end{align*}Applying H\"{o}lder's inequality, we obtain from the last estimate that \begin{align}\label{nabla est 1}
     \int_\Omega   \big|\hspace{0.5pt} \nabla u \hspace{0.5pt}\big|^2  \leq  2 \int_\Omega f \cdot u \hspace{1.5pt}\lesssim\hspace{1.5pt} \| f \|_{L^2} \| \hspace{0.5pt}u \hspace{0.5pt} \|_{L^2}.
\end{align}

We estimate the $H^1$-norm of $u$. By the fundamental theorem of calculus, \begin{align*}
    u(x', w) = u(x', d ) - \int_w^d \partial_3 u (x', x_3) \hspace{1pt} \mathrm{d} x_3 = - \int_w^d \partial_3 u (x', x_3) \hspace{1pt} \mathrm{d} x_3 \hspace{20pt}\text{for any $w \in [\hspace{0.5pt}0, d\hspace{0.5pt}]$.}  
\end{align*} It then turns out \begin{align*}
    \left|\hspace{2pt}\int_{\mathbb{T}^{2}} u(x', w) \hspace{1pt}\mathrm{d} x' \hspace{1pt}\right| = \left|\hspace{2pt} \int_{\mathbb{T}^{2}}\int_w^d \partial_3 u (x', x_3) \hspace{1pt} \mathrm{d} x_3 \hspace{1pt}\mathrm{d} x' \hspace{1pt}\right| \hspace{1.5pt}\lesssim\hspace{1.5pt}\big\|\hspace{0.5pt}\partial_3 u \hspace{0.5pt}\big\|_{L^2} \hspace{20pt}\text{for any $w \in [\hspace{0.5pt}0, d\hspace{0.5pt}]$.} 
\end{align*}Using this estimate and the Poincar\'{e}'s inequality on torus, we obtain, for any $w \in [\hspace{0.5pt}0, d\hspace{0.5pt}]$, that \begin{align*}\big\|\hspace{1pt}u(\cdot, w)\hspace{1pt}\big\|_{L^2(\mathbb{T}^{2})} &\leq \left\|\hspace{1.5pt} u(\cdot, w) - \int_{\mathbb{T}^{2}} u(x', w) \hspace{1pt}\mathrm{d} x' \hspace{1pt}\right\|_{L^2(\mathbb{T}^{2})} + \hspace{2pt}\left|\hspace{2pt}\int_{\mathbb{T}^{2}} u(x', w) \hspace{1pt}\mathrm{d} x' \hspace{1pt}\right|\\[2mm]
&\lesssim \big\| \hspace{1pt} \nabla' u\hspace{0.5pt}(\cdot, w) \hspace{1pt}\big\|_{L^2(\mathbb{T}^{2})} + \big\|\hspace{0.5pt}\partial_3 u \hspace{0.5pt}\big\|_{L^2}.
\end{align*} Take square on both sides of the above estimate and integrate the variable $w$ from $0$ to $d$. It follows \begin{align}\label{Poincare ineq}\|\hspace{0.5pt}u \hspace{0.5pt}\|_{L^2} \lesssim \|\hspace{0.5pt} \nabla u \hspace{0.5pt}\|_{L^2}.\end{align} Applying this estimate to the right-hand side of \eqref{nabla est 1}, we get
\begin{align}\label{nabla est 2} \| \nabla u \|_{L^2} \hspace{1.5pt}\lesssim\hspace{1.5pt} \| f \|_{L^2}.
\end{align}

Now, we estimate the $L^2$-norm of $\nabla^2 u$. Denote by $B'_r$ the open ball in $\mathbb{R}^{2}$ with center $0$ and radius $r$. Let $\eta$ be a smooth cut-off function that is  compactly supported on $B_4'$. Furthermore, $\eta$ is equivalently equal to $1$ on $B_2'$. Multiplying $\eta$ on both sides of \eqref{stokes} yields \begin{align*}
    - \Delta \big( \eta \hspace{0.5pt} u \big) + \nabla \hspace{0.5pt} \big( \eta  \left( q - c_q\right)\big) = f \hspace{0.3pt}
 \eta - u \hspace{0.3pt} \Delta \eta - 2 \hspace{0.3pt} \nabla \eta \cdot \nabla u + \left( q - c_q \right)\nabla \eta \hspace{20pt}\text{in $\Omega' := B'_4 \times (\hspace{0.5pt}0, d \hspace{0.5pt})$.} 
\end{align*}Here, $c_q$ is the average of $q$ over $\Omega'$. Note that $\mathrm{div} \big( \eta \hspace{0.5pt} u \big) = u \cdot \nabla \eta$.  Proposition 2.2 and Remark 2.6 in Chapter 1 of \cite{Temam1984} then induce  \begin{align*}
    \big\| \hspace{0.5pt}\nabla^2 u \hspace{0.5pt}\big\|_{L^2} \hspace{1.5pt}\lesssim\hspace{1.5pt}\big\| \hspace{0.5pt} \eta u \hspace{0.5pt}\big\|_{H^2(\Omega')} \hspace{1.5pt}\lesssim \hspace{1.5pt} \big\| \hspace{0.5pt} f \hspace{0.3pt}
 \eta - u \hspace{0.3pt} \Delta \eta - 2 \hspace{0.3pt} \nabla \eta \cdot \nabla u + \left( q - c_q \right)\nabla \eta \hspace{0.8pt}\big\|_{L^2(\Omega')} + \big\| \hspace{0.5pt}u \cdot \nabla \eta \hspace{0.8pt}\|_{H^1(\Omega')}.
\end{align*}Taking into account \eqref{Poincare ineq}-\eqref{nabla est 2}, we can keep estimating the right-hand side above and get \begin{align}\label{nabla sqr est 1}
    \big\| \hspace{0.5pt}\nabla^2 u \hspace{0.5pt}\big\|_{L^2} \hspace{1.5pt}\lesssim\hspace{1.5pt} \| \hspace{0.6pt}f \hspace{0.5pt}\|_{L^2} + \big\| \hspace{0.5pt}q - c_q  \hspace{0.5pt}\big\|_{L^2(\Omega')} \hspace{1.5pt}\lesssim\hspace{1.5pt} \| \hspace{0.6pt}f \hspace{0.5pt}\|_{L^2} + \big\| \hspace{0.5pt}\nabla q \hspace{0.5pt}\big\|_{H^{-1}(\Omega')}.
\end{align}The second estimate above uses Proposition 1.2 in Chapter 1 of \cite{Temam1984}. See also \cite{necas1965}. Suppose $\varphi$ is a smooth $3$-vector field that is compactly supported on $\Omega'$. We take inner product with $\varphi$ on both sides of \eqref{stokes} and integrate over $\Omega'$. Through integration by parts, it follows \begin{align*}
    \int_{\Omega'} \varphi \cdot \nabla q = \int_{\Omega'} f \cdot \varphi - \nabla u : \nabla \varphi \hspace{1.5pt}\lesssim \hspace{1.5pt} \| \hspace{0.6pt}f \hspace{0.5pt}\|_{L^2} \| \hspace{0.6pt} \varphi \hspace{0.5pt}\|_{L^2(\Omega')} + \| \hspace{0.6pt} \nabla u \hspace{0.5pt}\|_{L^2} \| \hspace{0.6pt} \nabla \varphi \hspace{0.5pt}\|_{L^2(\Omega')} \hspace{1.5pt}\lesssim \hspace{1.5pt} \| \hspace{0.6pt}f \hspace{0.5pt}\|_{L^2} \| \hspace{0.6pt}\varphi \hspace{0.5pt}\|_{H^1(\Omega')}.
\end{align*} The last estimate in the above uses \eqref{nabla est 2}. Therefore, we get $\big\| \hspace{0.5pt}\nabla q \hspace{0.5pt}\big\|_{H^{-1}(\Omega')}  \lesssim \| \hspace{0.6pt}f \hspace{0.5pt}\|_{L^2}$. Applying this estimate to the right-hand side of \eqref{nabla sqr est 1} then completes the proof.
\end{proof}

We also need to control the Hessian of $\phi$. \begin{lem}\label{est nabla 2 phi} Assume $\phi \in H^2(\Omega)$ and satisfies \eqref{bc}. Then the following estimate holds: \begin{align*}
    \int_{\Omega} \big|\hspace{0.5pt}\nabla^2 \phi \hspace{0.5pt}\big|^2 \hspace{1.5pt}\lesssim\hspace{1.5pt}\int_{\Omega} \big(\Delta \phi\big)^2 + \int_\Omega \big|\hspace{0.5pt}\nabla \phi \hspace{0.5pt}\big|^2.
\end{align*}
\end{lem}
\begin{proof}[\bf Proof] According to an approximation argument, we assume without loss of generality that $\phi$ is smooth throughout $\overline{\Omega}$. Applying integration by parts, we get
\begin{align*}
    \int_{\Omega} \big|\hspace{0.5pt}\nabla^2 \phi \hspace{0.5pt}\big|^2 = \int_{\text{P}} \nabla \phi \cdot \nabla \partial_3 \phi - \int_{\text{P}} \partial_3 \phi \hspace{1pt}\Delta \phi   - \int_{\text{H}} \nabla \phi \cdot \nabla \partial_{3} \phi   + \int_{\text{H}} \partial_3 \phi \hspace{1pt}\Delta \phi + \int_{\Omega} \big(\Delta \phi\big)^2.
\end{align*}
Note that \begin{align*}\nabla \phi \cdot \nabla \partial_3 \phi - \partial_3 \phi \hspace{1pt}\Delta \phi = \nabla' \phi \cdot \nabla' \partial_3 \phi - \partial_3 \phi \hspace{1.5pt}\Delta' \phi, \h{15pt}\text{where $\Delta' = \partial_{11} + \partial_{22}$.}\end{align*} Using the boundary conditions of $\phi$ on $\mathrm H \cup \mathrm P$, we then obtain \begin{align*}
    \int_{\text{P}} \nabla \phi \cdot \nabla \partial_3 \phi - \int_{\text{P}} \partial_3 \phi \hspace{0.8pt}\Delta \phi   &- \int_{\text{H}} \nabla \phi \cdot \nabla \partial_{3} \phi   + \int_{\text{H}} \partial_3 \phi \hspace{0.8pt}\Delta \phi \\[1mm]
    &= L_\text{H} \int_\text{H} \cos 2 \phi \hspace{1pt} \big| \nabla' \phi \big|^2 - \frac{L_\text{H}}{2} \int_\text{H} \sin 2 \phi \hspace{1pt} \Delta' \phi  = 2 L_\text{H} \int_\text{H} \cos 2 \phi \hspace{1pt} \big| \nabla' \phi \big|^2.
\end{align*}The last equality above uses integration by parts with respect to the tangential variables and the periodic boundary condition of $\phi$. Therefore, 
\begin{align}\label{Hessian delta id}
    \int_{\Omega} \big|\hspace{0.5pt}\nabla^2 \phi \hspace{0.5pt}\big|^2  = \int_{\Omega} \big(\Delta \phi\big)^2 + 2 L_\text{H} \int_\text{H} \cos 2 \phi \hspace{1pt} \big| \nabla' \phi \big|^2 \hspace{1.5pt}\leq\hspace{1.5pt}\int_{\Omega} \big(\Delta \phi\big)^2 + 2 L_\text{H} \int_\text{H} \big| \nabla' \phi \big|^2.
\end{align}Using integration by parts with respect to the normal variable, we obtain \begin{align*}
    \int_\text{H} \big| \nabla' \phi \big|^2 = - \int_\Omega \partial_3 \big| \nabla' \phi \big|^2 = - 2 \int_\Omega \nabla' \phi \cdot \partial_3 \nabla' \phi.
\end{align*}It then turns out \begin{align}\label{hes, del}
    \int_{\Omega} \big|\hspace{0.5pt}\nabla^2 \phi \hspace{0.5pt}\big|^2 \hspace{1.5pt}\leq\hspace{1.5pt}\int_{\Omega} \big(\Delta \phi\big)^2 - 4 L_\text{H} \int_\Omega \nabla' \phi \cdot \partial_3 \nabla' \phi \hspace{1.5pt}\leq\hspace{1.5pt}\int_{\Omega} \big(\Delta \phi\big)^2 + \frac{1}{2} \int_{\Omega} \big|\hspace{0.5pt}\nabla^2 \phi \hspace{0.5pt}\big|^2 + c_* \int_\Omega \big|\hspace{0.5pt}\nabla' \phi \hspace{0.5pt}\big|^2.
\end{align} The constant $c_*$ depends only on $L_{\mathrm H}$. The proof is completed.
\end{proof}We now discuss the higher-order energy estimate of the global classical solution of IBVP. \begin{lem}\label{high ord est}
Suppose $(u, \phi)$ is a global classical solution of $\mathrm{IBVP}$. $A(t)$ is defined by \begin{equation*}
    A(t) := \int_{\Omega} \big|\hspace{0.5pt}\nabla u \hspace{0.5pt}\big|^2 + \big|\hspace{0.5pt}\Delta \phi + \frac{h^2}{2} \sin 2 \phi \hspace{0.5pt}\big|^2.
\end{equation*}
Then we have
\begin{align} \label{HigherE}
    \frac{\mathrm{d} \hspace{0.2pt}A}{\mathrm{d}
    \hspace{0.7pt}t} + \int_\Omega \big|\hspace{0.5pt} S\hspace{0.3pt}u\hspace{0.5pt}\big|^2 + \big| \hspace{0.8pt}\nabla \big( \Delta \phi + \frac{h^2}{2} \sin 2 \phi \big) \hspace{0.5pt}\big|^2 \hspace{1.5pt}\leq\hspace{1.5pt}A\hspace{0.5pt}\mathcal{Q}\big(\mathcal{E} + A\big), \h{15pt}\text{where $S u := - \Delta u + \nabla p$.}
\end{align} The single-variable function $\mathcal{Q}$ is a non-constant polynomial with non-negative coefficients.
\end{lem}
\begin{proof}[\bf Proof] 
Since $(u, \phi)$ is a global classical solution of IBVP, direct calculations show that\begin{align*}
    \frac{1}{2} \hspace{0.5pt} \frac{\mathrm{d} \hspace{0.2pt}A}{\mathrm{d}
    \hspace{0.7pt}t} = - \int_\Omega \Delta u \cdot \partial_t u + \int_\Omega \big( \Delta \phi + \frac{h^2}{2} \sin 2 \phi \big) \big( \Delta \hspace{0.5pt}\partial_t \phi + h^2 \cos 2 \phi \hspace{1.5pt} \partial_t \phi \big).
\end{align*}Here, we apply the no-slip boundary condition \eqref{bdry of u}. Using the equation of $\phi$ in \eqref{sEL}, we can rewrite this identity as follows: \begin{align}\label{high ord energy iden}
    \frac{1}{2} \hspace{0.5pt} \frac{\mathrm{d} \hspace{0.2pt}A}{\mathrm{d}
    \hspace{0.7pt}t} = I_1 + I_2 + I_3 + I_4, 
\end{align}where the four terms on the right-hand side above are given by \begin{align*}
    &I_1 := - \int_\Omega \Delta u \cdot \partial_t u, \\[1mm] &I_2 := \int_\Omega \big( \Delta \phi + \frac{h^2}{2} \sin 2 \phi \big) \hspace{1pt} \Delta \hspace{0.5pt}\big( \Delta \phi + \frac{h^2}{2} \sin 2 \phi \big),\\[1mm] &I_3 := - \int_\Omega \big( \Delta \phi + \frac{h^2}{2} \sin 2 \phi \big) \hspace{1pt} \Delta \hspace{0.5pt} \big( u \cdot \nabla \phi \big), \\[1mm] &I_4 := \int_\Omega h^2 \cos 2 \phi 
 \hspace{1pt}\big( \Delta \phi + \frac{h^2}{2} \sin 2 \phi \big)\hspace{1pt} \big( \Delta \phi + \frac{h^2}{2} \sin 2 \phi - u \cdot \nabla \phi \big).
\end{align*}

We now estimate these four terms successively.\vspace{0.3pc}

\noindent\textbf{Estimate for $I_1$.} It follows from $\div u = 0$ in $\Omega$ and $u = 0$ on $\partial \hspace{0.5pt} \Omega$ that
\begin{equation*}
    I_1 =   \int_\Omega S \hspace{0.3pt} u \cdot \partial_t u.
\end{equation*}In light of the equation for $u$ in \eqref{sEL}, it then turns out \begin{align}\label{eq of I1}
    I_1 = - \int_\Omega \big|\hspace{0.5pt} S\hspace{0.3pt}u\hspace{0.5pt}\big|^2 - \int_\Omega \big(S\hspace{0.3pt}u\big) \cdot \big( u \cdot \nabla u\big) - \int_\Omega \big( \Delta \phi + \frac{h^2}{2} \sin 2 \phi \big) \hspace{0.5pt} \big(S\hspace{0.3pt}u \big) \cdot \nabla \phi.
\end{align}According to  Gagliardo-Nirenberg inequality, $\nabla u$ can be estimated by \begin{align}\label{gn for nabla u}
\| \hspace{0.3pt}\nabla u\hspace{0.3pt}\|_{L^4} \hspace{1.5pt}\lesssim\hspace{1.5pt}\| \hspace{0.3pt}u\hspace{0.3pt}\|_{L^2} + \| \hspace{0.3pt}u\hspace{0.3pt}\|^{\frac{1}{4}}_{H^1} \big\| \hspace{0.3pt}\nabla^2 u\hspace{0.3pt}\big\|^{\frac{3}{4}}_{L^2}.
\end{align}
By  H\"{o}lder, Sobolev and the last estimate, it satisfies \begin{align*}
    \left| \hspace{2pt}\int_\Omega \big(S\hspace{0.3pt}u\big) \cdot \big( u \cdot \nabla u\big) \hspace{1pt}\right| \hspace{1.5pt}\lesssim\hspace{1.5pt} \| \hspace{1pt}S\hspace{0.7pt} u \hspace{1pt}\|_{L^2} \| \hspace{0.3pt}u\hspace{0.3pt}\|_{L^4} \| \hspace{0.3pt}\nabla u\hspace{0.3pt}\|_{L^4} \hspace{1.5pt}\lesssim\hspace{1.5pt} \| \hspace{1pt}S\hspace{0.7pt} u \hspace{1pt}\|_{L^2} \| \hspace{0.3pt}u\hspace{0.3pt}\|_{H^1} \left[\hspace{1pt}\| \hspace{0.3pt}u\hspace{0.3pt}\|_{L^2} + \| \hspace{0.3pt}u\hspace{0.3pt}\|^{\frac{1}{4}}_{H^1} \big\| \hspace{0.3pt}\nabla^2 u\hspace{0.3pt}\big\|^{\frac{3}{4}}_{L^2}\right].
\end{align*} With Poincar\'{e} inequality in \eqref{Poincare ineq}, the last estimate leads to \begin{align*}
    \left| \hspace{2pt}\int_\Omega \big(S\hspace{0.3pt}u\big) \cdot \big( u \cdot \nabla u\big) \hspace{1pt}\right| \hspace{1.5pt}\lesssim\hspace{1.5pt}  A\hspace{1pt}\big\| \hspace{1pt}S\hspace{0.7pt} u \hspace{1pt}\big\|_{L^2}  + A^{\frac{5}{8}} \hspace{1pt} \big\| \hspace{1pt}S\hspace{0.7pt} u \hspace{1pt}\big\|_{L^2}  \big\| \hspace{0.3pt}\nabla^2 u\hspace{0.3pt}\big\|^{\frac{3}{4}}_{L^2}.
\end{align*} Applying Lemma \ref{Stokes est} to the right-hand side above gives us \begin{align*}
    \left| \hspace{2pt}\int_\Omega \big(S\hspace{0.3pt}u\big) \cdot \big( u \cdot \nabla u\big) \hspace{1pt}\right| \hspace{1.5pt}\lesssim\hspace{1.5pt} A\hspace{1pt}\big\| \hspace{1pt}S\hspace{0.7pt} u \hspace{1pt}\big\|_{L^2}  + A^{\frac{5}{8}} \hspace{1pt} \big\| \hspace{1pt}S\hspace{0.7pt} u \hspace{1pt}\big\|^{\frac{7}{4}}_{L^2}.
\end{align*}By Young's inequality, it then turns out \begin{align*}
    \left| \hspace{2pt}\int_\Omega \big(S\hspace{0.3pt}u\big) \cdot \big( u \cdot \nabla u\big) \hspace{1pt}\right| \hspace{1.5pt}\leq\hspace{1.5pt} \frac{1}{12} \hspace{1pt} \| \hspace{1pt}S\hspace{0.7pt} u \hspace{1pt}\|^2_{L^2} + c_* \big(A^2 + A^5\big). 
\end{align*}
Plugging the last estimate into \eqref{eq of I1}, we get
\begin{equation}\label{Inq I.1.}
    I_1 \leq - \frac{11}{12} \int_\Omega \big|\hspace{0.5pt} S\hspace{0.3pt}u\hspace{0.5pt}\big|^2  + c_* \big(A^2 + A^5\big) - \int_\Omega \big( \Delta \phi + \frac{h^2}{2} \sin 2 \phi \big) \hspace{0.5pt} \big(S\hspace{0.3pt}u \big) \cdot \nabla \phi.
\end{equation}\vspace{0.2pc}

\noindent\textbf{Estimate for $I_2$.} Using the equation of $\phi$ in \eqref{sEL} and the homogeneous Dirichlet boundary conditions of  $(u, \phi)$ on P, we can apply integration by parts and rewrite $I_2$ as follows:
\begin{align}\label{basic dep of I2}
    I_2 = - \int_\Omega \big| \hspace{0.8pt}\nabla \big( \Delta \phi + \frac{h^2}{2} \sin 2 \phi \big) \hspace{0.5pt}\big|^2 - \int_\text{H} \big( \Delta \phi + \frac{h^2}{2} \sin 2 \phi \big) \hspace{0.5pt} \partial_3 \big( \partial_t \phi + u \cdot \nabla \phi \big). 
\end{align}
The weak anchoring condition of $\phi$ induces 
\begin{align*}
    - \int_\text{H} \big( \Delta \phi + \frac{h^2}{2} \sin 2 \phi \big) \hspace{0.5pt} \partial_3 \partial_t \phi &= \frac{L_\text{H}}{2} \int_\text{H} \big( \Delta \phi + \frac{h^2}{2} \sin 2 \phi \big) \hspace{0.5pt} \partial_t \big(  \sin 2 \phi \big) \\[1mm]
    &= L_\text{H} \int_\text{H} \big( \Delta \phi + \frac{h^2}{2} \sin 2 \phi \big) \hspace{0.5pt} \cos 2 \phi \hspace{1.5pt} \partial_t \phi.
\end{align*}Note that $u = 0$ on H. The last equalities and the equation of $\phi$ in \ref{sEL} infer
\begin{align}\label{I.2.1}
- \int_\text{H} \big( \Delta \phi + \frac{h^2}{2} \sin 2 \phi \big) \hspace{0.5pt} \partial_3 \partial_t \phi = L_\text{H} \int_\text{H} \big( \Delta \phi + \frac{h^2}{2} \sin 2 \phi \big)^2 \hspace{0.5pt} \cos 2 \phi  
    \hspace{1.5pt}\lesssim\hspace{1.5pt} \int_\text{H} \big( \Delta \phi + \frac{h^2}{2} \sin 2 \phi \big)^2.
\end{align}
By the fundamental theorem of calculus, \begin{align*}
    \int_\text{H} \big( \Delta \phi + \frac{h^2}{2} \sin 2 \phi \big)^2 &= - 2   \int_{\Omega} \big( \Delta \phi + \frac{h^2}{2} \sin 2 \phi \big) \partial_3 \big( \Delta \phi + \frac{h^2}{2} \sin 2 \phi \big).
\end{align*}Apply this equality to the right-hand side of \eqref{I.2.1} and then use Young's inequality. It follows \begin{align}\label{I.2.2}
- \int_\text{H} \big( \Delta \phi + \frac{h^2}{2} \sin 2 \phi \big) \hspace{0.5pt} \partial_3 \partial_t \phi \hspace{1.5pt}\leq\hspace{1.5pt} \frac{1}{12} \int_\Omega  \big| \hspace{0.8pt}\nabla \big( \Delta \phi + \frac{h^2}{2} \sin 2 \phi \big) \hspace{0.5pt}\big|^2 + c_* \int_\Omega \big( \Delta \phi + \frac{h^2}{2} \sin 2 \phi \big)^2. 
\end{align}

We continue to use $\big(u, \phi\big) = 0$ on P and $u = 0$ on H, from which yield 
\begin{align}\label{Exp of 2nd}
    - \int_{\text{H}} \big( \Delta \phi + \frac{h^2}{2} \sin 2 \phi \big) \partial_3 \big(u \cdot \nabla \phi\big) &= - \int_\text{H} \big( \Delta \phi + \frac{h^2}{2} \sin 2 \phi \big) \hspace{1pt}\big(\partial_3 u\big) \cdot \nabla \phi  \nonumber\\[2mm]
    &\h{-155pt} = \int_{\Omega} \partial_3 \left[ \big( \Delta \phi + \frac{h^2}{2} \sin 2 \phi \big) \hspace{1pt}\big(\partial_3 u\big) \cdot \nabla \phi \right] \\[2mm]
    &\h{-155pt} = \int_{\Omega} \partial_3 \big( \Delta \phi + \frac{h^2}{2} \sin 2 \phi \big) \hspace{1pt} \big(\partial_3 u\big) \cdot  \nabla \phi  +  \big( \Delta \phi + \frac{h^2}{2} \sin 2 \phi \big)\hspace{1pt}  \partial_{33} u  \cdot  \nabla \phi    +  \big( \Delta \phi + \frac{h^2}{2} \sin 2 \phi \big)\hspace{1pt}  \partial_3 u  \cdot  \nabla \partial_3 \phi. \nonumber
\end{align}

The first integrand in the last line of \eqref{Exp of 2nd} can be estimated by \begin{align*}
    \int_{\Omega} \partial_3 \big( \Delta \phi + \frac{h^2}{2} \sin 2 \phi \big) \hspace{1pt} \big(\partial_3 u\big) \cdot  \nabla \phi \hspace{1.5pt}\leq\hspace{1.5pt} \frac{1}{12} \int_\Omega  \big| \hspace{0.8pt}\nabla \big( \Delta \phi + \frac{h^2}{2} \sin 2 \phi \big) \hspace{0.5pt}\big|^2 + c_* \int_\Omega \big| \hspace{0.5pt}\nabla u \hspace{0.5pt} \big|^2 \hspace{1pt} \big| \hspace{0.5pt}\nabla \phi \hspace{0.5pt} \big|^2. 
\end{align*}Applying H\"{o}lder inequality, \eqref{gn for nabla u}, and Sobolev inequality induce \begin{align*}
    &\int_{\Omega} \partial_3 \big( \Delta \phi + \frac{h^2}{2} \sin 2 \phi \big) \hspace{1pt} \big(\partial_3 u\big) \cdot  \nabla \phi \\[2mm] 
&\hspace{80pt} \leq\hspace{1.5pt} \frac{1}{12} \int_\Omega  \big| \hspace{0.8pt}\nabla \big( \Delta \phi + \frac{h^2}{2} \sin 2 \phi \big) \hspace{0.5pt}\big|^2 + c_* \hspace{0.5pt}  \big\| \hspace{0.5pt}\nabla u \hspace{0.5pt}\big\|^2_{L^4}  \hspace{0.5pt}\big\| \hspace{0.5pt}\nabla \phi \hspace{0.5pt}\big\|^2_{L^4}\\[2mm] 
    &\hspace{80pt} \leq \hspace{1.5pt} \frac{1}{12} \int_\Omega  \big| \hspace{0.8pt}\nabla \big( \Delta \phi + \frac{h^2}{2} \sin 2 \phi \big) \hspace{0.5pt}\big|^2 + c_* \left[\hspace{0.5pt}\| \hspace{0.3pt}u\hspace{0.3pt}\|^2_{L^2} + \| \hspace{0.3pt}u\hspace{0.3pt}\|^{\frac{1}{2}}_{H^1} \big\| \hspace{0.3pt}\nabla^2 u\hspace{0.3pt}\big\|^{\frac{3}{2}}_{L^2}\hspace{0.5pt}\right] \hspace{0.5pt}\big\| \hspace{0.5pt}\nabla \phi \hspace{0.5pt}\big\|^2_{H^1}.
\end{align*}Using \eqref{Poincare ineq} and Lemmas \ref{Stokes est}-\ref{est nabla 2 phi}, we reduce the last estimate to \begin{align}\label{I2.2.3.1}
    &\int_{\Omega} \partial_3 \big( \Delta \phi + \frac{h^2}{2} \sin 2 \phi \big) \hspace{1pt} \big(\partial_3 u\big) \cdot  \nabla \phi \\[1mm] &\h{50pt}\leq\hspace{1.5pt} \frac{1}{12} \int_\Omega  \big| \hspace{0.8pt}\nabla \big( \Delta \phi + \frac{h^2}{2} \sin 2 \phi \big) \hspace{0.5pt}\big|^2 + A \hspace{1pt}\mathcal{Q}\big(\mathcal{E} + A \big) + A^{\frac{1}{4}}\mathcal{Q}\big(\mathcal{E} + A \big) \hspace{1pt} \big\| \hspace{0.3pt}S\hspace{0.3pt} u\hspace{0.3pt}\big\|^{\frac{3}{2}}_{L^2}\nonumber\\[2mm]
    &\h{50pt}\leq\hspace{1.5pt} \frac{1}{12} \int_\Omega  \big| \hspace{0.8pt}\nabla \big( \Delta \phi + \frac{h^2}{2} \sin 2 \phi \big) \hspace{0.5pt}\big|^2 + \frac{1}{12} \int_\Omega \big|\hspace{0.5pt} S\hspace{0.3pt}u\hspace{0.5pt}\big|^2 + A\hspace{1pt}\mathcal{Q}\big(\mathcal{E} + A \big). \nonumber
\end{align}Here, $\mathcal{Q}$ is a single-variable non-constant polynomial whose coefficients are all non-negative. This polynomial can vary in different estimates below. \vspace{0.2pc}

The second integrand in the last line of \eqref{Exp of 2nd} can be estimated by  
\begin{align*}
    \int_{\Omega} \big( \Delta \phi + \frac{h^2}{2} \sin 2 \phi \big)\hspace{1pt}  \partial_{33} u  \cdot  \nabla \phi &\hspace{1.5pt}\leq\hspace{1.5pt} \big\| \hspace{0.5pt}\nabla^2 u\hspace{0.5pt}\big\|_{L^2} \left(\int_{\Omega} \big( \Delta \phi + \frac{h^2}{2} \sin 2 \phi \big)^2 \hspace{1pt}\big|\hspace{0.5pt}\nabla \phi\hspace{0.5pt}\big|^2\right)^{\frac{1}{2}} \\[2mm]
    &\hspace{1.5pt}\leq\hspace{1.5pt} \frac{1}{12} \int_\Omega \big|\hspace{0.5pt} S\hspace{0.3pt}u\hspace{0.5pt}\big|^2 + c_* \left(\int_{\Omega} \big( \Delta \phi + \frac{h^2}{2} \sin 2 \phi \big)^4\right)^{\frac{1}{2}} \left( \int_\Omega \big|\hspace{0.5pt}\nabla \phi\hspace{0.5pt}\big|^4 \right)^{\frac{1}{2}}.
\end{align*}We apply the Gagliardo-Nirenberg inequality to estimate $\Delta \phi + \frac{h^2}{2} \sin 2 \phi$ as follows: \begin{align}\label{gn del phi}
    &\big\| \hspace{0.5pt}\Delta \phi + \frac{h^2}{2} \sin 2 \phi \hspace{0.5pt}\big\|_{L^4}  \lesssim\hspace{1.5pt} \big\| \hspace{0.5pt}\nabla \hspace{1pt}\big( \Delta \phi + \frac{h^2}{2} \sin 2 \phi \big) \hspace{0.5pt}\big\|^{\frac{3}{4}}_{L^2} \hspace{1.5pt}\big\| \hspace{0.5pt}\Delta \phi + \frac{h^2}{2} \sin 2 \phi \hspace{0.5pt}\big\|^{\frac{1}{4}}_{L^2} + \big\| \hspace{0.5pt}\Delta \phi + \frac{h^2}{2} \sin 2 \phi \hspace{0.5pt}\big\|_{L^2}. 
\end{align}The last two estimates and the fact that  $\big\| \nabla \phi \big\|^2_{L^4} \lesssim \big\| \nabla \phi \big\|^2_{L^2} + \big\| \nabla^2 \phi \big\|^2_{L^2} \lesssim 1 + \mathcal{E} + A$ yield \begin{align}\label{I2.2.3.2}
    &\int_{\Omega} \big( \Delta \phi + \frac{h^2}{2} \sin 2 \phi \big)\hspace{1pt}  \partial_{33} u  \cdot  \nabla \phi \leq\hspace{1.5pt}\frac{1}{12} \int_\Omega \big|\hspace{0.5pt} S\hspace{0.3pt}u\hspace{0.5pt}\big|^2 \nonumber\\[2mm] &\h{20pt} + \mathcal{Q}\big(\mathcal{E} + A \big)
\left[\hspace{1pt}\big\| \hspace{0.5pt}\nabla \hspace{1pt}\big( \Delta \phi + \frac{h^2}{2} \sin 2 \phi \big) \hspace{0.5pt}\big\|^{\frac{3}{2}}_{L^2} \hspace{1.5pt}\big\| \hspace{0.5pt}\Delta \phi + \frac{h^2}{2} \sin 2 \phi \hspace{0.5pt}\big\|^{\frac{1}{2}}_{L^2} + \big\| \hspace{0.5pt}\Delta \phi + \frac{h^2}{2} \sin 2 \phi \hspace{0.5pt}\big\|^2_{L^2} \right] \nonumber\\[2mm]
&\hspace{20pt}\leq\hspace{1.5pt} \frac{1}{12} \int_\Omega \big|\hspace{0.5pt} S\hspace{0.3pt}u\hspace{0.5pt}\big|^2 + \frac{1}{12} \int_\Omega  \big| \hspace{0.8pt}\nabla \big( \Delta \phi + \frac{h^2}{2} \sin 2 \phi \big) \hspace{0.5pt}\big|^2 + A \hspace{1pt}\mathcal{Q}\big(\mathcal{E} + A \big).\end{align}\vspace{0.2pc}

As for the third integrand in the last line of \eqref{Exp of 2nd}, it can be controlled by \begin{align*}
    \int_{\Omega}   \big( \Delta \phi + \frac{h^2}{2} \sin 2 \phi \big)\hspace{1pt}  \partial_3 u  \cdot  \nabla \partial_3 \phi \hspace{1.5pt}\leq\hspace{1.5pt} \big\| \hspace{0.5pt}\nabla^2 \phi \hspace{0.5pt}\big\|_{L^2} \hspace{.5pt}\big\| \hspace{0.5pt}\nabla u \hspace{0.5pt}\big\|_{L^4}  \hspace{.5pt}\big\| \hspace{0.5pt}\Delta \phi + \frac{h^2}{2} \sin 2 \phi \hspace{0.5pt}\big\|_{L^4}.
\end{align*}Recall the Gagliardo-Nirenberg estimates in \eqref{gn for nabla u} and \eqref{gn del phi} and apply Lemmas \ref{Stokes est}-\ref{est nabla 2 phi}. The above estimate can be reduced to \begin{align*}
    &\int_{\Omega}   \big( \Delta \phi + \frac{h^2}{2} \sin 2 \phi \big)\hspace{1pt}  \partial_3 u  \cdot  \nabla \partial_3 \phi\\[1mm] &\h{30pt}\lesssim\hspace{1pt} A^{\frac{1}{4}}\big(1 + \mathcal{E} + A \big)^{\frac{1}{2}} \hspace{.5pt}\left[\hspace{1pt} A^{\frac{3}{8}} + \big\| \hspace{0.3pt}S\hspace{0.3pt} u\hspace{0.3pt}\big\|^{\frac{3}{4}}_{L^2}\right] \hspace{.5pt}\left[ A^{\frac{3}{8}} + \big\| \hspace{0.5pt}\nabla \hspace{1pt}\big( \Delta \phi + \frac{h^2}{2} \sin 2 \phi \big) \hspace{0.5pt}\big\|^{\frac{3}{4}}_{L^2} \right].
\end{align*}By Young's inequality, we obtain from the last estimate that \begin{align}\label{I2.2.3.3}
    &\int_{\Omega}   \big( \Delta \phi + \frac{h^2}{2} \sin 2 \phi \big)\hspace{1pt}  \partial_3 u  \cdot  \nabla \partial_3 \phi \nonumber\\[1mm] &\hspace{30pt}\leq \hspace{1.5pt}\frac{1}{12} \int_\Omega \big|\hspace{0.5pt} S\hspace{0.3pt}u\hspace{0.5pt}\big|^2 + \frac{1}{12} \int_\Omega  \big| \hspace{0.8pt}\nabla \big( \Delta \phi + \frac{h^2}{2} \sin 2 \phi \big) \hspace{0.5pt}\big|^2 + A \hspace{0.5pt}\mathcal{Q}\big(\mathcal{E} + A \big).
\end{align}

Combining this estimate, \eqref{I2.2.3.1}, and \eqref{I2.2.3.2}, we estimate the right-hand side of \eqref{Exp of 2nd} and get 
\begin{align*}
- &\int_{\text{H}} \big( \Delta \phi + \frac{h^2}{2} \sin 2 \phi \big) \partial_3 \big(u \cdot \nabla \phi\big) \hspace{1.5pt}\leq\hspace{1.5pt} \frac{1}{4} \int_\Omega \big|\hspace{0.5pt} S\hspace{0.3pt}u\hspace{0.5pt}\big|^2 + \frac{1}{4} \int_\Omega  \big| \hspace{0.8pt}\nabla \big( \Delta \phi + \frac{h^2}{2} \sin 2 \phi \big) \hspace{0.5pt}\big|^2 + A\hspace{0.5pt}\mathcal{Q}\big(\mathcal{E} + A \big).
\end{align*}In light of this estimate, \eqref{basic dep of I2} and \eqref{I.2.2}, it follows 
\begin{align}\label{est of I2}
    I_2 \hspace{1.5pt}\leq\hspace{1.5pt} \frac{1}{4} \int_\Omega \big|\hspace{0.5pt} S\hspace{0.3pt}u\hspace{0.5pt}\big|^2 - \frac{2}{3} \int_\Omega  \big| \hspace{0.8pt}\nabla \big( \Delta \phi + \frac{h^2}{2} \sin 2 \phi \big) \hspace{0.5pt}\big|^2 + A \hspace{0.5pt}\mathcal{Q}\big(\mathcal{E} + A \big).
\end{align}\vspace{0.2pc}

\noindent\textbf{Estimate for $I_3$.} The $I_3$-term can be split into \begin{align}\label{split of I3}
    I_3  &= \int_\Omega \big( \Delta \phi + \frac{h^2}{2} \sin 2 \phi \big) \hspace{1pt} \big(S\hspace{0.3pt}u\big) \cdot \nabla \phi -  \int_\Omega \big( \Delta \phi + \frac{h^2}{2} \sin 2 \phi \big) \hspace{1pt} \nabla p  \cdot \nabla \phi \nonumber\\[2mm]
    &- 2 \int_\Omega \big( \Delta \phi + \frac{h^2}{2} \sin 2 \phi \big) \hspace{1pt} \partial_i u \cdot \nabla \partial_i \phi - \int_\Omega \big( \Delta \phi + \frac{h^2}{2} \sin 2 \phi \big) \hspace{1pt} u \cdot \nabla 
 \Delta \phi.
\end{align}Note that the first term on the right-hand side of \eqref{split of I3} can be cancelled by the last term in \eqref{eq of I1}. We therefore only consider the rest three terms on the right-hand side of \eqref{split of I3}. \vspace{0.2pc}
 
By H\"older and Sobolev inequalities, the second term on the right-hand side of \eqref{split of I3} satisfies  
\begin{align*}
-  \int_\Omega \big( \Delta \phi + \frac{h^2}{2} \sin 2 \phi \big) \hspace{1pt} \nabla p  \cdot \nabla \phi &\hspace{1.5pt}\leq \hspace{1.5pt} \big\| \hspace{0.5pt}\nabla p \hspace{0.5pt}\big\|_{L^2} \big\| \hspace{0.5pt}\nabla \phi \hspace{0.5pt}\big\|_{L^4}  \hspace{.5pt}\big\| \hspace{0.5pt}\Delta \phi + \frac{h^2}{2} \sin 2 \phi \hspace{0.5pt}\big\|_{L^4}\\[1mm] &\hspace{1.5pt}\lesssim\hspace{1.5pt}\big\| \hspace{0.5pt}\nabla p \hspace{0.5pt}\big\|_{L^2} \big\| \hspace{0.5pt}\nabla \phi \hspace{0.5pt}\big\|_{H^1}  \hspace{.5pt}\big\| \hspace{0.5pt}\Delta \phi + \frac{h^2}{2} \sin 2 \phi \hspace{0.5pt}\big\|_{L^4}. 
\end{align*}Applying Lemmas \ref{Stokes est}-\ref{est nabla 2 phi} and the estimate \eqref{gn del phi}, we can bound the right-hand side above and get \begin{align*}
 -  \int_\Omega \big( \Delta \phi + \frac{h^2}{2} \sin 2 \phi \big) \hspace{1pt} \nabla p  \cdot \nabla \phi &\hspace{1.5pt}\lesssim \hspace{1.5pt} \big\| \hspace{0.5pt}S \hspace{0.3pt}u \hspace{0.5pt}\big\|_{L^2} \big( 1 + \mathcal{E} + A \big)^{\frac{1}{2}} \left[  A^{\frac{1}{2}} + A^{\frac{1}{8}} \big\| \hspace{0.5pt}\nabla \hspace{1pt}\big( \Delta \phi + \frac{h^2}{2} \sin 2 \phi \big) \hspace{0.5pt}\big\|^{\frac{3}{4}}_{L^2}  \right] \\[2mm]
 &\hspace{-30pt}\leq\hspace{1.5pt} A^{\frac{1}{2}}\mathcal{Q}\big(\mathcal{E} + A \big) \hspace{1pt} \big\| \hspace{0.5pt}S \hspace{0.3pt}u \hspace{0.5pt}\big\|_{L^2} + A^{\frac{1}{8}}\mathcal{Q}\big(\mathcal{E} + A \big) \hspace{1pt} \big\| \hspace{0.5pt}S \hspace{0.3pt}u \hspace{0.5pt}\big\|_{L^2} \hspace{1pt}\big\| \hspace{0.5pt}\nabla \hspace{1pt}\big( \Delta \phi + \frac{h^2}{2} \sin 2 \phi \big) \hspace{0.5pt}\big\|^{\frac{3}{4}}_{L^2}.
\end{align*}Young's inequality then infers\begin{align}\label{I.3.1}
& -  \int_\Omega \big( \Delta \phi + \frac{h^2}{2} \sin 2 \phi \big) \hspace{1pt} \nabla p  \cdot \nabla \phi \nonumber\\[1mm]
&\h{30pt}\leq\hspace{1.5pt}\frac{1}{24} \int_\Omega \big|\hspace{0.5pt} S\hspace{0.3pt}u\hspace{0.5pt}\big|^2 + \frac{1}{24} \int_\Omega  \big| \hspace{0.8pt}\nabla \big( \Delta \phi + \frac{h^2}{2} \sin 2 \phi \big) \hspace{0.5pt}\big|^2 + A \hspace{0.5pt}\mathcal{Q}\big(\mathcal{E} + A \big).
\end{align}\vspace{0.2pc}

The third term on the right-hand side of \eqref{split of I3} can be controlled in a way similar as \eqref{I2.2.3.3}. We give the estimate in the following without the proof. \begin{align}\label{I.3.2}
    & - 2 \int_\Omega \big( \Delta \phi + \frac{h^2}{2} \sin 2 \phi \big) \hspace{1pt} \partial_i u \cdot \nabla \partial_i \phi \nonumber\\[1mm] &\h{30pt}\leq \hspace{1.5pt}\frac{1}{24} \int_\Omega \big|\hspace{0.5pt} S\hspace{0.3pt}u\hspace{0.5pt}\big|^2 + \frac{1}{24} \int_\Omega  \big| \hspace{0.8pt}\nabla \big( \Delta \phi + \frac{h^2}{2} \sin 2 \phi \big) \hspace{0.5pt}\big|^2 + A\hspace{0.5pt}\mathcal{Q}\big(\mathcal{E} + A \big).
\end{align}\vspace{0.2pc}

The last term on the right-hand side of \eqref{split of I3} can be calculated by \begin{align*}
    & - \int_\Omega \big( \Delta \phi + \frac{h^2}{2} \sin 2 \phi \big) \hspace{1pt} u \cdot \nabla 
 \Delta \phi \\[1mm] 
 & \h{30pt} = - \frac{1}{2} \int_\Omega   u \cdot \nabla \hspace{1pt}\big(  
 \Delta  \phi + \frac{h^2}{2} \sin 2 \phi \big)^2  +  h^2\int_\Omega \cos 2 \phi \hspace{1pt}\big( \Delta \phi + \frac{h^2}{2} \sin 2 \phi \big) \hspace{1pt} u \cdot \nabla \phi\\[1mm]
 & \h{30pt}= h^2\int_\Omega \cos 2 \phi \hspace{1pt}\big( \Delta \phi + \frac{h^2}{2} \sin 2 \phi \big) \hspace{1pt} u \cdot \nabla \phi.
\end{align*}In the second equality above, we use $\mathrm{div} \hspace{0.5pt} u = 0$ in $\Omega$ and $u = 0$ on $\mathrm{H} \cup \mathrm P$. \vspace{0.2pc}

We now apply the last identity, \eqref{I.3.1} and \eqref{I.3.2} to the right-hand side of \eqref{split of I3}. It turns out \begin{align}\label{est of I3}
I_3  \hspace{1.5pt}\leq\hspace{1.5pt} &\frac{1}{12} \int_\Omega \big|\hspace{0.5pt} S\hspace{0.3pt}u\hspace{0.5pt}\big|^2  + \frac{1}{12} \int_\Omega  \big| \hspace{0.8pt}\nabla \big( \Delta \phi + \frac{h^2}{2} \sin 2 \phi \big) \hspace{0.5pt}\big|^2 + A\hspace{0.5pt}\mathcal{Q}\big(\mathcal{E} + A \big)\\[2mm]
&+ \int_\Omega \big( \Delta \phi + \frac{h^2}{2} \sin 2 \phi \big) \hspace{1pt} \big(S\hspace{0.3pt}u\big) \cdot \nabla \phi + h^2\int_\Omega \cos 2 \phi \hspace{1pt}\big( \Delta \phi + \frac{h^2}{2} \sin 2 \phi \big) \hspace{1pt} u \cdot \nabla \phi. \nonumber
\end{align}\vspace{0.2pc}

\noindent\textbf{Estimate for $I_4$.} According to the definition of $A(t)$, the $I_4$-term can be estimated as follows:
\begin{align}\label{est of I4}
    I_4 \hspace{1.5pt}\leq\hspace{1.5pt}  h^2 \hspace{0.3pt}A - h^2 \int_\Omega \cos 2 \phi 
 \hspace{1pt}\big( \Delta \phi + \frac{h^2}{2} \sin 2 \phi \big)\hspace{1pt} u \cdot \nabla \phi.
\end{align}\vspace{0.1pc}

The proof is finished by applying \eqref{Inq I.1.}, \eqref{est of I2}, \eqref{est of I3}-\eqref{est of I4} to the right-hand side of \eqref{high ord energy iden}.
\end{proof}

\subsection{The \texorpdfstring{$\omega$}{TEXT}-limit set of global classical hydrodynamic flow}\vspace{0.2pc}

We begin to study the asymptotic limit of the global classical solution of $\mathrm{IBVP}$. In this section, we show that the $\omega$-limit set of the global classical solution is a singleton. 

\begin{lem}\label{limit and unif bound}
Suppose $\big(u, \phi\big)$ is a classical solution of $\mathrm{IBVP}$ on $\big[\hspace{0.5pt}T_0, \infty \big)$ for some $T_0 \geq 0$. Then, 
\begin{align} \label{H1L2limit}
    \lim_{t \hspace{.5pt}\to \hspace{.5pt} \infty} \left\{ \hspace{1pt} \| \hspace{0.5pt} u \hspace{0.5pt}\|_{H^1} + \big\| \hspace{0.5pt}\Delta \phi + \frac{h^2}{2} \sin 2 \phi \hspace{0.5pt}\big\|_{L^2} \right\} = 0.
\end{align}Moreover, the following uniform boundedness holds:
\begin{equation} \label{H1H2bdd}
    \sup_{t \hspace{0.5pt} \geq \hspace{0.5pt} T_0} \Big\{ \hspace{1pt} \|\hspace{0.5pt} u  \hspace{0.5pt}\|_{H^1} + \|\hspace{0.5pt} \phi  \hspace{0.5pt}\|_{H^2} \Big\} < \infty.
\end{equation}
\end{lem}

\begin{proof}[\bf Proof]
Since $\phi = 0$ on P, same proof for \eqref{Poincare ineq} induces $\| \hspace{0.5pt}\phi\hspace{.5pt}\|_{L^2} \lesssim \| \hspace{0.5pt}\nabla \phi\hspace{.5pt}\|_{L^2}$. By the basic energy law in \eqref{basic energy est}, it turns out 
\begin{equation}\label{unif est 1}
    \sup_{t \hspace{1pt}\geq \hspace{1pt} T_0} \Big\{ \hspace{1pt} \|\hspace{0.5pt} u \hspace{0.5pt}\|^2_{L^2} + \big\|\hspace{0.5pt} \phi \hspace{0.5pt}\big\|_{H^1}^2 \Big\} + \int_{T_0}^{\infty} A(s) \hspace{1.5pt}\mathrm{d} s \hspace{1.5pt}\lesssim\hspace{1.5pt} \mathcal{E}(T_0).
\end{equation}
Taking into account \eqref{unif est 1} and Lemma \ref{high ord est}, we apply Lemma 6.2.1 and Remark 6.2.1 in \cite{Zheng2004} to get  
\begin{equation*}
    \lim_{t \hspace{1pt}\to \hspace{1pt} \infty} A(t) = \lim_{t \hspace{1pt}\to \hspace{1pt} \infty} \int_{\Omega} \big|\hspace{0.5pt}\nabla u \hspace{0.5pt}\big|^2 + \big|\hspace{0.5pt}\Delta \phi + \frac{h^2}{2} \sin 2 \phi \hspace{0.5pt}\big|^2 = 0.
\end{equation*}\eqref{H1L2limit} then follows by the above limit and \eqref{Poincare ineq}. Moreover, the uniform boundedness result in \eqref{H1H2bdd} holds with the last limit, \eqref{unif est 1}, and Lemma \ref{est nabla 2 phi}.
\end{proof}

Due to Lemma \ref{limit and unif bound}, if $\big(u, \phi\big)$ is a global classical solution of IBVP, its $\omega$-limit set, denoted by $\omega\big(u, \phi\big)$, is a subset of $\Sigma$, where 
\begin{align*}
        \Sigma := \Big\{ \hspace{1pt} \big(0, \phi_\infty\big) \colon \text{$\phi_\infty$ is a strong solution of (\ref{SG})}\hspace{1pt}\Big\}.
    \end{align*}
In addition, we have
\begin{lem}\label{uniq limit}
Suppose $(u, \phi)$ is a classical solution of $\mathrm{IBVP}$ on $[\hspace{0.5pt}T_0, \infty )$ for some $T_0 \geq 0$. Then $\omega\left(u, \phi\right)$ is a singleton. There is a constant $T_\rho > T_0$ such that
    \begin{align}\label{dist phi and phi infty}
        \big\|\hspace{1pt} \phi(t) - \phi_{\infty} \hspace{1pt}\big\|_{H^1} \hspace{1pt}\leq\hspace{1pt} \rho \hspace{20pt}\text{for all $t \geq T_\rho$.}
    \end{align} Here, $ (\hspace{0.2pt}0, \phi_\infty) \in \omega\left(u, \phi\right) \subset \Sigma$. The positive constant $\rho$ is given in Theorem \ref{LS general}. 
\end{lem}

\begin{proof}[\bf Proof]
Given $(0, \phi_\infty) \in \omega\left(u, \phi\right)$, there is a strictly increasing time sequence $\big\{ t_m \big\}$, which diverges to $\infty$ as $m \to \infty$, such that
\begin{align}\label{strong H1 conv}
    \lim_{m \hspace{1pt}\to\hspace{1pt} \infty} \|\hspace{1pt} u(t_m) \hspace{1pt}\|_{H^1} + \big\|\hspace{1pt} \phi (t_m) - \phi_{\infty} \hspace{1pt}\big\|_{H^1} = 0.
\end{align}Without loss of generality, we can assume $t_m > T_0$ for all $m$. Using \eqref{strong H1 conv}, we obtain
\begin{align}\label{conv of energ}
    \lim_{m \hspace{1pt}\to \hspace{1pt} \infty} \cE(t_m) = E\hspace{0.5pt}\big[\hspace{0.5pt}\phi_{\infty}\hspace{0.5pt}\big].
\end{align}

Assume there is  $T \geq T_0$ such that $\cE(T) = E\hspace{0.5pt}[\hspace{0.5pt}\phi_{\infty}\hspace{0.5pt}]$. In light of \eqref{conv of energ} and the basic energy law \eqref{energy}, we have $\cE(t) = E\hspace{0.5pt}[\hspace{0.5pt}\phi_{\infty}\hspace{0.5pt}]$ for all $t \in [\hspace{0.5pt}T, \infty )$. Moreover,
\begin{align}\label{id beyond T}
    \big\|\hspace{0.5pt} \nabla u \hspace{0.5pt}\big\|_{L^2} + \big\| \hspace{0.5pt}\Delta \phi + \frac{h^2}{2} \sin 2 \phi\hspace{0.5pt}\big\|_{L^2} = 0 \hspace{20pt}\text{on $(T, \infty)$.}
\end{align}
It turns out $u \equiv 0$ on $\Omega \times (T, \infty)$. According to the equation of $\phi$ in \eqref{sEL}, \eqref{id beyond T} further implies $\partial_t \phi \equiv 0$ on $\Omega \times (T, \infty)$. We then have $\phi = \phi_\infty$ on $\Omega \times (T, \infty)$. The lemma follows in this case.\vspace{0.2pc}

In the rest of the proof we assume $\cE(t) > E\hspace{0.5pt}[\hspace{0.5pt}\phi_{\infty}\hspace{0.5pt}]$ for all $t > T_0$. $\epsilon$ is an arbitrary positive number no more than $\rho$. By \eqref{strong H1 conv}, there is a natural number $N_\epsilon$, such that
\begin{align*}
    \big\|\hspace{1pt} \phi (t_m) - \phi_{\infty} \hspace{1pt}\big\|_{H^1} \hspace{1pt}\leq\hspace{1pt} \frac{\epsilon}{2} \hspace{20pt}\text{for all $m \geq N_\epsilon$.}
\end{align*}
For $m \geq N_\epsilon$, we define
\begin{align*}
    t_{m, *} := \sup \Big\{ \hspace{1pt} t \colon t \geq t_m \hspace{3pt}\text{and satisfies}\hspace{3pt} \big\|\hspace{1pt} \phi(s) - \phi_{\infty} \hspace{1pt}\big\|_{H^1} \hspace{0.5pt}\leq\hspace{0.5pt} \epsilon \hspace{4pt}\text{for all} \hspace{3pt} s \in [\hspace{0.5pt}t_m, t\hspace{0.5pt}] \hspace{1pt}\Big\}.
\end{align*}We claim there exists a natural number $M_{\epsilon} \geq N_\epsilon$ such that $t_{M_{\epsilon}, *} = \infty$. Therefore, $\omega(u, \phi)$ is singleton by the arbitrariness of $\epsilon$. Letting $\epsilon = \rho$, we also prove \eqref{dist phi and phi infty}.\vspace{0.2pc}
 
Now we assume $t_{m, *}$ is finite for all $m \geq N_\epsilon$. By \L{}ojasiewicz-Simon inequality in Theorem \ref{LS general}, \begin{align}\label{app of LS}
\big\|\hspace{1pt} E'\hspace{0.5pt}[\hspace{0.5pt}\phi\hspace{0.5pt}] \hspace{1.5pt}\big\|_{\mathscr{V}'} \geq \gamma \hspace{1pt}\big|\hspace{1pt} E\hspace{0.5pt}[\hspace{0.5pt}\phi\hspace{0.5pt}] - E\hspace{0.5pt}[\hspace{0.5pt}\phi_\infty\hspace{0.5pt}] \hspace{1.5pt}\big|^{1 - \theta} \hspace{20pt}\text{on $\big[\hspace{0.5pt}t_m, t_{m, *}\hspace{0.5pt}\big]$}.
\end{align} According to the computation of $E'$ in \eqref{1st deriv of E} and the boundary condition of $\phi$, we can integrate by parts and obtain \begin{align*}
    \big< \hspace{0.5pt}E'[\hspace{0.5pt}\phi\hspace{0.5pt}], \varphi \hspace{0.5pt}\big>_{\mathscr{V}'\times \mathscr{V}} = - \int_{\Omega} \big( \hspace{0.5pt}\Delta \phi + \frac{h^2}{2} \sin 2 \phi\hspace{0.5pt}\big)\hspace{1pt}\varphi \hspace{20pt}\text{for all $\varphi \in \mathscr{V} = H_\text{P}^1(\Omega)$.}
\end{align*}Applying H\"{o}lder and Poincar\'{e} inequalities on the right-hand side above induces \begin{align*}
    \big< \hspace{0.5pt}E'[\hspace{0.5pt}\phi\hspace{0.5pt}], \varphi \hspace{0.5pt}\big>_{\mathscr{V}'\times \mathscr{V}} \hspace{1pt}\leq\hspace{1pt} \big\| \hspace{0.5pt}\Delta \phi + \frac{h^2}{2} \sin 2 \phi\hspace{0.5pt}\big\|_{L^2} \| \hspace{0.5pt}\varphi\hspace{0.5pt}\|_{L^2} \hspace{1pt}\lesssim\hspace{1pt} \big\| \hspace{0.5pt}\Delta \phi + \frac{h^2}{2} \sin 2 \phi\hspace{0.5pt}\big\|_{L^2} \| \hspace{0.5pt}\varphi\hspace{0.5pt}\|_{H^1}.
\end{align*}By duality argument, it implies $$\big\|\hspace{1pt} E'\hspace{0.5pt}[\hspace{0.5pt}\phi\hspace{0.5pt}] \hspace{1.5pt}\big\|_{\mathscr{V}'} \hspace{1pt}\lesssim\hspace{1pt}\big\| \hspace{0.5pt}\Delta \phi + \frac{h^2}{2} \sin 2 \phi\hspace{0.5pt}\big\|_{L^2},$$ which can bound the left-hand side in \eqref{app of LS} from above and infer  \begin{align*}
\big\| \hspace{0.5pt}\Delta \phi + \frac{h^2}{2} \sin 2 \phi\hspace{0.5pt}\big\|_{L^2} \hspace{1pt}\gtrsim\hspace{1pt} \hspace{1pt}\big|\hspace{1pt} E\hspace{0.5pt}[\hspace{0.5pt}\phi\hspace{0.5pt}] - E\hspace{0.5pt}[\hspace{0.5pt}\phi_\infty\hspace{0.5pt}] \hspace{1.5pt}\big|^{1 - \theta} \hspace{20pt}\text{on $\big[\hspace{0.5pt}t_m, t_{m, *}\hspace{0.5pt}\big]$}.
 \end{align*}
Note that $\theta \in (\hspace{0.3pt}0, \frac{1}{2}\hspace{0.5pt}]$. By the last estimate, it follows
\begin{align*}
    \Big\{ \cE(t) - E[\phi_{\infty}] \hspace{0.5pt}\Big\}^{1 - \theta} \hspace{1pt}\lesssim\hspace{1pt} \Big\{ \|\hspace{0.5pt} u \hspace{0.5pt}\|_{L^2}^2 + \big\| \hspace{0.5pt}\Delta \phi + \frac{h^2}{2} \sin 2 \phi\hspace{0.5pt}\big\|_{L^2}^{\frac{1}{1 - \theta}} \hspace{1pt}\Big\}^{1 - \theta} \hspace{1pt}\lesssim\hspace{1pt} \|\hspace{0.5pt} u \hspace{0.5pt}\|_{L^2} + \big\| \hspace{0.5pt}\Delta \phi + \frac{h^2}{2} \sin 2 \phi\hspace{0.5pt}\big\|_{L^2},
\end{align*} where $t \in \big[\hspace{0.5pt}t_m, t_{m, *}\hspace{0.5pt}\big]$. Applying Poincar\'{e} inequality \eqref{Poincare ineq} then yields
\begin{align}\label{LS application}
    \Big\{ \cE(t) - E[\phi_{\infty}] \hspace{0.5pt}\Big\}^{1 - \theta}  \hspace{1pt}\lesssim\hspace{1pt} \|\hspace{0.5pt} \nabla u \hspace{0.5pt}\|_{L^2} + \big\| \hspace{0.5pt}\Delta \phi + \frac{h^2}{2} \sin 2 \phi\hspace{0.5pt}\big\|_{L^2}  \hspace{20pt}\text{for all $t \in \big[\hspace{0.5pt}t_m, t_{m, *}\hspace{0.5pt}\big]$}.
\end{align}
According to the differential version of the energy identity \eqref{basic energy est} and the above estimate, 
\begin{align*}
    - \frac{\mathrm{d}}{ \mathrm{d} \hspace{0.3pt}t} \hspace{1pt} \Big\{ \cE(t) - E[\phi_{\infty}] \hspace{.8pt}\Big\}^{\theta} &=   \theta \hspace{0.5pt}A \hspace{1pt}\Big\{ \cE(t) - E[\phi_{\infty}] \hspace{0.8pt}\Big\}^{\theta - 1}  \gtrsim\hspace{1pt} A^{\frac{1}{2}}\hspace{20pt}\text{for all $t \in \big[\hspace{0.5pt}t_m, t_{m, *}\hspace{0.5pt}\big]$}.
\end{align*}We integrate the above estimate from $t_m$ to $t$. It induces
\begin{align}\label{bd of Ahalf}
    \Big\{ \cE(t) - E[\phi_{\infty}] \hspace{0.5pt}\Big\}^{\theta} + C_0 \int_{t_m}^t A^{\frac{1}{2}} \hspace{1pt}\leq\hspace{1pt} \Big\{ \cE(t_m) - E[\phi_{\infty}] \hspace{0.5pt}\Big\}^{\theta} \hspace{20pt}\text{for any $t \in \big[\hspace{0.5pt} t_m, t_{m, *} \hspace{0.5pt}\big]$.}
\end{align}
The $L^2$-norm of $\phi(t_{m, *}) - \phi_\infty$ can be controlled by the triangle inequality as follows:
\begin{align}\label{tn*}
    \big\|\hspace{0.5pt} \phi(t_{m, *}) - \phi_{\infty} \hspace{0.5pt}\big\|_{L^2} &\hspace{1pt}\leq\hspace{1pt} \big\|\hspace{0.5pt} \phi(t_m) - \phi_{\infty} \hspace{0.5pt}\big\|_{L^2} + \big\|\hspace{0.5pt} \phi(t_{m, *}) - \phi(t_m) \hspace{0.5pt}\big\|_{L^2} \nonumber\\[1mm]&\hspace{1pt}\leq\hspace{1pt} \big\|\hspace{0.5pt} \phi(t_m) - \phi_{\infty} \hspace{0.5pt}\big\|_{L^2} + \int_{t_m}^{t_{m, *}} \big\|\hspace{0.5pt} \partial_{s} \phi \hspace{0.5pt}\big\|_{L^2}.
\end{align}Recalling the equation satisfied by $\phi$ in \eqref{sEL}, we can apply H\"{o}lder and Sobolev inequalities to get \begin{align*}
    \big\|\hspace{0.5pt} \partial_{s} \phi \hspace{0.5pt}\big\|_{L^2} \hspace{1pt}\leq\hspace{1pt}\big\|\hspace{0.5pt} u \cdot \nabla \phi \hspace{0.5pt}\big\|_{L^2} + \big\| \hspace{0.5pt}\Delta \phi + \frac{h^2}{2} \sin 2 \phi\hspace{0.5pt}\big\|_{L^2} \hspace{1pt}\leq\hspace{1pt}A^{\frac{1}{2}} + \|\hspace{0.5pt} u \hspace{0.5pt}\|_{H^1} \big\|\hspace{0.5pt}\nabla \phi \hspace{0.5pt}\big\|_{H^1},\hspace{15pt}\text{for all $s \in \big[\hspace{0.5pt}t_m, t_{m, *}\hspace{0.5pt}\big]$}.
\end{align*}By  \eqref{Poincare ineq} and the uniform boundedness of the $H^2$-norm of $\phi$ in \eqref{H1H2bdd}, it holds $$
    \big\|\hspace{0.5pt} \partial_{s} \phi \hspace{0.5pt}\big\|_{L^2} \hspace{1pt}\lesssim\hspace{1pt}\big(A(s)\big)^{\frac{1}{2}}, \hspace{20pt}\text{for all $s \in \big[\hspace{0.5pt}t_m, t_{m, *}\hspace{0.5pt}\big]$}.$$ By plugging this estimate to the right-hand side of \eqref{tn*} and then using \eqref{bd of Ahalf},  \begin{align*}
        \big\|\hspace{0.5pt} \phi(t_{m, *}) - \phi_{\infty} \hspace{0.5pt}\big\|_{L^2}  \hspace{1pt}\lesssim\hspace{1pt} \big\|\hspace{0.5pt} \phi(t_m) - \phi_{\infty} \hspace{0.5pt}\big\|_{L^2} + \int_{t_m}^{t_{m, *}} A^{\frac{1}{2}} \hspace{1pt}\lesssim\hspace{1pt} \big\|\hspace{0.5pt} \phi(t_m) - \phi_{\infty} \hspace{0.5pt}\big\|_{L^2} + \Big\{ \cE(t_m) - E[\phi_{\infty}] \hspace{0.5pt}\Big\}^{\theta}.
    \end{align*}
Taking into account the limits \eqref{strong H1 conv} and \eqref{conv of energ}, we have 
\begin{align}\label{L2 conv}
    \lim_{m \hspace{1pt}\to\hspace{1pt} \infty} \big\| \hspace{0.5pt}\phi(t_{m, *}) - \phi_{\infty} \hspace{0.5pt}\big\|_{L^2} = 0.
\end{align}Let $t = t_{m, *}$ in \eqref{bd of Ahalf}. It induces \begin{align}\label{energy conv at tn*}
    \cE(t_{m, *}) - E[\phi_{\infty}] = \frac{1}{2}  \int_\Omega \big| \hspace{0.5pt}u(t_{m, *}) \hspace{0.5pt}\big|^2 + E\big[ \phi(t_{m, *}) \big] - E[\phi_{\infty}] \longrightarrow 0 \hspace{10pt}\text{as $m \to \infty$.}
\end{align}Using the convergence of $u$ in \eqref{H1L2limit}, the compactness of the trace operator from $H_\text{P}^1(\Omega)$ to $L^2(H)$ and \eqref{L2 conv}, we obtain from \eqref{energy conv at tn*} that $$\big\| \hspace{0.5pt}\nabla \phi(t_{m, *}) \hspace{0.5pt}\big\|_{L^2} \longrightarrow \big\| \hspace{0.5pt}\nabla \phi_\infty \hspace{0.5pt}\big\|_{L^2} \hspace{20pt} \text{as $m \to \infty$.}$$ Then,  $\phi(t_{m, *})$ converges to $\phi_\infty$ strongly in $H^1(\Omega)$ as $m \to \infty$. Hence, $\big\|\hspace{0.5pt} \phi(t_{m, *}) - \phi_{\infty} \hspace{0.5pt}\big\|_{H^1} < \epsilon$ for sufficiently large $m$, which violates the definition of $t_{m, *}$. The proof is completed. 
\end{proof}

\subsection{Convergence rate to the asymptotic limit}

Continuing the last section, we now investigate the rate of convergence to the equilibrium solution $(0, \phi_\infty)$ as $t \to \infty$, along the global classical solution $(u, \phi)$ of $\mathrm{IBVP}$. Our main result is
\begin{prop}\label{convergence rate}
Suppose $(u, \phi)$ is a classical solution of $\mathrm{IBVP}$ on $\big[\hspace{0.5pt}T_0, \infty \big)$ for some $T_0 \geq 0$. $(0, \phi_\infty)$ is its unique long-time asymptotic limit. If $\theta$ is the \L{}ojasiewicz-Simon exponent in Theorem \ref{LS general} associated with the critical point $\phi_\infty$, then we have  
    \begin{itemize}
        \item[$\mathrm{(1).}$] If $0 < \theta < \frac{1}{2}$, then
        \begin{align*}
            \big\|\hspace{0.5pt} u(t) \hspace{0.5pt}\big\|_{H^1} + \big\|\hspace{0.5pt} \phi(t) - \phi_{\infty} \hspace{0.5pt}\big\|_{H^2} \hspace{1pt}\lesssim\hspace{1pt} (1 + t)^{- \frac{\theta}{1 - 2 \theta}} \hspace{20pt} \text{for all $t > T_0$.}
        \end{align*}

        \item[$\mathrm{(2).}$] If $\theta = \frac{1}{2}$, then for some positive constant $\kappa$, it holds
        \begin{align*}
            \big\|\hspace{0.5pt} u(t) \hspace{0.5pt}\big\|_{H^1} + \big\|\hspace{0.5pt} \phi(t) - \phi_{\infty} \hspace{0.5pt}\big\|_{H^2} \hspace{1pt}\lesssim\hspace{1pt} e^{- \kappa t} \hspace{20pt} \text{for all  $t > T_0$.}
        \end{align*}Here, $\kappa$ is a constant depending on $h$, $L_{\mathrm{H}}$, $\Omega$, $T_0$, and the value of $(u, \phi)$ at $T_0$.
    \end{itemize}
\end{prop}

\begin{proof}[\bf Proof]
Throughout the proof, we assume $\cE(t) > E[\phi_{\infty}]$ for all $t \geq T_0$. We also assume $t \geq T_\rho$ such that \eqref{dist phi and phi infty} in Lemma \ref{uniq limit} holds. The proof is divided into 3 steps.\vspace{0.3pc}

\noindent\textbf{Step 1: $L^2$-estimate of $\phi$.}\\[2mm]  Same proof for \eqref{LS application} induces \begin{align*}
     \cE(t) - E[\phi_{\infty}]  \hspace{1pt}\lesssim\hspace{1pt} A^{\frac{1}{2 - 2 \theta}}  \hspace{20pt}\text{for all $t \in \big[\hspace{0.5pt}T_\rho, \infty\hspace{0.5pt}\big)$}.
\end{align*}By the above estimate and the differential version of the energy identity in Lemma \ref{basic energy est},
\begin{align*}
    0 = \frac{\mathrm{d}}{\mathrm{d}\hspace{0.3pt}t} \hspace{0.2pt} \Big\{ \cE(t) - E[\phi_{\infty}] \hspace{0.5pt}\Big\} + A(t) \hspace{0.5pt}\geq\hspace{0.5pt} \frac{\mathrm{d}}{\mathrm{d} \hspace{0.3pt}t} \hspace{0.2pt}\Big\{ \cE(t) - E[\phi_{\infty}] \hspace{0.2pt}\Big\} + \mu \hspace{0.3pt}\Big\{ \cE(t) - E[\phi_{\infty}] \hspace{0.2pt}\Big\}^{2-2\theta} \hspace{20pt}\text{for all $t \geq T_\rho$.}
\end{align*}Here, $\mu > 0$ is constant. It can be adjusted suitably small in the estimates below. Applying this ODE inequality yields, for all $t \geq 10 \hspace{0.5pt}T_\rho$, that
\begin{align}\label{defn of K}
    \cE(t) - E[\phi_{\infty}] \hspace{1.5pt}\lesssim\hspace{1.5pt}K(t), \hspace{15pt}\text{where} \hspace{5pt} K(t) := \left\{\begin{array}{lcl} (1 + t)^{- \frac{1}{1 - 2 \theta}} \hspace{15pt}&\text{if $0 < \theta < \frac{1}{2}$;}\\[2mm]
    \hspace{10pt}e^{- \mu \hspace{0.2pt}t}  &\hspace{{-19.5pt}}\text{if $\theta = \frac{1}{2}$.}\end{array}\right. 
\end{align}
For all $t \geq 10 \hspace{0.5pt}T_\rho$, the uniform boundedness in \eqref{H1H2bdd} infers $\|\hspace{0.5pt} u \cdot \nabla \phi \hspace{0.5pt}\|_{L^2} \lesssim  \|\hspace{0.5pt} \nabla u \hspace{0.5pt}\|_{L^2}$. We then have
\begin{align*}
    \big\|\hspace{0.5pt} \phi(t) - \phi_{\infty} \hspace{0.5pt}\big\|_{L^2} \hspace{1pt}\leq\hspace{1pt} \int_t^{\infty} \big\|\hspace{0.5pt} \partial_{s} \phi \hspace{0.5pt}\big\|_{L^2} \hspace{1pt}\lesssim\hspace{1pt} \int_t^{\infty} \big\|\hspace{0.5pt} \nabla u \hspace{0.5pt}\big\|_{L^2} + \big\|\hspace{0.5pt} \Delta \phi + \frac{h^2}{2} \sin 2 \phi \hspace{0.5pt}\big\|_{L^2} \hspace{20pt} \text{for all $t \geq 10\hspace{0.5pt}T_\rho$.} \end{align*}By the two estimates above and the similar derivations for \eqref{bd of Ahalf}, it turns out
    \begin{align}\label{L2 estimate with decay}
    \big\|\hspace{0.5pt} \phi(t) - \phi_{\infty} \hspace{0.5pt}\big\|^{\frac{1}{\theta}}_{L^2} \hspace{1pt}\lesssim\hspace{1pt}  \cE(t) - E[\phi_{\infty}]  \lesssim  K(t)  \hspace{20pt}\text{for all $t \geq 10\hspace{0.5pt}T_\rho$.}
\end{align}\\[-3mm]
\noindent\textbf{Step 2: Convergence of $(u, \phi)$ in $L^2 \times H^1$.}\\[2mm] In light of the equation \eqref{sEL} satisfied by $(u, \phi)$ and the equation in \eqref{SG} satisfied by $\phi_\infty$, it holds
\begin{align}\label{variant of sel}
    \left\{ \begin{aligned}
        &\partial_t\big(\phi - \phi_{\infty}\big) +  u \cdot \nabla \phi = \Delta \big(\phi - \phi_{\infty}\big) + \frac{h^2}{2} \big(\sin 2 \phi - \sin 2 \phi_{\infty}\big), \\[2mm]
        &\hspace{10pt}\partial_t u +  u \cdot \nabla  u - \Delta u = - \nabla p - \big(\Delta \phi + \frac{h^2}{2} \sin 2 \phi\big) \hspace{0.5pt}\nabla \phi.
    \end{aligned} \right.
\end{align}
Multiply $\partial_t \big(\phi - \phi_{\infty}\big)$ and take the inner product with $u$ on both sides of the first and second equations above, respectively. Through integration by parts, we get
\begin{align*}
    \frac{1}{2} \hspace{1pt}\frac{\mathrm{d}}{\mathrm{d} \hspace{0.3pt}t} \hspace{1pt} \left[\hspace{2pt}  \int_{\Omega}  |\hspace{0.5pt}u\hspace{0.5pt}|^2 +  \big|\hspace{0.5pt}\nabla \phi - \nabla \phi_{\infty} \hspace{0.5pt}\big|^2 + \frac{h^2}{2} R_*\big(\phi_\infty, \phi\big) + \frac{L_\text{H}}{2} \int_\text{H} R_*\big(\phi_\infty, \phi\big) \hspace{1.5pt} \right]
    = - A,
\end{align*} where \begin{align}\label{defn R} R_*\big(\phi_\infty, \phi\big) := \cos 2 \phi - \cos 2 \phi_{\infty} + 2\hspace{0.3pt}(\phi - \phi_{\infty})  \sin 2 \phi_{\infty}.\end{align} Moreover, we can also multiply $\phi - \phi_{\infty}$ on the both sides of the first equation in \eqref{variant of sel} and integrate by part. It then turns out \begin{align*}
    &\frac{1}{2} \frac{\mathrm{d}}{\mathrm{d}\hspace{0.3pt}t} \int_{\Omega} \big|\hspace{0.5pt}\phi - \phi_{\infty}\hspace{0.5pt}\big|^2 + \int_{\Omega} \big|\hspace{0.5pt}\nabla \phi - \nabla \phi_{\infty} \hspace{0.5pt}\big|^2 = - \int_{\Omega} 
 \big(\phi - \phi_{\infty}\big) \hspace{0.2pt} u \cdot \nabla \phi_\infty  \\[2mm]
    &\h{30pt}+ \frac{h^2}{2} \int_{\Omega} \big(\sin 2 \phi - \sin 2 \phi_{\infty}\big) \big(\phi - \phi_{\infty}\big) + \frac{L_\text{H}}{2} \int_{\text{H}} \big(\sin 2 \phi - \sin 2 \phi_{\infty}\big) \big(\phi - \phi_{\infty}\big).
\end{align*}Denote by $G(t)$ the function
\begin{align}\label{defn G}
    G(t) := \frac{1}{2} \int_{\Omega} |\hspace{0.5pt}u\hspace{0.5pt}|^2 + \big|\hspace{0.5pt}\phi - \phi_{\infty}\hspace{0.5pt}\big|^2 + \big|\hspace{0.5pt}\nabla  \phi - \nabla \phi_{\infty} \hspace{0.5pt}\big|^2 + \frac{h^2}{2} R_*\big(\phi_\infty, \phi\big) + \frac{L_\text{H}}{4} \int_\text{H} R_*\big(\phi_\infty, \phi\big).
\end{align}The above arguments then induce
\begin{align}\label{energy id of G} \frac{\mathrm{d}\hspace{0.3pt}G}{\mathrm{d}\hspace{0.3pt} t} + A  &+  \int_{\Omega}  \big|\hspace{0.5pt}\nabla \phi - \nabla \phi_{\infty} \hspace{0.5pt}\big|^2 = - \int_{\Omega} 
 \big(\phi - \phi_{\infty}\big) \hspace{0.2pt} u \cdot \nabla \phi_\infty \\[2mm]
    & + \frac{h^2}{2} \int_{\Omega} \big(\sin 2 \phi - \sin 2 \phi_{\infty}\big) \big(\phi - \phi_{\infty}\big) + \frac{L_\text{H}}{2} \int_{\text{H}} \big(\sin 2 \phi - \sin 2 \phi_{\infty}\big) \big(\phi - \phi_{\infty}\big).\nonumber
\end{align}By the fundamental theorem of calculus, \begin{align}\label{L2 integral on H}
    \int_\text{H} \big(\phi - \phi_\infty\big)^2 = - \int_\Omega \partial_3  \big(\phi - \phi_\infty\big)^2 = - 2 \int_\Omega \big(\phi - \phi_\infty\big) \big(\partial_3 \phi - \partial_3 \phi_\infty\big).
\end{align}The right-hand side of \eqref{energy id of G} can therefore be estimated as follows:
\begin{align*}
   \frac{\mathrm{d}\hspace{0.3pt}G}{\mathrm{d}\hspace{0.3pt} t} + A +  \int_{\Omega}  \big|\hspace{0.5pt}\nabla \phi - \nabla \phi_{\infty} \hspace{0.5pt}\big|^2 &\hspace{1pt}\lesssim\hspace{1pt}   \int_{\Omega} 
 |\hspace{0.5pt}u\hspace{0.5pt}| \hspace{1pt} \big|\phi - \phi_{\infty}\big|    +  \int_{\Omega}  \big(\phi - \phi_{\infty}\big)^2 +  \int_{\text{H}} \big(\phi - \phi_{\infty}\big)^2\\[2mm]
 &\hspace{1pt}= \hspace{1pt}\int_{\Omega} 
 |\hspace{0.5pt}u\hspace{0.5pt}| \hspace{1pt} \big|\phi - \phi_{\infty}\big|    +  \int_{\Omega}  \big(\phi - \phi_{\infty}\big)^2 - 2 \int_\Omega \big(\phi - \phi_\infty\big) \big(\partial_3 \phi - \partial_3 \phi_\infty\big).
\end{align*}It then turns out by Poincar\'{e} inequality \eqref{Poincare ineq} and Young's inequality that \begin{align}\label{energy ineq of G}
    \frac{\mathrm{d}\hspace{0.3pt}G}{\mathrm{d}\hspace{0.3pt} t} + \frac{1}{2} \left[ \hspace{1pt}A +  \int_{\Omega}  \big|\hspace{0.5pt}\nabla \phi - \nabla \phi_{\infty} \hspace{0.5pt}\big|^2 \hspace{1.5pt}\right] &\hspace{1pt}\lesssim\hspace{1pt}\int_{\Omega}  \big(\phi - \phi_{\infty}\big)^2.
\end{align}Keep applying  \eqref{Poincare ineq} and notice the definition of $G$ in \eqref{defn G}. We rewrite the above estimate by
\begin{align*}
    \frac{\mathrm{d} \hspace{0.3pt}G}{\mathrm{d}\hspace{0.3pt}t}  + C_1\hspace{0.2pt} G  \hspace{1pt}\leq\hspace{1pt} C_2 \int_\Omega \big( \phi - \phi_\infty\big)^2 + \frac{C_1}{4} \left[\hspace{1pt}h^2 \int_{\Omega} R_*\big(\phi_\infty, \phi\big) + L_\text{H} \int_{\text{H}} R_*\big(\phi_\infty, \phi\big)\hspace{1pt}\right].
\end{align*}Here, $C_1$ and $C_2$ are positive constants. Recall the definition of $R_*\big(\phi_\infty, \phi\big)$ in \eqref{defn R}. It holds
\begin{align}\label{control of R}
\Big|\hspace{1pt}R_*\big(\phi_\infty, \phi\big)\hspace{1pt}\Big| \hspace{1.5pt}\lesssim\hspace{1.5pt} \big( \phi - \phi_{\infty} \big)^2.
\end{align} The last two estimates and \eqref{L2 integral on H} then yield
\begin{align*}
\frac{\mathrm{d} \hspace{0.3pt}G}{\mathrm{d}\hspace{0.3pt}t}  + C_1\hspace{0.2pt} G  \hspace{1.5pt}\lesssim\hspace{1.5pt}  \int_\Omega \big( \phi - \phi_\infty\big)^2 +  \int_{\text{H}} \big( \phi - \phi_\infty\big)^2 = \int_\Omega \big( \phi - \phi_\infty\big)^2 - 2 \int_\Omega \big(\phi - \phi_\infty\big) \big(\partial_3 \phi - \partial_3 \phi_\infty\big).
\end{align*}Using Young's inequality and the decay estimate in \eqref{L2 estimate with decay}, we get 
\begin{align}\label{conv control of G}
 \frac{\mathrm{d} \hspace{0.3pt}G}{\mathrm{d}\hspace{0.3pt}t}  + \frac{C_1}{2}\hspace{0.2pt} G  \hspace{1.5pt}\lesssim\hspace{1.5pt}\big( K(t) \big)^{2 \theta} \hspace{20pt}\text{for all $t \geq 10\hspace{0.3pt}T_\rho$,}\end{align}which further implies \begin{align*}
    G(t)  &\hspace{1.5pt}\leq \hspace{1.5pt} G(10\hspace{0.3pt}T_\rho) \hspace{1pt}\exp\Big\{5\hspace{0.5pt}C_1 T_\rho - \frac{C_1}{2} t\hspace{0.5pt}\Big\} + c_* \int_{10\hspace{0.3pt}T_\rho}^t \big(K(s)\big)^{2\theta} e^{- \frac{C_1 ( t - s)}{2}} \hspace{1.5pt}\mathrm{d} s \nonumber\\[2mm]    &\hspace{1.5pt}\lesssim\hspace{1.5pt} e^{ - \frac{C_1}{2} t} + \int_{10\hspace{0.3pt}T_\rho}^{\frac{t +  10 \hspace{0.3pt}T_{\rho}}{2}} \big(K(s)\big)^{2\theta} e^{- \frac{C_1 ( t - s)}{2}} \hspace{1.5pt}\mathrm{d} s + \int_{\frac{t + 10\hspace{0.3pt}T_\rho}{2}}^t \big(K(s)\big)^{2\theta} e^{- \frac{C_1 ( t - s)}{2}} \hspace{1.5pt}\mathrm{d} s \nonumber\\[2mm]&\hspace{1.5pt}\lesssim\hspace{1.5pt}e^{ - \frac{C_1}{4} t} + \left[ \hspace{1pt}K\left(\frac{t + 10\hspace{0.3pt}T_\rho}{2}\right)\hspace{1pt}\right]^{2\theta}\hspace{30pt}\text{for all $t \geq 10\hspace{0.3pt}T_\rho$.}\end{align*}If $0 < \theta < \frac{1}{2}$, it follows from the last estimate that 
\begin{align}\label{decay est of G}
    G(t) \hspace{1.5pt}\lesssim\hspace{1.5pt} \big( K( t/2) \big)^{2\theta} \hspace{20pt}\text{for all $t \geq 10\hspace{0.3pt}T_\rho$}.
\end{align}If $\theta = \frac{1}{2}$, we can choose the parameter $\mu$ in $K(t)$ (see \eqref{defn of K}) suitably small such that \eqref{decay est of G} still holds.  The smallness of $\mu$ depends on $C_1$. \eqref{defn G} and \eqref{decay est of G} then induce
\begin{align*}
 \int_{\Omega} |\hspace{0.5pt}u\hspace{0.5pt}|^2  + \big|\hspace{0.5pt}\nabla  \phi - \nabla \phi_{\infty} \hspace{0.5pt}\big|^2 \hspace{1.5pt}\lesssim\hspace{1.5pt}  \big( K(t/2) \big)^{2\theta} + \int_\Omega \big(\phi - \phi_\infty\big)^2 +  \int_\text{H} \big(\phi - \phi_\infty\big)^2 \hspace{20pt}\text{for all $t \geq 10\hspace{0.3pt}T_\rho$.}
\end{align*}Here we also use \eqref{control of R}. Now we apply \eqref{L2 integral on H} and \eqref{L2 estimate with decay} to estimate the right-hand side above. It turns out \begin{align*}
     \int_{\Omega} |\hspace{0.5pt}u\hspace{0.5pt}|^2  + \big|\hspace{0.5pt}\nabla  \phi - \nabla \phi_{\infty} \hspace{0.5pt}\big|^2 \hspace{1.5pt}\lesssim\hspace{1.5pt}  \big( K(t/2) \big)^{2\theta}  - 2 \int_\Omega \big(\phi - \phi_\infty\big) \big(\partial_3 \phi - \partial_3 \phi_\infty\big) \hspace{20pt}\text{for all $t \geq 10\hspace{0.3pt}T_\rho$.}
\end{align*}Using Young's inequality and \eqref{L2 estimate with decay} again, we reduce the last estimate to \begin{align}\label{H1L2 estimate}
\int_{\Omega} |\hspace{0.5pt}u\hspace{0.5pt}|^2  + \big|\hspace{0.5pt}\nabla  \phi - \nabla \phi_{\infty} \hspace{0.5pt}\big|^2 \hspace{1.5pt}\lesssim\hspace{1.5pt}  \big( K(t/2) \big)^{2\theta} \hspace{20pt}\text{for all $t \geq 10\hspace{0.3pt}T_\rho$.}
\end{align}\vspace{0.2pc}

\noindent\textbf{Step 3: Convergence of $(u, \phi)$ in $H^1 \times H^2$.}\\[2mm]
Recall Lemma \ref{high ord est} and the uniform boundedness in \eqref{H1H2bdd}. It holds
\begin{align}\label{linear ineq of A}
    \frac{\mathrm{d} \hspace{0.2pt}A}{\mathrm{d}
    \hspace{0.7pt}t}  \hspace{1.5pt}\leq\hspace{1.5pt}c_1 A \h{15pt}\text{for all $t \geq 10 T_\rho$.}
\end{align}Here, $c_1 > 0$ is a constant. According to \eqref{energy ineq of G} and \eqref{L2 estimate with decay}, we also have \begin{align*}
    \frac{\mathrm{d}\hspace{0.3pt}G}{\mathrm{d}\hspace{0.3pt} t} + \frac{1}{2} \left[ \hspace{1pt}A +  \int_{\Omega}  \big|\hspace{0.5pt}\nabla \phi - \nabla \phi_{\infty} \hspace{0.5pt}\big|^2 \hspace{1.5pt}\right] &\hspace{1pt}\leq\hspace{1pt}c_2\hspace{0.1pt} \big(K(t)\big)^{2\theta} \hspace{20pt}\text{for some constant $c_2 > 0$ and all $t \geq 10 \hspace{0.3pt}T_\rho$.}
\end{align*}Multiplying the both sides of the above estimate by $4c_1$ and then adding \eqref{linear ineq of A}, we obtain \begin{align*}
 \frac{\mathrm{d} \hspace{0.2pt}A}{\mathrm{d}
    \hspace{0.7pt}t}+  c_1 \hspace{.1pt}A &\hspace{1pt}\leq\hspace{1pt}- 4\hspace{0.2pt}c_1\frac{\mathrm{d}\hspace{0.3pt}G}{\mathrm{d}\hspace{0.3pt} t} + 4\hspace{0.2pt}c_1c_2\hspace{0.1pt} \big(K(t)\big)^{2\theta} \hspace{20pt}\text{for all $t \geq 10 \hspace{0.3pt}T_\rho$.}
\end{align*}By solving this ODE inequality, it turns out,  for all $t \geq 10 \hspace{0.3pt}T_\rho$, that \begin{align*}
    &e^{c_1 t} A(t)  \hspace{1.5pt}\leq \hspace{1.5pt} e^{10 \hspace{0.2pt}c_1 T_\rho} A(10 \hspace{0.3pt}T_\rho) - 4\hspace{0.2pt}c_1 \int_{10 T_\rho}^t e^{c_1 s}\hspace{0.5pt} \frac{\mathrm d G}{\mathrm d s} \hspace{2pt}\mathrm{d} s + 4\hspace{0.2pt}c_1c_2\hspace{0.1pt}\int_{10 T_\rho}^t e^{c_1 s}\hspace{0.3pt} \big(K(s)\big)^{2\theta} \hspace{2pt}\mathrm{d} s\\[2mm]
    &\hspace{10pt}=\hspace{1.5pt}
    e^{10 \hspace{0.2pt}c_1 T_\rho} 
 \Big[ \hspace{1pt}A(10 \hspace{0.3pt}T_\rho) + 4 \hspace{0.2pt}c_1 G(10 T_\rho) \hspace{1pt}\Big] - 4\hspace{0.2pt}c_1  e^{c_1t} G(t)  + 4 \hspace{0.2pt}c_1 
    \int_{10 T_\rho}^t e^{c_1 s}\hspace{0.3pt} \Big[ \hspace{1pt}c_1G(s) + c_2 \big(K(s)\big)^{2\theta} \hspace{1pt}\Big] \hspace{2pt}\mathrm{d} s.
\end{align*}Using the decay estimate of $G(t)$ in \eqref{decay est of G}, the definition of $G$ in \eqref{defn G} and the estimate in \eqref{control of R}, we keep estimating the right-hand side above and get \begin{align*}
    e^{c_1 t} A(t) &\hspace{1.5pt}\lesssim \hspace{1.5pt} 1 + e^{c_1 t} \left[ \hspace{1pt} \int_\Omega \big(\phi - \phi_\infty\big)^2 + \int_\text{H} \big(\phi - \phi_\infty\big)^2 \hspace{1pt}\right] +  \int_{10 T_\rho}^t e^{c_1 s}\hspace{0.3pt}  \big(K(s/2)\big)^{2\theta} \hspace{2pt}\mathrm{d} s.
\end{align*}By \eqref{L2 estimate with decay}, \eqref{H1L2 estimate} and trace theorem, it turns out from the last estimate that \begin{align*}
    A(t) &\hspace{1.5pt}\lesssim \hspace{1.5pt} e^{- c_1 t} + \big( K(t/2)\big)^{2\theta} +  \int_{10 T_\rho}^t e^{- c_1 ( t -  s)}\hspace{0.3pt}  \big(K(s/2)\big)^{2\theta} \hspace{2pt}\mathrm{d} s.
\end{align*}Similar arguments for deriving \eqref{decay est of G} can be applied to obtain \begin{align}\label{decay est of A}A(t) \hspace{1.5pt}\lesssim\hspace{1.5pt} \big( K(t/4) \big)^{2 \theta} \hspace{20pt}\text{for all $t \geq 10\hspace{0.3pt}T_\rho$}.
\end{align}If $\theta = \frac{1}{2}$, we also choose $\mu$ in $K(t)$ suitably small. The smallness of $\mu$ depends on $c_1$.\vspace{0.2pc}

It remains to study the $L^2$-norm of the second-order derivatives of $\phi - \phi_\infty$. Applying integration by parts, we can use the same arguments as in the proof of Lemma \ref{est nabla 2 phi} to obtain \begin{align*}
    \int_\Omega \big|\hspace{1pt}\nabla^2 \phi - \nabla^2 \phi_\infty\hspace{1pt}\big|^2 &\hspace{1pt}=\hspace{1pt} 2 L_\text{H} \int_\text{H} \nabla' \big(\phi - \phi_\infty\big) \cdot  \big( \cos 2 \phi \hspace{1.5pt}\nabla' \phi - \cos 2 \phi_\infty \hspace{1pt}\nabla' \phi_\infty\big) + \int_\Omega \big(\Delta \phi - \Delta \phi_\infty\big)^2\\[2mm]
    &\hspace{1pt}\leq\hspace{1pt} \int_\Omega \big( \Delta \phi - \Delta \phi_\infty\big)^2 + c_* \int_\text{H} \big|\hspace{1pt}\phi - \phi_\infty\hspace{1pt}\big|^2 + \big|\hspace{1pt}\nabla' \phi - \nabla' \phi_\infty \hspace{1pt}\big|^2.
\end{align*}Note that \begin{align*}
    \int_\text{H} \big|\hspace{1pt} \nabla' \phi - \nabla' \phi_\infty \big|^2  = - 2 \int_\Omega \nabla' \big(\phi - \phi_\infty\big) \cdot \partial_3 \nabla' \big(\phi - \phi_\infty\big).
\end{align*}It then turns out \begin{align*}
 \frac{1}{2}\hspace{0.5pt}\int_\Omega \big|\hspace{1pt}\nabla^2 \phi - \nabla^2 \phi_\infty\hspace{1pt}\big|^2
    \hspace{1pt}\leq\hspace{1pt} \int_\Omega \big( \Delta \phi - \Delta \phi_\infty\big)^2 + c_* \int_\text{H} \big|\hspace{1pt}\phi - \phi_\infty\hspace{1pt}\big|^2 + c_* \int_\Omega  \big|\hspace{1pt}\nabla \phi - \nabla \phi_\infty \hspace{1pt}\big|^2.
\end{align*}Still use \eqref{L2 estimate with decay}, \eqref{H1L2 estimate} and trace theorem, it follows from the last estimate that \begin{align}\label{H2 est of phi 1}
    \int_\Omega \big|\hspace{1pt}\nabla^2 \phi - \nabla^2 \phi_\infty\hspace{1pt}\big|^2
    \hspace{1pt}\lesssim\hspace{1pt} \int_\Omega \big( \Delta \phi - \Delta \phi_\infty\big)^2 + \big( K(t/2)\big)^{2\theta} \hspace{20pt}\text{for all $t \geq 10 \hspace{0.3pt} T_\rho$.}
    \end{align} The $L^2$-norm of $\Delta \phi - \Delta \phi_\infty$ can be estimated by \eqref{decay est of A}. In fact, \begin{align*}
        \int_\Omega \big( \Delta \phi - \Delta \phi_\infty\big)^2 = \int_\Omega \big( \Delta \phi + \frac{h^2}{2} \sin 2 \phi + \frac{h^2}{2} \big( \sin 2 \phi_\infty - \sin 2 \phi\big)\big)^2 \hspace{1pt}\lesssim\hspace{1pt} A(t) + \int_\Omega \big(\phi - \phi_\infty\big)^2.
    \end{align*}Applying \eqref{decay est of A} and \eqref{L2 estimate with decay} to the right-hand side above yields \begin{align*}
        \int_\Omega \big( \Delta \phi - \Delta \phi_\infty\big)^2 \hspace{1pt}\lesssim\hspace{1pt}\big(K(t/4)\big)^{2 \theta} \hspace{20pt}\text{for all $t \geq 10\hspace{0.3pt}T_\rho$}.
    \end{align*}This estimate together with \eqref{H2 est of phi 1} induces the decay estimate of the $L^2$-norm of $\nabla^2 \phi - \nabla^2 \phi_\infty$.
\end{proof}

\section{Partial regularity of the suitable weak solutions}

In this section, we establish the regularity results for the global suitable weak solution $(u, \phi)$ of the IBVP under the small dissipation energy condition. Moreover, $u$ is shown converging to $0$ in $L^\infty(\Omega)$ as $t \to \infty$. The main results are read as follows:

\begin{thm} \label{partial regularity}
Suppose $(u, \phi)$ is a global suitable weak solution of the $\mathrm{IBVP}$. For any $\epsilon > 0$, there exists a sufficiently large time $T_{\epsilon}$, which depends on $\epsilon$ and the solution $(u, \phi)$, such that 
    \begin{equation} \label{smallness}
        \int_{T_{\epsilon}}^\infty \int_{\Omega} \big|\hspace{.5pt} \nabla u \hspace{.5pt}\big|^2 + \Big|\hspace{0.5pt}\Delta \phi + \frac{h^2}{2} \sin 2 \phi \hspace{0.5pt}\Big|^2 \hspace{1.5pt}\leq\hspace{1.5pt} \epsilon.
    \end{equation} 
If $\epsilon$ is sufficiently small, then $(u, \phi)$ is regular on $\overline{\Omega} \times \big[\h{1pt} T_{\epsilon} + 10, \infty \big)$. In addition,
    \begin{equation} \label{decay of sup u}
    \lim_{t \h{0.5pt}\to \h{0.5pt} \infty} \big\|\h{0.5pt} u(t) \h{0.5pt}\big\|_{L^{\infty}(\Omega)} = 0.
\end{equation}
\end{thm}

The remainder of this section is devoted to proving Theorem \ref{partial regularity}. The regularity at interior points has been investigated by Lin-Liu in \cite{LinLiu1996}. To be simple, we only consider the boundary case. The method we present here is based on a blow-up argument, which is motivated by that of Lin in \cite{Lin1998} and Seregin in \cite{Seregin2002} for the pure Navier-Stokes equations. The readers may also refer to Du-Hu-Wang \cite{DuHuWang2019} for the applications in the Beris-Edwards system. 

Before proceeding, we sketch the arguments in the following four sections. In Section \ref{unf of phi}, we prove a maximum principle and study the $L^\infty$-estimate of $\phi$ on $\overline{\Omega} \times [0,\infty)$. Sections \ref{dis imply L3} and \ref{L infty of u phi} are devoted to showing the $L^\infty$-estimate of $(u, \nabla \phi)$ near the boundary $\mathrm{H} \cup \mathrm{P}$  after a long time. With this boundedness result, in Section \ref{hoder of u}, we obtain the H\"{o}lder regularity of $u$ after a long time and verify the asymptotic limit \eqref{decay of sup u}. Throughout the following, the parabolic cylinder is denoted by
$$P_r(x, t) := \big(B_r(x) \cap  \Omega  \hspace{1pt} \big) \times (t - r^2, t).$$

\subsection{Maximum Principle and \texorpdfstring{$L^{\infty}$}{TEXT} -Estimates of \texorpdfstring{$\phi$}{TEXT}}\label{unf of phi}  
We study the advection-diffusion equation:
\begin{align}\label{sine-gordon with drft}
    \partial_t \phi - \Delta \phi + u \cdot \nabla \phi = h^2 \sin \phi \cos \phi \h{20pt}\text{on $\overline{\Omega} \times [\h{0.5pt}T_0, T\h{0.5pt}]$,}
\end{align}

\noindent where $T_0 \in [\h{0.5pt}0, T\h{0.5pt})$. $u$ is a divergence-free drift. The first result is about the maximum principle for the classic solutions to this equation subjecting to the boundary conditions in \eqref{bc}.
\begin{lem}\label{max principle}
Assume that $u \in C^\infty\big(\h{1pt}\overline{\Omega} \times [\h{0.5pt}T_0, T\h{0.5pt}]\h{1pt}\big)$ and $\phi$ is a smooth solution of \eqref{sine-gordon with drft} subjecting to the boundary condition \eqref{bc}. 
\begin{itemize}
    \item[$\mathrm{(1).}$] If $\phi \geq m_1 \pi$ at $T_0$ for some $m_1 \in \mathbb Z$, then $\phi \geq m_1 \pi$ at all $t \in [\h{0.5pt}T_0, T\h{0.5pt}]$. \vspace{0.3pc}

    \item[$\mathrm{(2).}$] If $\phi \leq m_2 \pi$ at $T_0$ for some $m_2 \in \mathbb Z$, then $\phi \leq m_2\pi$ at all $t \in [\h{0.5pt}T_0, T\h{0.5pt}]$. \vspace{0.3pc}

    \item[$\mathrm{(3).}$] If $0 \leq \phi \leq \pi$ at $T_0$ and $\phi(\cdot, T_0) \not\equiv 0$, then $0 < \phi < \pi$ on $\Omega \times (\h{0.5pt}T_0, T\h{0.5pt})$.
\end{itemize}
    
\end{lem}

\begin{proof}[\bf Proof] Notice that $\phi - m_1 \pi$ and $m_2\pi - \phi$ satisfies the same transported sine-Gordon equation in \eqref{sine-gordon with drft} and the boundary condition on H in \eqref{bc}. We change the variable by letting  \begin{align}\label{defn of gamma} \text{$\psi :=  \gamma \hspace{1.5pt}e^{L_\text{H} x_3 - \left( M_1 + L_\text{H}^2 \right) \hspace{1pt}t}$, where \text{$\gamma$ denotes either  $\phi - m_1 \pi$ or $m_2\pi - \phi$}.} \end{align} It then turns out \begin{align}\label{revision of eq bc} \left\{ \begin{array}{lcl}
\partial_t \psi +  v \cdot \nabla \psi \hspace{1.5pt}=\hspace{1.5pt} \Delta \psi + B \hspace{0.8pt} \psi \hspace{100pt}&\text{in $\Omega$,}\\[2mm]
\hspace{54.5pt}\psi \geq 0 &\text{on P,}\\[2mm]
\hspace{15.5pt}-  L_\text{H}^{-1} \hspace{1pt}\partial_3 \psi = e^{- \left(M_1 + L_\text{H}^2\right) \hspace{1pt}t} \hspace{1pt}\Big(\frac{1}{2} \sin 2 \gamma - \gamma \Big)   &\h{1pt}\text{on H.}
\end{array}
\right.
\end{align}In the above, $v := u + 2 \hspace{0.5pt} L_\text{H} \hspace{0.5pt} e_3$ with $e_3$ the unit positive direction along the $x_3$-variable. The coefficient $B$ is given by \begin{align}\label{defn of B}
 B = B(\phi, u_3) :=  L_\text{H} \hspace{0.5pt}u_3  - M_1 + \dfrac{h^2 \sin 2 \hspace{0.5pt}\gamma}{2 \hspace{0.5pt} \gamma}. 
\end{align}It satisfies $B \leq -1$ in $\Omega$ if the constant $M_1$ in \eqref{defn of B} is sufficiently large.

To prove (1) and (2) in the lemma, it suffices to show
\begin{align}\label{equivalency of (1)}
    \inf \left\{\hspace{1pt}\min_{ \overline{\Omega}} \psi(\cdot, t) : t \in (\h{0.5pt}T_0, T\h{0.5pt})\right\} \geq 0 \hspace{20pt}\text{if $\psi \geq 0$ at $T_0$}.
\end{align}Suppose that there is a $T_* \in (\h{0.5pt} T_0, T\h{0.5pt})$, so that the minimum of $\psi\left(\cdot, T_*\right)$ over $\overline{\Omega}$ is negative. Then
\begin{align*}
    \min_{\overline{\Omega}} \psi\left(\cdot, t_*\right) \hspace{1pt}=\hspace{1pt} \psi\left(x_*, t_*\right) \hspace{1pt}=\hspace{1pt}
    \min\Big\{\hspace{1pt}\psi\left(x, t\right) : \left(x, t\right) \in \overline{\Omega} \times [\hspace{0.5pt}T_0, T_*\hspace{0.5pt}] \hspace{1.5pt}\Big\} < 0, 
\end{align*}for some $t_* \in \left(\hspace{0.5pt}T_0, T_*\hspace{0.5pt}\right]$ and $x_* \in \overline{\Omega}$. It then turns out \begin{align}\label{sgn of partial t}
\partial_t \psi \left(x_*, t_*\right) \leq 0.\end{align}By the sign condition of $\psi$ on P (see \eqref{revision of eq bc}), the point $x_*$ is not on P. We also claim that $x_* \not \in \text{H}$. Otherwise, it holds $\partial_3 \psi \left(x_*, t_*\right) \geq 0$. By the boundary condition on H in \eqref{revision of eq bc}, \begin{align*}
    0 \geq -  L_\text{H}^{-1} \hspace{1pt}\partial_3 \psi \left(x_*, t_*\right) = e^{- \left(M_1 + L_\text{H}^2\right) \hspace{1pt}t_*} \hspace{1pt}\Big(\hspace{1pt}\frac{1}{2} \sin 2 \gamma \left(x_*, t_*\right) - \gamma \left(x_*, t_*\right) \Big).
\end{align*}However, the right-hand side above is strictly positive  because $\gamma\left(x_*, t_*\right) < 0$. By $x_* \in \Omega$, it follows $$
\nabla \psi\left(x_*, t_*\right) = 0 \hspace{20pt}\text{and}\hspace{20pt}\Delta \psi\left(x_*, t_*\right) \geq 0.$$ Since $B \hspace{0.5pt} \psi > 0$ at $(x_*, t_*)$, we obtain
\begin{align*}
    \partial_t \psi\left(x_*, t_*\right) = - \big(v \cdot \nabla  \psi \big) \left(x_*, t_*\right) + \Delta \psi \left(x_*, t_*\right) + \big(B \hspace{0.5pt} \psi \big) (x_*, t_*) > 0.
\end{align*}It violates \eqref{sgn of partial t}. Therefore, \eqref{equivalency of (1)} holds. We obtain (1) and (2) in the lemma.

Now we prove (3) in the lemma. In the following, we fix $m_1 = 0$ and $m_2 = 1$. If $\psi \geq 0$ at $T_0$, then it satisfies $\psi \geq 0$ on $\Omega \times [\hspace{0.5pt}T_0, T\h{0.5pt}]$. If $\psi = 0$ at some point  $(x_0, t_0) \in \Omega \times (T_0, T)$, then by strong maximum principle of parabolic equations (see Theorem 2.7 in \cite{Liebmann}), it holds $\psi \equiv 0$ on $\Omega \times [\hspace{0.5pt}T_0, t_0\hspace{0.5pt}]$. Therefore, $\gamma \equiv 0$ on $\Omega \times [\hspace{0.5pt}T_0, t_0\hspace{0.5pt}]$. Recall the definition of $\gamma$ in \eqref{defn of gamma}. If $\gamma = \phi$, then $\phi \equiv 0$ on $\Omega \times [\hspace{0.5pt}T_0, t_0\hspace{0.5pt}]$. This violates the non-equivalent-zero condition of $\phi$ at $T_0$. If $\gamma = \pi - \phi$, then $\phi \equiv \pi$ on $\Omega \times [\hspace{0.5pt}T_0, t_0\hspace{0.5pt}]$. It is a contradiction to the homogeneous Dirichlet boundary condition of $\phi$ on P. If the assumptions in the (3) of the lemma hold, then $\psi > 0$ on $\Omega \times (T_0, T)$. The proof is completed.
\end{proof}

We now apply this maximum principle and an approximation argument to obtain the uniform boundedness of $\phi$ over $\overline{\Omega} \times [0, \infty)$, where $(u, \phi)$ is a global suitable weak solution of the $\mathrm{IBVP}$. \begin{prop}\label{bounded of phi}Suppose $(u, \phi)$ is a global suitable weak solution of the $\mathrm{IBVP}$. Then it satisfies  $$ \big\|\h{1pt} \phi \h{1pt}\big\|_{L^\infty\left(\overline{\Omega} \h{0.5pt}\times \h{0.5pt} [\h{0.5pt}0, \infty) \right)} \leq M_{\phi_0},$$ where $M_{\phi_0}$ is a positive constant depending only on the $L^\infty$-norm of the initial angle $\phi_0$.
\end{prop}

\begin{proof}[\bf Proof] We divide the proof into 4 steps. \\[1mm]
\textbf{Step 1. Approximation.}
We extend to define $u \equiv 0$ on the complement set of $\overline{\Omega} \times [\h{0.5pt}0, \infty\h{0.5pt})$ in $\mathbb R^4$. Using the standard mollifier $\eta$ on $\mathbb R^4$, we define the mollification of $u$ by $u_{\delta} := \eta_{\delta} * u$. Here, $\delta > 0$ is a scaling parameter. $\eta_\delta$ is given by $$\eta_\delta (x, t) := \frac{1}{\delta^4} \h{1pt}\eta \left(\frac{x}{\delta}, \frac{t}{\delta}\right).$$  For any $T > 0$, Condition (1) in Definition \ref{suitable wk sol} and Proposition 3.2 in the Chapter 1 of \cite{Dibenedetto1993} infer that $u \in L^{\frac{10}{3}}\big( \Omega \times (0, T) \big)$. Hence, $u_{\delta} \to u$ strongly in $L^{\frac{10}{3}}\big(\Omega \times (0, T)\big)$ as $\delta \to 0$. In addition, the incompressibility condition $\div u_\delta = 0$ is preserved. 

Now, we fix a sequence $\big\{\delta_k\big\}$ converging to $0$ as $k \to \infty$. Using this sequence, we introduce the approximation of $\phi$ as follows:\vspace{0.2pc} 
\begin{equation} \label{eq of psi_delta}\left\{ \begin{aligned}
    \partial_t \psi^{\delta_k} - \Delta \psi^{\delta_k} + u_{\delta_k} \cdot \nabla \psi^{\delta_k} &= h^2 \sin \psi^{\delta_k} \cos \psi^{\delta_k} \quad\quad\hspace{17pt} \text{in } \Omega \times (0, \infty); \\[.5mm]
    \psi^{\delta_k} &= 0 \quad\quad\hspace{90pt} \text{on P}; \\[.5mm]
    -  \partial_3 \psi^{\delta_k} &= L_\text{H} \sin \psi^{\delta_k} \cos \psi^{\delta_k} \quad\quad\hspace{13.5pt} \text{on H}. 
\end{aligned} \right. \end{equation}

\noindent We also set $\psi^{\delta_k}(\cdot, 0) = \phi_0$. By (1) and (2) in Lemma \ref{max principle}, there is a positive constant $M_{\phi_0}$ such that 
\begin{align}\label{unif psi del}
    \max_{\overline{\Omega} \times [\h{0.5pt}0, T \h{0.5pt}]} \big|\h{.5pt} \psi^{\delta_k} \h{.5pt}\big| \leq M_{\phi_0} \h{20pt}\text{for any $k \in \mathbb N$ and $T > 0$.}
\end{align}

\noindent Here, $M_{\phi_0}$ depends only on the $L^\infty$-norm of $\phi_0$ on $\Omega$. Hence, we can find a $\psi^0 \in L^\infty\big(\Omega \times (0, T)\big)$ and a subsequence of $\big\{\delta_k\big\}$, which is still denoted by $\big\{\delta_k\big\}$, such that 
\begin{align*}
    \psi^{\delta_k} \rightharpoonup \psi^0 \quad\quad \text{weakly in } L^2\big(\Omega \times (0, T)\big) \h{3pt}\text{as $k \to \infty$}.
\end{align*}

\noindent \textbf{Step 2. Energy Estimate.} We claim the following energy estimate for any $\psi^{\delta_k}$:
\begin{equation} \label{L2LinftyL2H1 of psi_delta}
    \sup_{t \h{1pt}\in \h{1pt}(\h{0.5pt}0, \h{1pt} T \h{0.5pt})}\int_{\Omega} \big|\h{.5pt} \psi^{\delta_k} \h{.5pt}\big|^2  +  \int_0^{T} \int_{\Omega}  \big|\h{.5pt} \nabla \psi^{\delta_k} \h{.5pt}\big|^2 \h{1.5pt}\leq\h{1.5pt} M_{u, \h{1pt}\phi_0, \h{1pt}T}.
\end{equation}Here, $M_{u, \h{1pt}\phi_0, \h{1pt}T} > 0$ is a constant depending on $u$, $\phi_0$ and $T$. To prove \eqref{L2LinftyL2H1 of psi_delta}, we multiply $\psi^{\delta_k}$ on the both sides of \eqref{eq of psi_delta} and then integrate over $\Omega \times (0, t)$. It turns out for any $t \in (0, T)$ that
\begin{align}\label{energ eqa of psi k}
    \int_{\Omega \times \{t \}} \big|\h{.5pt} \psi^{\delta_k} \h{.5pt}\big|^2 &+ 2 \int_0^{t} \int_{\Omega} \big|\h{.5pt} \nabla \psi^{\delta_k} \h{.5pt}\big|^2 \\[2mm]
    &= \int_\Omega \phi_0^2 - 2\int_0^{t} \int_{\Omega} \psi^{\delta_k}  \h{1pt}\big(u_{\delta_k} \cdot \nabla \big)  \psi^{\delta_k} + h^2 \int_0^{t} \int_{\Omega}  \psi^{\delta_k} \h{0.5pt} \sin 2 \psi^{\delta_k} + L_\mathrm{H}  \int_0^{t} \int_{\mathrm{H}} \psi^{\delta_k} \h{0.5pt} \sin 2 \psi^{\delta_k}. \nonumber
\end{align}

\noindent Using the uniform boundedness of $\psi^{\delta_k}$ in Step 1 and Young's inequality, we have
\begin{align*}
    \left|\h{2pt} \int_0^{t} \int_\Omega \psi^{\delta_k} \big(u_{\delta_k} \cdot \nabla\big) \psi^{\delta_k} \right| \leq M_{\phi_0} \int_0^{t} \int_\Omega \big|\h{.5pt} u_{\delta_k} \h{.5pt}\big|^2 + \frac{1}{2} \int_0^{t} \int_\Omega \big|\h{.5pt} \nabla \psi^{\delta_k} \h{.5pt} \big|^2.
\end{align*}

\noindent Therefore,
\begin{align*}
    &\int_{\Omega \times \{ t \}} \big|\h{.5pt} \psi^{\delta_k} \h{.5pt}\big|^2 +  \int_0^{t} \int_{\Omega} \big|\h{.5pt} \nabla \psi^{\delta_k} \h{.5pt}\big|^2 \lesssim_{M_{\phi_0}} 1 + T + \int_0^{T} \int_\Omega \big|\h{.5pt} u_{\delta_k} \h{.5pt}\big|^2.
\end{align*}The estimate \eqref{L2LinftyL2H1 of psi_delta} is obtained since over $\Omega \times [\h{0.5pt}0, T\h{0.5pt}]$, the $L^2$-norm of $u_{\delta_k}$ is uniformly bounded by the $L^2$-norm of $u$. \vspace{0.3pc}

\noindent \textbf{Step 3. Strong $L^2$-convergence.} 
Suppose $\xi$ is a smooth function compactly supported in $\Omega$. The bracket $\big<\cdot, \cdot \big>$ is the duality between the Sobolev space $H_0^{1}(\Omega)$ and its dual space $H^{-1} (\Omega)$. Using \eqref{eq of psi_delta}, we have 
\begin{align*}
    \big<\partial_t \psi^{\delta_k}, \xi\big> &= \int_{\Omega} \psi^{\delta_k} \h{1pt}u_{\delta_k} \cdot \nabla \xi  - \int_{\Omega} \nabla \psi^{\delta_k} \cdot \nabla \xi + \frac{h^2}{2} \int_{\Omega} \xi \sin 2 \psi^{\delta_k}. 
\end{align*}

\noindent It can be estimated that
\begin{align*}
    \big<\partial_t \psi^{\delta_k}, \xi\big> \h{1.5pt}&\lesssim_{K_*}\h{1.5pt} \big\|\h{.5pt} \psi^{\delta_k} \h{.5pt}\big\|_{L^6(\Omega)} \big\|\h{.5pt} u_{\delta_k} \h{.5pt}\big\|_{L^3(\Omega)} \big\|\h{.5pt} \nabla \xi \h{.5pt}\big\|_{L^2(\Omega)} + \big\|\h{0.5pt} \nabla \psi^{\delta_k} \h{0.5pt}\big\|_{L^{2}(\Omega)}  \h{1pt} \big\| \h{0.5pt}\nabla \xi \h{0.5pt}\big\|_{L^2(\Omega)} + \big\|\h{0.5pt} \xi \h{0.5pt}\big\|_{L^2(\Omega)}\\[1mm]
    &\lesssim_{\h{0.5pt}u, \h{1pt}\phi_0} \big\|\h{.5pt} \xi \h{.5pt}\big\|_{H^{1}(\Omega)}+ \big\|\h{0.5pt} \nabla \psi^{\delta_k} \h{0.5pt}\big\|_{L^{2}(\Omega)}  \h{1pt} \big\| \h{0.5pt}\nabla \xi \h{0.5pt}\big\|_{L^2(\Omega)}.
\end{align*}

\noindent Take supreme over all $\xi$ with $ \|\h{0.5pt} \xi \h{0.5pt}\|_{H^{1}(\Omega)} \leq 1$ and integrate the $t$-variable from $0$ to $T$. It follows 
\begin{align*}
    \int_{0}^{T}\big\| \h{0.5pt}\partial_t \psi^{\delta_k} \h{0.5pt}\big\|^{2}_{H^{-1}(\Omega)} \lesssim_{\h{0.5pt}u, \h{1pt}\phi_0} T +  \int_0^{T} \int_{\Omega}  \big|\h{.5pt} \nabla \psi^{\delta_k} \h{.5pt}\big|^2 \leq M_{\h{0.5pt}u, \h{1pt}\phi_0, \h{1pt}T}.
\end{align*}

\noindent Note that $H_{\mathrm P}^1(\Omega)$ is compactly embedded into $L^2(\Omega)$. $L^2(\Omega)$ is continuously embedded into $H^{-1}(\Omega)$. By the Aubin-Lions compactness lemma, it follows that
\begin{align*}
    \psi^{\delta_k} \to \psi^0 \quad \text{strongly in } L^2\big(\Omega \times (0, T)\big) \h{3pt}\text{as $k \to \infty$}.
\end{align*}

\noindent \textbf{Step 4. Uniqueness of Limit.} We prove $\psi^0 = \phi$ almost everywhere on $\Omega \times (0, T)$. This uniqueness result finishes the proof of the $L^\infty$-estimate of $\phi$. Note that $\phi$ solves the third equation in \eqref{sEL} weakly. Taking $\psi^{\delta_k}$ as the test function, we obtain, for any $t \in (0, T)$, that \begin{align*}
    \int_{\Omega \times \{t\}} \psi^{\delta_k} \phi + \int_0^t \int_\Omega \nabla \psi^{\delta_k} \cdot \nabla \phi &= \int_\Omega \phi_0^2 + \int_0^t \int_\Omega \phi \h{1pt}\partial_t \psi^{\delta_k} \\[1mm]
    &+ \int_0^t \int_\Omega \phi\h{1pt}u \cdot \nabla \psi^{\delta_k} + \frac{h^2}{2} \int_0^t \int_\Omega \psi^{\delta_k} \sin 2 \phi + \frac{L_{\mathrm H}}{2} \int_0^t \int_{\mathrm H} \psi^{\delta_k} \sin 2 \phi.
\end{align*} We can also multiply $\phi$ on the first equation in \eqref{eq of psi_delta} and integrate by part. It turns out that \begin{align*}
    \int_0^t \int_\Omega \phi \h{1pt}\partial_t \psi^{\delta_k} &+ \int_0^t \int_\Omega \nabla \psi^{\delta_k} \cdot \nabla \phi \\[1mm]
    &= - \int_0^t \int_\Omega \phi\h{1pt}u_{\delta_k} \cdot \nabla \psi^{\delta_k} + \frac{h^2}{2} \int_0^t \int_\Omega \phi \sin 2 \psi^{\delta_k} + \frac{L_{\mathrm H}}{2} \int_0^t \int_{\mathrm H} \phi \sin 2 \psi^{\delta_k}.
\end{align*} Summing the last two equalities induces \begin{align*}
    \int_{\Omega \times \{t\}} \psi^{\delta_k} \phi &+ 2\int_0^t \int_\Omega \nabla \psi^{\delta_k} \cdot \nabla \phi  = \int_\Omega \phi_0^2 + \int_0^t \int_\Omega \phi\h{1pt}\big(u - u_{\delta_k}\big) \cdot \nabla \psi^{\delta_k} \\[1mm]
     &+ \frac{h^2}{2} \int_0^t \int_\Omega \psi^{\delta_k} \sin 2 \phi + \phi \sin 2 \psi^{\delta_k} + \frac{L_{\mathrm H}}{2} \int_0^t \int_{\mathrm H} \psi^{\delta_k} \sin 2 \phi + \phi \sin 2 \psi^{\delta_k}.
\end{align*}By this equality, \eqref{energ eqa of psi k}, and \eqref{energy eqa phi}, it follows that \begin{align}\label{L2 est of error}
    \int_{\Omega \times \{t\}} \big|\h{.5pt} \phi - \psi^{\delta_k} \h{.5pt}\big|^2 &+ 2 \int_0^t\int_{\Omega} \big|\h{.5pt} \nabla \phi - \nabla \psi^{\delta_k} \h{.5pt}\big|^2 = 2 \int_0^t\int_{\Omega} \bigl(\phi - \psi^{\delta_k}\bigr) \bigl( u_{\delta_k} - u \bigr) \cdot \nabla   \psi^{\delta_k}   \\[1mm]
   & + h^2 \int_0^t\int_{\Omega} (\sin 2 \phi - \sin 2 \psi^{\delta_k}) (\phi - \psi^{\delta_k}) + L_{\mathrm{H}} \int_0^t\int_{\mathrm{H}} (\sin 2 \phi - \sin 2 \psi^{\delta_k}) (\phi - \psi^{\delta_k}). \nonumber
\end{align} 

By the fundamental theorem of calculus,
\begin{align*}
    \int_{\mathrm{H}} \bigl(\phi - \psi^{\delta_k} \bigr)^2 = - \int_{\Omega} \partial_3 \bigl(\phi - \psi^{\delta_k} \bigr)^2 = - 2 \int_{\Omega} \bigl( \phi - \psi^{\delta_k} \bigr) \bigl( \partial_3 \phi - \partial_3 \psi^{\delta_k} \bigr).
\end{align*}
Using Young's inequality then infers
\begin{align}\label{fir est}
    \left|\h{2pt}  h^2  \int_0^t\int_{\Omega} (\sin 2 \phi - \sin 2 \psi^{\delta_k}) (\phi - \psi^{\delta_k}) \right. &+ \left.  L_{\mathrm{H}}  \int_0^t\int_{\mathrm{H}} (\sin 2 \phi - \sin 2 \psi^{\delta_k}) (\phi - \psi^{\delta_k})\h{2pt} \right| \nonumber\\[2mm]
    &\leq \frac{1}{2} \int_0^t\int_\Omega \big|\h{.5pt} \nabla \phi - \nabla \psi^{\delta_k} \h{.5pt}\big|^2 + K_* \int_0^t \int_\Omega \big|\h{.5pt} \phi - \psi^{\delta_k} \h{.5pt}\big|^2.
\end{align}

For the first term on the right-hand side of \eqref{L2 est of error}, we use H\"{o}lder's inequality and \eqref{L2LinftyL2H1 of psi_delta} to estimate it as follows: 
\begin{align}\label{second est}
    &\left|\h{1pt} \int_0^t\int_{\Omega} \bigl(\phi - \psi^{\delta_k}\bigr) \bigl( u_{\delta_k} - u \bigr) \cdot \nabla   \psi^{\delta_k} \h{1pt} \right|\nonumber\\[1mm] &\h{30pt}\leq \left( \int_0^t\int_\Omega \big|\h{.5pt} u_{\delta_k} - u \h{.5pt}\big|^{\frac{10}{3}} \right)^{\frac{3}{10}} \left( \int_0^t\int_\Omega \big|\h{.5pt} \nabla\phi - \nabla\psi^{\delta_k} \h{.5pt}\big|^2 \right)^{\frac{1}{2}} \left(\int_0^t\int_{\Omega} \big|\h{.5pt} \phi - \psi^{\delta_k} \h{.5pt}\big|^5 \right)^{\frac{1}{5}} \nonumber\\[2mm]
    &\h{30pt}\leq M_{*} \left( \int_0^t\int_\Omega \big|\h{.5pt} u_{\delta_k} - u \h{.5pt}\big|^{\frac{10}{3}} \right)^{\frac{3}{10}}   \left(1 + \int_0^t\int_{\Omega} \big|\h{.5pt} \phi \h{.5pt}\big|^5 \right)^{\frac{1}{5}}.
\end{align} Here, $M_* > 0$ is a constant depending on $u$, $\phi$, $\phi_0$ and $T$. We also use the uniform bound \eqref{unif psi del} in the last estimate. By \eqref{infty222}, Lemma \ref{est nabla 2 phi}, and Proposition 3.2 in the Chapter 1 of \cite{Dibenedetto1993}, we obtain $$\text{$\nabla \phi \in L^{\frac{10}{3}} \big( \Omega \times [\h{0.5pt}0, T\h{0.5pt}] \h{1pt} \big)$, which yields}\h{3pt} \phi \in L^\infty\big((0, T); L^2(\Omega)\big) \cap L^{\frac{10}{3}}\big((0, T); W^{1, \frac{10}{3}}(\Omega) \big).$$ Applying Proposition 3.2 in the Chapter 1 of \cite{Dibenedetto1993} again induces $\phi \in L^{\frac{50}{9}} \big( \Omega \times [\h{0.5pt}0, T\h{0.5pt}] \h{1pt} \big)$. The estimate in \eqref{second est} can then be reduced to \begin{align*}
    \left|\h{1pt} \int_0^t\int_{\Omega} \bigl(\phi - \psi^{\delta_k}\bigr) \bigl( u_{\delta_k} - u \bigr) \cdot \nabla   \psi^{\delta_k} \h{1pt} \right| \leq M_{*} \left( \int_0^t\int_\Omega \big|\h{.5pt} u_{\delta_k} - u \h{.5pt}\big|^{\frac{10}{3}} \right)^{\frac{3}{10}}. 
\end{align*}By this estimate and \eqref{fir est}, it then turns out from \eqref{L2 est of error} that \begin{align*} 
    \int_{\Omega \times \{t\}} \big|\h{.5pt} \phi - \psi^{\delta_k} \h{.5pt}\big|^2 \leq M_{*} \left( \int_0^t\int_\Omega \big|\h{.5pt} u_{\delta_k} - u \h{.5pt}\big|^{\frac{10}{3}} \right)^{\frac{3}{10}}  + K_* \int_0^t \int_\Omega \big|\h{.5pt} \phi - \psi^{\delta_k} \h{.5pt}\big|^2 \h{20pt}\text{for any $t \in (0, T)$}. \nonumber
\end{align*} By Gronwall's inequality, the above estimate is reduced to \begin{align}\label{conver of psi delta}
    \int_{\Omega \times \{t\}} \big|\h{.5pt} \phi - \psi^{\delta_k} \h{.5pt}\big|^2 \leq M_{*} \left( \int_0^T\int_\Omega \big|\h{.5pt} u_{\delta_k} - u \h{.5pt}\big|^{\frac{10}{3}} \right)^{\frac{3}{10}} \h{20pt}\text{for any $t \in (0, T)$}.
\end{align}
\noindent Thus, $\phi = \psi^0$ almost everywhere in $\Omega \times (0, T)$ by taking $k \to \infty$ in the last estimate.
\end{proof}

\subsection{Small dissipation energy implies the smallness of \texorpdfstring{$L^3$}{TEXT}-integrals}\label{dis imply L3} Given a suitable weak solution $(u, \phi)$, we consider, for all $r \leq r_0 := d\h{0.5pt}/\h{0.5pt}2$, the following dimensionless quantities:
\begin{align*}
    &\mathrm{A}(r) :=\sup_{t \h{1pt}\in \h{1pt}[\h{0.5pt}t_0 - r^2, \h{1pt} t_0 \h{0.5pt}]}  r^{-1} \int_{B^{\pm}_r(x_0) \times \{\h{0.5pt} t \h{0.5pt}\}} \big|\hspace{.5pt} u \hspace{.5pt}\big|^2 + \big|\hspace{.5pt} \nabla \phi \hspace{.5pt}\big|^2,\hspace{40pt} \mathrm{B}(r) :=  r^{-1} \int_{P_r(z_0)} \big|\hspace{.5pt} \nabla u \hspace{.5pt}\big|^2 + \big|\hspace{.5pt} \nabla^2 \phi\hspace{.5pt}\big|^2, \\[1mm]
    &\mathrm{C}(r) :=  r^{-2} \int_{P_r(z_0)} \big|\hspace{.5pt}u\hspace{.5pt}\big|^3 + \big|\hspace{.5pt}\nabla \phi\hspace{.5pt}\big|^3, \hspace{113.5pt} \mathrm{D}(r) :=  r^{-2} \int_{P_r(z_0)} \big|\hspace{.5pt}p - [\h{0.5pt}p\h{0.5pt}]_{x_0, r}\hspace{.5pt}\big|^{\frac{3}{2}}.
\end{align*}Here, $z_0 = (x_0, t_0)$. $p$ is the induced pressure. $B^{\pm}_r(x_0)$ is the half ball $B_r(x_0) \cap \Omega$. We choose ``$+$" superscript if  $x_0 \in \mathrm H$. If $x_0 \in \mathrm P$, then we choose ``$-$" superscript. In the definition of $\mathrm{D}(r)$, the notation $[\h{0.5pt}p\h{0.5pt}]_{x_0, r}$ is the average of the pressure $p$ on $B_r(x_0) \cap \Omega$. We also define
\begin{equation*}
    \mathrm{D}_1(r) :=  r^{-\frac{3}{2}} \int_{t_0 - r^2}^{t_0} \left(\int_{B^{\pm}_r(x_0)} \big|\hspace{.5pt} \nabla p\hspace{.5pt}\big|^{\frac{9}{8}} \right)^{\frac{4}{3}}.
\end{equation*} \vspace{.1pc}

\noindent By the Poincar\'e-Sobolev inequality, it satisfies \begin{align}\label{D, D_1}\mathrm{D}(r) \lesssim_{K_*} \mathrm{D}_1(r). \end{align} Here and in the following,  $K_* > 0$ is a universal constant. It depends, possibly, only on $h$ and $L_{\mathrm H}$.
\vspace{0.4pc}

\noindent \textbf{I. Estimate of the dissipation energy.} \vspace{0.4pc}

The smallness of dissipation energy is concluded in the following lemma.
\begin{lem}\label{dis small}
Suppose $(u, \phi)$ is a global suitable weak solution of the $\mathrm{IBVP}$. Then it satisfies the global energy inequality in  \eqref{infty222}. Hence, for any $\epsilon > 0$, we have \eqref{smallness} for some sufficiently large time $T_\epsilon$. There also exists a radius \begin{align*}
    r_\epsilon = \min \left\{ r_0, \frac{\epsilon}{K_0} \right\} 
\end{align*} such that the following holds. Given any $r \leq r_\epsilon$, we can find a large time $t_{\epsilon, \h{0.5pt}r}$ such that
    \begin{align*}
        \mathrm{B}(r) \h{1.5pt}\leq\h{1.5pt} \epsilon, \h{20pt}\text{for any $z_0 = (x_0, t_0)$ with $t_0 \geq t_{\epsilon,\h{0.5pt}r} + 10$.}
    \end{align*}Here, $K_0$ is a positive constant depending on $h$, $L_\mathrm{H}$, $\Omega$,  and the solution $(u, \phi)$.
\end{lem}

\begin{proof}[\bf Proof]

Utilizing Lemma \ref{est nabla 2 phi}, or equivalently \eqref{hes, del}, it holds \begin{align*}
        \int_{P_{r}(z_0)} \big|\h{.5pt} \nabla^2 \phi \h{.5pt}\big|^2 &\h{1.5pt} \lesssim_{K_*}\h{1.5pt} \int_{t_0 - r^2}^{t_0} \int_\Omega\big|\h{.5pt} \Delta \phi \h{.5pt}\big|^2 + \int_{t_0 - r^2}^{t_0} \int_{\Omega} \big|\h{.5pt} \nabla \phi \h{.5pt}\big|^2 \\[1.5mm] &\h{1.5pt}\lesssim_{K_*}\h{1.5pt} \int_{t_0 - r^2}^{t_0}\int_{\Omega} \Big|\h{1pt} \Delta \phi + \frac{h^2}{2} \sin 2 \phi \h{1pt}\Big|^2 +  r^2 \h{1pt}  \big|\h{0.5pt} \Omega \h{0.5pt}\big| + r^2  \sup_{t\h{0.8pt}\in\h{0.8pt}[\h{0.8pt}t_0 - r^2, \h{0.5pt} t_0\h{0.6pt}]} \int_\Omega \big|\h{.5pt} \nabla \phi \h{.5pt}\big|^2.
\end{align*}Here, $t_0 \geq 10$. The radius $r \in \left(0, \min \big\{1, r_0\big\}\right)$. Recall \eqref{infty222}. It then turns out\begin{align*}
    r^{-1}\int_{P_{r}(z_0)} \big|\h{.5pt} \nabla^2 \phi \h{.5pt}\big|^2  \h{1.5pt}\leq\h{1.5pt} K_* \h{0.5pt}r^{-1}\int_{t_0 - 1}^{\infty}\int_{\Omega} \Big|\h{1pt} \Delta \phi + \frac{h^2}{2} \sin 2 \phi \h{1pt}\Big|^2 + K_*\big(\h{1pt}|\h{0.5pt} \Omega \h{0.5pt}|   + 1 \big) \h{0.5pt}r.
\end{align*}For any $\epsilon > 0$, we take $$r_\epsilon = \min\left\{ r_0, \h{1pt}\frac{\epsilon}{K_0} \right\}, \h{15pt}\text{where $K_0 := 2 K_*\big(\h{1pt} |\h{0.5pt} \Omega \h{0.5pt} |   + 1 \big)$}. $$ Given $r \leq r_\epsilon$, we then apply \eqref{smallness} to find a $t_{\epsilon,\h{0.5pt}r}$ such that \begin{align*}
    K_* \h{0.5pt} r^{-1}\int_{t_{\epsilon,\h{0.5pt}r}}^{\infty}\int_{\Omega} \Big|\h{1pt} \Delta \phi + \frac{h^2}{2} \sin 2 \phi \h{1pt}\Big|^2 +  r^{-1} \int_{t_{\epsilon,\h{0.5pt}r}}^{\infty} \int_{\Omega} \big|\h{.5pt} \nabla u \h{.5pt}\big|^2\leq \frac{\epsilon}{2}.
\end{align*} The proof is completed.
\end{proof}

\noindent\textbf{II. Some preliminary estimates of $\mathrm A$, $\mathrm B$, $\mathrm C$, $\mathrm D$.}\vspace{0.4pc}

In this part, we introduce some estimates for our future study of the $L^3$-integrals of $u$ and $\nabla \phi$. 
\begin{lem} \label{inequality of C}For any $0 < \rho \leq r \leq r_0$, it satisfies
\begin{equation*}
    \mathrm{C}(\rho) \hspace{1.5pt}\lesssim_{K_*}\hspace{1.5pt} \left(\frac{\rho}{r}\right)^3 \mathrm{A}^{\frac{3}{2}}(r) + \left(\frac{r}{\rho}\right)^{3} \mathrm{A}^{\frac{3}{4}}(r) \h{1pt}\mathrm{B}^{\frac{3}{4}}(r).
\end{equation*}
\end{lem}
\noindent We omit the proof of this lemma, which can be shown by following the proof of Lemma 6.2 in \cite{DuHuWang2019}.  \vspace{0.2pc}

In the next lemma, we introduce a local energy estimate for the suitable weak solution $(u, \phi)$.
\begin{lem}\label{loc est} For any $0 < \rho  \leq r_0$, it satisfies 
\begin{align*}
    \mathrm{A}\left(\frac{\rho}{2}\right) + \mathrm{B}\left(\frac{\rho}{2}\right) \hspace{1.5pt}\lesssim_{K_*, \h{1pt}\phi_0}\hspace{1.5pt} \rho +  \mathrm{C}^{\frac{2}{3}}(\rho) + \mathrm{D}^{\frac{4}{3}}(\rho) + \mathrm{A}^{\frac{1}{2}}(\rho) \h{1.5pt} \mathrm{B}^{\frac{1}{2}}(\rho)\h{1.5pt}\mathrm{C}^{\frac{1}{3}}(\rho).
\end{align*}Here, $z_0 = (x_0, t_0)$ with $t_0 \geq r_0^2 + 10$.
\end{lem}\begin{proof}[\bf Proof] Choose a smooth test function $\varphi(x, t) = \varphi_1(x) \h{0.5pt}\varphi_2(t)$ such that $\varphi_1$ is compactly supported in $B_\rho(x_0)$. It is equal to $1$ on $B_{\rho/2}(x_0)$. As for $\varphi_2$, it is non-decreasing. In addition, it is identically equal to $1$ on $[\h{0.5pt}t_0 - \frac{\rho^2}{4}, t_0\h{0.5pt}]$ and is equal to $0$ if $t \leq t_0 - \rho^2$. We can also assume that $0 \leq \varphi_i \leq 1$, where $i = 1, 2$. Meanwhile, for some positive universal constant $K_*$, it holds  \begin{align*} \rho \h{1pt} \big|\hspace{1.5pt} \nabla \varphi \hspace{.5pt} \big|  + \rho^2 \h{1pt} \big|\hspace{.5pt} \partial_t \varphi \hspace{.5pt}\big| + \rho^2 \h{1pt}\big|\hspace{.5pt} \nabla^2 \varphi \hspace{.5pt} \big| \hspace{1.5pt}\leq\hspace{1.5pt} K_* \h{20pt} \text{ on $P_\rho(z_0)$}.
\end{align*}Replacing the test function in \eqref{energy ineq} with $\varphi^2$ yields, for any $T \in [\h{.5pt} t_0 - \frac{\rho^2}{4}, t_0 \h{.5pt}]$, that
\begin{align}\label{energy in lem 5}
      \rho^{-1}\int_{\Omega \times \{\h{0.5pt}T\h{0.5pt}\}} \varphi^2 \left( \h{1pt} |u|^2 + |\nabla \phi|^2\right)  &+ \rho^{-1}\int_0^T\int_\Omega \varphi^2\left(\h{1pt} | \nabla u |^2 + |\nabla^2 \phi |^2 \right) \\[1mm] & \h{-45pt} \leq K_* \hspace{.5pt}\rho + K_*\h{0.5pt}\rho^{-3} \int_{P_\rho(z_0)} \big|\hspace{.5pt} u \hspace{.5pt} \big|^2 + \big|\hspace{.5pt} \nabla \phi \hspace{.5pt} \big|^2  +  2 \rho^{-1} \int_0^T \int_\Omega \left(u \cdot \nabla \phi\right) \nabla \phi \cdot \nabla \varphi^2  \nonumber\\[1mm]
     &\h{-45pt} + K_*\h{0.5pt}\rho^{-2} \int_{P_\rho(z_0)}  \big|\hspace{.5pt} u \hspace{.5pt} \big| \h{1pt}\Big[ \h{1pt} \big|\hspace{.5pt} p - [\h{.8pt} p \h{.5pt}]_{x_0,\, \rho} \hspace{.8pt}\big| + \big| \h{1pt} |\hspace{.5pt} u \hspace{.5pt} |^2 -  [\h{0.8pt}|\hspace{.8pt} u \hspace{0.8pt}|^2  \h{0.8pt}]_{x_0,\, \rho} \h{1pt}\big| + \big| \h{1pt} |\hspace{.5pt} \nabla \phi \hspace{.5pt} |^2 -  [\h{0.8pt}|\hspace{.5pt} \nabla \phi \hspace{.5pt} |^2 \h{0.8pt}]_{x_0,\, \rho} \h{1pt}\big| \h{1pt}\Big]  \nonumber\\[1mm]
     &\h{-45pt} + K_*\h{0.5pt}\rho^{-2} \int_0^T\int_\text{H} | \nabla' \phi |^2 \h{1pt} \varphi - \rho^{-1}\int_0^T \Big[\h{1pt}\int_\text{P}   \left(\partial_3  \phi \right)^2 \h{1pt} \partial_3 \varphi^2 -  \frac{L_\text{H}^2}{4} \int_{\mathrm H} \partial_3 \varphi^2 \h{1pt} \sin^2 2 \phi \h{1pt} \Big]. \nonumber
\end{align}

\noindent Here, we also use the incompressibility condition of $u$. \vspace{0.2pc}

Using the H\"older inequality, we have  
\begin{align*}
      \rho^{-2} \int_{P_\rho(z_0)} \big|\hspace{.5pt} u \hspace{.5pt} \big| \h{1pt}\big|\hspace{.8pt} p - [\h{.5pt} p \h{.5pt}]_{x_0,\, \rho} \hspace{.8pt}\big|  \hspace{1.5pt}\lesssim_{K_*}\hspace{1.5pt} \mathrm{C}^{\frac{1}{3}}(\rho) \h{1.5pt}\mathrm{D}^{\frac{2}{3}}(\rho).
\end{align*}

Still by the H\"older inequality, it turns out
\begin{align*}
      &\rho^{-2} \int_{P_\rho(z_0)} \big|\hspace{.5pt} u \hspace{.5pt} \big|\h{1pt} \Big[\h{1pt}\big|\h{1pt} |\hspace{.5pt} u \hspace{.5pt} |^2 -  [\h{1pt} |\hspace{.5pt} u \hspace{.5pt} |^2 \h{1pt}]_{x_0,\, \rho} \h{1pt}\big| + \big|\h{1pt} |\hspace{.5pt} \nabla \phi \hspace{.5pt} |^2 -  [\h{1pt} |\hspace{.5pt} \nabla \phi \hspace{.5pt} |^2 \h{1pt}]_{x_0,\, \rho} \h{1pt}\big| \h{1pt}\Big] \\[1mm]
&\hspace{6pt}\lesssim_{K_*}\hspace{1.5pt} \rho^{-2} \int_{t_0 - \rho^2}^{t_0} \big\|\h{.5pt} u \h{.5pt}\big\|_{L^3(B^{\pm}_\rho(x_0) )}  \h{1.5pt}\Big[\h{1pt} \left\|\h{1pt} |\hspace{.5pt} u \hspace{.5pt} |^2 - [\h{1pt} |\hspace{.5pt} u \hspace{.5pt} |^2 \h{1pt}]_{x_0,\, \rho} \h{1pt}\right\|_{L^{\frac{3}{2}}(B^{\pm}_\rho(x_0))} + \left\|\h{1pt} |\hspace{.5pt} \nabla \phi \hspace{.5pt} |^2 -  [\h{1pt} |\hspace{.5pt} \nabla \phi \hspace{.5pt} |^2 \h{1pt}]_{x_0,\, \rho} \h{1pt}\right\|_{L^{\frac{3}{2}}(B^{\pm}_\rho(x_0))} \Big].
\end{align*}
Applying the Sobolev-Poincar\'{e} inequality and H\"{o}lder inequality, we obtain
\begin{align*}
    \Big\|\h{1pt} |\hspace{.5pt} u \hspace{.5pt} |^2 - [\h{1pt} |\hspace{.5pt} u \hspace{.5pt} |^2 \h{1pt}]_{x_0,\, \rho} \h{1pt}\Big\|_{L^{\frac{3}{2}}(B^{\pm}_\rho(x_0))} &+ \Big\|\h{1pt} |\hspace{.5pt} \nabla \phi \hspace{.5pt} |^2 -  [\h{1pt} |\hspace{.5pt} \nabla \phi \hspace{.5pt} |^2 \h{1pt}]_{x_0,\, \rho} \h{1pt}\Big\|_{L^{\frac{3}{2}}(B^{\pm}_\rho(x_0))}\\[1.5mm]& \h{-62pt}  \lesssim_{K_*} \left( \int_{B^{\pm}_\rho(x_0)} \big|\h{.5pt} u \h{.5pt} \big|^2 \right)^{\frac{1}{2}} \left( \int_{B^{\pm}_\rho(x_0)} \big|\h{.5pt} \nabla u \h{.5pt} \big|^2 \right)^{\frac{1}{2}} + \left( \int_{B^{\pm}_\rho(x_0)} \big|\h{.5pt} \nabla \phi \h{.5pt} \big|^2 \right)^{\frac{1}{2}} \left(\int_{B^{\pm}_\rho(x_0)} \big|\h{.5pt} \nabla^2 \phi \h{.5pt} \big|^2 \right)^{\frac{1}{2}}. 
\end{align*}
Therefore, 
\begin{align*}
    \rho^{-2} \int_{P_\rho(z_0)} \big|\hspace{.5pt} u \hspace{.5pt} \big|\h{1pt} \Big[\h{1pt}\big|\h{1pt} |\hspace{.5pt} u \hspace{.5pt} |^2 -  [\h{1pt} |\hspace{.5pt} u \hspace{.5pt} |^2 \h{1pt}]_{x_0,\, \rho} \h{1pt}\big| + \big|\h{1pt} |\hspace{.5pt} \nabla \phi \hspace{.5pt} |^2 -  [\h{1pt} |\hspace{.5pt} \nabla \phi \hspace{.5pt} |^2 \h{1pt}]_{x_0,\, \rho} \h{1pt}\big| \h{1pt}\Big]\h{1.5pt}\lesssim_{K_*}\h{1.5pt}  \mathrm{A}^{\frac{1}{2}}(\rho) \h{1.5pt} \mathrm{B}^{\frac{1}{2}}(\rho) \h{1.5pt}\mathrm{C}^{\frac{1}{3}}(\rho).
\end{align*}

Recall the boundary conditions in \eqref{bc}. We can apply the integration by parts to obtain \begin{align*}
    & K_* \rho^{-2} \int_0^T\int_\text{H} | \nabla' \phi |^2 \h{1pt} \varphi  - \rho^{-1}\int_0^T \left(\int_\text{P}   \left(\partial_3  \phi \right)^2 \h{1pt} \partial_3 \varphi^2 -  \frac{L_\text{H}^2}{4} \int_{\mathrm H} \partial_3 \varphi^2 \h{1pt} \sin^2 2 \phi \right)   \\[2mm]
    &\h{30pt}=\h{1.5pt}  - K_* \rho^{-2} \int_0^T \int_{\Omega}  2 \h{1pt}\varphi  \h{0.5pt}\nabla' \phi \cdot \partial_{3} \nabla' \phi  +   \big|\hspace{.5pt} \nabla' \phi \hspace{.5pt}\big|^2  \h{1pt}\partial_3 \varphi - \rho^{-1} \int_0^T \int_\Omega \partial_3 \left( \left(\partial_3  \phi \right)^2 \h{1pt} \partial_3 \varphi^2 \right).
\end{align*}
Since 
\begin{align*}
    - \int_\Omega \partial_3 \left( \left(\partial_3  \phi \right)^2 \h{1pt} \partial_3 \varphi^2 \right) = - \int_{\Omega} 4 \left(\partial_3 \phi\right) \left(\partial_{33} \phi\right) \h{1pt} \varphi \h{1pt} (\partial_3 \varphi) + 2 \left(\partial_3 \phi\right)^2  \big[  \left(\partial_3 \varphi\right)^2 +  \varphi \h{1pt}  \partial_{33} \varphi \h{1pt} \big], 
\end{align*}
we then obtain by the last two equalities and Young's inequality that
\begin{align*}
     K_* \rho^{-2} \int_0^T\int_\text{H} | \nabla' \phi |^2 \h{1pt} \varphi &- \rho^{-1}\int_0^T \left(\int_\text{P}   \left(\partial_3  \phi \right)^2 \h{1pt} \partial_3 \varphi^2 -  \frac{L_\text{H}^2}{4} \int_{\mathrm H} \partial_3 \varphi^2 \h{1pt} \sin^2 2 \phi \right)    \\[1.5mm]
    &  \leq \sigma \rho^{-1} \int_0^T \int_\Omega \varphi^2 \big| \nabla^2 \phi \big|^2  + c_\sigma\h{1pt} \rho^{-3}  \int_{P_\rho(z_0)} \big| \nabla \phi \big|^2.
\end{align*}Here, $\sigma > 0$ is a smll positive number.\vspace{0.2pc}

Now we deal with the term
\begin{align*}
    \rho^{-1} \int_0^T \int_\Omega (u \cdot \nabla \phi) \h{1pt}\nabla \phi \cdot \nabla \varphi^2.
\end{align*}

\noindent Using the boundary condition $u = 0$ on $\mathrm{P} \cup \mathrm{H}$ and the incompressibility condition of $u$, we perform integration by parts and obtain
\begin{align*}
    \rho^{-1} \int_0^T \int_\Omega (u \cdot \nabla \phi) \h{1pt}\nabla \phi \cdot \nabla \varphi^2  = - \rho^{-1} \int_0^T \int_\Omega \phi \h{1pt} u   \h{1pt}\cdot \nabla \big(\nabla \phi \cdot \nabla \varphi^2 \big).
\end{align*}

\noindent By the boundedness of $\phi$ in Proposition \ref{bounded of phi}, it turns out
\begin{align*}
    \rho^{-1} \left|\h{1.5pt} \int_0^T \int_\Omega (u \cdot \nabla \phi) \h{1pt} \nabla \phi \cdot \nabla \varphi^2 \h{1pt}\right| \h{1.5pt}&\leq\h{1.5pt} \sigma \rho^{-1} \int_0^T \int_{\Omega} \big|\h{.5pt} \nabla^2 \phi \h{.5pt}\big|^2 \varphi^2 + M_{\phi_0} \h{1pt}\sigma^{-1}\rho^{-3}\int_{P_\rho(z_0)} \big|\h{.5pt} u \h{.5pt}\big|^2 + \big|\h{.5pt} \nabla \phi \h{.5pt}\big|^2.
\end{align*}Here, $\sigma > 0$ is a small constant.

Apply all the above arguments to the right-hand side of \eqref{energy in lem 5} and take $\sigma$ small. It follows 
\begin{align}\label{temp energy}
    \rho^{-1}\int_{\Omega \times \{\h{0.5pt}T\h{0.5pt}\}} \varphi^2 \left( \h{1pt} |u|^2 + |\nabla \phi|^2\right)  &+ \rho^{-1}\int_0^T\int_\Omega \varphi^2\left(\h{1pt} | \nabla u |^2 + |\nabla^2 \phi |^2 \right)  \\[1mm] 
    & \h{-46pt} \lesssim_{K_*, \h{1pt}\phi_0}\hspace{1.5pt}\rho + \mathrm{C}^{\frac{1}{3}}(\rho) \h{1.5pt}\mathrm{D}^{\frac{2}{3}}(\rho) +  \mathrm{A}^{\frac{1}{2}}(\rho) \h{1.5pt} \mathrm{B}^{\frac{1}{2}}(\rho) \h{1.5pt}\mathrm{C}^{\frac{1}{3}}(\rho) + \rho^{-3} \int_{P_\rho(z_0)} \big|\hspace{.5pt} u \hspace{.5pt} \big|^2 + \big|\hspace{.5pt} \nabla \phi \hspace{.5pt} \big|^2 \nonumber \\[1mm] 
    & \h{-46pt} \lesssim_{K_*, \h{1pt}\phi_0}\hspace{1.5pt}\rho + \mathrm{C}^{\frac{2}{3}}(\rho) + \mathrm{C}^{\frac{1}{3}}(\rho) \h{1.5pt}\mathrm{D}^{\frac{2}{3}}(\rho) +  \mathrm{A}^{\frac{1}{2}}(\rho) \h{1.5pt} \mathrm{B}^{\frac{1}{2}}(\rho) \h{1.5pt}\mathrm{C}^{\frac{1}{3}}(\rho). \nonumber
\end{align}
The proof is completed by taking supreme over $T \in [\h{0.5pt} t_0 - \frac{\rho^2}{4}, t_0 \h{0.5pt}]$ in \eqref{temp energy}. 
\end{proof} 

Recall \eqref{D, D_1}. To estimate $\mathrm{D}(\rho)$, it suffices to control $\mathrm{D}_1(\rho)$. 
\begin{lem} \label{inequality of D1}
For any $0 < 2 \rho \leq r \leq r_0$, it satisfies
\begin{equation*}
    \mathrm{D}_1(\rho) \hspace{1.5pt}\lesssim_{K_*}\hspace{1.5pt} \left(\dfrac{\rho}{r} \right)^{2}  \big[\h{1pt} \mathrm{D}_1(r) + \mathrm{B}^{\frac{3}{4}}(r) \h{1pt}\big] + \left(\frac{r}{\rho} \right)^{\frac{3}{2}} \big[\h{1pt} \mathrm{A}^{\frac{1}{2}}(r) \h{1pt} \mathrm{B}(r) + \mathrm{A}^{\frac{3}{4}}(r) \h{1pt}\mathrm{B}^{\frac{3}{4}}(r) \h{1pt} \big].
\end{equation*}
\end{lem}

\begin{proof}[\bf Proof]

Utilizing the H\"older inequality, we have
\begin{align*}
    &\int_{t_0 - \rho^2}^{t_0} \big\| \h{1pt} u \cdot \nabla  u + \nabla \cdot \big(\nabla \phi \odot \nabla  \phi \big) \h{1pt}\big\|^{\frac{3}{2}}_{L^{\frac{9}{8}}(B^{\pm}_{\rho}(x_0))} \\[1mm]
    &\h{30pt}\lesssim_{K_*} \int_{t_0 - \rho^2}^{t_0} \big\| \h{0.5pt}u\h{0.5pt}\big\|^{\frac{3}{2}}_{L^{\frac{18}{7}}(B^{\pm}_{\rho}(x_0))} \h{1pt}\big\| \h{0.5pt}\nabla u \h{0.5pt}\big\|^{\frac{3}{2}}_{L^{2}(B^{\pm}_{\rho}(x_0))} + \big\| \h{0.5pt}\nabla \phi \h{0.5pt}\big\|^{\frac{3}{2}}_{L^{\frac{18}{7}}(B^{\pm}_{\rho}(x_0))} \h{1pt}\big\| \h{0.5pt}\nabla^2 \phi \h{0.5pt}\big\|^{\frac{3}{2}}_{L^{2}(B^{\pm}_{\rho}(x_0))}.
\end{align*}

\noindent To control the $L^{\frac{18}{7}}$-norms above, we apply the Gagliardo-Nirenberg inequality and get \begin{align*}
    &\h{10pt}\big\| \h{0.5pt}u\h{0.5pt}\big\|_{L^{\frac{18}{7}}(B^{\pm}_{\rho}(x_0))}\lesssim_{K_*} \big\| \h{0.5pt}\nabla u\h{0.5pt}\big\|^{\frac{1}{3}}_{L^{2}(B^{\pm}_{\rho}(x_0))} \h{1pt}\big\| \h{0.5pt}u\h{0.5pt}\big\|^{\frac{2}{3}}_{L^{2}(B^{\pm}_{\rho}(x_0))} + \rho^{- \frac{1}{3}} \big\|\h{.5pt} u \h{.5pt} \big\|_{L^2(B^\pm_\rho(x_0))}, \\[1mm]
     &\big\| \h{0.5pt}\nabla \phi\h{0.5pt}\big\|_{L^{\frac{18}{7}}(B^{\pm}_{\rho}(x_0))} \lesssim_{K_*} \big\| \h{0.5pt}\nabla^2 \phi\h{0.5pt}\big\|^{\frac{1}{3}}_{L^{2}(B^{\pm}_{\rho}(x_0))} \h{1pt}\big\| \h{0.5pt}\nabla \phi\h{0.5pt}\big\|^{\frac{2}{3}}_{L^{2}(B^{\pm}_{\rho}(x_0))} + \rho^{- \frac{1}{3}} \big\|\h{.5pt} \nabla \phi \h{.5pt} \big\|_{L^2(B^\pm_\rho(x_0))}.
\end{align*}

\noindent Then it follows that
\begin{align*}
    &\int_{t_0 - \rho^2}^{t_0} \big\| \h{1pt} u \cdot \nabla  u + \nabla \cdot \big(\nabla \phi \odot \nabla  \phi \big) \h{1pt}\big\|^{\frac{3}{2}}_{L^{\frac{9}{8}}(B^{\pm}_{\rho}(x_0))} \\[2mm]
    &\h{30pt}\lesssim_{K_*} \int_{t_0 - \rho^2}^{t_0} \big\| \h{0.5pt} \nabla u\h{0.5pt}\big\|^{2}_{L^{2}(B^{\pm}_{\rho}(x_0))} \h{1pt}\big\| \h{0.5pt} u \h{0.5pt}\big\|_{L^{2}(B^{\pm}_{\rho}(x_0))} + \rho^{- \frac{1}{2}}\big\| \h{0.5pt} \nabla u \h{0.5pt}\big\|^{\frac{3}{2}}_{L^{2}(B^{\pm}_{\rho}(x_0))} \h{1pt}\big\| \h{0.5pt} u \h{0.5pt}\big\|^{\frac{3}{2}}_{L^{2}(B^{\pm}_{\rho}(x_0))}\\[2mm]
    &\h{43pt}+ \int_{t_0 - \rho^2}^{t_0} \big\| \h{0.5pt} \nabla^2 \phi \h{0.5pt}\big\|^{2}_{L^{2}(B^{\pm}_{\rho}(x_0))} \h{1pt}\big\| \h{0.5pt} \nabla \phi \h{0.5pt}\big\|_{L^{2}(B^{\pm}_{\rho}(x_0))} + \rho^{- \frac{1}{2}}\big\| \h{0.5pt} \nabla^2 \phi \h{0.5pt}\big\|^{\frac{3}{2}}_{L^{2}(B^{\pm}_{\rho}(x_0))} \h{1pt}\big\| \h{0.5pt} \nabla \phi \h{0.5pt}\big\|^{\frac{3}{2}}_{L^{2}(B^{\pm}_{\rho}(x_0))}.
\end{align*}

\noindent Therefore, 
\begin{align*}
    \left\|\h{1pt} u \cdot \nabla u + \nabla \cdot \big(\nabla \phi \odot \nabla  \phi \big) \h{1pt} \right\|_{\frac{9}{8}, \frac{3}{2}, P_\rho(z_0)} \hspace{1.5pt}\lesssim_{K_*}\hspace{1.5pt} \rho \h{1pt} \mathrm{A}^{\frac{1}{3}}(\rho) \h{1pt} \mathrm{B}^{\frac{2}{3}}(\rho) + \rho \h{1pt} \mathrm{A}^{\frac{1}{2}}(\rho) \h{1pt}\mathrm{B}^{\frac{1}{2}}(\rho).
\end{align*}

Let $\big(v, q_1 \big)$ be a solution to the following initial boundary value problem:
\begin{equation}\label{eq 11} \left\{ \begin{aligned}
    \partial_t v - \Delta v + \nabla q_1 &= - u \cdot \nabla u - \nabla \cdot \big(\nabla \phi \odot \nabla  \phi \big) \quad  \h{9pt} \text{in } P_{\rho}(z_0); \\[.5mm]
    \operatorname{div} v &= 0 \quad\hspace{127pt} \text{in } P_{\rho}(z_0); \\[.5mm]
    v &= 0 \quad\hspace{127pt} \text{on } \mathscr{P} P_{\rho}(z_0).
\end{aligned} \right. \end{equation} 
Here, $\mathscr{P} P_{\rho}(z_0)$ denotes the parabolic boundary of $\mathscr{P} P_{\rho}(z_0)$. By Theorem 1.1 in \cite{Solonnikov2003}, 
\begin{align}\label{est 11}
     \rho^{-2} \| \h{0.5pt} v \h{0.5pt}\|_{\frac{9}{8}, \frac{3}{2}, P_\rho(z_0)} + \rho^{-1} \|\h{0.5pt} \nabla v \h{0.5pt}\|_{\frac{9}{8},\frac{3}{2}, P_\rho(z_0)} +   \|\h{0.5pt} \nabla q_1 \h{0.5pt}\|_{\frac{9}{8}, \frac{3}{2}, P_\rho(z_0)}  \lesssim_{K_*}  \rho\h{1pt}\mathrm{A}^{\frac{1}{3}}(\rho) \h{1pt} \mathrm{B}^{\frac{2}{3}}(\rho) + \rho\h{1pt}\mathrm{A}^{\frac{1}{2}}(\rho) \h{1pt} \mathrm{B}^{\frac{1}{2}}(\rho). 
\end{align}

Define $w := u - v$ and $q_2 := p - [\h{0.5pt}p\h{0.5pt}]_{x_0, \rho} - q_1$. It then holds
\begin{equation}\label{eq 12} \left\{ \begin{aligned}
    \partial_t w - \Delta w + \nabla q_2 &= 0 \quad \h{5.5pt}\text{in } P_{\rho}(z_0); \\[1mm]
    \operatorname{div} w &= 0 \quad \h{6pt} \text{in } P_{\rho}(z_0); \\[1mm]
    w &= 0 \quad \h{6pt}\text{on } \Big\{ \partial B^{\pm}_{\rho}(x_0) \cap \mathrm{H} \Big\} \times [\h{0.5pt}t_0 - \rho^2, t_0\h{0.5pt}].
\end{aligned} \right. \end{equation}
Using Proposition 2 in \cite{Seregin2003}, we obtain  
\begin{align*}
    \|\h{0.5pt} \nabla q_2 \h{0.5pt}\|_{9, \frac{3}{2}, P_{ \rho\h{0.5pt}/\h{0.5pt}2}(z_0)} \h{1.5pt}\lesssim_{K_*}\h{1.5pt} \rho^{- \frac{7}{3}} \left[\h{1pt} \rho^{-2} \|\h{0.5pt} w \h{0.5pt}\|_{\frac{9}{8}, \frac{3}{2}, P_{\rho}(z_0)} + \rho^{-1} \|\h{0.5pt} \nabla w \h{0.5pt}\|_{\frac{9}{8}, \frac{3}{2}, P_{\rho}(z_0)} + \rho^{-1} \|\h{0.5pt}q_2 \h{0.5pt}\|_{\frac{9}{8}, \frac{3}{2}, P_{\rho}(z_0)} \h{1pt}\right].
\end{align*}
We continue to estimate the right-hand side by the triangle and the Poincar\'e inequalities. Hence,
\begin{align*}
     \rho^{\frac{7}{3}} \|\h{0.5pt} \nabla q_2 \h{0.5pt}\|_{9, \frac{3}{2}, P_{ \rho\h{0.5pt}/\h{0.5pt}2}(z_0)}  \lesssim_{K_*}    \rho^{-1}\|\h{0.5pt} \nabla u \h{0.5pt}\|_{\frac{9}{8}, \frac{3}{2}, P_{\rho}(z_0)} + \|\h{0.5pt} \nabla p \h{0.5pt}\|_{\frac{9}{8}, \frac{3}{2}, P_{\rho}(z_0)} + \rho^{-1}\|\h{0.5pt} \nabla v \h{0.5pt}\|_{\frac{9}{8}, \frac{3}{2}, P_{\rho}(z_0)} +  \|\h{0.5pt} \nabla q_1 \h{0.5pt}\|_{\frac{9}{8}, \frac{3}{2}, P_{\rho}(z_0)}.
\end{align*} Here, we also assume $[\h{0.5pt}q_1\h{0.5pt}]_{x_0, \rho} = 0$. Note that the H\"older inequality infers
\begin{align*}
    \|\h{1pt} \nabla u \h{1pt}\|_{\frac{9}{8}, \frac{3}{2}, P_\rho(z_0)} + \|\h{1pt} \nabla^2 \phi \h{1pt}\|_{\frac{9}{8},  \frac{3}{2}, P_\rho(z_0)} \hspace{1.5pt}\lesssim_{K_*}\hspace{1.5pt} \rho^2 \h{1pt}\mathrm{B}^{\frac{1}{2}}(\rho).
\end{align*}
It then turns out, from the above arguments, that
\begin{align*}
   \rho^{\frac{4}{3}} \|\h{0.5pt} \nabla q_2 \h{0.5pt}\|_{9, \frac{3}{2}, P_{\rho\h{0.5pt}/\h{0.5pt}2}(z_0)}  \lesssim_{K_*}  \mathrm{B}^{\frac{1}{2}}(\rho) + \mathrm{D}_1^{\frac{2}{3}}(\rho) + \mathrm{A}^{\frac{1}{3}}(\rho)\h{1.5pt} \mathrm{B}^{\frac{2}{3}}(\rho) + \mathrm{A}^{\frac{1}{2}}(\rho) \h{1.5pt} \mathrm{B}^{\frac{1}{2}}(\rho).
\end{align*}

For any $0 < 2\rho \leq r \leq r_0$, we have
\begin{align*}
    \mathrm{D}_1(\rho) &\h{1.5pt}\lesssim_{K_*}\h{1.5pt} \rho^{-\frac{3}{2}} \int_{t_0 - \rho^2}^{t_0} \left( \int_{B^{\pm}_{\rho}(x_0)} \big|\hspace{.5pt} \nabla q_1 \hspace{.5pt}\big|^{\frac{9}{8}} \right)^{\frac{4}{3}} + \rho^2 \int_{t_0 - \rho^2}^{t_0} \left( \int_{B^{\pm}_{\rho}(x_0)} \big|\hspace{.5pt} \nabla q_2 \hspace{.5pt}\big|^{9} \right)^{\frac{1}{6}} \\[2mm]
    &\hspace{1.5pt}\lesssim_{K_*}\hspace{1.5pt} \rho^{-\frac{3}{2}} \int_{t_0 - r^2}^{t_0} \left( \int_{B^{\pm}_{r}(x_0)} \big|\hspace{.5pt} \nabla q_1 \hspace{.5pt}\big|^{\frac{9}{8}} \right)^{\frac{4}{3}} + \rho^2 \int_{t_0 - r^2/4}^{t_0} \left( \int_{B^{\pm}_{r\h{0.5pt}/\h{0.5pt}2}(x_0)} \big|\hspace{.5pt} \nabla q_2 \hspace{.5pt}\big|^{9} \right)^{\frac{1}{6}}. \end{align*} Our estimates for $q_1$ and $q_2$ above then induce \begin{align*}
    \mathrm D_1(\rho) &\hspace{1.5pt}\lesssim_{K_*}\hspace{1.5pt} \left(\frac{r}{\rho}\right)^{\frac{3}{2}} \big[ \h{1pt}\mathrm{A}^{\frac{1}{2}}(r) \h{1.5pt} \mathrm{B}(r) + \mathrm{A}^{\frac{3}{4}}(r) \h{1.5pt} \mathrm{B}^{\frac{3}{4}}(r) \h{1pt}\big] + \left(\frac{\rho}{r}\right)^2 \big[ \h{1pt} \mathrm{B}^{\frac{3}{4}}(r) + \mathrm{D}_1(r)  \h{1pt} \big] .
\end{align*}

\noindent The desired is thus obtained.
\end{proof}\vspace{0.2pc}

\noindent \textbf{III. The smallness assumptions.} This part is devoted to verifying some smallness assumptions that will be used in Section \ref{L infty of u phi}. 

\begin{lem} \label{smallness of A C D}
For any $\epsilon > 0$ suitably small, there exists a radius $\rho_\epsilon < r_\epsilon$ such that
\begin{align*}
    \max \Bigl\{ \mathrm{A}(\rho_\epsilon),\ \mathrm{C}(\rho_\epsilon),\ \mathrm{D}_1(\rho_\epsilon) \Bigr\} \h{1.5pt}\leq\h{1.5pt} \epsilon \h{20pt}\text{for any $z_0 = (x_0, t_0)$ with $t_0$ suitably large.}
\end{align*}The largeness of $t_0$ depends on $\epsilon$.
\end{lem}

\begin{proof}[\bf Proof.] We divide the proof into 3 steps. \\[2mm]
\textbf{Step 1. Iterative argument.} For any $\rho \in (0, r_\epsilon]$ and $\theta \in (0, \frac{1}{2}\h{0.5pt}]$, it follows from Lemma \ref{loc est} that
\begin{align*}
    \mathrm{A}^{\frac{3}{2}}\left(\frac{1}{2} \h{0.5pt} \theta \rho \right) + \mathrm{B}^{\frac{3}{2}}\left(\frac{1}{2} \h{0.5pt} \theta \rho\right) \h{1.5pt}\lesssim_{K_*,\h{1pt}\phi_0}\h{1.5pt} \mathrm{C}\left(\theta \rho\right) + \mathrm{A}^{\frac{3}{2}}\left(\theta \rho\right) \mathrm{B}^{\frac{3}{2}}\left(\theta \rho\right) + \mathrm{D}_1^{2}\left(\theta \rho\right) + \left(\theta \rho\right)^{\frac{3}{2}}.
\end{align*}

\noindent Also, by Lemma \ref{inequality of D1},
\begin{align*}
    \mathrm{D}_1^2\left(\theta \rho\right) \hspace{1.5pt}\lesssim_{K_*}\hspace{1.5pt} \theta^4 \h{1pt} \Bigl[ \mathrm{B}^{\frac{3}{2}}(\rho) + \mathrm{D}_1^2(\rho) \Bigr] + \theta^{-3} \h{1pt} \Bigl[ \mathrm{A}(\rho) \h{1.5pt}\mathrm{B}^2(\rho) + \mathrm{A}^{\frac{3}{2}}(\rho) \h{1.5pt} \mathrm{B}^{\frac{3}{2}}(\rho) \Bigr].
\end{align*}

\noindent Applying the above estimates, together with Lemma \ref{inequality of C}, we obtain
\begin{align*}
    \mathrm{A}^{\frac{3}{2}}\left(\frac{1}{2} \h{0.5pt}\theta \rho\right) + \mathrm{B}^{\frac{3}{2}}\left(\frac{1}{2} \h{0.5pt} \theta \rho\right) \h{1.5pt}&\lesssim_{K_*, \h{1pt}\phi_0}\h{1.5pt} \theta^3 \mathrm{A}^{\frac{3}{2}}(\rho) + \theta^{-3} \mathrm{A}^{\frac{3}{4}}(\rho) \h{1.5pt} \mathrm{B}^{\frac{3}{4}}(\rho) + \theta^{-3} \mathrm{A}^{\frac{3}{2}}(\rho) \h{1.5pt}\mathrm{B}^{\frac{3}{2}}(\rho) \\[1mm]
    &\h{27pt}+\left(\theta \rho\right)^{\frac{3}{2}} +  \theta^4 \h{1pt} \Bigl[ \mathrm{B}^{\frac{3}{2}}(\rho) + \mathrm{D}_1^2(\rho) \Bigr] + \theta^{-3} \h{1pt} \Bigl[ \mathrm{A}(\rho) \h{1.5pt} \mathrm{B}^2(\rho) + \mathrm{A}^{\frac{3}{2}}(\rho) \h{1.5pt} \mathrm{B}^{\frac{3}{2}}(\rho) \Bigr].
\end{align*}

\noindent We introduce a new quantity
$\mathrm{E}(\rho)  :=  \mathrm{A}^{\frac{3}{2}}(\rho) + \mathrm{D}^2_1(\rho)$ and apply the Young's inequality. The last two estimates then yield \begin{align*}
    \mathrm{E}\left(\frac{1}{2} \h{0.5pt}\theta \rho\right) \h{1.5pt} \lesssim_{K_*, \h{1pt}\phi_0}\h{1.5pt} \mathrm{E}(\rho) \Bigl[ \theta^3 + \theta^{-3} \h{1pt} \mathrm{B}^{\frac{3}{2}}(\rho) + \theta^4 \Bigr] + \Bigl[ \theta^{-9} \h{1pt} \mathrm{B}^{\frac{3}{2}}(\rho) + \theta^4 \h{1pt} \mathrm{B}^{\frac{3}{2}}(\rho) + \theta^{-15} \h{1pt}\mathrm{B}^6(\rho)  \Bigr] + \left(\theta \rho\right)^{\frac{3}{2}}.
\end{align*} 
From Lemma \ref{dis small}, for any $0 < \rho \leq r_\epsilon$, there exists a $T_{\epsilon, \h{1pt}\rho}$ large enough such that
\begin{align}\label{est B}
     \mathrm{B}^{\frac{3}{4}}(\rho) < \epsilon^9 \h{15pt}\text{for any $z_0 = (x_0, t_0)$ with $t_0 \geq T_{\epsilon, \h{1pt}\rho}$}.
\end{align}

\noindent We then take $\theta = \epsilon$ and conclude from the last two estimates that \begin{align*}
    \mathrm{E}\left(\frac{1}{2} \epsilon \rho \right) \h{1.5pt}\leq\h{1.5pt} \frac{1}{2} \mathrm{E}(\rho) + \frac{1}{4} \epsilon^2 \h{15pt}\text{for any $0 < \rho \leq r_\epsilon$ and $z_0 = (x_0, t_0)$ with $t_0 \geq T_{\epsilon, \h{1pt}\rho}$.}
\end{align*}Here, the smallness of $\epsilon$ depends on $K_*$ and $\phi_0$. Iterating this inequality yields that
\begin{align}\label{ind ar}
    \mathrm{E}\left(\left(\frac{\epsilon}{2} \right)^k r_\epsilon \right) \h{1.5pt}\leq\h{1.5pt} \frac{1}{2^{k-1}} \h{1pt}\mathrm{E}\left( \frac{\epsilon}{2}\h{1pt}r_\epsilon\right) + \frac{1}{2} \epsilon^2 \h{15pt}\text{for any $z_0 = (x_0, t_0)$ with $t_0 \geq T_{k, \h{1pt} \epsilon}$.}
\end{align}Here, $T_{k, \h{1pt} \epsilon} > 0$ is a large time. \vspace{0.4pc}

\noindent \textbf{Step 2. Estimates of $\mathrm A(r_\epsilon)$ and $\mathrm D_1(r_\epsilon)$.} First, for some positive constant $M_{u, \h{1pt}\phi}$, which depends on the suitable weak solution $(u, \phi)$, we have
\begin{align}\label{est of A}
    \mathrm{A}(r_\epsilon) &\h{1.5pt}\leq\h{1.5pt} r_\epsilon^{-1} \sup_{t \h{1pt}\geq\h{1pt} 0} \int_{\Omega \times \{ t \}} \big|\h{.5pt} u \h{.5pt}\big|^2 + \big|\h{.5pt} \nabla \phi \h{.5pt}\big|^2 \h{1.5pt}\leq\h{1.5pt} M_{u, \h{1pt}\phi} \h{1.5pt}r_\epsilon^{-1}.
\end{align}

Now we bound $\mathrm{D}_1(r_\epsilon)$. Arguing by contradiction, we can find a $t_* \in \big[\h{0.5pt}t_0 - 2 r_\epsilon^2, t_0 - r_\epsilon^2\h{0.5pt}\big]$ such that
\begin{align*}
    \int_{\Omega \times \{ t_* \}} \big|\h{.5pt} \nabla u \h{.5pt}\big|^2  \leq M_{u, \h{1pt}\phi} \h{1.5pt} r_\epsilon^{-2}.
\end{align*}

\noindent Since $u(t_*, \cdot) \in H^1(\Omega) = B_{2, 2}^1(\Omega) \hookrightarrow B_{\frac{9}{8}, \frac{3}{2}}^{\frac{2}{3}}(\Omega)$, by Theorem 1.1 in \cite{Solonnikov2003}, it follows that
\begin{align*}
   r_\epsilon^{\frac{3}{2}} \h{1pt} \mathrm{D}_1(r_\epsilon) &\leq   \int_{t_*}^{t_0} \left(\int_{\Omega} \big|\hspace{.5pt} \nabla p \hspace{.5pt}\big|^{\frac{9}{8}} \right)^{\frac{4}{3}}\\[1mm]  &\lesssim_{K_*} \left(\int_{\Omega \times \{ t_* \}} \big|\h{.5pt} \nabla u \h{.5pt}\big|^2\right)^{\frac{3}{4}} + \Big\|\h{.5pt} u \cdot \nabla u + \nabla \cdot \big(\nabla \phi \odot \nabla  \phi \big)  \h{.5pt}\Big\|_{L^{\frac{3}{2}}\big(\big(t_0 - 2 \h{0.5pt} r_\epsilon^2, t_0\big) ;\, L^{\frac{9}{8}}(\Omega)\big)}^{\frac{3}{2}}.
\end{align*}

\noindent The second term in the last line above can be estimated the same as in the proof of Lemma \ref{inequality of D1}. Utilizing H\"older inequality, we obtain
\begin{align*}
    &\int_{t_0 - 2 r_\epsilon^2}^{t_0} \big\| \h{1pt} u \cdot \nabla  u + \nabla \cdot \big(\nabla \phi \odot \nabla  \phi \big) \h{1pt}\big\|^{\frac{3}{2}}_{L^{\frac{9}{8}}(\Omega)}  \lesssim_{K_*} \int_{t_0 - 2 r_\epsilon^2}^{t_0} \big\| \h{0.5pt}u\h{0.5pt}\big\|^{\frac{3}{2}}_{L^{\frac{18}{7}}(\Omega)} \h{1pt}\big\| \h{0.5pt}\nabla u \h{0.5pt}\big\|^{\frac{3}{2}}_{L^{2}(\Omega)} + \big\| \h{0.5pt}\nabla \phi \h{0.5pt}\big\|^{\frac{3}{2}}_{L^{\frac{18}{7}}(\Omega)} \h{1pt}\big\| \h{0.5pt}\nabla^2 \phi \h{0.5pt}\big\|^{\frac{3}{2}}_{L^{2}(\Omega)}.
\end{align*}

\noindent To control the $L^{\frac{18}{7}}$-norms above, we apply the Gagliardo-Nirenberg inequality. Hence, \begin{align*}
    &\h{10pt}\big\| \h{0.5pt}u\h{0.5pt}\big\|_{L^{\frac{18}{7}}(\Omega)}\lesssim_{K_*} \big\| \h{0.5pt}\nabla u\h{0.5pt}\big\|^{\frac{1}{3}}_{L^{2}(\Omega)} \h{1pt}\big\| \h{0.5pt}u\h{0.5pt}\big\|^{\frac{2}{3}}_{L^{2}(\Omega)} + \big\|\h{.5pt} u \h{.5pt} \big\|_{L^2(\Omega)}, \\[1mm]
     &\big\| \h{0.5pt}\nabla \phi\h{0.5pt}\big\|_{L^{\frac{18}{7}}(\Omega)} \lesssim_{K_*} \big\| \h{0.5pt}\nabla^2 \phi\h{0.5pt}\big\|^{\frac{1}{3}}_{L^{2}(\Omega)} \h{1pt}\big\| \h{0.5pt}\nabla \phi\h{0.5pt}\big\|^{\frac{2}{3}}_{L^{2}(\Omega)} + \big\|\h{.5pt} \nabla \phi \h{.5pt} \big\|_{L^2(\Omega)}.
\end{align*}

\noindent Then it follows that
\begin{align*}
    \int_{t_0 - 2 r_\epsilon^2}^{t_0} \big\| \h{1pt} u \cdot \nabla  u + \nabla \cdot \big(\nabla \phi \odot \nabla  \phi \big) \h{1pt}\big\|^{\frac{3}{2}}_{L^{\frac{9}{8}}(\Omega)} \lesssim_{K_*} &\int_{t_0 - 2 r_\epsilon^2}^{t_0} \big\| \h{0.5pt} \nabla u\h{0.5pt}\big\|^{2}_{L^{2}(\Omega)} \h{1pt}\big\| \h{0.5pt} u \h{0.5pt}\big\|_{L^{2}(\Omega)} + \big\| \h{0.5pt} \nabla u \h{0.5pt}\big\|^{\frac{3}{2}}_{L^{2}(\Omega)} \h{1pt}\big\| \h{0.5pt} u \h{0.5pt}\big\|^{\frac{3}{2}}_{L^{2}(\Omega)}\\[1mm]
    + &\int_{t_0 - 2 r_\epsilon^2}^{t_0} \big\| \h{0.5pt} \nabla^2 \phi \h{0.5pt}\big\|^{2}_{L^{2}(\Omega)} \h{1pt}\big\| \h{0.5pt} \nabla \phi \h{0.5pt}\big\|_{L^{2}(\Omega)} + \big\| \h{0.5pt} \nabla^2 \phi \h{0.5pt}\big\|^{\frac{3}{2}}_{L^{2}(\Omega)} \h{1pt}\big\| \h{0.5pt} \nabla \phi \h{0.5pt}\big\|^{\frac{3}{2}}_{L^{2}(\Omega)}.
\end{align*}

\noindent By Lemma \ref{est nabla 2 phi}, we have
\begin{align*}
    \int_\Omega \big|\h{.5pt} \nabla^2 \phi \h{.5pt}\big|^2 \h{1.5pt}\lesssim_{K_*}\h{1.5pt} \int_{\Omega} \big|\h{.5pt} \Delta \phi + \frac{h^2}{2} \sin 2\phi \h{.5pt}\big|^2 + \big| \Omega \big| + \int_\Omega \big|\h{.5pt} \nabla \phi \h{.5pt}\big|^2.
\end{align*}

\noindent The last two estimates and \eqref{infty222} infer that
\begin{align*}
    &\int_{t_0 - 2 r_\epsilon^2}^{t_0} \big\| \h{1pt} u \cdot \nabla  u + \nabla \cdot \big(\nabla \phi \odot \nabla  \phi \big) \h{1pt}\big\|^{\frac{3}{2}}_{L^{\frac{9}{8}}(\Omega)} \h{1.5pt}\leq M_{u, \h{1pt} \phi}.
\end{align*}
We then conclude from the above discussions that \begin{align}\label{est D_1}
\mathrm D_1(r_\epsilon) \leq M_{u, \h{1pt}\phi} \h{1.5pt}r_\epsilon^{-3}.
\end{align}

\noindent \textbf{Step 3.} Applying \eqref{est of A} and \eqref{est D_1} to the right-hand side of \eqref{ind ar} yields \begin{align*}
    \mathrm{E}\left(\left(\frac{\epsilon}{2} \right)^k r_\epsilon \right) &\leq \frac{1}{2^{k-1} }\h{1pt}\frac{K_*}{\epsilon^{3}} \h{1pt}  \mathrm{E}\left( r_\epsilon\right) + \frac{1}{2} \h{1pt}\epsilon^2 \leq \frac{1}{2^{k-1}} \h{1pt} \frac{M_{u, \h{1pt}\phi}}{\epsilon^{3} \h{1pt}r_\epsilon^{6}}   + \frac{1}{2} \epsilon^2 \h{1.5pt} \h{15pt}\text{for any $z_0 = (x_0, t_0)$ with $t_0 \geq T_{k, \h{1pt} \epsilon}$.}
\end{align*}

\noindent By this estimate and Lemma $\ref{inequality of C}$, it further turns out that
\begin{align*}
    \mathrm{C}\left(\left(\frac{\epsilon}{2}\right)^k r_\epsilon \right) &\hspace{1.5pt}\leq\hspace{1.5pt} K_* \left(\frac{\epsilon}{2}\right)^3 \mathrm{A}^{\frac{3}{2}}\left(\left(\frac{\epsilon}{2}\right)^{k-1} r_\epsilon \right) + K_* \left(\frac{\epsilon}{2}\right)^{-3} \mathrm{A}^{\frac{3}{4}}\left(\left(\frac{\epsilon}{2}\right)^{k-1} r_\epsilon \right) \h{1pt}\mathrm{B}^{\frac{3}{4}}\left(\left(\frac{\epsilon}{2} \right)^{k-1} r_\epsilon \right) \\[2mm]
    &\h{1.5pt}\leq\h{1.5pt} K_* \left(\frac{1}{2^{k-2}} \h{1pt} \frac{M_{u, \h{1pt}\phi}}{r_\epsilon^{6}}   + \frac{1}{2} \epsilon^5 \right) + K_* \h{1pt}\epsilon^{-3} \left( \frac{1}{2^{k-2}} \h{1pt} \frac{M_{u, \h{1pt}\phi}}{\epsilon^{3} \h{1pt}r_\epsilon^{6}}   + \frac{1}{2} \epsilon^2 \right)^{\frac{1}{2}} \mathrm{B}^{\frac{3}{4}}\left(\left(\frac{\epsilon}{2}\right)^{k-1} r_\epsilon \right).
\end{align*}

\noindent We take $k = k_\epsilon$ suitably large and denote $\rho_\epsilon := \left(\frac{\epsilon}{2}\right)^{k_\epsilon} r_\epsilon$. The last two estimates then infer that 
\begin{align*}
    \mathrm{E}\left(\rho_\epsilon \right) + \mathrm{C}\left(\rho_\epsilon \right) \leq \frac{3}{4}\h{1pt}\epsilon^2 + \epsilon^{-3} \h{1pt}\mathrm{B}^{\frac{3}{4}}\left(\left(\frac{\epsilon}{2}\right)^{k_\epsilon-1} r_\epsilon \right) \h{15pt}\text{for any $z_0 = (x_0, t_0)$ with $t_0 \geq T_{k_\epsilon, \h{1pt} \epsilon}$.}
\end{align*}

\noindent The proof of the lemma then follows by the last estimate and Lemma \ref{dis small}.
\end{proof}
 
\begin{cor} \label{smallness of extra term}
    For the $\rho_\epsilon$ found in Lemma \ref{smallness of A C D}, we have
 \begin{align*}
        \rho_\epsilon^{-5} \int_{P_{\rho_\epsilon}(z_0)}  \big|\h{.5pt} \phi - (\phi)_{z_0, \rho_\epsilon} \big|^3 \leq K_* \epsilon \h{20pt}\text{for all $t \geq t_\epsilon$}.
    \end{align*} Here, $t_\epsilon$ is a sufficiently large time.
\end{cor}

\begin{proof}[\bf Proof.]
For some universal positive constant $K_*$, it holds that 
\begin{align}\label{tei angle}
    \int_{P_{\rho_\epsilon}(z_0)}  \big|\h{.5pt} \phi - (\phi)_{z_0, \rho_\epsilon} \big|^3 \lesssim_{K_*} \int_{P_{\rho_\epsilon}(z_0)}  \big|\h{.5pt} \phi - [\phi]_{x_0, \rho_\epsilon} \big|^3 + \int_{P_{\rho_\epsilon}(z_0)}  \big|\h{1.5pt} [\phi]_{x_0, \rho_\epsilon} - (\phi)_{z_0, \rho_\epsilon} \h{1pt} \big|^3.
\end{align}

Using Poincar\'e's inequality and Lemma \ref{smallness of A C D}, we have
\begin{align}\label{fir poin ca}
    \rho_\epsilon^{-5} \int_{P_{\rho_\epsilon}(z_0)} \big|\h{.5pt} \phi - [\phi]_{x_0, \rho_\epsilon} \big|^3 \lesssim_{K_*} \rho_\epsilon^{-2} \int_{P_{\rho_\epsilon}(z_0)} \big|\h{.5pt} \nabla \phi \h{.5pt} \big|^3 \leq \mathrm{C}(\rho_\epsilon) \leq  \epsilon.
\end{align}

For the second term on the right-hand side of \eqref{tei angle}, it can be estimated  by
\begin{align}\label{aver est}
    \int_{P_{\rho_\epsilon}(z_0)} \big|\h{1.5pt} [\phi]_{x_0, \rho_\epsilon} - (\phi)_{z_0, \rho_\epsilon} \big|^3 \lesssim_{K_*} \rho_\epsilon^3 \int_{t_0 - \rho_\epsilon^2}^{t_0} \big|\h{1.5pt} [\phi]_{x_0, \rho_\epsilon}(t) - (\phi)_{z_0, \rho_\epsilon} \big|^3 \h{2pt}\mathrm d t.
\end{align}

\noindent Note that
\begin{align*}
    [\phi]_{x_0, \rho_\epsilon}(t) - (\phi)_{z_0, \rho_\epsilon} &= [\phi]_{x_0, \rho_\epsilon}(t) - \rho_\epsilon^{-2} \int^{t_0}_{t_0 - \rho_\epsilon^2}  \big|\h{.5pt} B_{\rho_\epsilon}(x_0) \h{.5pt}\big|^{-1} \int_{B_{\rho_\epsilon}(x_0)} \phi(y, s) \h{1pt}\D y \h{1pt} \D s \\[2mm]
    &=  \rho_\epsilon^{-2} \int_{t_0 - \rho_\epsilon^2}^{t_0} [\phi]_{x_0, \rho_\epsilon}(t) - [\phi]_{x_0, \rho_\epsilon}(s) \h{1pt} \D s.
\end{align*}

\noindent The problem is therefore reduced to estimating the last line above. Integrating the equation of $\phi$ over the ball $B_{\rho_\epsilon}(x_0)$, we obtain
\begin{align*}
    \partial_t \int_{B_{\rho_\epsilon}(x_0)} \phi + \int_{B_{\rho_\epsilon}(x_0)} u \cdot \nabla \phi = \int_{B_{\rho_\epsilon}(x_0)} \Delta \phi + \frac{h^2}{2} \sin 2 \phi.
\end{align*}

\noindent Then we integrate with respect to time for $t_0 - \rho^2_\epsilon \leq s \leq t \leq t_0$. Hence,
\begin{align*}
    \int_{B_{\rho_\epsilon}(x_0) \times \{ t \}} \phi - \int_{B_{\rho_\epsilon}(x_0) \times \{ s \}} \phi = - \int_s^t \int_{B_{\rho_\epsilon}(x_0)} u \cdot \nabla \phi + \int_s^t \int_{B_{\rho_\epsilon}(x_0)} \Delta \phi + \frac{h^2}{2} \sin 2 \phi.
\end{align*}

\noindent By H\"older's inequality, it follows that
\begin{align*}
    &\rho_\epsilon^3 \h{1pt} \bigl|\h{1.5pt} [\phi]_{x_0, {\rho_\epsilon}}(t) - [\phi]_{x_0, {\rho_\epsilon}}(s) \h{1pt}\bigr| \\[2mm]
    &\h{20pt}\lesssim_{K_*}  \rho_\epsilon^{\frac{5}{3}}\left(\int_{P_{\rho_\epsilon}(z_0)} \big|\h{.5pt} u \h{.5pt}\big|^3 \right)^{\frac{1}{3}} \left(\int_{P_{\rho_\epsilon}(z_0)} \big|\h{.5pt} \nabla \phi \h{.5pt}\big|^3 \right)^{\frac{1}{3}} + \rho_\epsilon^{\frac{5}{2}}\left( \int_{P_{\rho_\epsilon}(z_0)} \big|\h{.5pt} \Delta \phi + \frac{h^2}{2} \sin 2 \phi \h{.5pt}\big|^2 \right)^{\frac{1}{2}} \\[2mm]
    &\h{20pt}\lesssim_{K_*} \rho_\epsilon^3 \h{1.5pt} \mathrm{C}^{\frac{2}{3}}(\rho_\epsilon) + \rho_\epsilon^{\frac{5}{2}}\left( \int_{P_{\rho_\epsilon}(z_0)} \big|\h{.5pt} \Delta \phi + \frac{h^2}{2} \sin 2 \phi \h{.5pt}\big|^2 \right)^{\frac{1}{2}}.
\end{align*}

\noindent The estimate \eqref{aver est} can then be reduced to \begin{align*}
    \rho_\epsilon^{-5}\int_{P_{\rho_\epsilon}(z_0)} \big|\h{1.5pt} [\phi]_{x_0, \rho_\epsilon} - (\phi)_{z_0, \rho_\epsilon} \big|^3 \lesssim_{K_*}   \mathrm{C}^{2}(\rho_\epsilon) + \left( \rho_\epsilon^{- 1} \int_{P_{\rho_\epsilon}(z_0)} \big|\h{.5pt} \Delta \phi + \frac{h^2}{2} \sin 2 \phi \h{.5pt}\big|^2 \right)^{\frac{3}{2}}.  
\end{align*}The proof then follows by applying this estimate and \eqref{fir poin ca} to the right-hand side of \eqref{tei angle}. Here, we also use Lemmas \ref{smallness of A C D} and \ref{dis small}.
\end{proof}

\subsection{\texorpdfstring{$L^\infty$}{TEXT}-estimate induced by some small \texorpdfstring{$L^3$}{TEXT}-integrals}\label{L infty of u phi}

The $L^\infty$-estimates of $u$ and $\nabla \phi$ are investigated in this section. Since Lin-Liu has already discussed the interior case in \cite{LinLiu1996}, our main arguments are devoted to proving the boundary case. See Lemma \ref{decay lem} below. In the following discussions, the spatial average of the pressure is denoted by
\begin{equation*}
    [\h{0.5pt}p\h{0.5pt}]_{x_0, r}(t) := \frac{1}{\big|\hspace{.5pt}B_r(x_0) \cap  \Omega  \hspace{.5pt}\big|}\int_{B_r(x_0) \hspace{.5pt}\cap \h{1pt} \Omega } p(x, t) \hspace{1.5pt}\mathrm{d} x.
\end{equation*} For any $z = (x, t)$ with $ x \in \mathrm H \cup \mathrm P$ and a function $\psi$ over $P_r(z)$, we define \begin{align*}
    (\psi)_{z, r} := \left\{\begin{array}{lcl} \displaystyle \frac{1}{|\h{0.5pt} P_{r}(z) \h{0.5pt}| } \int_{P_{r}(z)} \psi \h{15pt} &\text{if $ x \in  \mathrm H$,}\\[6mm] \h{30pt}0, &\text{if $x \in  \mathrm P$.}\end{array}\right. \h{15pt} 
\end{align*}  

\begin{lem} \label{decay lem}
There exist a universal small constant $\theta_0 > 0$ and a constant $\epsilon_0 > 0$ such that if 
\begin{equation*}
  F(r, z_0) :=  r^{-2} \int_{P_r(z_0)} \big|\hspace{.5pt} u \hspace{.5pt}\big|^3 + \big|\hspace{.5pt} \nabla \phi\hspace{.5pt} \big|^3 + r^{-5} \int_{P_r(z_0)} \big|\hspace{.5pt} \phi - (\phi)_{z_0, r} \hspace{.5pt}\big|^3 + \Bigl(r^{-2} \int_{P_r(z_0)} \big|\hspace{.5pt} p - [\h{0.5pt}p\h{0.5pt}]_{x_0, r} \hspace{.5pt}\big|^{\frac{3}{2}} \Bigr)^2 \hspace{1.5pt}\leq\hspace{1.5pt} \epsilon_0^3,
\end{equation*}for some $r > 0$ and $z_0 = (x_0, t_0) \in \big(\mathrm {H \cup P}\big) \times (r^2, \infty),$ then we have
\begin{equation*}
     F\big(\theta_0 r, z_0\big) \hspace{1.5pt}\leq\hspace{1.5pt}  \theta_0^{3} \max \Bigl\{\h{0.5pt} \theta_0^{-9}\h{0.5pt}r^3,  F\big(r, z_0\big) \h{0.5pt}\Bigr\}.
\end{equation*} Here, $\epsilon_0$ is small enough. $\theta_0$ is universal, meaning that it depends possibly only on $h$ and $L_{\mathrm H}$.
\end{lem}

\begin{proof}[\bf Proof] 
The proof is divided into 4 steps.\\[2mm]
\textbf{Step 1.} We construct a blow-up sequence. Suppose the conclusion is false. Then for a $\theta_0 \in (0, \frac{1}{4})$ to be determined later, we can find $r_i > 0$ and $z_i = (x_i, t_i) \in \big(\mathrm {H \cup P} \big) \times \big(r_i^2, \infty\big)$ such that
\begin{equation}\label{small ener}
    F(r_i, z_i) := \epsilon_i^3 \longrightarrow 0 \h{15pt}\text{as $i \to \infty$}.
\end{equation}
Meanwhile, it satisfies
\begin{equation} \label{tau_blow_up}  F(\theta_0 r_i, z_i) > \theta_0^{3} \max\Big\{\theta_0^{-9} r_i^3, \epsilon_i^3\Big\}.
 \end{equation}
\eqref{small ener}-\eqref{tau_blow_up} infer that \begin{align}\label{est of r_i}  r_i^3 \leq 8 \h{0.5pt}\theta_0 \h{.5pt}\epsilon_i^3 \longrightarrow 0 \h{15pt}\text{ as $i \to \infty$.} \end{align}

Assume either $\big\{x_i\big\} \subset  \mathrm H$ or $\big\{x_i\big\} \subset  \mathrm P$. $B^{\pm}_r(x)$ denotes the parts of $B_r(x)$ lying in $\big\{x_3 \gtrless 0\big\}$, respectively.  If $x = 0$, we simply drop $0$ from the notation $B^{\pm}_r(0)$. Now, we assume all $r_i$ are sufficiently small and introduce the blow-up sequence:
\begin{equation}\label{defn blow up}
    \big(u_i, \phi_i, p_i\big)(x, t) :=   \left( \frac{r_i u}{\epsilon_i}, \h{0.5pt} \frac{\phi - (\h{0.5pt}\phi\h{0.5pt})_{z_i, \frac{r_i}{2}}}{\epsilon_i}, \h{0.5pt} \frac{r_i^2 \big(\h{1pt}p - [\h{0.5pt}p\h{0.5pt}]_{x_i, r_i}\big)}{\epsilon_i}\right)\big(x_i + r_i\h{0.5pt} x, t_i + r_i^2 \h{0.5pt} t\big) \h{15pt}\text{for $(x, t) \in Q_1$.}
\end{equation}Here, given $r > 0$, the notation $Q_r$ is particularly used to denote the cylinder $B_r^{\pm} \times (- r^2, 0\h{0.5pt})$. The $\pm$ is determined by the sequence $\big\{x_i\big\}$. In light of \eqref{sEL}, $(u_i, \phi_i, p_i)$ is a suitable weak solution of the scaled system:
\begin{equation} \label{rescaled system} \left\{ \begin{aligned}
    \partial_t u_i + \epsilon_i \h{1pt}u_i \cdot \nabla u_i - \Delta u_i &= - \nabla p_i - \epsilon_i \h{1pt}\nabla \cdot \big(\nabla \phi_i \odot \nabla  \phi_i\big), \\[1mm]
    \operatorname{div} u_i &= 0, \\[1mm]
    \partial_t \phi_i + \epsilon_i \h{1pt}u_i \cdot \nabla  \phi_i - \Delta \phi_i &=  \frac{h^2 \h{1pt}r_i^2 \h{1pt} \epsilon_i^{-1}}{2}    \sin 2 \left(\epsilon_i \h{0.5pt} \phi_i + (\h{0.5pt}\phi\h{0.5pt})_{z_i, \frac{r_i}{2}}\right) 
\end{aligned} \right. \h{25pt}\text{on $Q_1$.}\end{equation}Moreover,
\begin{align} \label{rescaled L3 estimate} 
    &\mathrm{(1).} \h{5pt}\int_{Q_{1}} \big|\hspace{.5pt} u_i \hspace{.5pt}\big|^3 + \big|\hspace{.5pt} \nabla \phi_i \hspace{.5pt} \big|^3 + \big|\h{0.5pt} \phi_i - (\phi_i)_{0, 1}^* \h{0.5pt}\big|^3 + \left(\int_{Q_{1}} \big| \hspace{.5pt} p_i \hspace{.5pt} \big|^{\frac{3}{2}} \right)^2 = 1, \nonumber\\[2mm]
    &\mathrm{(2).} \h{6pt}\theta_0^{-2} \int_{Q_{\theta_0}} \big| \hspace{.5pt} u_i \hspace{.5pt} \big|^3 + \big| \hspace{.5pt} \nabla \phi_i \hspace{.5pt} \big|^3 + \theta_0^{-5} \int_{Q_{\theta_0}} \big|\h{0.5pt} \phi_i - (\phi_i)_{0, \theta_0}^* \h{0.5pt}\big|^3 + \left( \theta_0^{-2} \int_{Q_{\theta_0}} \big|\hspace{.5pt} p_i - [\h{0.5pt}p_i\h{0.5pt}]^*_{0, \theta_0} \hspace{.5pt}\big|^{\frac{3}{2}} \right)^2 > \theta_0^{3}.
  \end{align} In the item (1) of the above, if $\big\{x_i\big\} \subset \mathrm H$, then $(\psi)_{0, r}^*$ is the average of $\psi$ over $Q_r$. If $\big\{x_i\big\} \subset \mathrm P$, then $(\psi)_{0, r}^*$ is equal to $0$. In the item (2) of \eqref{rescaled L3 estimate}, the notation $[\h{0.5pt}p_i\h{0.5pt}]^*_{0, \theta_0}$ is the average of $p_i$ over $B_{\theta_0}^{\pm}$ at time $t$. Same as before, the $\pm$ is determined by $\big\{x_i\big\}$. We also have the following boundary condition for $\phi_i$. If $\big\{x_i\big\} \subset  \mathrm H$, then
\begin{equation}\label{scaled wk anchoring}
    \partial_3 \phi_i = - \frac{L_\mathrm{H} \h{0.5pt} r_i \h{0.5pt}\epsilon_i^{-1}}{2}   \sin  2\left(\epsilon_i \h{0.5pt}\phi_i +  (\h{0.5pt}\phi\h{0.5pt})_{z_i, \frac{r_i}{2}}\right)  \quad\quad \text{on } B'_1 \times (-1, 0\h{0.5pt}).
\end{equation} Here, $B'_r$ is the flat boundary of $B_r^{\pm}$.  When it satisfies $\big\{x_i\big\} \subset  \mathrm P$, we have \begin{equation}\label{scaled wk anchoring Diri}
    \phi_i = 0  \quad\quad \text{on } B'_1 \times (-1, 0\h{0.5pt}).
\end{equation}

\noindent\textbf{Step 2.} 
We claim that there exists a universal constant $K_* > 0$ such that for all $i$, it holds
\begin{equation} \label{rescaled weak}
    \sup_{t \h{1pt}\in \h{1pt}[- \frac{1}{4}, \h{1pt} 0\h{0.5pt}]}\int_{B^{\pm}_{1/2}}   |u_i|^2 + |\nabla \phi_i|^2  +  \int_{Q_{1/2}}   | \nabla u_i |^2 + |\nabla^2 \phi_i |^2 \leq K_*.
\end{equation}
The constant $K_*$ is universal in the sense that it depends possibly only on $h$ and $L_{\mathrm{H}}$. The $\pm$ is determined by the sequence $\big\{x_i\big\}$. \vspace{0.2pc}

To show this energy estimate, we delve into the generalized energy inequality \eqref{energy ineq}. In the following, $\varphi = \varphi^*_1(x) \h{1pt} \varphi^*_2(t)$ is a smooth test function, where $\varphi^*_1$ is compactly supported on $B_1$ and is equivalently equal to $1$ on $B_{1/2}$. $\varphi^*_2$ is defined on $(- \infty, 0\h{1pt}]$ such that $\varphi^*_2$ is equivalently equal to $0$ for all $t \leq -1$ and is equivalently equal to 1 on $[- 1/4, 0\h{1pt}]$. $\varphi^*_1$ and $\varphi^*_2$ are all non-negative functions with the maximum values no more than $1$. With the function $\varphi$, we define $$ \varphi_i (x, t) := \varphi \left( \frac{x - x_i}{r_i}, \frac{t - t_i}{r_i^2} \right) \h{15pt}\text{for all $(x, t) \in P_{r_i}(z_i)$}.$$ Then we replace the test function in \eqref{energy ineq} with $\varphi_i^2$. \vspace{0.4pc}

\noindent \textbf{I. Estimates of $R\left(\phi, \varphi_i^2\right)$.} \vspace{0.4pc} 

\noindent According to \eqref{defn of R}, $R(\phi, \varphi_i^2)$ is given as follows:
\begin{align*} R\left(\phi, \varphi_i^2\right)  :=    - \int_\text{P}   \left(\partial_3  \phi \right)^2 \h{1pt} \partial_3 \varphi_i^2 +  \frac{L_\text{H}^2}{4} \int_{\mathrm H} \partial_3 \varphi_i^2 \h{1pt} \sin^2 2 \phi   + \int_\text{H} | \nabla' \phi |^2 \h{1pt} \partial_3 \varphi_i^2 + 2L_\text{H} \int_\text{H} \varphi_i^2 \h{1pt}  |\nabla' \phi|^2 \cos 2 \phi. 
\end{align*} If $\big\{x_i \big\} \subset \mathrm P$, then the last two integrals above vanish. Hence, we consider the case when $\big\{x_i \big\} \subset \mathrm H$. The consequence in this part is also valid when $\big\{x_i \big\} \subset \mathrm P$. \vspace{0.4pc}

For the first two terms in $R(\phi, \varphi_i^2)$, we note that 
\begin{align*}
    - \int_\text{P}   \left(\partial_3  \phi \right)^2 \h{1pt} \partial_3 \varphi_i^2 +  \frac{L_\text{H}^2}{4} \int_{\mathrm H} \partial_3 \varphi_i^2 \h{1pt} \sin^2 2 \phi = - \int_\Omega \partial_3 \left( \left(\partial_3  \phi \right)^2 \h{1pt} \partial_3 \varphi_i^2 \right).
\end{align*} 

\noindent It then follows
\begin{align*}
    - \int_\Omega \partial_3 \left( \left(\partial_3  \phi \right)^2 \h{1pt} \partial_3 \varphi_i^2 \right) = - \int_{\Omega} 4 \left(\partial_3 \phi\right) \left(\partial_{33} \phi\right) \h{1pt} \varphi_i \h{1pt} (\partial_3 \varphi_i) + 2 \left(\partial_3 \phi\right)^2  \big[  \left(\partial_3 \varphi_i\right)^2 +  \varphi_i \h{1pt}  \partial_{33} \varphi_i \h{1pt} \big]. 
\end{align*}
Using the Young's inequality infers
\begin{align}\label{est 1 in R}
    \int_{0}^{t_i} \left| \h{2pt} \int_\Omega \partial_3 \left( \left(\partial_3  \phi \right)^2 \h{1pt} \partial_3 \varphi_i^2 \right) \h{1pt}\right| \leq \frac{1}{16}  \int_{P_{r_i}(z_i)} \big|\h{.5pt} \partial_{33} \phi \h{.5pt}\big|^2 \h{1pt} \varphi_i^2 + K_* \h{1pt}r_i^{-2} \int_{P_{r_i}(z_i)} \big|\h{.5pt} \partial_3 \phi \h{.5pt}\big|^2.
\end{align}

For the third term in $R(\phi, \varphi_i^2)$, we make use of the boundary condition $\nabla' \phi = 0$ on P. Therefore,
\begin{align*}
    \int_\text{H} \big|\h{.5pt} \nabla' \phi \h{.5pt}\big|^2 \h{1pt} \partial_3 \varphi_i^2  =   - \int_{\Omega} 4  \h{1pt}\nabla' \phi \cdot \partial_3 \nabla' \phi   \h{2pt} \varphi_i \h{1pt} \partial_3 \varphi_i - 2 \int_{\Omega} \big|\h{.5pt} \nabla' \phi \h{.5pt}\big|^2 \h{1pt} \big[ \left(\partial_3 \varphi_i\right)^2 +  \varphi_i \left(\partial_{33} \varphi_i\right) \big].
\end{align*}
Same derivations as in \eqref{est 1 in R} imply
\begin{align}\label{est 2 in R}
    \int_{0}^{t_i} \left|\h{2pt} \int_\text{H} \big|\h{.5pt} \nabla' \phi \h{.5pt}\big|^2 \h{1pt} \partial_3 \varphi_i^2 \h{1pt}\right| \leq \frac{1}{16} \int_{P_{r_i}(z_i)} \big|\h{.5pt} \partial_3 \nabla' \phi \h{.5pt}\big|^2 \varphi_i^2 + K_*\h{1pt}r_i^{-2} \int_{P_{r_i}(z_i)} \big|\h{.5pt} \nabla' \phi \h{.5pt}\big|^2.
\end{align}

For the last term in $R(\phi, \varphi_i^2)$, we still use the boundary condition $\phi = 0$ on P, and obtain
\begin{align*}
    \int_{\mathrm{H}} \big|\hspace{.5pt} \nabla' \phi \hspace{.5pt}\big|^2 \varphi_i^2 = - 2 \int_{\Omega}  \varphi_i^2 \h{0.5pt}\nabla' \phi \cdot \partial_{3} \nabla' \phi  +   \big|\hspace{.5pt} \nabla' \phi \hspace{.5pt}\big|^2  \varphi_i \h{1pt}\partial_3 \varphi_i.
\end{align*}Applying the Young's inequality then infers
\begin{align}\label{est of nabla' on H}
     \int_0^{t_i}\int_{\mathrm{H}} \big|\hspace{.5pt} \nabla' \phi \hspace{.5pt}\big|^2 \varphi_i^2  \leq \frac{1}{16 L_{\mathrm H}} \int_{P_{r_i}(z_i)} \big|\hspace{.5pt} \partial_{3} \nabla' \phi \hspace{.5pt}\big|^2 \varphi_i^2 + K_* \h{1pt}r_i^{-2}\int_{P_{r_i}(z_i)} \big|\hspace{.5pt} \nabla' \phi \hspace{.5pt}\big|^2.
\end{align}

By the above estimates \eqref{est 1 in R}-\eqref{est of nabla' on H}, 
\begin{align}\label{est R for use} 
 &\int_0^{t_i} \left|\h{2pt}R\left(\phi, \varphi_i^2 \right) \h{1pt} \right| \leq  \frac{1}{4} \int_{P_{r_i}(z_i)} \big|\h{.5pt} \nabla^2 \phi \h{.5pt}\big|^2 \h{1pt} \varphi_i^2 + K_* \h{0.5pt} r_i^{-2}\int_{P_{r_i}(z_i)} \big|\hspace{.5pt} \nabla \phi \hspace{.5pt}\big|^2.
\end{align} 

\noindent \textbf{II. Estimates of the integrals from the second-order normal derivative.} \vspace{0.2pc}  

\noindent In this part, we fix $t \in [\h{0.5pt} - \frac{1}{4}, 0\h{0.5pt}]$ and consider the following boundary integrals:
\begin{align}\label{rm bdr term}
         \int_{\mathrm{H} \times \{\h{0.5pt} t_i + r_i^2 t\h{0.5pt}\}} \varphi_i^2 \sin^2 \phi +  \int_0^{t_i + r_i^2 t}\int_{\mathrm{H}}   2 \left(\cos 2\phi\right) \left| \nabla' \phi\right|^2 \varphi_i^2  +   \left(\sin 2\phi\right)  \nabla' \phi \cdot \nabla' \varphi_i^2 -    \partial_s \varphi_i^2 \h{1pt} \sin^2 \phi.
\end{align}Same as in Part I, we assume $\big\{x_i \big\} \subset \mathrm H$. Otherwise,  if $\big\{x_i \big\} \subset \mathrm P$, all integrals in \eqref{rm bdr term} vanish.\vspace{0.4pc}

Denote by $\mathrm S\h{0.5pt}[\h{0.5pt}\phi\h{0.5pt}]$ the constant \begin{align*}
    \mathrm S\h{0.5pt}[\h{0.5pt}\phi\h{0.5pt}] :=  \sin^2 \left(\big(\phi\big)_{z_i, \frac{r_i}{2}} + \epsilon_i \big(\phi_i\big)^*_{0, 1} \right),
\end{align*}where $\big(\phi_i\big)^*_{0, 1}$ is the average of $\phi_i$ over $Q_1$. Since it satisfies 
\begin{align*}
    \int_{\mathrm H \times \{\h{0.5pt} t_i + r_i^2 t \h{0.5pt}\}} \varphi_i^2 \sin^2 \phi &- \int_{t_i - r_i^2}^{t_i + r_i^2 t} \int_{\mathrm H} \partial_s \varphi_i^2 \h{1pt}\sin^2 \phi  \\[1mm]
    &=     \int_{\mathrm H \times \{\h{0.5pt} t_i + r_i^2 t\h{0.5pt}  \}} \varphi_i^2 \Bigl(\sin^2 \phi - \mathrm S\h{0.5pt}[\h{0.5pt}\phi\h{0.5pt}] \h{1pt}\Bigr) -  \int_{t_i - r_i^2}^{t_i + r_i^2 t} \int_{\mathrm H} \partial_s \varphi_i^2 \Bigl(\sin^2 \phi - \mathrm S \h{0.5pt}[\h{0.5pt}\phi\h{0.5pt}] \h{1pt} \Bigr),
\end{align*} we then obtain, by the mean value theorem and the change of variables, that \begin{align}\label{es 1,null} &\left|\h{2pt} \int_{\mathrm H \times \{\h{0.5pt} t_i + r_i^2 t\h{0.5pt} \}} \varphi_i^2 \sin^2 \phi - \int_{t_i - r_i^2}^{t_i + r_i^2 t} \int_{\mathrm H} \partial_s \varphi_i^2 \h{1pt}\sin^2 \phi \h{2pt}\right| \\[1mm]  &\h{100pt} \lesssim_{K_*} r_i^2\h{0.5pt}\epsilon_i\int_{B'_1 \times \{ \h{0.5pt}t\h{0.5pt} \}} \varphi^2 \left|\h{1pt}  \phi_i -   \big(\phi_i\big)^*_{0, 1} \h{1pt} \right| +  r_i^{2} \h{0.5pt}\epsilon_i  \int_{-1}^{t} \int_{B'_1} \varphi \left|\h{1pt} \phi_i -  \big(\phi_i\big)^*_{0, 1} \right|\nonumber\\[1mm]
&\h{38pt}\lesssim_{K_*} r_i^2\h{0.5pt}\epsilon_i + r_i^2\h{0.5pt}\epsilon_i \int_{B'_1 \times \{ \h{0.5pt}t\h{0.5pt} \}} \varphi^4 \left( \phi_i -   \big(\phi_i\big)^*_{0, 1}  \right)^2 +  r_i^{2} \h{0.5pt}\epsilon_i  \int_{-1}^{t} \int_{B'_1} \varphi^2 \left( \phi_i -  \big(\phi_i\big)^*_{0, 1} \right)^2. \nonumber
\end{align}Applying the integration by parts with respect to the $x_3$-variable yields \begin{align*}
    &\int_{B'_1 \times \{ \h{0.5pt}t\h{0.5pt} \}} \varphi^4 \left(  \phi_i -   \big(\phi_i\big)^*_{0, 1} \right)^2 +  \int_{-1}^{t} \int_{B'_1} \varphi^2 \left( \phi_i -  \big(\phi_i\big)^*_{0, 1} \right)^2 \\[1mm]
    &\h{20pt}\lesssim_{K_*} \int_{B^{+}_1 \times \{\h{0.5pt}t\h{0.5pt}\}} \varphi^2  \h{1pt}\big|\h{1pt}  \nabla \phi_i \h{1pt} \big|^2 + \varphi^2 \left(  \phi_i -   \big(\phi_i\big)^*_{0, 1} \right)^2  +  \int_{Q_1}  \big| \nabla \phi_i \big|^2 + \left( \phi_i -  \big(\phi_i\big)^*_{0, 1} \right)^2.
\end{align*}The last two integrals above are uniformly bounded by (1) in \eqref{rescaled L3 estimate}. It then turns out \begin{align}\label{before L2}
     \int_{B'_1 \times \{ \h{0.5pt}t\h{0.5pt} \}} \varphi^4 \left(  \phi_i -   \big(\phi_i\big)^*_{0, 1} \right)^2  &+  \int_{-1}^{t} \int_{B'_1} \varphi^2 \left( \phi_i -  \big(\phi_i\big)^*_{0, 1} \right)^2 \\[1mm]
    & \lesssim_{K_*} 1 + \int_{B^{+}_1 \times \{\h{0.5pt}t\h{0.5pt}\}}  \varphi^2  \h{1pt}\big|\h{1pt}  \nabla \phi_i \h{1pt} \big|^2 + \varphi^2 \left(  \phi_i -   \big(\phi_i\big)^*_{0, 1} \right)^2. \nonumber
\end{align}

Multiply $\varphi^2 \left(\phi_i - \big(\phi_i\big)^*_{0, 1}\right)$ on the both sides of the third equation in \eqref{rescaled system} and integrate. It holds \begin{align*}
\int_{B^{+}_1 \times \{ t \}  }\varphi^2  \left(\phi_i - \big(\phi_i\big)^*_{0, 1}\right)^2  = &\int_{-1}^t\int_{B^{+}_1} \left(\phi_i - \big(\phi_i\big)^*_{0, 1}\right)^2 \partial_s \varphi^2 - 2 \epsilon_i \h{0.5pt} \varphi^2 \left(\phi_i - \big(\phi_i\big)^*_{0, 1}\right)u_i \cdot \nabla  \phi_i\\[1mm]
&\h{-60pt}+ \int_{-1}^t\int_{B^+_1}2 \varphi^2 \left(\phi_i - \big(\phi_i\big)^*_{0, 1}\right) \Delta \phi_i +  \frac{h^2 \h{1pt}r_i^2 }{\h{1pt} \epsilon_i} \varphi^2 \left(\phi_i - \big(\phi_i\big)^*_{0, 1}\right)   \sin 2 \left(\epsilon_i \h{0.5pt} \phi_i + (\h{0.5pt}\phi\h{0.5pt})_{z_i, \frac{r_i}{2}}\right).
\end{align*}Utilizing the boundedness of $\varphi$ and its derivatives, we obtain from this equality the estimate: \begin{align}\label{L2 of phii}
    \int_{B^{+}_1 \times \{ t \}  }\varphi^2  \left(\phi_i - \big(\phi_i\big)^*_{0, 1}\right)^2  \lesssim_{K_*} & 1  +  \int_{Q_1}     \varphi^2 \big(\Delta \phi_i\big)^2 + \int_{Q_1} \big|\h{0.5pt} u_i \h{0.5pt}\big|^3 + \big|\h{0.5pt}\nabla  \phi_i \h{0.5pt}\big|^3 + \left| \h{1pt}\phi_i - \big(\phi_i\big)^*_{0, 1} \h{1pt}\right|^3.
\end{align}Here, we also use Young's inequality. Apply (1) of \eqref{rescaled L3 estimate}. The above estimate is reduced to \begin{align*}
    &\int_{B^{+}_1 \times \{ t \}  }\varphi^2  \left(\phi_i - \big(\phi_i\big)^*_{0, 1}\right)^2  \lesssim_{K_*}  1  +  \int_{Q_1}     \varphi^2 \big(\Delta \phi_i\big)^2.
\end{align*} This estimate and \eqref{es 1,null}-\eqref{before L2} induce \begin{align}\label{time slice}
&\left|\h{2pt} \int_{\mathrm H \times \{\h{0.5pt} t_i + r_i^2 t\h{0.5pt} \}} \varphi_i^2 \sin^2 \phi - \int_{t_i - r_i^2}^{t_i + r_i^2 t} \int_{\mathrm H} \partial_s \varphi_i^2 \h{1pt}\sin^2 \phi \h{2pt}\right| \\[1mm] 
&\h{100pt}\lesssim_{K_*} r_i^2\h{0.5pt}\epsilon_i + r_i^2\h{0.5pt}\epsilon_i   \int_{B^{+}_1 \times \{\h{0.5pt}t\h{0.5pt}\}}  \varphi^2  \h{1pt}\big|\h{1pt}  \nabla \phi_i \h{1pt} \big|^2 + r_i^2\h{0.5pt}\epsilon_i  \int_{Q_1}     \varphi^2 \big(\Delta \phi_i\big)^2. \nonumber
\end{align}

We continue to estimate integrals in \eqref{rm bdr term}. Applying \eqref{est of nabla' on H}  induces
\begin{align}\label{2nd inegral}
    \left|\h{1.5pt} \int_0^{t_i + r_i^2 t}\int_{\mathrm{H}}  2 \left(\cos 2\phi\right) \left| \nabla' \phi\right|^2 \varphi_i^2 \h{1.5pt}\right| \leq \frac{1}{8L_{\mathrm H}} \int_{P_{r_i}(z_i)} \big|\hspace{.5pt} \partial_{3} \nabla' \phi \hspace{.5pt}\big|^2 \varphi_i^2 + K_* \h{1pt}r_i^{-2}\int_{P_{r_i}(z_i)} \big|\hspace{.5pt} \nabla' \phi \hspace{.5pt}\big|^2.
\end{align}

For the rest integral in \eqref{rm bdr term}, we use the boundary condition of $\phi$ on P and integrate by parts with respect to $x_3$. Consequently, it turns out
\begin{align*}
    & - \int_{0}^{t_i + r_i^2 t} \int_{\mathrm{H}} (\sin 2 \phi) \nabla' \phi \cdot \nabla' \varphi_i^2  \\[1mm]
    &=    \int_0^{t_i + r_i^2 t} \int_{\Omega} 2 \h{0.5pt} \cos 2 \phi  \left(\partial_3 \phi\right) \left(\nabla' \phi\right) \cdot \nabla' \varphi_i^2  +   \sin 2 \phi \h{1pt} (\partial_3 \nabla' \phi) \cdot \nabla' \varphi_i^2 +    \sin 2 \phi \h{1pt} (\nabla' \phi) \cdot \partial_3 \nabla \varphi_i^2.
\end{align*}

\noindent Direct estimates infer that \begin{align}\label{third est requ}
    \left|\h{1pt}\int_{0}^{t_i + r_i^2 t} \int_{\mathrm{H}} (\sin 2 \phi) \nabla' \phi \cdot \nabla' \varphi_i^2 \h{1pt}\right| \leq \frac{1}{8 L_{\mathrm H}} \int_{P_{r_i}(z_i)} \big|\hspace{.5pt} \partial_{3} \nabla' \phi \hspace{.5pt}\big|^2 \varphi_i^2 + K_* r_i^3  + K_* r_i^{-2} \int_{P_{r_i}(z_i)} \big|\h{0.5pt}\nabla \phi \h{0.5pt}\big|^2. 
\end{align} 

From the last estimate and \eqref{time slice}-\eqref{2nd inegral}, the integrals in \eqref{rm bdr term} are bounded by
\begin{align*}
    \frac{1}{4 L_{\mathrm H}} \int_{P_{r_i}(z_i)} \big|\h{.5pt} \nabla^2 \phi \h{.5pt}\big|^2 \varphi_i^2 + K_* r_i^2\left( \epsilon_i +   \epsilon_i \int_{B_1^+ \times \{ t\}}  \varphi^2  \h{1pt}\big|\h{1pt}  \nabla \phi_i \h{1pt} \big|^2 +  \epsilon_i  \int_{Q_1}     \varphi^2 \big|\nabla^2 \phi_i\big|^2 +  r_i^{- 4} \int_{P_{r_i}(z_i)} \big|\h{.5pt} \nabla \phi \h{.5pt}\big|^2\right).
\end{align*}
Here, we also use \eqref{est of r_i}.\vspace{0.4pc}

\noindent\textbf{III. Energy estimate induced by the generalized energy inequality \eqref{energy ineq}}\vspace{0.4pc}

Using the H\"{o}lder and Young's inequalities, and the incompressibility condition, we can bound from above the  integrals over $\Omega \times [\h{0.5pt}0, t_i + r_i^2 t\h{0.5pt}]$ on the right-hand side of \eqref{energy ineq} by \begin{align*}
        r_i^{-1}\int_{P_{r_i}(z_i)} \left| u \right|^3 + | \nabla \phi |^3 + r_i^{-2}\int_{P_{r_i}(z_i)} |u|^2 + |\nabla \phi |^2 +   r_i^{-1} \left(\int_{P_{r_i}(z_i)} \left| u \right|^3 \right)^{\frac{1}{3}} \left( \int_{P_{r_i}(z_i)} \big|\h{1pt} p - [\h{0.5pt}p\h{0.5pt}]_{x_i, r_i} \big|^{\frac{3}{2}} \right)^{\frac{2}{3}},
\end{align*}up to a coefficient $K_*$. By this bound and the consequences in Parts I and II, \eqref{energy ineq} infers

\begin{align*}
       & \int_{B^{\pm}_1 \times \{\h{0.5pt}t\h{0.5pt}\}}  \varphi^2 \left( \h{.5pt} |\h{0.5pt}u_i\h{0.5pt}|^2 + |\h{0.5pt}\nabla \phi_i\h{0.5pt}|^2 \right)  +  2  \int_{-1}^{t }\int_{B^{\pm}_1} \varphi^2 \left( \h{0.5pt}  | \nabla u_i |^2 + |\nabla^2 \phi_i |^2 \h{0.5pt}\right) \\[1mm] &\h{30pt} \leq K_* \h{0.5pt}\frac{r_i }{ \epsilon_i}   + \frac{1}{2}  \int_{Q_1} \varphi^2 \h{1pt} \big| \nabla^2 \phi_i \big|^2   + K_* \h{0.5pt} \frac{r_i}{  \epsilon_i} \int_{Q_1} \varphi^2 \big| \nabla^2 \phi_i \big|^2 + K_* \h{0.5pt} \frac{r_i}{\epsilon_i} \int_{B_1^{\pm} \times \{t\}} \varphi^2 \h{0.5pt}\big| \nabla \phi_i \big|^2   \\[1mm]
       &\h{30pt} + K_*  \epsilon_i \int_{Q_1} \left| u_i \right|^3 + | \nabla \phi_i |^3  + K_*  \int_{Q_1} |u_i|^2 + |\nabla \phi_i |^2 +  K_*   \left(  \int_{Q_1} \left| u_i \right|^3 \right)^{\frac{1}{3}} \left(   \int_{Q_1} \big|\h{1pt} p_i \big|^{\frac{3}{2}} \right)^{\frac{2}{3}}.
\end{align*}
 
\noindent Here, the change of variables is also applied. We now utilize \eqref{est of r_i} and (1) in \eqref{rescaled L3 estimate}. The last estimate can then be  reduced to \begin{align*}
    & \int_{B^{\pm}_1 \times \{\h{0.5pt}t\h{0.5pt}\}}  \varphi^2 \left( \h{.5pt} |\h{0.5pt}u_i\h{0.5pt}|^2 + |\h{0.5pt}\nabla \phi_i\h{0.5pt}|^2 \right)  +  2  \int_{-1}^{t }\int_{B^{\pm}_1} \varphi^2 \left( \h{0.5pt}  | \nabla u_i |^2 + |\nabla^2 \phi_i |^2 \h{0.5pt}\right) \\[1mm] &\h{30pt} \leq K_*   + \frac{1}{2}  \int_{Q_1} \varphi^2 \h{1pt} \big| \nabla^2 \phi_i \big|^2   + K_* \h{0.5pt} \theta_0^{\frac{1}{3}} \int_{Q_1} \varphi^2 \big| \nabla^2 \phi_i \big|^2 + K_* \h{0.5pt} \theta_0^{\frac{1}{3}} \int_{B_1^{\pm} \times \{t\}} \varphi^2 \h{0.5pt}\big| \nabla \phi_i \big|^2.
\end{align*}Take $\theta_0$ sufficiently small with the smallness depending on $K_*$ only. The above estimate then yields
\begin{align*}
       & \sup_{t \h{1pt}\in \h{1pt}[ - \frac{1}{4}, 0 \h{0.8pt}]} \h{1.5pt}\int_{B^{\pm}_1 \times \{\h{0.5pt}t\h{0.5pt}\}}  \varphi^2 \left( \h{.5pt} |\h{0.5pt}u_i\h{0.5pt}|^2 + |\h{0.5pt}\nabla \phi_i\h{0.5pt}|^2 \right)  +   \int_{Q_1} \varphi^2 \left( \h{0.5pt}  | \nabla u_i |^2 + |\nabla^2 \phi_i |^2 \h{0.5pt}\right)  \leq K_*.
\end{align*}
  
\noindent The claim \eqref{rescaled weak} follows since $\varphi \equiv 1$ on $Q_{1/2}$.\vspace{0.4pc}

\noindent \textbf{Step 3.} We give some compactness results on the convergence of $(u_i, \phi_i)$. A decay estimate for the pressure is also derived. \vspace{0.4pc}

\noindent \textbf{I. Strong $L^3$-convergence of $\phi_i$}\vspace{0.2pc}

Applying \eqref{rescaled weak}, we obtain from Proposition 3.2 in the Chapter 1 of \cite{Dibenedetto1993} that \begin{align}\label{10/3 est}
    \| u_i \|_{L^{\frac{10}{3}}(Q_{1/2})} + \| \nabla \phi_i \|_{L^{\frac{10}{3}}(Q_{1/2})} \hspace{1.5pt}\leq\hspace{1.5pt} K_*, \h{15pt}\text{for all $i$.}
\end{align}
Use this estimate, \eqref{rescaled weak}, and the equation of $\phi_i$ in \eqref{rescaled system}. It turns out that \begin{align}\label{bound of partIt}
    \int_{Q_{1/2}} \big|\h{0.5pt} \partial_t \phi_i \h{0.5pt}\big|^{\frac{5}{3}}  &\lesssim_{K_*}  1 + \int_{Q_{1/2}} \big| u_i \big|^{\frac{5}{3}} \h{1pt}\big| \nabla  \phi_i \big|^{\frac{5}{3}} + \big| \nabla^2 \phi_i \big|^{\frac{5}{3}} \nonumber\\[1mm]
    &\lesssim_{K_*} 1 + \left(\int_{Q_{1/2}} \big| u_i \big|^{\frac{10}{3}} \right)^{\frac{1}{2}} \left( \int_{Q_{1/2}} \big| \nabla  \phi_i \big|^{\frac{10}{3}} \right)^{\frac{1}{2}} + \left(\int_{Q_{1/2}}\big| \nabla^2 \phi_i \big|^{2}\right)^{\frac{5}{6}} \leq K_*. 
\end{align} Here, $K_*$ is also a universal constant depending possibly only on $h$ and $L_{\mathrm H}$. From our construction of the blow-up sequence, either the average of $\phi_i$ on $Q_{1/2}$ is $0$ or $\phi_i = 0$ on $B'_{1/2} \times ( - \frac{1}{4}, 0 \h{0.5pt})$. We then can apply the $L^2$-estimate of $\nabla \phi_i$ in \eqref{rescaled weak}, the last estimate of $\partial_t \phi_i$,  and Poincar\'{e}'s inequality to obtain that $\phi_i$ is uniformly bounded in $W^{1, \frac{5}{3}}(Q_{1/2})$. By the Sobolev embedding, the $L^{\frac{20}{7}}$-norm of $\{\phi_i\}$ over $Q_{1/2}$ is uniformly bounded from above by the universal constant $K_*$. Using the compactness of the Sobolev embedding, up to a subsequence, $\{\phi_i\}$ converges to a limit function $\phi_*$ as $i \to \infty$, strongly in $L^2(Q_{1/2})$. Since the $L^2$-norm of $\nabla^2 \phi_i$ over $Q_{1/2}$ is uniformly bounded from above by $K_*$, we can keep extracting a subsequence such that $\nabla^2 \phi_i$ converges to $\nabla^2 \phi_*$ weakly in $L^2(Q_{1/2})$. By lower semi-continuity, it turns out \begin{align}\label{bd phi- phi*}
    \int_{Q_{1/2}} \big| \nabla^2 \phi_* \big|^2 \leq  \liminf_{i \to \infty} \int_{Q_{1/2}} \big| \nabla^2 \phi_i \big|^2 \leq K_*.
\end{align} 

We now fix the subsequence and verify the strong $L^3$-convergence of $\{ \phi_i \}$ to $\phi_*$. Utilizing the Gagliardo–Nirenberg inequality induces \begin{align*}
     \int_{Q_{1/2}} \big|\h{1pt}\phi_i - \phi_* \h{1pt}\big|^3 &\lesssim_{K_*} \int_{- 1/4}^0 \big\| \nabla^2 \phi_i - \nabla^2 \phi_* \big\|^{\frac{3}{4}}_{L^2(B^{\pm}_{1/2})} \h{2pt}\big\|   \phi_i - \phi_* \big\|^{\frac{9}{4}}_{L^2(B^{\pm}_{1/2})} + \big\|   \phi_i - \phi_* \big\|^{3}_{L^2(B^{\pm}_{1/2})}\\[1mm]
    & \h{- 50pt}\lesssim_{K_*} \left(\int_{Q_{1/2}}  \big| \nabla^2 \phi_i - \nabla^2 \phi_* \big|^{2}  \right)^{\frac{3}{8}} \left(\int_{- 1/4}^0\big\|   \phi_i - \phi_* \big\|^{\frac{18}{5}}_{L^2(B^{\pm}_{1/2})} \right)^{\frac{5}{8}} + \int_{- 1/4}^0 \big\|   \phi_i - \phi_* \big\|^{3}_{L^2(B^{\pm}_{1/2})}.
\end{align*}Using \eqref{bd phi- phi*} and the $L^2$-estimate of $\nabla^2 \phi_i$ in \eqref{rescaled weak}, we reduce the above estimate to \begin{align}\label{L3 over Q1/2}
     \int_{Q_{1/2}} \big|\h{1pt}\phi_i - \phi_* \h{1pt}\big|^3 \lesssim_{K_*} \left(\int_{- 1/4}^0\big\|   \phi_i - \phi_* \big\|^{\frac{18}{5}}_{L^2(B^{\pm}_{1/2})} \right)^{\frac{5}{8}} + \int_{- 1/4}^0 \big\|   \phi_i - \phi_* \big\|^{3}_{L^2(B^{\pm}_{1/2})}.
\end{align} To control the $L^2$-norm of $\phi_i$, we first note that \begin{align*}
    \int_{Q_{1/2}} \big|\h{0.5pt} \phi_i \h{0.5pt}\big|^2 \leq K_* \h{15pt}\text{for all $i$ and some universal constant $K_*$.}
\end{align*}Therefore, for each $i$, there exists a $t_i \in [ - \frac{1}{4}, 0 \h{0.5pt}]$ such that \begin{align*}
    \int_{B^{\pm}_{1/2} \times \{ \h{0.5pt}t_i \h{0.5pt}\}} \big|\h{0.5pt} \phi_i \h{0.5pt}\big|^2 \leq 8K_*. 
\end{align*}Applying this estimate, the upper boundedness of the $L^{\frac{20}{7}}$-norm of $\{\phi_i\}$ over $Q_{1/2}$, and \eqref{bound of partIt}, we get, for all $t \in [-\frac{1}{4}, 0\h{0.5pt}]$, that
\begin{align*}
    \int_{B^{\pm}_{1/2} \times \{ \h{0.5pt}t  \h{0.5pt}\}} \big|\h{0.5pt} \phi_i \h{0.5pt}\big|^2 &= \int_{B^{\pm}_{1/2} \times \{ \h{0.5pt}t_i \h{0.5pt}\}} \big|\h{0.5pt} \phi_i \h{0.5pt}\big|^2 + 2 \int_{t_i}^t  \int_{B^{\pm}_{1/2}}  \phi_i \h{1pt} \partial_s \phi_i \\[1mm]
    &\h{20pt}\lesssim_{K_*} 1 +  \left(\int_{Q_{1/2}} \big|\h{0.5pt}\phi_i \h{0.5pt}\big|^{\frac{20}{7}}\right)^{\frac{7}{20}}\left(\int_{Q_{1/2}} \big|\h{0.5pt}\partial_s \phi_i \h{0.5pt}\big|^{\frac{5}{3}}\right)^{\frac{3}{5}} \leq K_*.
\end{align*}Since $\{\phi_i\} \to \phi_*$ strongly in $L^2(Q_{1/2})$, 
then $\{\phi_i(\cdot, t)\} \to \phi_*(\cdot, t)$ strongly in $L^2(B^{\pm}_{1/2})$ for almost all $t \in [-\frac{1}{4}, 0\h{0.5pt}]$. The last estimate yields \begin{align*}
   \int_{B^{\pm}_{1/2} \times \{ \h{0.5pt}t  \h{0.5pt}\}} \big|\h{0.5pt} \phi_* \h{0.5pt}\big|^2 = \lim_{i \h{0.5pt}\to\h{0.5pt} \infty} \int_{B^{\pm}_{1/2} \times \{ \h{0.5pt}t  \h{0.5pt}\}} \big|\h{0.5pt} \phi_i \h{0.5pt}\big|^2  \leq K_* \h{20pt}\text{for almost all $t \in [-\frac{1}{4}, 0\h{0.5pt}]$}
\end{align*} We apply the last two estimates to the right-hand side of \eqref{L3 over Q1/2}. It follows \begin{align*}
    \int_{Q_{1/2}} \big|\h{1pt}\phi_i - \phi_* \h{1pt}\big|^3 \lesssim_{K_*} \left(\int_{Q_{1/2}} \big|\h{1pt}   \phi_i - \phi_* \h{1pt}\big|^{2}  \right)^{\frac{5}{8}} + \int_{Q_{1/2}} \big|\h{0.5pt}   \phi_i - \phi_* \h{0.5pt}\big|^{2} \longrightarrow 0 \h{20pt}\text{as $i \to \infty$}.
\end{align*}\vspace{0.2pc}

\noindent \textbf{II. Strong $L^3$-convergence of $(u_i, \nabla \phi_i)$} \vspace{0.2pc}

By (1) in \eqref{rescaled L3 estimate}, \eqref{rescaled weak}, and the duality argument, it turns out from the equation \eqref{rescaled system} that
\begin{equation*} 
    \big\|\h{0.5pt} \partial_t u_i \h{0.5pt}\big\|_{L_t^{\frac{3}{2}} \h{0.5pt}W^{-1, \frac{3}{2}}_x(Q_{1/2})} + \big\|\h{0.5pt} \partial_t \nabla \phi_i \h{0.5pt}\big\|_{L_t^{\frac{3}{2}} \h{0.5pt}W^{-1, \frac{3}{2}}_x(Q_{1/2})} \hspace{1.5pt}\leq\hspace{1.5pt} K_* \h{15pt}\text{for all $i$.}
\end{equation*}Recall \eqref{10/3 est}. We now take $$p = 3, \h{10pt}q = \frac{10}{3}, \h{10pt}X = H^1\big(B^{\pm}_{1/2}; \mathbb R^3\big), \h{10pt}B = L^3\big(B^{\pm}_{1/2}; \mathbb R^3\big), \h{10pt} Y = W^{-1, \frac{3}{2}}\big(B^{\pm}_{1/2}; \mathbb R^3\big)$$ in the Aubin-Lions' type compactness lemma. See Corollary 6 of \cite{S1986}. We then obtain 
\begin{equation} \label{L3 convergence}
    \big(u_i, \nabla \phi_i\big) \hspace{1.5pt}\longrightarrow\hspace{1.5pt} \big(u_*, \nabla \phi_*\big) \quad \text{strongly in } L^3\big(Q_{1/2}\big),
\end{equation}up to a subsequence. Here, $u_* \in L^3(Q_{1/2})$. $\phi_*$ is the same as we obtain in Part I of this step.\vspace{0.4pc}

\noindent \textbf{III. Decay estimate of $p_i$} \vspace{0.2pc}

In what follows, the norm of the space $L^p_t L^q_x(Q_r)$ is denoted by $\|\h{0.2pt} \cdot \h{0.2pt}\|_{q,\, p,\, Q_r}$. The standard space-time Sobolev spaces are used, as in \cite{Seregin2003}. 

Suppose $\big(u_i^{(1)}, p_i^{(1)}\big)$ satisfy the initial boundary value problem:
\begin{equation*} \left\{ \begin{aligned}
    \partial_t u_i^{(1)} - \Delta u_i^{(1)} + \nabla p_i^{(1)} &= - \epsilon_i \h{1pt}u_i \cdot \nabla u_i - \epsilon_i \h{1pt}\nabla \cdot \big(\nabla \phi_i \odot \nabla  \phi_i\big) \hspace{.5pt} \quad\quad \text{in } Q_{1/2}, \\[1mm]
    \operatorname{div} u_i^{(1)} &= 0 \quad\hspace{160.5pt} \text{in } Q_{1/2}, \\[1mm]
    u_i^{(1)} &= 0 \quad\hspace{160.5pt} \text{on } \mathscr P Q_{1/2}, 
\end{aligned} \right. \end{equation*}where $\mathscr P Q_{1/2}$ is the parabolic boundary of $Q_{1/2}$. By Theorem 1.1 in \cite{Solonnikov2003}, 
\begin{equation}\label{est of u_1 1st}
    \big\| \h{1pt} u_i^{(1)} \big\|_{W^{2,1}_{\frac{9}{8}, \frac{3}{2}}(Q_{1/2})} + \big\|\h{1pt} \nabla p_i^{(1)} \big\|_{\frac{9}{8}, \frac{3}{2},\, Q_{1/2}} \hspace{1.5pt}\lesssim_{K_*}\hspace{1.5pt}  \epsilon_i \h{0.5pt}\big\| \h{1pt} u_i \cdot \nabla  u_i + \nabla \cdot \big(\nabla \phi_i \odot \nabla  \phi_i\big) \h{1pt}\big\|_{\frac{9}{8}, \frac{3}{2}, \,Q_{1/2}}.
\end{equation} Utilizing H\"{o}lder's inequality, we have \begin{align*}
    &\int_{- 1/4}^0 \big\| \h{1pt} u_i \cdot \nabla  u_i + \nabla \cdot \big(\nabla \phi_i \odot \nabla  \phi_i\big) \h{1pt}\big\|^{\frac{3}{2}}_{L^{\frac{9}{8}}(B^{\pm}_{1/2})} \\[1mm]
    &\h{30pt}\lesssim_{K_*} \int_{- 1/4}^0 \big\| \h{0.5pt}u_i\h{0.5pt}\big\|^{\frac{3}{2}}_{L^{\frac{18}{7}}(B^{\pm}_{1/2})} \h{1pt}\big\| \h{0.5pt}\nabla u_i\h{0.5pt}\big\|^{\frac{3}{2}}_{L^{2}(B^{\pm}_{1/2})} + \big\| \h{0.5pt}\nabla \phi_i\h{0.5pt}\big\|^{\frac{3}{2}}_{L^{\frac{18}{7}}(B^{\pm}_{1/2})} \h{1pt}\big\| \h{0.5pt}\nabla^2 \phi_i\h{0.5pt}\big\|^{\frac{3}{2}}_{L^{2}(B^{\pm}_{1/2})}.
\end{align*}
To control the $L^{\frac{18}{7}}$-norms above, we apply the Gagliardo-Nirenberg inequality and get \begin{align*}
    &\big\| \h{0.5pt}u_i\h{0.5pt}\big\|_{L^{\frac{18}{7}}(B^{\pm}_{1/2})}\lesssim_{K_*} \big\| \h{0.5pt}\nabla u_i\h{0.5pt}\big\|^{\frac{1}{3}}_{L^{2}(B^{\pm}_{1/2})} \h{1pt}\big\| \h{0.5pt}u_i\h{0.5pt}\big\|^{\frac{2}{3}}_{L^{2}(B^{\pm}_{1/2})} + \big\| \h{0.5pt}u_i\h{0.5pt}\big\|_{L^{2}(B^{\pm}_{1/2})}, \\[1mm]
     &\big\| \h{0.5pt}\nabla \phi_i\h{0.5pt}\big\|_{L^{\frac{18}{7}}(B^{\pm}_{1/2})} \lesssim_{K_*} \big\| \h{0.5pt}\nabla^2 \phi_i\h{0.5pt}\big\|^{\frac{1}{3}}_{L^{2}(B^{\pm}_{1/2})} \h{1pt}\big\| \h{0.5pt}\nabla \phi_i\h{0.5pt}\big\|^{\frac{2}{3}}_{L^{2}(B^{\pm}_{1/2})} + \big\| \h{0.5pt}\nabla \phi_i\h{0.5pt}\big\|_{L^{2}(B^{\pm}_{1/2})}.
\end{align*}By \eqref{rescaled weak} and the last three estimates, \eqref{est of u_1 1st} can be reduced to  \begin{align}\label{est of u_1}
\big\| \h{1pt} u_i^{(1)} \big\|_{W^{2,1}_{\frac{9}{8}, \frac{3}{2}}(Q_{1/2})} + \big\|\h{1pt} \nabla p_i^{(1)} \big\|_{\frac{9}{8}, \frac{3}{2},\, Q_{1/2}}  \lesssim_{K_*}  \epsilon_i 
\end{align}

Now, we decompose $(u_i, p_i)$ into $$u_i = u_i^{(1)} + u_i^{(2)},\h{15pt} p_i = p_i^{(1)} + p_i^{(2)}.$$ The sequence $\bigl(u_i^{(2)}, p_i^{(2)}\bigr)$ then satisfy
\begin{equation*} \left\{ \begin{aligned}
    \partial_t u_i^{(2)} - \Delta u_i^{(2)} + \nabla p_i^{(2)} &= 0 \quad\quad \text{in } Q_{1/2}, \\[1mm]
    \operatorname{div} u_i^{(2)} &= 0 \quad\quad \text{in } Q_{1/2}; \\[1mm]
    u_i^{(2)} &= 0 \quad\quad \text{on } B'_{1/2} \times (- \frac{1}{4}, 0\h{0.5pt}).
\end{aligned} \right. \end{equation*}
Applying Proposition 2 in \cite{Seregin2003} to the above system induces
\begin{align}\label{basic est of pI2}
    \big\|\h{1pt} p^{(2)}_i \h{1pt}\big\|_{W^{1,0}_{9, \frac{3}{2}} (Q_{1/4})} &\h{1.5pt}\lesssim_{K_*}\h{1.5pt}  \big\|\h{1pt} u^{(2)}_i \h{1pt}\big\|_{W^{1, 0}_{\frac{9}{8}, \frac{3}{2}}(Q_{1/2})} + \big\|\h{1pt} p^{(2)}_i \h{1pt}\big\|_{\frac{9}{8}, \frac{3}{2},\, Q_{1/2}}.
\end{align}
Using triangle inequality and H\"{o}lder inequality, we obtain \begin{align}\label{control u_*}\big\|\h{1pt} u^{(2)}_i \h{1pt}\big\|_{W^{1, 0}_{\frac{9}{8}, \frac{3}{2}}(Q_{1/2})} \leq \big\|\h{1pt} u^{(1)}_i \h{1pt}\big\|_{W^{1, 0}_{\frac{9}{8}, \frac{3}{2}}(Q_{1/2})} + \big\|\h{1pt} u_i \h{1pt}\big\|_{W^{1, 0}_{2, 2} (Q_{1/2})} \leq K_*.\end{align} If we assume $p^{(1)}_i$ has $0$ average over $B^{\pm}_{1/2}$, then  by triangle inequality and Poincar\'{e} inequality, \begin{align}\label{control P_*}\big\|\h{1pt} p^{(2)}_i \h{1pt}\big\|_{\frac{9}{8}, \frac{3}{2}, \, Q_{1/2}} \leq \big\|\h{1pt} \nabla p^{(1)}_i \h{1pt}\big\|_{\frac{9}{8}, \frac{3}{2}, \,Q_{1/2}} + \big\|\h{1pt} p_i \h{1pt}\big\|_{\frac{3}{2}, \frac{3}{2}, \,Q_{1/2}} \leq K_*.\end{align}
Therefore, $\big\|\h{1pt} \nabla p_i^{(2)} \h{1pt}\big\|_{9, \frac{3}{2}, \, Q_{1/4}}\hspace{1.5pt}\leq\hspace{1.5pt} K_*$, which together with the estimate of $\nabla p_i^{(1)}$ in \eqref{est of u_1} infers
\begin{align*}
    \int_{Q_{\theta_0}} \big|\hspace{.8pt} p_i - [\h{.5pt} p_i \h{.5pt}]^*_{0, \theta_0} \hspace{.5pt}\big|^{\frac{3}{2}} &\hspace{1.5pt}\lesssim_{K_*}\hspace{1.5pt}   \theta_0^{\frac{1}{2}} \int_{- \theta_0^2}^0 \left(\int_{B^{\pm}_{\theta_0}} \big|\hspace{.5pt} \nabla p_i^{(1)} \hspace{.5pt} \big|^{\frac{9}{8}} \right)^{\frac{4}{3}} +  \theta_0^{4} \int_{- \theta_0^2}^0 \left(\int_{B^{\pm}_{\theta_0}} \big|\hspace{.5pt} \nabla p_i^{(2)} \hspace{.5pt}\big|^{9} \right)^{\frac{1}{6}} \h{1.5pt}\lesssim_{K_*}\h{1.5pt} \theta_0^{\frac{1}{2}} \h{1pt}\epsilon_i^{\frac{3}{2}} + \theta_0^4.
\end{align*}
Here, the Sobolev, H\"{o}lder, and Poincar\'{e} inequality are also used. Hence, 
\begin{equation*}
    \limsup_{i\h{1pt}\to\h{1pt}\infty} \int_{Q_{\theta_0}} \big|\hspace{.8pt} p_i - [\h{.5pt} p_i \h{.5pt}]^*_{0, \theta_0} \hspace{.5pt}\big|^{\frac{3}{2}} \lesssim_{K_*} \h{0.5pt}\theta_0^4.
\end{equation*}

\noindent\textbf{Step 4.} Using the results obtained in Step 3, we take $i \to \infty$ in (2) of \eqref{rescaled L3 estimate} and arrive at \begin{align}\label{contra}
 \theta_0^{3} \leq K_* \h{0.5pt}   \theta_0^4 + \theta_0^{-2} \int_{Q_{\theta_0}} \big| \hspace{.5pt} u_* \hspace{.5pt} \big|^3 + \big| \hspace{.5pt} \nabla \phi_* \hspace{.5pt} \big|^3 + \theta_0^{-5} \int_{Q_{\theta_0}} \big|\h{0.5pt} \phi_* - (\phi_*)_{0, \theta_0}^* \h{0.5pt}\big|^3.
\end{align} Recall \eqref{small ener}. We can also take $i \to \infty$ in the first equation of \eqref{rescaled system}. $u_*$ then solves the following boundary value problem:\begin{align*}
    \partial_t u_* - \Delta u_* = - \nabla p_*, \hspace{20pt} \operatorname{div} u_* = 0 \quad\quad \text{in } Q_{1/2}; \h{20pt} u_* = 0 \quad\quad\text{on $B'_{1/2} \times [- 1/4, 0\h{0.8pt}]$}.
\end{align*} Apply Lemma 1 in \cite{Seregin2003} to this Stokes equation. $u_*$ is therefore $\frac{1}{3}$-H\"older continuous on the closure of $Q_{1/8}$. The semi-H\"{o}lder norm over $\overline{Q_{1/8}}$ is bounded from above by $K_*$. Here, we use \eqref{control u_*}-\eqref{control P_*}. Since $u_* = 0$ on $B'_{1/2}$, it follows that
\begin{align}\label{decay_1}
    \theta_0^{-2} \int_{Q_{\theta_0}} \big|\hspace{.5pt} u_* \hspace{.5pt} \big|^3  \hspace{1.5pt}\lesssim_{K_*}\hspace{1.5pt}  \theta_0^4.
\end{align} Applying this estimate to the right-hand side of \eqref{contra} induces \begin{align}\label{contra2}
 \theta_0^{3} \leq K_* \h{0.5pt}   \theta_0^4 + \theta_0^{-2} \int_{Q_{\theta_0}}  \big| \hspace{.5pt} \nabla \phi_* \hspace{.5pt} \big|^3 + \theta_0^{-5} \int_{Q_{\theta_0}} \big|\h{0.5pt} \phi_* - (\phi_*)_{0, \theta_0}^* \h{0.5pt}\big|^3 \h{15pt}\text{for all $\theta_0 \in \big(0, \frac{1}{8}\big)$.}
\end{align}

Since $\big\{r_i\big\}$ and $ \big\{ r_i^2 \h{0.5pt} \epsilon_i^{-1} \big\} \to 0$ as $i \to \infty$, taking $i \to \infty$ in third equation of \eqref{rescaled system} then induces
\begin{align*}
    \partial_t \phi_* - \Delta \phi_* = 0  \hspace{20pt}\text{in $Q_{1/2}$}. 
\end{align*}The boundary condition of $\phi_*$ on $B'_{1/2}$ is different when the spatial domain is $B^+_{1/2}$ or $B^-_{1/2}$.\vspace{0.4pc}

\noindent\textbf{Case 1.} If $\big\{x_i \big\} \subset \mathrm H$, by \eqref{scaled wk anchoring}, the boundary condition of $\phi_i$ on $B'_{1/2}$ can be rewritten as \begin{align*}
    \partial_3 \phi_i = - \frac{L_\mathrm{H}}{2} \h{1pt}\frac{ \h{0.5pt} r_i}{\epsilon_i} \left(  \sin  2\left(\epsilon_i \h{0.5pt}\phi_i +  (\h{0.5pt}\phi\h{0.5pt})_{z_i, \frac{r_i}{2}}\right) - \sin  2 \h{1pt}(\h{0.5pt}\phi\h{0.5pt})_{z_i, \frac{r_i}{2}} \right) - \frac{L_\mathrm{H}}{2} \h{1pt}\frac{ \h{0.5pt} r_i}{\epsilon_i} \sin  2 \h{1pt}(\h{0.5pt}\phi\h{0.5pt})_{z_i, \frac{r_i}{2}}.  
\end{align*} There is a $\mu_* \in [-1, 1]$ so that the right-hand side above converges to $\mu_* \h{0.5pt}L_{\mathrm H} \h{0.8pt} \theta_0^{\frac{1}{3}}$ as $i \to \infty$, up to a subsequence. Hence, \begin{equation} \label{sbc}
    \partial_3 \phi_* = \mu_* \h{0.5pt}L_{\mathrm H} \h{0.8pt} \theta_0^{\frac{1}{3}}   \quad\quad \text{on } B'_{1/2} \times (- \frac{1}{4}, 0\h{0.5pt}).
\end{equation}

\noindent\textbf{Case 2.} If  $\big\{x_i \big\} \subset \mathrm P$, by \eqref{scaled wk anchoring Diri}, the boundary condition  of $\phi_*$ is given as follows:
\begin{equation} \label{sbc Diri}
    \phi_* = 0 \quad\quad \text{on } B'_{1/2} \times (- \frac{1}{4}, 0\h{0.5pt}).
\end{equation} 

We now estimate the function $\phi_*$ by separately discussing the above two cases. \vspace{0.4pc} 
 
\noindent\textbf{I.1. Spatial gradient estimates of $\phi_*$ in Case 1.} \vspace{0.2pc}

Recalling \eqref{sbc}, we extend $\psi_0 := \phi_* -   \mu_* \h{0.5pt} L_{\mathrm{H}} \h{0.5pt}\theta_0^{\frac{1}{3}} \h{0.5pt}x_3$ to $B_{1/2}$ by even extension. Define $$Q^*_{r} := B_r \times \big(- r^2, 0\h{0.5pt}\big).$$ Then, $\psi_0$  solves the parabolic equation \begin{align}\label{eq of psi-0}\partial_t \psi_0 = \Delta \psi_0 \h{15pt}\text{on $Q^*_{1/2}$}.\end{align} As shown on P53 of \cite{Liebmann}, $\nabla \psi_0$  is uniformly bounded on $Q^*_{1/4}$ with its $L^\infty$-norm satisfying \begin{align}\label{Linfty of nabla}
  \big\| \nabla \psi_0 \big\|^2_{L^\infty(Q^*_{1/4})} \leq  K_* \int_{Q^*_{1/2}} \big| \nabla \psi_0 \big|^2 \leq K_*.
\end{align}Utilize Theorem 4.7 in \cite{Liebmann}. $\nabla \psi_0$ is also $\frac{1}{3}$-H\"{o}lder continuous over $\overline{Q^*_{1/4}}$ with its semi-H\"{o}lder norm bounded from above by a universal constant $K_*$. Since $\partial_3 \psi_0 = 0$ on the flat boundary $B'_{1/2} \times (- \frac{1}{4}, 0\h{0.5pt})$,  same as \eqref{decay_1}, it holds \begin{align*}
    \theta_0^{-2} \int_{Q_{\theta_0}} \big|\hspace{.5pt} \partial_3 \psi_0 \hspace{.5pt} \big|^3  \hspace{1.5pt}\lesssim_{K_*}\hspace{1.5pt}  \theta_0^4.
\end{align*}Therefore, \begin{align*}
    \theta_0^{-2} \int_{Q_{\theta_0}} \big|\hspace{.5pt} \partial_3 \phi_* \hspace{.5pt} \big|^3  \hspace{1.5pt}\lesssim_{K_*}\hspace{1.5pt}  \theta_0^4.
\end{align*} As for the tangential derivatives, we use \eqref{Linfty of nabla} and apply Lemma 4.20 in \cite{Liebmann} to $\nabla' \phi_*$. It follows \begin{align*}
    \int_{Q_{\theta_0}} \big| \nabla' \phi_*\big|^3 \lesssim_{K_*} \int_{Q_{\theta_0}} \big| \nabla' \phi_*\big|^2 \lesssim_{K_*} \theta_0^7 \int_{Q_{1/2}} \big| \nabla' \phi_*\big|^2 \lesssim_{K_*} \theta_0^7.
\end{align*} Combining the last two estimates induces \begin{align}\label{est in even}
\theta_0^{-2} \int_{Q_{\theta_0}} \big|\hspace{.5pt} \nabla \phi_* \hspace{.5pt} \big|^3  \hspace{1.5pt}\lesssim_{K_*}\hspace{1.5pt} \theta_0^4 \h{15pt}\text{for all $\theta_0 \in \big(0, \frac{1}{4}\big)$.}
\end{align}

\noindent\textbf{I.2. $L^3$-estimate of $\phi_*$ in Case 1.} \vspace{0.2pc}

To estimate the $L^3$-integral of $\phi_* - (\phi_*)^*_{0, \theta_0}$ over $Q_{\theta_0}$, we also need to bound the $L^3$-integral of $\partial_t \phi_*$ over $Q_{\theta_0}$. Note that for $j = 1, 2, 3$, the spatial derivative $\partial_j \psi_0$ solves the same equation as $\psi_0$ in \eqref{eq of psi-0}. Using Proposition 7.14 in \cite{Liebmann} infers \begin{align*} \int_{Q^*_{\theta_0}} \big|\h{0.5pt} \nabla^2 \psi_0 \h{0.5pt}\big|^3 \lesssim_{K_*} \theta_0^{-3} \int_{Q^*_{36 \theta_0}}
\big|\h{0.5pt} \nabla \psi_0 \h{0.5pt}\big|^3 \lesssim_{K_*} \theta^3_0 \h{20pt}\text{for all $\theta_0 \in \big(0, \frac{1}{144} \big)$}.
\end{align*} It turns out \begin{align}\label{est of part t}\int_{Q_{\theta_0}} \big|\h{0.5pt} \partial_t \phi_* \h{0.5pt}\big|^3 = \int_{Q_{\theta_0}} \big|\h{0.5pt} \partial_t \psi_0 \h{0.5pt}\big|^3 \lesssim_{K_*} \theta_0^3 \h{20pt} \text{for all $\theta_0 \in \big(0, \frac{1}{144} \big)$}.
\end{align}
By Poincar\'{e} inequality, \begin{align*}
    \int_{Q_{\theta_0}} \big|\h{0.5pt} \phi_* - (\phi_*)_{0, \theta_0}^* \h{0.5pt}\big|^3 \lesssim_{K_*} \theta_0^3 \int_{Q_{\theta_0}} \big|\h{0.5pt} \nabla \phi_* \h{0.5pt}\big|^3 + \theta_0^6 \int_{Q_{\theta_0}} \big|\h{0.5pt} \partial_t \phi_* \h{0.5pt}\big|^3.
\end{align*}Applying \eqref{est in even}-\eqref{est of part t} to the right-hand side above gives us\begin{align}\label{control of L3 in even}
\theta_0^{-5}\int_{Q_{\theta_0}} \big|\h{0.5pt} \phi_* - (\phi_*)_{0, \theta_0}^* \h{0.5pt}\big|^3 \lesssim_{K_*} \theta_0^4 \h{20pt} \text{for all $\theta_0 \in \big(0, \frac{1}{144} \big)$}.
\end{align}

\noindent\textbf{II. Estimates of $\phi_*$ in Case 2.} \vspace{0.2pc}

In light of \eqref{sbc Diri}, we extend $\phi_*$ to $Q^*_{1/2}$ via the odd extension with respect to the $x_3$-variable. The extended $\phi_*$ satisfies the same equation as $\psi_0$ in \eqref{eq of psi-0}. Using the same argument for deriving \eqref{Linfty of nabla}, we obtain \begin{align}\label{Linfty of nabla phi*}
  \big\| \nabla \phi_* \big\|^2_{L^\infty(Q^*_{1/4})} \leq  K_* \int_{Q^*_{1/2}} \big| \nabla \phi_* \big|^2 \leq K_*.
\end{align}
Still utilize Theorem 4.7 in \cite{Liebmann}. $\nabla \phi_*$ is also $\frac{1}{3}$-H\"{o}lder continuous over $\overline{Q^*_{1/4}}$ with its semi-H\"{o}lder norm bounded from above by a universal constant $K_*$. Since $\nabla' \phi_* = 0$ on the flat boundary $B'_{1/2} \times (- \frac{1}{4}, 0\h{0.5pt})$,  same as \eqref{decay_1}, it turns out \begin{align*}
    \theta_0^{-2} \int_{Q_{\theta_0}} \big|\hspace{.5pt} \nabla' \phi_* \hspace{.5pt} \big|^3  \hspace{1.5pt}\leq\hspace{1.5pt} K_* \h{0.5pt} \theta_0^4.
\end{align*}

\noindent As for the normal derivative, we use \eqref{Linfty of nabla phi*} and apply Lemma 4.20 in \cite{Liebmann} to $\partial_3 \phi_*$. It holds
\begin{align*}
    \int_{Q_{\theta_0}} \big|\h{.5pt} \partial_3 \phi_* \h{.5pt}\big|^3 \h{1.5pt}\lesssim_{K_*}\h{1.5pt}  \int_{Q_{\theta_0}} \big|\h{.5pt} \partial_3 \phi_* \h{.5pt}\big|^2 \lesssim_{K_*}  \theta_0^7 \int_{Q_{1/2}} \big|\h{.5pt} \partial_3 \phi_* \h{.5pt}\big|^2 \lesssim_{K_*} \theta_0^7 \h{20pt}\text{for all $\theta_0 \in \big(0, \frac{1}{4}\big)$}. 
\end{align*}We can apply Lemma 4.20 in \cite{Liebmann} to $\partial_3 \phi_*$ since it satisfies $$\partial_{33} \phi_* = \partial_t \phi_* - \Delta' \phi_* = 0 \h{20pt}\text{ on $B'_{1/2} \times \big(- \frac{1}{4}, 0 \big)$}$$ by the boundary condition \eqref{sbc Diri}. Therefore, we still obtain the decay estimate \eqref{est in even} for the $\phi_*$ in case 2 by combining the last two estimates. Furthermore, by Poincar\'{e} inequality, \begin{align}\label{control of L3 in odd}
\theta_0^{-5}\int_{Q_{\theta_0}} \big|\h{0.5pt} \phi_* \h{0.5pt}\big|^3 \lesssim_{K_*} \theta_0^{-2}\int_{Q_{\theta_0}} \big|\h{0.5pt} \nabla \phi_* \h{0.5pt}\big|^3 \lesssim_{K_*} \theta_0^4 \h{20pt} \text{for all $\theta_0 \in \big(0, \frac{1}{4} \big)$}.
\end{align}

We now complete the proof of this lemma. Applying \eqref{est in even} and \eqref{control of L3 in even} if Case 1 holds, or \eqref{est in even} and \eqref{control of L3 in odd} if Case 2 holds to the right-hand side of \eqref{contra2} yields \begin{align*} 
 \theta_0^{3} \leq K_* \h{0.5pt}   \theta_0^4  \h{20pt}\text{for all $\theta_0 \in \big(0, \frac{1}{144}\big)$.}
\end{align*}This is impossible if we take $\theta_0$ to be $\frac{1}{2}\min \left\{ \frac{1}{144}, \frac{1}{K_*} \right\}$. $K_*$ is the constant in the last estimate.
\end{proof}

Iteratively applying Lemma \ref{decay lem}, we have
\begin{lem} \label{iter}
Fix $r > 0$ and $z_0 = (x_0, t_0) \in \big(\mathrm {H \cup P}\big) \times (r^2, \infty)$. Assume that
\begin{equation}\label{small assumption}
    \max \Big\{ \h{0.5pt} \theta_0^{-9} \h{0.5pt} r^3, \h{2pt}F(r, z_0) \h{0.5pt}\Big\} \hspace{1.5pt}\leq\hspace{1.5pt} \epsilon_0^3,
\end{equation}where $\epsilon_0$ is given in Lemma \ref{decay lem}. Then, 
\begin{align*}\max \Bigl\{\h{0.5pt} \theta_0^{-9}\h{0.5pt}\big(\theta^k_0\h{0.5pt}r\big)^3,  \h{0.5pt}F\big(\theta^k_0 \h{0.5pt} r, z_0\big) \h{0.5pt}\Bigr\}  \hspace{1.5pt}\leq\hspace{1.5pt}  \theta_0^{3k} \max \Bigl\{\h{0.5pt} \theta_0^{-9}\h{0.5pt}r^3,  F\big(r, z_0\big) \h{0.5pt}\Bigr\} \h{20pt}\text{for all $k \in \mathbb N \cup \{\h{0.5pt}0\h{0.5pt}\}$.}
\end{align*} Equivalently, 
\begin{align}\label{average est}
     \big(\theta^k_0 \h{0.5pt} r\big)^{-3}F\big(\theta^k_0 \h{0.5pt} r, z_0\big) \hspace{1.5pt}\leq\hspace{1.5pt}  r^{-3} \max \Bigl\{\h{0.5pt} \theta_0^{-9}\h{0.5pt}r^3,  F\big(r, z_0\big) \h{0.5pt}\Bigr\} \h{20pt}\text{for all $k \in \mathbb N \cup \{\h{0.5pt}0\h{0.5pt}\}$.}
\end{align}
\end{lem}

A direct corollary of Lemma \ref{iter} is read as follows. \begin{cor}\label{aver on cylinder}
    Assume the same $r$ and $z_0$ as in Lemma \ref{iter}. Then we have \begin{align*}
        \rho^{-3}  F\big(\rho, z_0\big) \leq 8 \h{0.5pt}\theta_0^{-8}  r^{-3} \max \Bigl\{\h{0.5pt} \theta_0^{-9}\h{0.5pt}r^3,  F\big(r, z_0\big) \h{0.5pt}\Bigr\} \h{40pt}\text{for all $\rho \in (0, r \h{0.5pt}]$.}
    \end{align*}
\end{cor}

The average of $|\h{0.5pt}u\h{0.5pt}|^2 + |\h{0.5pt}\nabla \phi \h{0.5pt}|^2$ on $B^{\pm}_{\rho\h{0.5pt}/\h{0.5pt}2}(x_0)$ at time $t_0$ can  be bounded by Corollary \ref{aver on cylinder} and \eqref{energy ineq}.  

\begin{prop}
Assume the same $r$ and $z_0$ as in Lemma \ref{iter}. Then we have
\begin{align*}
    \rho^{-3} \int_{B^{\pm}_{\rho \h{0.5pt}/\h{0.5pt} 2}(x_0) \times \{\h{0.5pt} t_0 \h{0.5pt}\}} \big|\h{.5pt} u \h{.5pt}\big|^2 + \big|\h{.5pt} \nabla \phi \h{.5pt}\big|^2 \h{1.5pt}\lesssim_{K_*}\h{1.5pt} 1 + \max \left\{ \theta^{-9}_0, \frac{F\big(r, z_0 \big)}{r^3} \right\}\h{20pt}\text{for all $\rho \in (0, r]$.}
\end{align*}
\end{prop}

\begin{proof}[\bf Proof.]
Choose the cut-off function $\varphi$ supported in $P_\rho(z_0)$, and with the properties:
\begin{align*}
    &\mathrm{(1).} \h{5pt}0 \hspace{1.5pt}\leq\hspace{1.5pt} \varphi \hspace{1.5pt}\leq\hspace{1.5pt} 1 \h{15pt}\text{in $P_\rho(z_0)$}; \h{20pt} \mathrm{(2).}\h{5pt}\rho \h{0.5pt}\big|\hspace{1.5pt} \nabla \varphi \hspace{.5pt} \big| +  \rho^2\h{0.5pt}\big|\hspace{.5pt} \partial_t \varphi \hspace{.5pt}\big| + \rho^2 \h{0.5pt} \big|\hspace{.5pt} \nabla^2 \varphi \hspace{.5pt} \big| \hspace{1.5pt}\leq K_* \h{15pt}\text{in $P_{\rho}(z_0)$}; \\[1mm] &\mathrm{(3).}\h{5pt} \varphi \hspace{1.5pt}\equiv\hspace{1.5pt} 1 \h{34pt}\text{ in } P_{\rho \h{0.5pt}/\h{0.5pt} 2}(z_0).
\end{align*}
Replacing the test function in \eqref{energy ineq} with $\varphi^2$, we obtain
\begin{align*}
    &\int_{B_{\rho}^{\pm}(x_0) \times \{ \h{0.5pt} t_0 \h{0.5pt}\} } \varphi^2 \left(\h{0.5pt} \big|\h{.5pt} u \h{.5pt}\big|^2 + \big|\h{.5pt} \nabla \phi \h{.5pt}\big|^2 \h{0.5pt} \right) +    \int_{P_{\rho}(z_0)} \varphi^2 \left(\h{0.5pt} \big|\h{.5pt} \nabla u \h{.5pt}\big|^2 + \big|\h{.5pt} \nabla^2 \phi\h{.5pt}\big|^2 \h{0.5pt}\right)\\[1mm]
    &\h{50pt}\lesssim_{K_*} \rho^3    + \big| \h{0.5pt}\mathrm{T}(\rho, z_0) \h{0.5pt}\big|  +  \rho^{-2} \int_{P_\rho(z_0)}  \big|\hspace{.5pt} u \hspace{.5pt} \big|^2 + \big|\hspace{.5pt} \nabla \phi \hspace{.5pt} \big|^2\\[1mm]
     &\h{63pt}+  \rho^{-1} \int_{P_\rho(z_0)}\big|\hspace{.5pt} u \hspace{.5pt} \big|^3 + \big|\hspace{.5pt} \nabla \phi \hspace{.5pt}\big|^3  + \rho^{-1} \left( \int_{P_\rho(z_0)} \big|\h{.5pt} u \h{.5pt}\big|^3 \right)^{\frac{1}{3}}\left(\int_{P_\rho(z_0)}   \big|\hspace{.5pt} p - [\h{.5pt} p \h{.5pt}]_{x_0,\, \rho} \hspace{.5pt}\big|^{\frac{3}{2}} \right)^{\frac{2}{3}}.
\end{align*}
We have applied the same arguments for \eqref{est R for use} and  \eqref{2nd inegral}-\eqref{third est requ} in the above estimate. Moreover,
\begin{align*}
    \mathrm{T}(\rho, z_0) := L_{\mathrm H}\int_{B'_{\rho}(x_0) \times \{\h{0.5pt} t_0 \h{0.5pt} \}} \varphi^2 \Bigl(\sin^2 \phi - \sin^2 (\phi)_{z_0, \rho} \Bigr) - L_{\mathrm H}\int_{t_0 - \rho^2}^{t_0} \int_{B'_{\rho}(x_0)} \partial_t \varphi^2 \h{1pt} \Bigl( \sin^2 \phi - \sin^2 (\phi)_{z_0, \rho} \Bigr).
\end{align*}
The last energy estimate can be reduced by Corollary \ref{aver on cylinder} as follows:
\begin{align}\label{reduc energy}
   \int_{B_{\rho}^{\pm}(x_0) \times \{ \h{0.5pt} t_0 \h{0.5pt}\} } \varphi^2 \left(\h{0.5pt} \big|\h{.5pt} u \h{.5pt}\big|^2 + \big|\h{.5pt} \nabla \phi \h{.5pt}\big|^2 \h{0.5pt} \right) &+    \int_{P_{\rho}(z_0)} \varphi^2  \h{0.5pt}  \big|\h{.5pt} \nabla^2 \phi\h{.5pt}\big|^2 \nonumber\\[1mm]
   &\lesssim_{K_*} \big| \h{0.5pt}\mathrm{T}(\rho, z_0) \h{0.5pt}\big| + \rho^3    +  \rho^3 \max \left\{\h{0.5pt} \theta_0^{-9}, \h{0.5pt}\frac{F\big(r, z_0\big)}{r^3} \h{0.5pt}\right\}.
\end{align}

If $x_0 \in \mathrm P$, then $\mathrm{T}(\rho, z_0) = 0$. The proof is finished. Now, we assume $x_0 \in \mathrm H$ and estimate $\mathrm{T}(\rho, z_0)$ in a similar fashion as in \eqref{es 1,null}. In this way, it turns out
\begin{align*}
    \big|\h{0.5pt}\mathrm{T}(\rho, z_0) \h{0.5pt}\big|  \lesssim_{K_*}  \int_{B_\rho'(x_0) \times \{\h{0.5pt} t_0 \h{0.5pt}\}} \varphi^2 \h{.5pt} \big|\h{.5pt} \phi - (\phi)_{z_0,\, \rho} \h{.5pt}\big| + \rho^{-2} \int_{t_0 - \rho^2}^{t_0} \int_{B_\rho'(x_0)} \varphi \h{1pt} \big|\h{.5pt} \phi - (\phi)_{z_0,\, \rho} \h{.5pt}\big|.
\end{align*}Using H\"{o}lder's inequality, we reduce the last estimate to \begin{align*}
    \big|\h{0.5pt}\mathrm{T}(\rho, z_0) \h{0.5pt}\big|  \lesssim_{K_*}  \rho \left(\int_{B_\rho'(x_0) \times \{\h{0.5pt} t_0 \h{0.5pt}\}} \varphi^4 \h{.5pt} \big|\h{.5pt} \phi - (\phi)_{z_0,\, \rho} \h{.5pt}\big|^2 \right)^{\frac{1}{2}} +   \left(\int_{t_0 - \rho^2}^{t_0} \int_{B_\rho'(x_0)} \varphi^2 \h{1pt} \big|\h{.5pt} \phi - (\phi)_{z_0,\, \rho} \h{.5pt}\big|^2 \right)^{\frac{1}{2}}.
\end{align*}
Apply the integration by parts with respect to the $x_3$-variable. The integrals on the right-hand side above can be estimated by
\begin{align*}
    &\int_{B_\rho'(x_0) \times \{\h{0.5pt} t_0 \h{0.5pt}\}} \varphi^4 \h{0.5pt}\big|\h{.5pt} \phi - (\phi)_{z_0,\, \rho} \h{.5pt}\big|^2 \lesssim_{K_*} \rho \h{0.5pt}\int_{B_\rho^{+}(x_0) \times \{\h{0.5pt} t_0 \h{0.5pt} \}} \varphi^2 \h{.5pt} \big|\h{.5pt} \nabla \phi \h{.5pt}\big|^2 + \rho^{-1} \int_{B_\rho^{+}(x_0) \times \{\h{0.5pt} t_0 \h{0.5pt} \}} \varphi^2 \h{1pt}\big|\h{0.5pt}\phi - (\phi)_{z_0,\, \rho} \h{0.5pt}\big|^2, \\[1.5mm]
&\int_{t_0 - \rho^2}^{t_0} \int_{B_\rho'(x_0)} \varphi^2 \big|\h{.5pt} \phi - (\phi)_{z_0,\, \rho} \h{.5pt}\big|^2 \lesssim_{K_*} \rho \int_{P_\rho(z_0)} \big|\h{.5pt} \nabla \phi \h{.5pt}\big|^2 + \rho^{-1}\int_{P_\rho(z_0)}  \big|\h{.5pt} \phi - (\phi)_{z_0,\, \rho} \h{.5pt} \big|^2.
\end{align*}
Therefore, \begin{align*}
    \big|\h{0.5pt}\mathrm{T}(\rho, z_0) \h{0.5pt}\big| \h{3pt} \lesssim_{K_*}    &\h{3pt}\rho^{\frac{3}{2}} \left(\int_{B_\rho^{+}(x_0) \times \{\h{0.5pt} t_0 \h{0.5pt} \}} \varphi^2 \h{.5pt} \big|\h{.5pt} \nabla \phi \h{.5pt}\big|^2\right)^{\frac{1}{2}} + \rho^{\frac{1}{2}} \left(\int_{B_\rho^{+}(x_0) \times \{\h{0.5pt} t_0 \h{0.5pt} \}} \varphi^2 \h{1pt}\big|\h{0.5pt}\phi - (\phi)_{z_0,\, \rho} \h{0.5pt}\big|^2 \right)^{\frac{1}{2}}\\[1mm]
    &\h{3pt}+   \left(\rho \int_{P_\rho(z_0)} \big|\h{.5pt} \nabla \phi \h{.5pt}\big|^2 + \rho^{-1}\int_{P_\rho(z_0)}  \big|\h{.5pt} \phi - (\phi)_{z_0,\, \rho} \h{.5pt} \big|^2\right)^{\frac{1}{2}}.
\end{align*}
By Corollary \ref{aver on cylinder}, \begin{align}\label{est inter}
     \rho \int_{P_\rho(z_0)} \big|\h{.5pt} \nabla \phi \h{.5pt}\big|^2  +  \rho^{-1} \int_{P_\rho(z_0)} \big|\h{.5pt} \phi - (\phi)_{z_0,\, \rho} \h{.5pt} \big|^2 \h{1.5pt}\lesssim_{K_*}\h{1.5pt} \rho^6 \max \left\{\h{0.5pt} \theta_0^{-9}, \h{0.5pt}\frac{F\big(r, z_0\big)}{r^3} \h{0.5pt}\right\} + \rho^6.
\end{align} The last two estimates induces \begin{align*}
\big|\h{0.5pt}\mathrm{T}(\rho, z_0) \h{0.5pt}\big| \h{3pt} \lesssim_{K_*}    &\h{3pt} \sigma \int_{B_\rho^{+}(x_0) \times \{\h{0.5pt} t_0 \h{0.5pt} \}} \varphi^2 \h{.5pt} \big|\h{.5pt} \nabla \phi \h{.5pt}\big|^2 + \sigma^{-1} \h{0.5pt} \rho^{3} \nonumber\\[1mm]
    & + \rho^3 \max \left\{\h{0.5pt} \theta_0^{-9}, \h{0.5pt}\frac{F\big(r, z_0\big)}{r^3} \h{0.5pt}\right\} + \rho^{\frac{1}{2}} \left(\int_{B_\rho^{+}(x_0) \times \{\h{0.5pt} t_0 \h{0.5pt} \}} \varphi^2 \h{1pt}\big|\h{0.5pt}\phi - (\phi)_{z_0,\, \rho} \h{0.5pt}\big|^2 \right)^{\frac{1}{2}}.
\end{align*}
Here, $\sigma > 0$ is a small positive number. Apply this estimate to the right-hand side of \eqref{reduc energy} and take $\sigma$ small enough. The smallness of $\sigma$ depends on the universal constant $K_*$. Then, \begin{align}\label{inter est 1}
    &\int_{B_{\rho}^{+}(x_0) \times \{ \h{0.5pt} t_0 \h{0.5pt}\} } \varphi^2 \left(\h{0.5pt} \big|\h{.5pt} u \h{.5pt}\big|^2 + \big|\h{.5pt} \nabla \phi \h{.5pt}\big|^2 \h{0.5pt} \right) +    \int_{P_{\rho}(z_0)} \varphi^2  \h{0.5pt}  \big|\h{.5pt} \nabla^2 \phi\h{.5pt}\big|^2 \nonumber\\[1mm]
   &\h{60pt}\lesssim_{K_*}   \rho^3 + \rho^3 \max \left\{\h{0.5pt} \theta_0^{-9}, \h{0.5pt}\frac{F\big(r, z_0\big)}{r^3} \h{0.5pt}\right\} + \rho^{\frac{1}{2}} \left(\int_{B_\rho^{+}(x_0) \times \{\h{0.5pt} t_0 \h{0.5pt} \}} \varphi^2 \h{1pt}\big|\h{0.5pt}\phi - (\phi)_{z_0,\, \rho} \h{0.5pt}\big|^2 \right)^{\frac{1}{2}}.
\end{align}
We are left to estimate the last term in \eqref{inter est 1} above. 

Multiply $\varphi^2 \bigl(\phi - (\phi)_{z_0,\, \rho} \bigr)$ on both sides of the third equation in \eqref{sEL} and integrate. It holds
\begin{align*}
    \int_{B_{\rho}^+(x_0) \times \{ \h{0.5pt}t_0 \h{0.5pt} \}} \varphi^2 \big|\h{0.5pt} \phi - (\phi)_{z_0,\, \rho} \h{0.5pt}\big|^2 &=  \int_{P_{\rho}(z_0)} \big|\h{0.5pt} \phi - (\phi)_{z_0,\, \rho} \h{0.5pt}\big|^2 \h{0.5pt}\partial_t \varphi^2 - 2 \varphi^2 \big(\phi - (\phi)_{z_0,\, \rho} \big)\h{0.5pt} u \cdot \nabla \phi \\[1mm]
    &+ \int_{P_{\rho}(z_0)} 2 \varphi^2 \big(\phi - (\phi)_{z_0,\, \rho} \big) \h{0.5pt}\Delta \phi + h^2 \big(\phi - (\phi)_{z_0,\, \rho} \big)\h{0.5pt} \sin 2 \phi.
\end{align*}
By H\"{o}lder inequality, Corollary \ref{aver on cylinder}, and the bounds of $\varphi$,
\begin{align*}
    \int_{B_{\rho}^+(x_0) \times \{ t_0 \}} \varphi^2 \Bigl( \phi - (\phi)_{z_0,\, \rho} \Bigr)^2 &\h{1.5pt}\lesssim_{K_*}\h{1.5pt} \rho^5 + \rho^5 \max \left\{\h{0.5pt} \theta_0^{-9}, \h{0.5pt}\frac{F\big(r, z_0\big)}{r^3} \h{0.5pt}\right\}   \\[1mm]
    &\h{15pt} +  \rho^{\frac{7}{2}} \left( \int_{P_{\rho}(z_0)} \varphi^2 \h{0.5pt}\big|\h{.5pt} \nabla^2 \phi \h{.5pt} \big|^2 \right)^{\frac{1}{2}} \left(\max \left\{\h{0.5pt} \theta_0^{-9}, \h{0.5pt}\frac{F\big(r, z_0\big)}{r^3} \h{0.5pt}\right\}\h{0.5pt}\right)^{\frac{1}{3}}.
\end{align*}
Apply this estimate to the right-hand side of \eqref{inter est 1} and then use Young's inequality. It follows \begin{align*}
    &\int_{B_{\rho}^{+}(x_0) \times \{ \h{0.5pt} t_0 \h{0.5pt}\} } \varphi^2 \left(\h{0.5pt} \big|\h{.5pt} u \h{.5pt}\big|^2 + \big|\h{.5pt} \nabla \phi \h{.5pt}\big|^2 \h{0.5pt} \right) +    \int_{P_{\rho}(z_0)} \varphi^2  \h{0.5pt}  \big|\h{.5pt} \nabla^2 \phi\h{.5pt}\big|^2 \nonumber\\[1mm]
   &\h{60pt}\lesssim_{K_*}   \sigma \int_{P_{\rho}(z_0)} \varphi^2 \h{0.5pt}\big|\h{.5pt} \nabla^2 \phi \h{.5pt} \big|^2 +  \sigma^{-\frac{1}{3}} \h{1pt}\rho^3 +  \sigma^{-\frac{1}{3}} \h{1pt}\rho^3 \max \left\{\h{0.5pt} \theta_0^{-9}, \h{0.5pt}\frac{F\big(r, z_0\big)}{r^3} \h{0.5pt}\right\},
\end{align*}where $\sigma$ is an arbitrary number in $(0, 1)$. The proof for $x_0 \in \mathrm H$ is also obtained by taking $\sigma$ suitably small. The smallness depends on a universal positive number. 
\end{proof}

\subsection{Uniform H\"older continuity of \texorpdfstring{$u$}{u} and the proof of  (\ref{decay of sup u})}\label{hoder of u} In this section, we first give an improved decay estimate for $u$ near $\mathrm H \cup \mathrm P$. Then we obtain a uniform H\"{o}lder estimate for large time $t$, which, by using Arzel\`{a}-Ascoli theorem, leads to the $L^\infty$-convergence of $u$ to $0$ as $t \to \infty$. Note that, in the following arguments, we always take time $t \geq T_\star$ for some large $T_\star$ such that \begin{align}\label{unif bounds of u nab phi}
 \|\h{0.5pt} u \h{0.5pt}\|_{L^\infty(\h{0.5pt}\Omega \times [\h{0.5pt} T_\star, \infty \h{0.5pt}) \h{0.5pt})} + \|\h{0.5pt} \nabla \phi \h{0.5pt}\|_{L^\infty(\h{0.5pt}\Omega \times [\h{0.5pt} T_\star, \infty \h{0.5pt}) \h{0.5pt})}  \leq K.
\end{align}Here, $K > 0$ is a constant.
 
\begin{lem} \label{decay lem u only}There exist a small constant $\theta_0 > 0$ and a constant $\epsilon_0 > 0$ such that if 
\begin{equation*}  
  r^{-2} \int_{P_r(z_0)} \big| \hspace{.5pt} u  \hspace{.5pt}\big|^3 + \left(r^{-2} \int_{P_r(z_0)} \big|\hspace{.5pt} p - [\h{0.5pt}p\h{0.5pt}]_{x_0, r} \hspace{.5pt} \big|^{\frac{3}{2}} \right)^2 \hspace{1.5pt}\leq\hspace{1.5pt} \epsilon_0^3,
\end{equation*}  for some $r \in (0, 1)$ and $z_0 = (x_0, t_0) \in \big(\mathrm {H \cup P}\big) \times \big(T_\star + 10, \infty\big)$,   then  
    \begin{align}
                   \big(\theta_0 r\big)^{-2} \int_{P_{\theta_0 r}(z_0)} \big|\hspace{.5pt} u  \hspace{.5pt}\big|^3 &+ \left( \big(\theta_0 r\big)^{-2}  \int_{P_{\theta_0 r}(z_0)} \big| \hspace{.5pt} p - [\h{0.5pt}p\h{0.5pt}]_{x_0, \theta_0 r} \hspace{.5pt} \big|^{\frac{3}{2}} \right)^2 \nonumber\\[1mm]
            &\leq \h{1.5pt}   \theta_0^{3.5} \max \left\{r^{3.5}, \h{1pt}  r^{-2} \int_{P_r(z_0)} \big| \hspace{.5pt} u  \hspace{.5pt}\big|^3 + \left(r^{-2} \int_{P_r(z_0)} \big|\hspace{.5pt} p - [\h{0.5pt}p\h{0.5pt}]_{x_0, r} \hspace{.5pt} \big|^{\frac{3}{2}} \right)^2 \right\}. 
    \end{align}
Here, $\epsilon_0$ is small enough. $\theta_0$ only depends on $h$, $L_{\mathrm H}$, and $K$ in \eqref{unif bounds of u nab phi}.
\end{lem}

\begin{proof}[\bf Proof] We divide the proof into 4 steps. \vspace{0.4pc}

\noindent \textbf{Step 1. Blow-up sequence.}\vspace{0.4pc}

Similar to Step 1 in the proof of Lemma \ref{decay lem}, we first construct a blow-up sequence. Suppose the conclusion of the current lemma is false. Then for a $\theta_0 \in (0, \frac{1}{4})$ to be determined later, we can find $r_i \in (0, 1)$ and $z_i = (x_i, t_i) \in \big(\mathrm {H \cup P} \big) \times \big(T_\star + 10, \infty\big)$ such that
\begin{equation}\label{small ener for u}
    r_i^{-2} \int_{P_{r_i}(z_i)} \big| \hspace{.5pt} u  \hspace{.5pt}\big|^3 + \left(r_i^{-2} \int_{P_{r_i}(z_i)} \big|\hspace{.5pt} p - [\h{0.5pt}p\h{0.5pt}]_{x_i, r_i} \hspace{.5pt} \big|^{\frac{3}{2}} \right)^2 := \lambda_i^3 \longrightarrow 0 \h{15pt}\text{as $i \to \infty$}.
\end{equation}
Meanwhile, it satisfies
\begin{align} \label{tau_blow_up for u}  \big(\theta_0 r_i\big)^{-2} \int_{P_{\theta_0 r_i}(z_i)} \big|\hspace{.5pt} u  \hspace{.5pt}\big|^3 &+ \left( \big(\theta_0 r_i\big)^{-2}  \int_{P_{\theta_0 r_i}(z_i)} \big| \hspace{.5pt} p - [\h{0.5pt}p\h{0.5pt}]_{x_i, \theta_0 r_i} \hspace{.5pt} \big|^{\frac{3}{2}} \right)^2 > \h{1.5pt}   \theta_0^{3.5} \max \left\{r_i^{3.5}, \h{1pt}  \lambda_i^3 \right\}. 
 \end{align}
\eqref{small ener for u}-\eqref{tau_blow_up for u} infer that \begin{align}\label{est of r_i for u}  r_i^{3.5} \leq 8 \h{1pt}\theta^{- 7.5}_0 \h{.5pt}\lambda_i^3 \longrightarrow 0 \h{15pt}\text{ as $i \to \infty$.} \end{align}

Assuming either $\big\{x_i\big\} \subset  \mathrm H$ or $\big\{x_i\big\} \subset  \mathrm P$, we introduce the blow-up sequence as follows:
\begin{equation}\label{defn blow up for u}
    \big(u_i, \phi_i, p_i\big)(x, t) :=   \left( \frac{r_i u}{\lambda_i}, \h{0.5pt} \frac{\phi}{\lambda_i}, \h{0.5pt} \frac{r_i^2 \big(\h{1pt}p - [\h{0.5pt}p\h{0.5pt}]_{x_i, r_i}\big)}{\lambda_i}\right)\big(x_i + r_i\h{0.5pt} x, t_i + r_i^2 \h{0.5pt} t\big) \h{15pt}\text{for $(x, t) \in Q_1$.}
\end{equation}In light of \eqref{sEL}, $(u_i, \phi_i, p_i)$ in \eqref{defn blow up for u} solves
\begin{equation} \label{rescaled system of u} \left\{ \begin{aligned}
    \partial_t u_i + \lambda_i \h{1pt}u_i \cdot \nabla u_i - \Delta u_i &= - \nabla p_i - \lambda_i \h{1pt}\nabla \cdot \big(\nabla \phi_i \odot \nabla  \phi_i\big), \\[1mm]
    \operatorname{div} u_i &= 0,  
\end{aligned} \right. \h{25pt}\text{on $Q_1$.}\end{equation}Moreover,
\begin{align} \label{rescaled L3 estimate for u} 
    &\mathrm{(1).} \h{5pt}\int_{Q_{1}} \big|\hspace{.5pt} u_i \hspace{.5pt}\big|^3    + \left(\int_{Q_{1}} \big| \hspace{.5pt} p_i \hspace{.5pt} \big|^{\frac{3}{2}} \right)^2 = 1, \nonumber\\[2mm]
    &\mathrm{(2).} \h{6pt}\theta_0^{-2} \int_{Q_{\theta_0}} \big| \hspace{.5pt} u_i \hspace{.5pt} \big|^3  + \left( \theta_0^{-2} \int_{Q_{\theta_0}} \big|\hspace{.5pt} p_i - [\h{0.5pt}p_i\h{0.5pt}]^*_{0, \theta_0} \hspace{.5pt}\big|^{\frac{3}{2}} \right)^2 > \theta_0^{3.5} \max \Big\{ r_i^{3.5} \h{1pt}\lambda_i^{-3}, 1 \Big\}.
  \end{align}From the first equation in \eqref{rescaled L3 estimate for u}, we can assume, after passing to a subsequence, that
\begin{equation*}
    \big(u_i, p_i\big) \rightharpoonup \big(u_*, p_*\big) \quad\quad \text{weakly in } L^3(Q_1) \times L^{\frac{3}{2}}(Q_1).
\end{equation*}

\noindent \textbf{Step 2. Uniform energy estimate and strong $L^3$-convergence of $\big\{u_i\big\}$.}\vspace{0.4pc}

We replace the test function in \eqref{energy of u} with $\varphi_i^2$, where $\varphi_i$ is given in Step 2 of the proof of Lemma \ref{decay lem}. Fix $t \in [ - \frac{1}{4}, 0 \h{0.5pt}]$ and integrate the time variable from $t_i - r_i^2$ to $t_i + r_i^2 t$. It turns out 
\begin{align*}
     \int_{\Omega\times \{\h{0.5pt}t_i + r_i^2 t\h{0.5pt}\}} \varphi_i^2 \h{0.5pt} \big|\h{.5pt} u \h{.5pt}\big|^2 &+ 2 \int_0^{\h{0.5pt}t_i + r_i^2 t\h{0.5pt}} \int_\Omega \varphi_i^2\h{0.5pt} \big|\h{.5pt} \nabla u \h{.5pt}\big|^2 = 2 \int_0^{\h{0.5pt}t_i + r_i^2 t\h{0.5pt}} \int_\Omega \varphi_i^2 \h{1.5pt} \nabla u : \big( \nabla \phi \odot \nabla \phi \big) \\[1.5mm]
    &+ \int_0^{\h{0.5pt}t_i + r_i^2 t\h{0.5pt}} \int_\Omega \left( u \cdot \nabla \varphi_i^2\right) \left( 2 p +  |u|^2 \right) + 2 \left(u \cdot \nabla \phi\right) \nabla \phi \cdot \nabla \varphi_i^2 +  \big|\h{.5pt} u \h{.5pt}\big|^2  \left( \partial_t \varphi_i^2 + \Delta \varphi_i^2\right).  \nonumber
\end{align*}

\noindent Using the boundedness of $u$ and $\nabla \phi$ in \eqref{unif bounds of u nab phi}, Young's inequality, H\"{o}lder's inequality, and  the boundedness of $\varphi_i$ and its derivatives, we conclude that
\begin{align*}
        \int_{\Omega\times \{\h{0.5pt}t_i + r_i^2 t\h{0.5pt}\}} \varphi_i^2 \h{0.5pt} \big|\h{.5pt} u \h{.5pt}\big|^2 &+  \int_0^{\h{0.5pt}t_i + r_i^2 t\h{0.5pt}} \int_\Omega \varphi_i^2\h{0.5pt} \big|\h{.5pt} \nabla u \h{.5pt}\big|^2  \\[1mm]
       &\lesssim_{K} r_i^4  + r_i^{-2}\int_{P_{r_i}(z_i)} \big|\h{.5pt} u \h{.5pt}\big|^2  +   r_i^{-1} \left(\int_{P_{r_i}(z_i)} \big|\h{.5pt} u \h{.5pt}\big|^3 \right)^{\frac{1}{3}} \left( \int_{P_{r_i}(z_i)} \big|\h{1pt} p - [\h{0.5pt}p\h{0.5pt}]_{x_i, r_i} \big|^{\frac{3}{2}} \right)^{\frac{2}{3}}.
\end{align*}

\noindent Apply the change of variables and then take supreme over $t \in [-\frac{1}{4}, 0\h{0.5pt}]$. We arrive at \begin{align*}
     \sup_{t \h{1pt}\in \h{1pt}[- \frac{1}{4}, \h{1pt} 0\h{0.5pt}]}\int_{B_{1/2}^{\pm}\times \{\h{0.5pt}t\h{0.5pt}\}}  \big|\h{.5pt} u_i \h{.5pt}\big|^2 +     \int_{Q_{1/2}}  \big|\h{.5pt} \nabla u_i \h{.5pt}\big|^2 \lesssim_{K}  r_i^3  \h{0.5pt}\lambda_i^{-2}  +    \int_{Q_1} \big|\h{.5pt} u_i \h{.5pt}\big|^2  +     \left(\int_{Q_1} \big|\h{.5pt} u_i \h{.5pt}\big|^3 \right)^{\frac{1}{3}} \left( \int_{Q_1} \big|\h{.5pt} p_i\h{0.5pt} \big|^{\frac{3}{2}} \right)^{\frac{2}{3}}.
\end{align*}
Utilizing \eqref{est of r_i for u} and (1) in \eqref{rescaled L3 estimate for u}, we obtain \begin{align}\label{energy esti for ui}
     \sup_{t \h{1pt}\in \h{1pt}[- \frac{1}{4}, \h{1pt} 0\h{0.5pt}]}\int_{B_{1/2}^{\pm}\times \{\h{0.5pt}t\h{0.5pt}\}}  \big|\h{.5pt} u_i \h{.5pt}\big|^2 +     \int_{Q_{1/2}}  \big|\h{.5pt} \nabla u_i \h{.5pt}\big|^2 \lesssim_{K} 1  \h{20pt}\text{for large $i$.}
\end{align} We then can keep extracting a subsequence, which is still denoted by $\big\{u_i\big\}$, such that
\begin{equation}\label{wk conv of ui}
    u_i \hspace{1.5pt}\rightharpoonup\hspace{1.5pt} u_* \quad\quad \text{weakly in } L_t^{2} \hspace{1pt}H_x^1 \big(Q_{1/2} \big).
\end{equation}

Suppose $\eta$ is a smooth vector field compactly supported in $B^{\pm}_{1/2}$. The bracket $\big<\cdot, \cdot \big>$ is the duality between $W_0^{1, 3}\big(B^{\pm}_{1/2}; \mathbb R^3\big)$ and its dual space. Using \eqref{defn blow up for u}, we have \begin{align*}
    \big<\partial_t u_i, \eta\big> = \lambda_i \int_{B^{\pm}_{1/2}} \big(u_i \odot u_i \big) : \nabla \eta  - \int_{B^{\pm}_{1/2}} \nabla u_i : \nabla \eta +  \int_{B^{\pm}_{1/2}} p_i \h{1pt}\mathrm{div} \eta + \lambda_i \int_{B^{\pm}_{1/2}}  \big(\nabla \phi_i \odot \nabla  \phi_i\big) : \nabla \eta.
\end{align*}In light of the fact that \begin{align}\label{unif est ui and nabla phii}
\|\h{0.5pt} u_i \h{0.5pt}\|_{L^\infty(\h{0.5pt}Q_{1/2} \h{0.5pt})} + \|\h{0.5pt} \nabla \phi_i \h{0.5pt}\|_{L^\infty(\h{0.5pt}Q_{1/2})}  \leq  K \h{0.5pt}r_i \h{0.5pt}\lambda_i^{-1},
\end{align} it holds \begin{align*}
    \big<\partial_t u_i, \eta\big> \lesssim_{K}  r_i^2 \h{1pt}\lambda_i^{-1} \h{1pt} \big\|\h{0.5pt} \nabla \eta \h{0.5pt}\big\|_{L^3(B^{\pm}_{1/2})}  + \big\|\h{0.5pt}\nabla u_i \h{0.5pt}\big\|_{L^{\frac{3}{2}}(B^{\pm}_{1/2})}  \h{1pt}\big\| \h{0.5pt}\nabla \eta \h{0.5pt}\big\|_{L^3(B^{\pm}_{1/2})} +  \big\|\h{0.5pt}p_i \h{0.5pt}\big\|_{L^{\frac{3}{2}}(B^{\pm}_{1/2})} \h{1pt}\big\| \h{0.5pt} \nabla \eta\h{0.5pt}\big\|_{L^3(B^{\pm}_{1/2})}.
\end{align*}Take supreme over all $\eta$ with $ \|\h{0.5pt} \eta \h{0.5pt}\|_{W^{1,3}(B^{\pm}_{1/2})} \leq 1$ and integrate the $t$-variable from $- \frac{1}{4}$ to $0$. It follows \begin{align*}
    \int_{- 1/4}^0\big\| \h{0.5pt}\partial_t u_i \h{0.5pt}\big\|^{\frac{3}{2}}_{W^{-1, \frac{3}{2}}(B^{\pm}_{1/2})} \lesssim_{K}  r_i^3 \h{1pt}\lambda^{- \frac{3}{2}}_i    + \int_{- 1/4}^0\big\|\h{0.5pt}\nabla u_i \h{0.5pt}\big\|^{\frac{3}{2}}_{L^{\frac{3}{2}}(B^{\pm}_{1/2})}  +  \int_{-1/4}^0\big\|\h{0.5pt}p_i \h{0.5pt}\big\|^{\frac{3}{2}}_{L^{\frac{3}{2}}(B^{\pm}_{1/2})} \lesssim_K 1 \h{10pt}\text{for large $i$.}
\end{align*} Here, we also use \eqref{est of r_i for u}, \eqref{energy esti for ui}, and (1) in \eqref{rescaled L3 estimate for u}.

On the other hand, we can obtain from \eqref{energy esti for ui} and Proposition 3.2 in the Chapter 1 of \cite{Dibenedetto1993} that \begin{align*} 
    \| u_i \|_{L^{\frac{10}{3}}(Q_{1/2})}   \hspace{1.5pt}\lesssim_K\hspace{1.5pt} 1 \h{20pt}\text{for all $i$.}
\end{align*}Then, we use the Aubin-Lions compactness lemma (see \cite{S1986}) and get \begin{align}\label{str L3 ui} u_i \longrightarrow u_* \h{10pt}\text{strongly in $L^3(Q_{1/2})$.}
\end{align}

\noindent\textbf{Step 3. Uniform decay estimate of $p_i$.}\vspace{0.4pc}

We first consider an $L^2$-estimate of $\nabla^2 \phi_i$. Recall \eqref{energy ineq}. It turns out \begin{align*}
    &   \int_{P_{r_i}(z_i)} \varphi_i^2 \h{0.5pt} \big|\h{0.5pt}\nabla^2 \phi \h{0.5pt}\big|^2 \lesssim_K  r_i^2 + r_i^{-1} \left(\int_{P_{r_i}(z_i)} \big|\h{.5pt} u \h{.5pt}\big|^3 \right)^{\frac{1}{3}} \left( \int_{P_{r_i}(z_i)} \big|\h{1pt} p - [\h{0.5pt}p\h{0.5pt}]_{x_i, r_i} \big|^{\frac{3}{2}} \right)^{\frac{2}{3}}.
\end{align*}
Here, we use the boundedness of $\varphi_i$ and its derivatives. \eqref{unif bounds of u nab phi} is also used to control the $L^\infty$-norms of $u$ and $\nabla \phi$. Apply the change of variables and (1) in \eqref{rescaled L3 estimate for u}. The last estimate is reduced to \begin{align}\label{Hession of phii}
 \int_{Q_{1/2}} \big|\h{0.5pt}\nabla^2 \phi_i \h{0.5pt}\big|^2 \leq \int_{Q_1} \varphi^2 \h{0.5pt} \big|\h{0.5pt}\nabla^2 \phi_i \h{0.5pt}\big|^2 \lesssim_K  r_i \h{1pt}\lambda_i^{-2}  +  \left(\int_{Q_1} \big|\h{.5pt} u_i \h{.5pt}\big|^3 \right)^{\frac{1}{3}} \left( \int_{Q_1} \big|\h{1pt} p_i \big|^{\frac{3}{2}} \right)^{\frac{2}{3}} \leq r_i \h{1pt}\lambda_i^{-2}  + 1.
\end{align}

Same as Part III of Step 3 in the proof Lemma \ref{decay lem}, we decompose $(u_i, p_i)$ into $$u_i = u_i^{(1)} + u_i^{(2)} \h{15pt}\text{and} \h{15pt} p_i = p_i^{(1)} + p_i^{(2)},$$ where $\big(u_i^{(1)}, p_i^{(1)}\big)$ satisfy the initial boundary value problem:
\begin{equation*} \left\{ \begin{aligned}
    \partial_t u_i^{(1)} - \Delta u_i^{(1)} + \nabla p_i^{(1)} &= - \lambda_i \h{1pt}u_i \cdot \nabla u_i - \lambda_i \h{1pt}\nabla \cdot \big(\nabla \phi_i \odot \nabla  \phi_i\big) \hspace{.5pt} \quad\quad \text{in } Q_{1/2}, \\[1mm]
    \operatorname{div} u_i^{(1)} &= 0 \quad\hspace{160.5pt} \text{in } Q_{1/2}, \\[1mm]
    u_i^{(1)} &= 0 \quad\hspace{160.5pt} \text{on } \mathscr P Q_{1/2}. 
\end{aligned} \right. \end{equation*}Still by Theorem 1.1 in \cite{Solonnikov2003}, 
\begin{equation*} 
    \big\| \h{1pt} u_i^{(1)} \big\|_{W^{2,1}_{\frac{9}{8}, \frac{3}{2}}(Q_{1/2})} + \big\|\h{1pt} \nabla p_i^{(1)} \big\|_{\frac{9}{8}, \frac{3}{2},\, Q_{1/2}} \hspace{1.5pt}\lesssim_{K_*}\hspace{1.5pt}  \lambda_i \h{0.5pt}\big\| \h{1pt} u_i \cdot \nabla  u_i + \nabla \cdot \big(\nabla \phi_i \odot \nabla  \phi_i\big) \h{1pt}\big\|_{\frac{9}{8}, \frac{3}{2}, \,Q_{1/2}}.
\end{equation*} Using \eqref{unif est ui and nabla phii}, \eqref{energy esti for ui}, and \eqref{Hession of phii} to control the right-hand side above, we get \begin{align}\label{est ui1}
\big\| \h{1pt} u_i^{(1)} \big\|_{W^{2,1}_{\frac{9}{8}, \frac{3}{2}}(Q_{1/2})} + \big\|\h{1pt} \nabla p_i^{(1)} \big\|_{\frac{9}{8}, \frac{3}{2},\, Q_{1/2}} &\hspace{1.5pt}\lesssim_{K}\hspace{1.5pt}r_i \h{0.5pt}\big\| \h{1pt} \nabla  u_i \h{0.5pt}\big\|_{2, \,2, \,Q_{1/2}} + r_i \big\|\h{0.5pt}\nabla^2 \phi_i \h{1pt}\big\|_{2, \,2, \,Q_{1/2}} \nonumber\\[1mm]
&\h{1.5pt}\lesssim_K\h{1.5pt} r_i + r_i^{\frac{3}{2}} \h{0.5pt}\lambda_i^{-1} \longrightarrow 0 \h{20pt}\text{as $i \to \infty$.}\end{align}

The pressure $p_i^{(2)}$ can be estimated in the same way as in the proof of Lemma \ref{decay lem}. Using triangle inequality, H\"{o}lder inequality, \eqref{energy esti for ui} and \eqref{est ui1}, we obtain \begin{align*} \big\|\h{1pt} u^{(2)}_i \h{1pt}\big\|_{W^{1, 0}_{\frac{9}{8}, \frac{3}{2}}(Q_{1/2})} \leq \big\|\h{1pt} u^{(1)}_i \h{1pt}\big\|_{W^{1, 0}_{\frac{9}{8}, \frac{3}{2}}(Q_{1/2})} + \big\|\h{1pt} u_i \h{1pt}\big\|_{W^{1, 0}_{2, 2} (Q_{1/2})} \leq K.\end{align*} If we assume $p^{(1)}_i$ has $0$ average over $B^{\pm}_{1/2}$, then  by triangle inequality and Poincar\'{e} inequality, \begin{align*} \big\|\h{1pt} p^{(2)}_i \h{1pt}\big\|_{\frac{9}{8}, \frac{3}{2}, \, Q_{1/2}} \leq \big\|\h{1pt} \nabla p^{(1)}_i \h{1pt}\big\|_{\frac{9}{8}, \frac{3}{2}, \,Q_{1/2}} + \big\|\h{1pt} p_i \h{1pt}\big\|_{\frac{3}{2}, \frac{3}{2}, \,Q_{1/2}} \leq K.\end{align*}
Applying the last two estimates to \eqref{basic est of pI2}, we get $\big\|\h{1pt} \nabla p_i^{(2)} \h{1pt}\big\|_{9, \frac{3}{2}, \, Q_{1/4}}\hspace{1.5pt}\leq\hspace{1.5pt} K$, which together with the estimate of $\nabla p_i^{(1)}$ in \eqref{est ui1} infers
\begin{align*}
    \int_{Q_{\theta_0}} \big|\hspace{.8pt} p_i - [\h{.5pt} p_i \h{.5pt}]^*_{0, \theta_0} \hspace{.5pt}\big|^{\frac{3}{2}} &\hspace{1.5pt}\lesssim_{K_*}\hspace{1.5pt}   \theta_0^{\frac{1}{2}} \int_{- \theta_0^2}^0 \left(\int_{B^{\pm}_{\theta_0}} \big|\hspace{.5pt} \nabla p_i^{(1)} \hspace{.5pt} \big|^{\frac{9}{8}} \right)^{\frac{4}{3}} +  \theta_0^{4} \int_{- \theta_0^2}^0 \left(\int_{B^{\pm}_{\theta_0}} \big|\hspace{.5pt} \nabla p_i^{(2)} \hspace{.5pt}\big|^{9} \right)^{\frac{1}{6}}\\[1mm]
    &\h{1.5pt}\lesssim_K \h{1.5pt}\theta_0^{\frac{1}{2}} \left(r_i + r_i^{\frac{3}{2}} \h{0.5pt}\lambda_i^{-1}\right)^{\frac{3}{2}} + \theta_0^4.
\end{align*}
Here, the Sobolev, H\"{o}lder, and Poincar\'{e} inequality are also used. Hence, 
\begin{equation}\label{limsup pre}
    \limsup_{i\h{1pt}\to\h{1pt}\infty} \int_{Q_{\theta_0}} \big|\hspace{.8pt} p_i - [\h{.5pt} p_i \h{.5pt}]^*_{0, \theta_0} \hspace{.5pt}\big|^{\frac{3}{2}} \lesssim_{K} \h{0.5pt}\theta_0^4.
\end{equation}

\noindent \textbf{Step 4.} Using \eqref{str L3 ui} and \eqref{limsup pre}, we take $i \to \infty$ in (2) of \eqref{rescaled L3 estimate for u} and arrive at  \begin{align}\label{contra for u}
 \theta_0^{3.5} \lesssim_{K}  \h{0.5pt}   \theta_0^4 + \theta_0^{-2} \int_{Q_{\theta_0}} \big| \hspace{.5pt} u_* \hspace{.5pt} \big|^3.
\end{align}Recall the second estimate in \eqref{est ui1}. We can take $i \to \infty$ in \eqref{rescaled system of u}. The limit $(u_*, p_*)$ solves the following linear equations:
\begin{align*}\partial_t u_* - \Delta u_* = - \nabla p_*, \hspace{20pt} \operatorname{div} u_* = 0 \quad\quad \text{in } Q_{1/2} \qquad \text{with } u_* = 0 \text{ on } B'_{1/2} \times (- 1/4, 0). 
\end{align*}
Note that $u_*$ can be estimated in the same way as in \eqref{decay_1} with the constant depending on $K$. Hence, \eqref{contra for u} can be reduced to $\theta_0^{3.5} \lesssim_K  \h{0.5pt}   \theta_0^4$, which is impossible if we take $\theta_0$ suitably small with the smallness depending on $K$.  
\end{proof}

Iteratively applying Lemma \ref{decay lem u only}, we have
\begin{lem} \label{iter Holder}
Fix $r \in (0, 1)$ and $z_0 = (x_0, t_0) \in \big(\mathrm {H \cup P}\big) \times \big(T_\star + 10, \infty\big)$. Assume that
\begin{equation}\label{small assumption for holder}
    H(r, z_0) := \max \left\{r^{3.5}, \h{1pt}  r^{-2} \int_{P_r(z_0)} \big| \hspace{.5pt} u  \hspace{.5pt}\big|^3 + \left(r^{-2} \int_{P_r(z_0)} \big|\hspace{.5pt} p - [\h{0.5pt}p\h{0.5pt}]_{x_0, r} \hspace{.5pt} \big|^{\frac{3}{2}} \right)^2 \right\} \hspace{1.5pt}\leq\hspace{1.5pt} \epsilon_0^3,
\end{equation}where $\epsilon_0$ is given in Lemma \ref{decay lem u only}. Then, \begin{align*}
    &H\big(\theta_0^k \h{0.5pt}r, z_0\big)  \leq  \theta_0^{3.5 \h{0.5pt} k} H(r, z_0) \h{20pt}\text{for any $k \in \mathbb N \cup \{ 0 \}$.}
\end{align*}
\end{lem}

A direct corollary of Lemma \ref{iter Holder} is read as follows.
  \begin{prop} \label{Holder regularity of u only}
Assume the same $r$ and $z_0$ as in Lemma \ref{iter Holder}. Then we have  
    \begin{equation*} \begin{aligned}
    \rho^{-5.5} \int_{P_{\rho}(z_0)} \big|\hspace{.5pt} u \hspace{.5pt}\big|^3  \hspace{1.5pt}\lesssim_{K}\hspace{1.5pt}  \frac{H(r, z_0)}{r^{3.5}} \h{25pt}\text{for any $\rho \in (\h{0.5pt}0, r\h{0.5pt}]$.}
\end{aligned} \end{equation*}
\end{prop}  

We now finish the proof of Theorem \ref{partial regularity}.
\begin{proof}[\bf Proof of \eqref{decay of sup u} in Theorem \ref{partial regularity}]
Recall that $z_0 = (x_0, t_0)$. Slightly modifying the proof of Proposition \ref{Holder regularity of u only} (see also \cite{LinLiu1996}), we can obtain a similar estimate as in Proposition \ref{Holder regularity of u only} for $u$ at the interior point $x_0 \in \Omega$. By Lemma 4.3 in \cite{Liebmann}, $u(t, \cdot)$ is uniformly bounded in $C^{\frac{1}{6}}(\overline{\Omega})$ for large $t$. Let $\{t_n\}$ be a sequence diverging to $\infty$ as $n \to \infty$. We can apply Arzel\`{a}-Ascoli theorem to extract a subsequence, which is still denoted by $\{t_n\}$, such that $u(t_n, \cdot)$ converges to some $u_*$ uniformly in $L^\infty(\Omega)$ as $n \to \infty$. In light of \eqref{H1L2limit} in Lemma \ref{limit and unif bound}, the limit $u_*$ must be identically $0$. \eqref{decay of sup u} then follows.
\end{proof}

\section{P-HAN transition along the classical hydrodynamic flow}

We study the P-HAN transition along a classical flow in this section. Our main result is \begin{prop}\label{P-Han along classical}
    Assume that $(u, \phi)$ is a global classical solution of $\mathrm{IBVP}$ on $\big[\hspace{0.5pt}T_0, \infty)$, where $T_0 > 0$ is a positive time. If we keep assuming that \begin{align}
\label{sgn of phi}
0 \leq \phi \leq \pi \hspace{15pt}\text{and}\hspace{15pt} \phi \not\equiv 0 \hspace{20pt}\text{on $\Omega \times \big\{\hspace{0.5pt}T_0\hspace{0.5pt}\big\}$},
\end{align}then the asymptotic limit $(0, \phi_\infty)$ of the solution $(u, \phi)$ can be determined as follows: \begin{align*}
    (0, \phi_\infty) = (0, 0) \hspace{15pt}\text{if $d \leq d_c$;} \hspace{20pt}(0, \phi_\infty) = (0, \phi_*) \hspace{15pt}\text{if $d > d_c$.}
\end{align*}When $d > d_c$, the limit $\phi_*$ is the unique non-negative global minimizer of $E$ in $H_\mathrm{P}^1(\Omega)$. Furthermore,\begin{itemize}
    \item[$\mathrm{(1).}$] If $d = d_c$, $(u, \phi)$ satisfies the algebraic decay as in (1) of Proposition \ref{convergence rate}. \vspace{0.3pc}
    \item[$\mathrm{(2).}$] If $d \neq d_c$, $(u, \phi)$ satisfies the exponential convergence as in (2) of Proposition \ref{convergence rate}. 
\end{itemize}  
\end{prop}

\begin{proof}[\bf Proof] The proof is divided into three steps. \vspace{0.2pc}

\noindent \textbf{Step 1.} In this step, we prove $\phi_\infty > 0$ on $\Omega \cup \text{H}$ if $d > d_c$.\vspace{0.2pc} 

Notice that $\phi$ satisfies \eqref{sgn of phi}. By (3) in Lemma \ref{max principle}, $0 < \phi < \pi$ on $\Omega \times (T_0, \infty)$. Therefore, $0 \leq \phi_\infty \leq \pi$ on $\Omega$. Since $(0, \phi_\infty)$ is a stationary solution of IBVP, then either $\phi_\infty \equiv 0$ on $\Omega$, or $\phi_\infty > 0$ on $\Omega$. Here we still use (3) in Lemma \ref{max principle}. If $\phi_\infty > 0$ on $\Omega$, then we claim that $\phi_\infty > 0$ on H. Otherwise, there is $x_* \in \text{H}$, so that $\phi_\infty\left(x_*\right) = 0$. It then holds that $\phi_\infty \in \big(0, \frac{\pi}{4}\big)$ on the upper-half ball $B^+_{\rho_0}(x_*)$, for some $\rho_0 > 0$ suitably small. Hence, $- \Delta \phi_\infty \geq 0$ on $B^+_{\rho_0}\left(x_*\right)$. Applying Hopf's lemma induces that $\partial_3 \phi_\infty \left(x_*\right) > 0$. However, by the boundary condition of $\phi_\infty$ on H, it turns out $\partial_3 \phi_\infty \left(x_* \right) = 0$. We get a contradiction. Therefore, if $\phi_\infty > 0$ on $\Omega$, then $\phi_\infty > 0$ on $\Omega \cup \text{H}$.

To complete the first step, we are left to show that  $\phi_\infty \not \equiv 0$ on $\Omega$ if $d > d_c$. According to Lemma \ref{mon of lambda1} and the fact that $\lambda_1(d_c) = 1$, we have $\lambda = \lambda_1(d) < 1$ when $d > d_c$. There is a constant $\epsilon \in (0, 1)$ suitably small, so that \begin{align}\label{parameter ineq}
    \lambda_1^2 \hspace{1pt}< \hspace{1pt} \frac{1 - \epsilon}{1 + \epsilon}.
\end{align}Suppose to the contrary that $\phi_{\infty} \equiv 0$. Then by Proposition \ref{convergence rate} and Morrey's inequality, $\|\hspace{0.5pt}\phi(t) \hspace{0.5pt}\|_{L^\infty}$ converges to $0$ as $t \to \infty$. Therefore, there is a time  $T_1 > T_0$, so that
\begin{align}\label{low bd of sin}
    \sin  2 \hspace{0.5pt}\phi  \hspace{1pt}\geq \hspace{1pt} 2 \left(1 - \epsilon\right) \phi \hspace{1pt}>\hspace{1pt} 0 \hspace{20pt} \text{on } \Omega \times (T_1, \infty).
\end{align}Since $d > d_c$, we have a non-negative and non-trivial eigenfunction, denoted by $\phi_1$, associated with the principal eigenvalue $\mathrm{R}^{3 \mathrm{D}} = \lambda_1^2$ in \eqref{pos of evalue}. $\phi_1$ is strictly positive on $\Omega \cup \text{H}$ and satisfies the boundary value problem \eqref{Seignprob}. Simply denoting by $\alpha$ the constant $1 + \epsilon$, we multiply $\phi_1^{\alpha}$ on the both sides of the equation of $\phi$ in \eqref{sEL}. Through the integration by parts, we obtain
\begin{align*}
    \frac{\mathrm{d}}{\mathrm{d}\hspace{0.3pt}t} \int_{\Omega} \phi \hspace{0.8pt} \phi_1^{\alpha} = \int_{\Omega} \phi \hspace{0.8pt}\Delta \phi_1^{\alpha} + \phi \hspace{0.8pt}u \cdot \nabla  \phi_1^{\alpha} + \frac{h^2}{2}\hspace{0.5pt} \phi_1^{\alpha} \sin 2 \phi  - \int_\text{H}   \alpha \hspace{0.5pt} L_\text{H} \hspace{0.5pt} \lambda_1^2 \hspace{0.5pt} \phi \hspace{0.8pt}\phi_1^{\alpha} - \frac{L_\text{H}}{2} \hspace{0.5pt} \phi_1^{\alpha} \sin 2 \phi.
\end{align*}
Now we claim that for some $T_2 > T_1$, it satisfies
\begin{align}\label{Integral Monotonicity}
    \frac{\mathrm{d}}{\mathrm{d}\hspace{0.3pt}t} \int_{\Omega} \phi \hspace{0.8pt} \phi_1^{\alpha} \hspace{1pt}\geq\hspace{1pt} 0 \hspace{20pt} \text{for all $t > T_2$}.
\end{align}
In fact, according to \eqref{low bd of sin}, we have for all $t > T_1$ that
\begin{align*}
    &\int_{\Omega} \phi \hspace{0.8pt}\Delta \phi_1^{\alpha} + \phi \hspace{0.8pt}u \cdot \nabla  \phi_1^{\alpha} + \frac{h^2}{2}\hspace{0.5pt} \phi_1^{\alpha} \sin 2 \phi  - \int_\text{H}   \alpha \hspace{0.5pt} L_\text{H} \hspace{0.5pt} \lambda_1^2 \hspace{0.5pt} \phi \hspace{0.8pt}\phi_1^{\alpha} - \frac{L_\text{H}}{2} \hspace{0.5pt} \phi_1^{\alpha} \sin 2 \phi \\[2mm]
    & \hspace{20pt}\geq\hspace{1pt} \int_{\Omega} \phi \hspace{1pt}\Big\{ \hspace{1pt} \Delta \phi_1^{\alpha} - \|\hspace{0.5pt} u \hspace{0.5pt}\|_{L^{\infty}} \big|\hspace{0.5pt}\nabla \phi_1^{\alpha}\hspace{0.5pt}\big| + h^2 \left(1 - \epsilon\right) \phi_1^{\alpha}  \hspace{1pt}\Big\} + L_\text{H} \int_\text{H}  \phi \hspace{0.8pt}\phi_1^{\alpha} \hspace{1pt} \Big\{  \left(1 - \epsilon\right) - \alpha \hspace{0.5pt} \lambda_1^2 \hspace{1pt}\Big\}.
\end{align*}In light of \eqref{parameter ineq}, the integral on H on the right-hand side above is non-negative. Thus, \begin{align}\label{low bd of rhd}
\frac{\mathrm{d}}{\mathrm{d}\hspace{0.3pt}t} \int_{\Omega} \phi \hspace{0.8pt} \phi_1^{\alpha} \hspace{1pt}\geq\hspace{1pt}\int_{\Omega} \phi \hspace{1pt}\Big\{ \hspace{1pt} \Delta \phi_1^{\alpha} - \|\hspace{0.5pt} u \hspace{0.5pt}\|_{L^{\infty}} \big|\hspace{0.5pt}\nabla \phi_1^{\alpha}\hspace{0.5pt}\big| + h^2 \left(1 - \epsilon\right) \phi_1^{\alpha}  \hspace{1pt}\Big\}.
\end{align}
Direct computations show that
\begin{align*}
    \nabla \phi_1^{\alpha} = \alpha \hspace{0.5pt} \phi_1^{\alpha - 1} \hspace{0.5pt} \nabla \phi_1 \hspace{15pt}\text{ and } \hspace{15pt} \Delta \phi_1^{\alpha} = \alpha\hspace{0.5pt} \phi_1^{\alpha - 2} \hspace{0.5pt} \Big\{ \hspace{1pt} \left(\alpha - 1\right) \big|\hspace{0.5pt}\nabla \phi_1\hspace{0.2pt}\big|^2 - h^2 \hspace{0.5pt}\lambda_1^2 \hspace{0.8pt}\phi_1^2 \hspace{1pt}\Big\}.
\end{align*}
Applying the Young's inequality, it then follows that
\begin{align}\label{low rhd 2}
    &\int_{\Omega} \phi \hspace{1pt}\Big\{ \hspace{1pt} \Delta \phi_1^{\alpha} - \|\hspace{0.5pt} u \hspace{0.5pt}\|_{L^{\infty}} \big|\hspace{0.5pt}\nabla \phi_1^{\alpha}\hspace{0.5pt}\big| + h^2 \left(1 - \epsilon\right) \phi_1^{\alpha}  \hspace{1pt}\Big\} \nonumber\\[2mm]
    & \hspace{20pt}\geq \hspace{1pt}  \alpha \int_{\Omega} \phi \hspace{0.8pt} \phi_1^{\alpha - 2} \left\{ \hspace{1.5pt} \Big[\hspace{1pt}\alpha - 1 - \frac{1}{2} \hspace{0.5pt}\|\hspace{0.5pt} u \hspace{0.5pt}\|_{L^{\infty}} \Big] \hspace{1.5pt} \big|\hspace{0.5pt}\nabla \phi_1\hspace{0.2pt}\big|^2 + \Big[\hspace{1.5pt}h^2 \left(\frac{1 - \epsilon}{\alpha} - \lambda_1^2\right) - \frac{1}{2} \hspace{0.5pt}\|\hspace{0.5pt} u \hspace{0.5pt}\|_{L^{\infty}} \Big] \hspace{1.5pt} \phi_1^2 \hspace{1.5pt} \right\}.
\end{align}
Now we choose $T_2 > T_1$ suitably large, so that  
\begin{align*}
    \sup_{t \hspace{.5pt}>\hspace{.5pt} T_2}\hspace{0.5pt} \big\|\hspace{0.5pt} u(t) \hspace{0.5pt}\big\|_{L^{\infty}} \hspace{1pt}\leq\hspace{1pt}   \min \left\{ \alpha - 1, \hspace{1pt}h^2 \left(\frac{1 - \epsilon}{\alpha} - \lambda_1^2\right)\hspace{1pt} \right\}.
\end{align*}
Therefore, \eqref{Integral Monotonicity} holds by \eqref{low bd of rhd}, \eqref{low rhd 2} and the last estimate of $u$. \eqref{Integral Monotonicity} further yields that
\begin{align*}
    \int_{\Omega \times \{\hspace{0.5pt} t \hspace{0.5pt}\}} \phi \hspace{0.8pt} \phi_1^{\alpha} \hspace{1pt}\leq\hspace{1pt} \lim_{s \hspace{0.5pt}\to \hspace{0.5pt} \infty} \int_{\Omega \times \{\hspace{0.5pt}s\hspace{0.5pt}\}} \phi \hspace{0.8pt} \phi_1^{\alpha} \hspace{1pt} = \hspace{1pt}0 \hspace{20pt} \text{for all $t > T_2$.}
\end{align*}
Note that $\phi_1$ is strictly positive on $\Omega \cup \text{H}$, and $\phi$ is strictly positive on $\Omega$. The left-hand side above must be strictly positive for all $t > T_2$. This is a contradiction to  the last estimate.\vspace{0.3pc}

\noindent \textbf{Step 2.} In this step, we show that $\phi_\infty < \frac{\pi}{2}$ on $\Omega \cup \text{H}$. \vspace{0.2pc}

We only need to prove \begin{align}\label{upp of phiinfty}
\max_{\overline{\Omega}} \phi_\infty \hspace{1pt}\leq \hspace{1pt}\pi\big/\hspace{0.5pt}2.
\end{align}Once the above estimate of $\phi_\infty$ holds, we can follow the similar arguments as in the proof of Lemma \ref{sign of min}, in particular the last paragraph in the proof there, to show that the inequality in \eqref{upp of phiinfty} is strict. Now we suppose on the contrary that \eqref{upp of phiinfty} fails. Then it satisfies $\phi_\infty(y_*) > \pi \big/ 2$, where $y_* \in \Omega \cup \text{H}$ is a maximum point of $\phi_\infty$. If $y_* \in \Omega$, then we have $\Delta \phi_\infty\left(y_*\right) \leq 0$. But by the equation of $\phi_\infty$ and the fact that $\phi_\infty < \pi$ on $\Omega$, it holds \begin{align*}
     \Delta \phi_\infty\left(y_*\right) = - \frac{h^2}{2} \sin 2 \hspace{0.5pt} \phi_\infty\left(y_*\right) > 0.
\end{align*}The maximum point of $\phi_\infty$ must lie on H. Note that $\phi_\infty(y_*) > \pi \big/ 2$. There is a suitably small radius, denoted by $r_0$, so that $\pi\big/2 < \phi_\infty < \pi$ on  $B^+_{r_0}(y_*)$. It turns out by the equation of $\phi_\infty$ that $\Delta \phi_\infty > 0$ on  $B^+_{r_0}(y_*)$. According to Hopf lemma, it follows that $- \partial_3 \phi_\infty\left(y_*\right) > 0$. However, this is impossible since by the boundary condition of $\phi_\infty$ on H,\begin{align*}
    - \partial_3 \phi_\infty\left(y_*\right) = \frac{L_\text{H}}{2} \sin 2\hspace{0.5pt} \phi_\infty\left(y_*\right) \hspace{1pt}\leq \hspace{1pt} 0.
\end{align*}We therefore prove the assertion in \eqref{upp of phiinfty}.\vspace{0.2pc}

\noindent\textbf{Step 3.} We complete the proof in this step. First, we determine the asymptotic limit $\phi_\infty$. If $d \leq d_c$, then by (1) in Proposition \ref{threshold of thickness}, we have $\phi_\infty \equiv 0$ on $\Omega$. If $d > d_c$, then by the results obtained from Steps 1 and 2 above, we have $\phi_\infty \in (0, \frac{\pi}{2})$ on $\Omega \cup \text{H}$. According to Lemma \ref{sign of min} and the uniqueness result in Lemma \ref{unique of positive bd solu}, $\phi_\infty = \phi_*$, where $\phi_*$ is the unique non-negative global minimizer of $E$ in $H_\text{P}^1(\Omega)$. If $d = d_c$, the algebraic decay rate in (1) of Proposition \ref{P-Han along classical} results from (1) in Proposition \ref{convergence rate}. If $d \neq d_c$, then by Corollary \ref{LS exp}, the \L ojasiewicz-Simon exponent associated with $\phi_\infty$ is equal to $\frac{1}{2}$. Here we use the fact that $\phi_\infty = 0$ if $d < d_c$, while $\phi_\infty = \phi_*$ if $d > d_c$. In either case, $\phi_\infty$ is the global minimizer of the energy $E$ in $H_\text{P}^1(\Omega)$. By (2) in Proposition \ref{convergence rate}, the exponential convergence rate in (2) of Proposition \ref{P-Han along classical} follows. \end{proof}

\section{P-HAN transition along the suitable weak solution}

In this last section, we prove Part (2) of Theorem \ref{thm1.2}. Since the suitable weak solution $(u, \phi)$ is classical after a long time, the proof can be obtained by Proposition \ref{P-Han along classical}, combined with the following two lemmas.  
\begin{lem}If $0 \leq \phi_0 \leq \pi$ and $\phi_0 \not\equiv 0$ on $\Omega$, then $\phi$ satisfies \eqref{sgn of phi} at any large time $T_0$. 
\end{lem} \noindent We omit the proof of this lemma. It can be obtained by an approximation argument, using the approximation sequence $\big\{\psi^{\delta_k}\big\}$ in \eqref{eq of psi_delta}, which satisfies \eqref{conver of psi delta}, and Lemma \ref{max principle}. \vspace{0.2pc}

In the end, we prove a non-vanishing result of $\phi$.
\begin{lem}
    If $\phi_0 \geq 0$ and $\phi_0 \not\equiv 0$ in $\Omega$, then $\phi\left(t, \cdot\right) \not \equiv 0$ in $\Omega$ for large $t$.
\end{lem}
\begin{proof}[\bf Proof]
Suppose $(u, \phi)$ is smooth on $\overline{\Omega}\times (T_0, \infty)$. If it holds $\phi\left(T_1, \cdot\right) \equiv 0 $ on $\Omega$ for some $T_1 > T_0$, then we show in the following that $\phi_0 \equiv 0$ on $\Omega$. Hence, we obtain a contradiction. \vspace{0.2pc}

Recall the approximation $\big\{\psi^{\delta_k}\big\}$ in \eqref{eq of psi_delta}. By the maximum principle shown in Lemma \ref{max principle}, it turns out $\psi^{\delta_k} \geq 0$ in $\Omega \times (0, \infty)$ for all $k$.  Therefore,
\begin{equation} \label{eq of psi^delta}
    \partial_t \big(e^{h^2 t} \psi^{\delta_k} \big) - \Delta \big( e^{h^2 t} \psi^{\delta_k} \big) + u_{\delta_k} \cdot \nabla  \big(e^{h^2 t} \psi^{\delta_k}\big) \hspace{1.5pt}=\hspace{1.5pt} \frac{h^2}{2} \h{1pt}e^{h^2t}\left(\sin 2 \psi^{\delta_k} + 2 \psi^{\delta_k} \right) \hspace{1.5pt}\geq\hspace{1.5pt} 0.
\end{equation}
For any $x \in \Omega$, we fix a  $r_* < 2^{-1}\operatorname{dist}\left(x, \partial \Omega\right)$ such that $ T_1 = 4 \h{0.5pt} m \h{0.5pt} r_*^2$ for some $m \in \mathbb{N}$. It follows from \eqref{conver of psi delta} that $\big\{\psi^{\delta_k}\left(T_1, \cdot \right)\big\}$ converges to $0$ strongly in $L^2(\Omega)$ as $k \to \infty$. Up to a subsequence, which is still denoted by $\big\{\psi^{\delta_k}\left(T_1, \cdot\right)\big\}$, it holds that $\big\{\psi^{\delta_k}\left(T_1, \cdot\right)\big\}$ converges to $0$ almost everywhere in $\Omega$ as $k \to \infty$. Therefore, for any $\epsilon > 0$, we can find a $K = K(\epsilon, T_1) \in \mathbb N$ such that 
\begin{align*}
    \inf \Big\{\h{1pt}e^{h^2 t} \psi^{\delta_k}(y, t)  :  (y, t) \h{1pt}\in\h{1pt} P_{r_*}(x,\h{1pt} T_1) \h{1.5pt} \Big\}\leq \epsilon \h{20pt}\text{for any $k \geq K$.}
\end{align*}

\noindent Applying the Harnack inequality due to Ignatov-Kukavica-Ryzhik (see Lemma 3.1 in \cite{IgnatovaKukavicaRyzhik2016}), we obtain for some small $p_0 > 0$ that
\begin{align*}
    \int_{P_{r_*}(x, \h{1pt}T_1 - 3 \h{0.5pt} r_*^2)} \left|\h{.5pt} e^{h^2 t} \psi^{\delta_k}(z, t) \h{.5pt}\right|^{p_0} \D z \h{1pt} \D t \h{1.5pt}\leq\h{1.5pt} C \h{1pt}\epsilon^{p_0} \h{15pt}\text{for all $k \geq K$.}
\end{align*}

\noindent Here, $p_0$ and $C$ are independent of $k$. In light of \eqref{conver of psi delta}, we now take $k \to \infty$ and $\epsilon \to 0$ successively in the above estimate. It turns out that $\phi = 0$ almost everywhere in $P_{r_*}(x, T_1 - 3\h{0.5pt} r_*^2)$. The trace lemma then yields that $\phi\left(T_1 - 4 \h{0.5pt} r_*^2, \cdot\right) = 0$ almost everywhere in $B_{r_*}(x)$. Repeatedly applying the above arguments by $m - 1$ more times, we get $\phi_0 \equiv 0$ in $B_{r_*}(x)$. Since $x$ is an arbitrary point in $\Omega$, it follows that $\phi_0 \equiv 0$ in $\Omega$. The proof is completed.
\end{proof} 

\bibliographystyle{plain}
\bibliography{references}

\end{document}